\documentclass[11pt]{article}
\usepackage[total={6in, 8in}]{geometry}

\usepackage{nicefrac}
\usepackage[authoryear]{natbib}%

\usepackage[colorlinks,linkcolor=red,citecolor=blue, pagebackref=true,backref=true]{hyperref}

\usepackage{amsmath,amssymb,amsthm}%
\usepackage{comment}
\usepackage{color}
\usepackage{setspace}
\usepackage{pdflscape}
\usepackage{enumitem}
\usepackage{algorithm}
\usepackage{algpseudocode}
\usepackage{dsfont}
\usepackage{svg}
\usepackage{pdfpages}
\usepackage[font=footnotesize]{caption} %

\usepackage[english]{babel}
\usepackage{macros}

\usepackage{setspace}

\usepackage{amsmath}
\makeatletter
\let\save@mathaccent\mathaccent
\newcommand*\if@single[3]{%
  \setbox0\hbox{${\mathaccent"0362{#1}}^H$}%
  \setbox2\hbox{${\mathaccent"0362{\kern0pt#1}}^H$}%
  \ifdim\ht0=\ht2 #3\else #2\fi
  }
\newcommand*\rel@kern[1]{\kern#1\dimexpr\macc@kerna}
\newcommand*\widebar[1]{\@ifnextchar^{{\wide@bar{#1}{0}}}{\wide@bar{#1}{1}}}
\newcommand*\wide@bar[2]{\if@single{#1}{\wide@bar@{#1}{#2}{1}}{\wide@bar@{#1}{#2}{2}}}
\newcommand*\wide@bar@[3]{%
  \begingroup
  \def\mathaccent##1##2{%
    \let\mathaccent\save@mathaccent
    \if#32 \let\macc@nucleus\first@char \fi
    \setbox\z@\hbox{$\macc@style{\macc@nucleus}_{}$}%
    \setbox\tw@\hbox{$\macc@style{\macc@nucleus}{}_{}$}%
    \dimen@\wd\tw@
    \advance\dimen@-\wd\z@
    \divide\dimen@ 3
    \@tempdima\wd\tw@
    \advance\@tempdima-\scriptspace
    \divide\@tempdima 10
    \advance\dimen@-\@tempdima
    \ifdim\dimen@>\z@ \dimen@0pt\fi
    \rel@kern{0.6}\kern-\dimen@
    \if#31
      \overline{\rel@kern{-0.6}\kern\dimen@\macc@nucleus\rel@kern{0.4}\kern\dimen@}%
      \advance\dimen@0.4\dimexpr\macc@kerna
      \let\final@kern#2%
      \ifdim\dimen@<\z@ \let\final@kern1\fi
      \if\final@kern1 \kern-\dimen@\fi
    \else
      \overline{\rel@kern{-0.6}\kern\dimen@#1}%
    \fi
  }%
  \macc@depth\@ne
  \let\math@bgroup\@empty \let\math@egroup\macc@set@skewchar
  \mathsurround\z@ \frozen@everymath{\mathgroup\macc@group\relax}%
  \macc@set@skewchar\relax
  \let\mathaccentV\macc@nested@a
  \if#31
    \macc@nested@a\relax111{#1}%
  \else
    \def\gobble@till@marker##1\endmarker{}%
    \futurelet\first@char\gobble@till@marker#1\endmarker
    \ifcat\noexpand\first@char A\else
      \def\first@char{}%
    \fi
    \macc@nested@a\relax111{\first@char}%
  \fi
  \endgroup
}
\makeatother
\usepackage{mathtools}

\usepackage{nameref,hyperref}
\usepackage[capitalize]{cleveref}

\Crefname{theorem}{Theorem}{Theorems}
\Crefname{lemma}{Lemma}{Lemmas}

\def\title#1{{\LARGE\bf  \begin{center} #1 \vspace{0pt} \end{center}  } }

\newcommand{\domain}{\Theta}
\begin{document}
 
 \begin{center}

  {\bf{\LARGE{Local Poisson Deconvolution for Discrete Signals  
  }}}

  \vspace*{.2in}

\vspace*{.2in}

{\large{
\begin{tabular}{cccc}
Shayan Hundrieser$^{\mathsection,\dagger,\ast}$ & Tudor Manole$^{\diamond, \ast}$ & Danila Litskevich$^\dagger$ 
& Axel Munk$^{\dagger, \ddagger}$
\end{tabular}
}}

\vspace{.15in}

\begin{tabular}{c}
	$^\mathsection$University of Twente, Department of Applied Mathematics \\
	$^\dagger$University of G\"ottingen, Institute for Mathematical Stochastics \\
	$^\diamond$Massachusetts Institute of Technology, Statistics and Data Science Center \\
  $^\ddagger$University Medical Center G\"ottingen, Cluster of Excellence ``Multiscale Bioimaging: \\from Molecular Machines to Networks of Excitable Cells'' (MBExC)   \\[0.15in]
\end{tabular}
\begin{tabular}{cc}
   				\texttt{s.hundrieser@utwente.nl}, \texttt{tmanole@mit.edu},  \\
   				 \texttt{danila.litskevich@gmail.com}, \texttt{munk@math.uni-goettingen.de}
\end{tabular}
\phantom{\footnote{Email}}

\blfootnote{$^\ast$ \!These authors contributed equally.}

\vspace{.15in}

\today

\vspace*{.15in}

\begin{abstract}
	We analyze the statistical problem of recovering an atomic signal, modeled as a discrete uniform distribution $\mu$, from a binned Poisson convolution model. 
	This question is motivated, among others, by super-resolution laser microscopy 
	applications, where precise estimation of $\mu$ provides insights into spatial formations of cellular protein assemblies.  
	Our main results quantify the {\it local} minimax risk of estimating $\mu$  
	for a broad class of smooth convolution kernels.   
	This local perspective enables us to sharply
	quantify  optimal estimation rates as a function of the clustering structure
	of the underlying signal. Moreover, our results are expressed under a
	 multiscale loss
	function, which reveals that
	different parts of the underlying signal can be recovered at different rates
	depending on their local geometry. Overall, these results paint an optimistic perspective
	on the Poisson deconvolution problem, showing that accurate recovery 
	is achievable under a much broader class of signals than suggested
	by existing global minimax analyses. Beyond Poisson deconvolution, 
	our results also allow us to establish the local minimax
	rate  of parameter estimation in   Gaussian mixture models with uniform weights. 
    
  We apply our methods to experimental super-resolution microscopy data to 
  identify the location and configuration of individual DNA origamis. In addition, we complement our findings with numerical experiments on runtime and statistical recovery that showcase the practical performance of our estimators and their trade-offs. 
\end{abstract}
\end{center}


\allowdisplaybreaks

 \pagenumbering{Roman}

\newpage
  \begingroup
 \hypersetup{linkcolor=black}
 {\renewcommand{\baselinestretch}{0.75} 
 \addtocontents{toc}{\protect\setcounter{tocdepth}{2}}
   \tableofcontents}
 \endgroup
 \newpage
 
 \pagenumbering{arabic}
\section{Introduction}\label{sec:intro}
Many fundamental tasks in the physical sciences 
and related areas are mathematically described as statistical inverse problems, wherein the objective is to reconstruct an underlying signal from indirect and noisy observations.
A~prominent example is the {\it Poisson deconvolution problem}~\citep{shepp1982maximum}, which
arises when the signal of interest is corrupted by convolutional blur, and is observed through
rare and discrete observations, such as photon emissions. 
This setting arises in a variety of metrology, sensing, and imaging applications, 
such as photon-limited medical imaging~\citep{willett2003platelets}, fluorescence  microscopy~\citep{aspelmeier2015modern}, 
collider physics~\citep{kuusela2015}, %
and telescopy \citep{bertero2009image, schmitt2012multichannel}.

Traditional statistical works on  deconvolution problems typically assume smoothness of the 
underlying signal, for instance in the H\"older or Sobolev sense \citep{carroll1988optimal,stefanski1990deconvolving,fan1991optimal, johnstone2004wavelet, bissantz2007non, meister2009deconvolution, pensky2009functional, dattner2011deconvolution,rousseau2024wasserstein}. Under these assumptions, minimax convergence rates for recovering the signal 
can be inherently slow, decaying for instance at a logarithmic rate in the case
of Gaussian deconvolution.
These rates reflect the worst-case ill-posed nature of such inverse problems,
and have led to the conventional theoretical view of deconvolution as a 
notoriously difficult statistical task. This pessimistic perspective stands in contrast to 
practice, where statistical deconvolution is routinely carried out  
with great success. This gap between theory and practice has long been noted, e.g., in the context of recovery of piecewise constant signals \citep{boysen2009} or from a Bayesian perspective \citep{efron2016empirical}, and can be attributed to the fact that typical
minimax analyses are carried out---by mathematical convenience---over excessively large function classes,
which are not reflective of the actual hardness of the   problem.

In this work, we aim to quantify this gap by analyzing a discrete Poisson deconvolution model that naturally arises in many applications, including modern optical imaging. It serves as a prototypical model, which is amenable to a {\it local} minimax analysis. 
The local perspective on minimax  problems is one of several
possible refinements of the usual minimax criterion in order to better reflect
the difficulty of an estimation problem. 
Concretely, in a local  minimax  analysis, the goal
is to characterize the worst-case estimation
risk over small balls of the parameter space, 
much like in the classical H\'ajek-Le Cam theory for semiparametric
models~\citep{hajek1972, vandervaart2002}. 
In nonparametric settings, local minimax rates of convergence
have been established in a variety of estimation problems, such as shape-constrained regression~\citep{meyer2000,chatterjee2015} and high-dimensional Gaussian mean estimation~\citep{neykov2022},
as well as a variety of hypothesis testing problems
(e.g.\,\cite{valiant2017,balakrishnan2019,wei2020a}).  
Closest to our work are local minimax estimation rates
for finite mixture models~\citep{gassiat2014,heinrich2018,wu2020},
which we discuss in further detail below.

Our analysis will involve two different types of locality. 
On the one hand, as in the aforementioned works, we will
sharply characterize the hardness of the discrete Poisson deconvolution problem 
over small balls of the parameter space, 
highlighting 
that the optimal estimation rate is highly heterogeneous
across different regions of the space, largely depending on the geometric
structure of the signal. These results will confirm that
the global worst-case rates of convergence are achieved by adversarial 
distributions which are atypical in practice. On the other hand, 
our results will rely on a multiscale loss function which is itself local, 
in the sense that it will allow
us to show that different parts of the underlying signal can be recovered
with varying degrees of accuracy, depending on the region of the parameter space
where this signal lies. 
It is worth emphasizing that these {\it doubly local} minimax rates will be achieved
by estimators which are entirely {\it adaptive}, in the sense that they will not require
any prior knowledge on the structure of the signal. Such estimators, and their computation will be addressed in detail.

We defer a concrete description of our model to Section~\ref{sec:model} below, 
however we begin by highlighting its main features.
We will make the assumption that the signal to be recovered is a \emph{finite superposition of uniformly weighted point sources}:
\begin{align}\label{eq:discretePopulation}
\mu = \frac{1}{k} \sum_{i=1}^{k} \delta_{\theta_i},
\end{align}
where the parameters $\theta_i\in \bbR^d$ are unknown.
In the most general setting, we impose no separation assumptions on these points, and we aim
to recover $\mu$ from observations of its convolution
with a smooth kernel, corrupted with Poissonian noise. 

Our inspiration for this model comes
from the field of super-resolution photonic imaging, which broadly refers
to a set of microscopy technologies that
are effectively diffraction-unlimited, and are able to make
observations at the unprecedented scale of tens of nanometers~\citep{schermelleh2019}. 
Such devices are routinely used to resolve {\it discrete} collections of light sources; 
for instance, resulting from photon emissions of subcellular fluorescent proteins after laser excitation~\citep{hell2003}. 
In such cases, the atoms $\theta_i$  
 represent the $k$ light sources
of interest (representing individual proteins), which can be viewed as indistinguishable in shape and size
for all intents and purposes of the imaging application at hand. 
Recovering the measure
 $\mu$ is equivalent to recovering
the unordered collection of these proteins. 

From a broader perspective, we argue that  the smallest scale of any imaging application can always be thought of as taking the form of signal~\eqref{eq:discretePopulation}, 
in the sense that any object under observation is ultimately composed of discrete
``particles'' which are effectively indistinguishable
in shape and size. While this perspective is impractical 
when observations are made at coarse scales, 
the ongoing technological advancements of a variety of scientific imaging technologies 
increasingly allow practitioners to make observations at such fine scales that
the discrete nature of the signal becomes relevant to characterize. 
Our model is tailored to these applications, and we expect it to continue
becoming more relevant as resolution improvements  are made
across scientific domains;
for further examples in super-resolution microscopy, see for instance~\cite{balzarotti2017nanometer,weber2021minsted}. 
To give a second example beyond microscopy, we note that the uniform measure~\eqref{eq:discretePopulation} 
can be used to model observations in particle accelerators, which      
consist of discrete collections of particle jets \citep{komiske2019}.  
Due to the limited resolution of detectors used in these experiments, 
the particle jets are measured with convolutional blur, 
and the deconvolution problem therefore arises prominently (and is typically
referred to as {\it unfolding}---see for instance~\cite{panaretos2011}).
 Model~\eqref{eq:discretePopulation}  is also natural
for astrophysical imaging, which often involves discrete stellar sources \citep{bertero2009image}, and the uniformity condition is reasonable when the stars in the observation window exhibit similar intensities. 

Our work is certainly not the first to study the behavior
of deconvolution problems under finer structural assumptions.
Perhaps the most common existing approach 
is to assume that the signal of interest admits a sparse representation in an appropriate
basis. For instance, if the signal admits a sparse
Fourier or wavelet representation, then the deconvolution problem can be reduced to a sparse
linear inverse problem which can be solved via compressive sensing methods~\citep{fan2002wavelet, johnstone2004wavelet, candes2006near,donoho2006compressed, raginsky2010compressed,giryes2014sparsity}. 
These approaches are orthogonal to our work: Although 
$k$-atomic measures $\mu$ are sparse in the spatial domain, they 
do not generally admit a sparse representation
in any countable basis which renders the deconvolution problem linear. 
In particular, model~\eqref{eq:discretePopulation} is nonlinear
when viewed as a function of the parameters $\theta_i$. 
Discrete deconvolution problems of this type
have been the object of study in at least two distinct literatures. 
The first is the literature on finite mixture models, which 
is the main source of inspiration for our work, and which we will discuss at length in Section~\ref{sec:mixtures} below.
The second is the {\it mathematical} theory of super-resolution, 
which is concerned with the recovery of discrete signals from  measurements 
under a low-pass filter (that is, under convolution with a band-limited
kernel); for instance, see~\cite{donoho1992superresolution, candes2013super, candes2014towards, tang2013compressed, moitra2015}. Like us, this line of work was  motivated by super-resolution 
advances in imaging technologies. Unlike our work, however, their restriction 
to band-limited kernels makes the problem so ill-posed that 
consistent recovery of the underlying signal is impossible without further separation assumptions. 
In contrast, we focus on convolution kernels which are not band-limited for which
recovery is possible without separation assumptions. Further, this allows us to provide a local minimax analysis of this problem, under a local
loss function which exhibits the heterogeneity in estimating different
subsets of the discrete signal.

\subsection{Problem Setting}
\label{sec:model} 

\begin{figure}[t!]
	\centering
	\includegraphics[width = 0.85\textwidth]{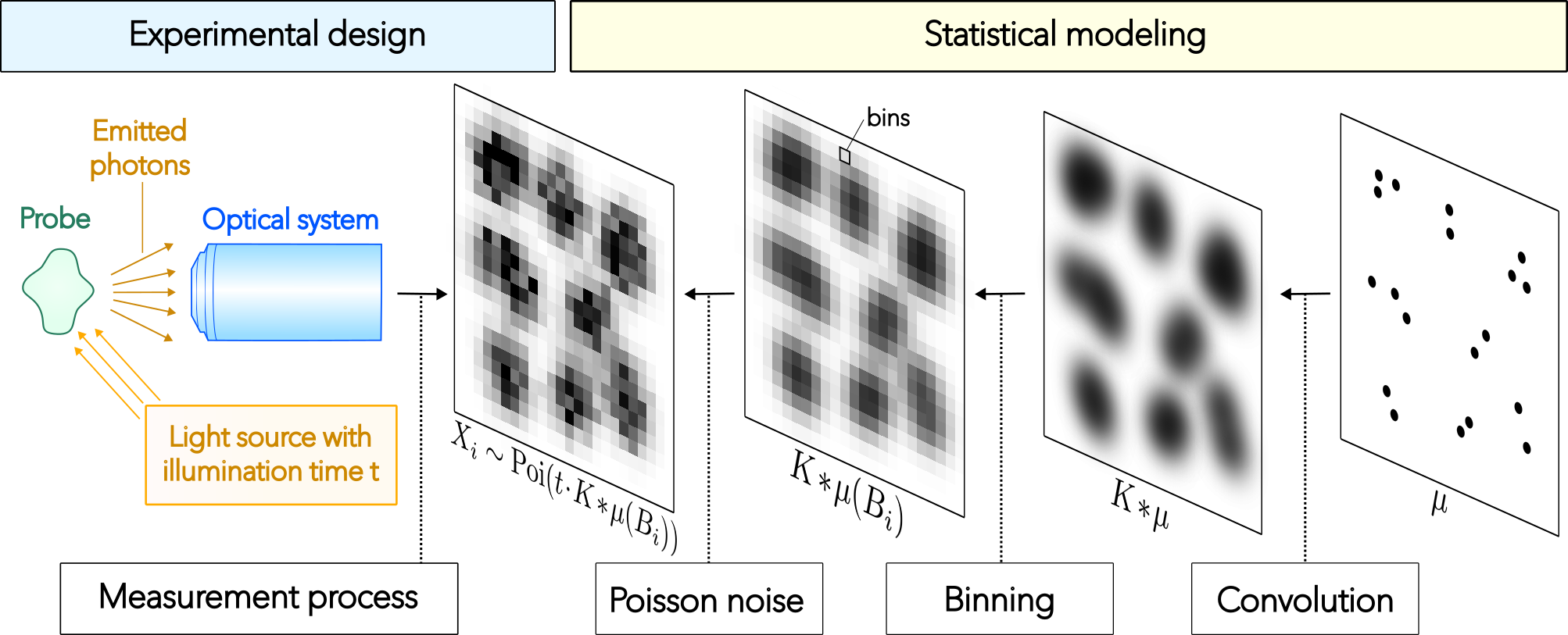}
	\caption{An example of the discrete Poisson deconvolution model \eqref{eq:model} arising in photon-limited fluorescence microscopy: In the physical experiment, a probe is illuminated for time $t>0$, causing the emission of photons which are recorded through an optical system on a binned detection domain. The statistical model \eqref{eq:model} assumes that measurements are generated by a uniform distribution $\mu$ on finitely many points which are convolved with a kernel $K$, discretized onto certain bins (corresponding to the detection domain), and corrupted by Poissonian
	 noise  proportional to the illumination time.}
	\label{fig:model}
\end{figure}

We now turn to formalizing our Poisson deconvolution model, using 
super-resolution microscopy as a running motivation. 
Concretely, our focus will be on the so-called 
{\it semiclassical} Poisson deconvolution model~\citep{mandel1959,kulaitis2021}, 
which is applicable, in particular, to scanning-mode super-resolution fluorescence microscopy
techniques such as Stimulation Emission Depletion (STED) 
microscopy~\citep{hell1994breaking, klar2000fluorescence,hell2003}. 
In such experimental setups, photon counts are collected sequentially 
over a finite set of bins $\{B_1, \dots, B_m\}$ which partition an observation 
window $\Omega\subseteq \RR^d$  into discrete regions, 
and each count is modeled as an independent Poisson-distributed random variable \citep{munk2020statistical}, 
see Figure \ref{fig:model}. Formally, we consider observations of the form:\begin{align}\label{eq:model}
	X_i \sim \textup{Poi}\left( t\cdot K\ast \mu(B_i) \right), \quad i = 1, \dots, m,
\end{align} 
where $K\colon \RR^d\to [0,\infty)$ is the convolution kernel associated with the microscope and\footnote{Here and throughout 
the manuscript, we use the following nonstandard notation:
for a map $f:\bbR^d\to\bbR$ and a set Borel $B \subseteq \bbR^d$, 
we write $f(B) = \int_B f(x)dx$. In particular, $f(B)$ does not denote
the direct image.} we write
$K\ast \mu(B_i) = \int_{B_i} K\ast \mu(x') \dif x' = \int_{B_i}\int_{\Theta}K(x' - x)\dif \mu(x)\dif x'$ to denote the (unscaled) intensity for bin $B_i$. Here, the atoms of $\mu$ are all assumed to lie within some (probe) domain $\Theta\subseteq \RR^d$ and 
the (known) real number $t > 0$ is to be understood as a proxy for the sample
size of the problem; physically it can be interpreted as
the laser illumination time of the experiment, or 
as a quantity proportional to the number of light pulses sent onto the probe.  Further, while we typically assume that $\Theta \subseteq \Omega$, this assumption is not necessary in all cases. 

It is worth noting that the semiclassical
model can be refined even further by incorporating
additional stochastic terms to model artifacts arising from detector noise \citep{aspelmeier2015modern}. 
We omit such effects to focus on the 
main features arising from deconvolution. 
We also stress that, while our main motivation comes from optics, 
many features of this model are more widely applicable. 
For instance, as we have already emphasized, the uniform
measure $\mu$ can serve as a natural model for particle jets in 
collider physics, which are  modeled precisely through the binned Poisson deconvolution
model that we described above~\citep{vandyk2014,kuusela2015}.

For theoretical purposes, we treat the number of  atoms $k$  as being known and fixed
throughout most the paper. Although unrealistic in most practical applications, 
the local minimax estimation risk for this problem
is challenging to characterize even when $k$ is known, 
and provides 
 important insights into the more general problem of estimating $\mu$ when $k$ is unknown. 
We postpone a detailed analysis under unknown $k$ to a separate work, but provide some discussion in Section~\ref{sec:discussion}.

Throughout our development, we place
the following two structural assumptions on the model. %

\begin{assumption}\label{ass:bins}
	The sets $\Theta, \Omega \subseteq \RR^d$, for $d\in \NN$, are convex and compact with non-empty interior, and the bins $\{B_i\}_{i = 1}^m$ %
	 form a (disjoint) partition of $\Omega$, such that each bin $B_i$ has positive Lebesgue measure and fulfills $\diam(B_i)\leq C m^{-1/d}$, with a constant $C=C(\Omega)>0$ only dependent on $\Omega$. 
\end{assumption}
\begin{assumption}\label{ass:kernel}One of the following settings hold for the probability kernel $K\colon \RR^d\to [0,\infty)$ in~\eqref{eq:model}.%
	\begin{enumerate}
		\item The kernel $K$ is compactly-supported, $h$-regular for $h\in\{1, \dots, k\}$ (see Definition~\ref{def:regularKernel}), and it holds that $\Theta + \supp(K) = \{\theta +x \,\colon \theta \in \Theta, x\in \supp(K)\} \subseteq \Omega$. 
		\item The kernel $K$ is given by a centered Gaussian density with a (symmetric) positive definite covariance matrix  $\Sigma\in \RR^{d\times d}$, i.e., $K(x) = ((2\pi)^d \det(\Sigma))^{-1/2}\exp(-\frac{1}{2}x^\top \Sigma^{-1}x)$. %
	\end{enumerate}
\end{assumption}

Assumption \ref{ass:kernel}(i) is perhaps the most relevant one 
for microscopy applications, since unbounded convolutional
kernels are often truncated  in practice. The condition
$\Theta + \supp(K) \subseteq \Omega$, which is made for technical
convenience, ensures that the blurred
signal $K\star\mu$ is always supported in the observation domain $\Omega$.
For unbounded kernels, such a condition is not possible, in which case 
it becomes necessary to characterize the tail behavior of the convolution
 $K\ast \mu$ based on the values
 that it takes on within the observation domain. 
We are able to carry out such an analysis for the special
case of Gaussian kernels satisfying   Assumption~\ref{ass:kernel}(ii). 
Such kernels also sometimes used
as an approximation of the point-spread function of STED microscopes~\citep{vondiezmann2017}.

\subsection{Theory: Main Results}

We begin by deriving the global minimax rate of estimation in model~\eqref{eq:model}, which will
serve as a benchmark for our subsequent local minimax rates. 
To this end, we adopt the Wasserstein distance~\citep{villani2009}
as a loss function for quantifying the risk of parameter estimation
over $\calU_k(\Theta)= \textstyle\{ \frac{1}{k}\sum_{i = 1}^{k} \delta_{\theta_i} \colon \theta_i \in \Theta\}$. For uniform measures $\mu = (1/k)\sum_{i=1}^{k} \delta_{\theta_i}\in \calU_k(\Theta)$ 
and $\nu = (1/k)\sum_{i=1}^{k} \delta_{\eta_i} \in \calU_k(\Theta)$ it can be defined as follows 
(see Section~\ref{subsec:notation} for a more general definition): 
$$W_p(\mu,\nu) = \min_{\sigma\in \calS(k)}\begin{cases} \left(\frac 1 k\sum_{i = 1}^{k} \|\theta_{\sigma(i)} - \eta_i\|^p\right)^{1/p} & \text{ if } p \in [1, \infty),\\
\max_{1 \leq i \leq k} \|\theta_{\sigma(i)} - \eta_i\| & \text{ if } p = \infty,
\end{cases}$$
where $\calS(k)$ is the permutation group on $\{1,\dots,k\}$ and $\|\cdot\|$ denotes the Euclidean norm. 
Our use of the Wasserstein distance is motivated 
by the fact that it equips $\calU_k(\Theta)$ with a metric
which is well-defined between discrete measures that are mutually singular,
and is consistent with the Euclidean metric on the ground space $\Theta\subseteq \RR^d$. 
Similar considerations have motivated the wide adoption of the 1-Wasserstein distance
in   mixture modeling problems, where it was first introduced in the works of~\cite{nguyen2013,nguyen2015}; see also \cite{bing2023} for an axiomatic justification in such problems. 

In what follows, we denote by $\calE_{t,m}(\domain)$ the set of estimators for~$\mu$, i.e., Borel-measurable functions $\hat\mu_{t,m}$ of an observation $(X_1,\dots,X_m)$ from model \eqref{eq:model}.

\begin{theorem}[Global minimax risk]
	\label{thm:global_minimax_risk}
	Fix $d,k\in \NN$ and $\Theta, \Omega \subseteq \RR^d$. Let $K$ be a kernel such that Assumptions~\ref{ass:bins} and~\ref{ass:kernel} are met. 
	Let $t \geq  1$ and $m \in \bbN$ satisfy  $m \geq  t^{(d+\gamma)/2}$
	for a fixed constant $\gamma > 0$. 
	Then, the minimax risk is   bounded as\footnote{Throughout this work we write $a\lesssim b$ to denote that there exists a positive constant $C>0$ such that $a\leq Cb$, and we write $a \asymp b$ whenever $a\lesssim b \lesssim a$. }%
	\begin{align}\label{eq:globalMinimaxRisk}
		\inf_{\hat\mu_{t,m} \in \calE_{t,m}(\domain)} \sup_{\mu\in \calU_k(\domain)} \bbE_\mu W_1(\hat\mu_{t,m},\mu) \asymp t^{-\frac 1 {2k}},
	\end{align}  
	where the implicit constants   depend only 
	on $\Theta, \Omega, \gamma, d, k, K$.
\end{theorem}

The upper bound in equation~\eqref{eq:globalMinimaxRisk} 
is achieved by 
a method-of-moments estimator when the kernel $K$ is compactly-supported
(Proposition~\ref{prop:mm_rate}), 
and by a maximum likelihood estimator (MLE)
when the kernel is Gaussian (Proposition~\ref{prop:mle_gaussian}). 
The minimax lower bound is developed in Proposition~\ref{prop:minimax_lower_bound}. Notably, since the low-dimensional setting is of main interest in our work, we make no attempt to sharply characterize the dependence of the minimax estimation risk  
on  $d$. Nevertheless, our proof permits making all constants explicit, and the lower bound is independent of $d$.

The main takeaway from 
Theorem~\ref{thm:global_minimax_risk}
is that optimal recovery of $\mu$ is largely driven by the number of components $k$,
 and degrades exponentially as $k$ increases. This aligns with the classical statistical deconvolution 
literature~\citep{stefanski1990deconvolving,fan1991optimal,dedecker2013minimax} which 
suffers from
logarithmic minimax rates for estimating smooth measures $\mu$ 
(roughly corresponding to $k = \infty$). 
The rate in \eqref{eq:globalMinimaxRisk} is also related to the minimax rate of estimating the components of a Gaussian mixture with $k$ possibly non-uniformly weighted components;
we provide a detailed comparison to this literature
in~\Cref{sec:mixtures}.

While Theorem~\ref{thm:global_minimax_risk}  sharply quantifies the minimax rate in terms of $t$, the condition on the number of bins $m$ is likely sub-optimal and is imposed to  ensure that the statistical error dominates the discretization error. 
Nevertheless, 
it is also relevant to assess the estimability of $\mu$ when $m$ is small relative to $t$. 
This regime turns out to be more subtle, since model~\eqref{eq:model} may not even be identifiable when $m$ is small; for instance, this occurs when $K$ is compactly-supported on a set of sufficiently small diameter. Remarkably, however, we show in Section~\ref{sec:mle_bounded}  
for the univariate setting that model~\eqref{eq:model}
is identifiable when $m$ is larger than an explicit constant that depends on $k$ and  $K$, provided that the kernel $K$
satisfies a condition that we term {\it root regularity} (Definition~\ref{def:root_regularity}), which includes the univariate Gaussian kernel. Under this condition, we show that the MLE is consistent even when $m$ is bounded (Proposition \ref{prop:consistency_LSE_MLE_finite_m}), which aligns with some of our numerical evidence (\Cref{sec:simulations}).  Our proof strategy is based on controlling the number of zero-sets of weighted sums of shifted functionals defined in terms of the kernel $K$. In the multivariate setting this is known to be much more difficult and is far beyond the scope of this paper. 

Shifting our focus back to our minimax analysis, we note that  the slow minimax rate of Theorem~\ref{thm:global_minimax_risk}
arises from the regime where all atoms of $\mu$ collapse into a single point.
In fact, the least-favorable choices of $\mu$ appearing in 
our minimax lower bound have atoms which are all within a distance of order 
$t^{-1/2k}$ of each other. This situation is clearly 
overly pessimistic for most practical purposes.  
To overcome this deficiency, we will now derive {\it local} minimax bounds
which show that 
the rate significantly improves when 
some of the atoms of $\mu$ are separated by a distance larger than $t^{-1/2k}$. 
We will also show that different parts of the measure $\mu$ can be estimated
at different rates, depending on the local
separation of atoms within these respective regions.

We formulate these results by replacing the parameter
space $\calU_k(\Theta)$ with an $\infty$-Wasserstein ball %
\begin{align*}
	\calU_k(\Theta; \mu_0, \rho)\coloneqq \big\{\mu \in\calU_k(\Theta) \;\colon W_\infty(\mu, \mu_0)< \rho\big\},
\end{align*}
of radius $\rho > 0$, 
centered at a measure  $\mu_0 \in \calU_k(\Theta)$
which will be chosen to admit some separation among its atoms.  Concretely, given
$1 \leq k_0 \leq k$  and a vector 
$r = (r_1,\dots,r_{k_0}) \in \bbN^{k_0}$ satisfying\footnote{We use the convention
$\bbN = \{1,2,\dots\}$, thus  $r_i\geq 1$ for all $i$. Furthermore, 
we write $|r| = \sum_i r_i$.} $|r|\coloneqq \sum_{i=1}^{k_0} r_i = k$,
we will choose $\mu_0$ to be any element of the set
\begin{align*}
	\calU_{k,k_0}(\Theta; r,  \delta) \coloneqq \left\{  \mu_0 = \textstyle \frac 1 k \sum_{j=1}^{k_0}r_j\delta_{\theta_{0j}} \in \calU_{k}(\Theta) :
	\min_{i\neq j} 
		\|\theta_{0i} - \theta_{0j}\| \geq \delta  \displaystyle \right\},
		\quad \text{where } \delta > 0.
\end{align*}
Heuristically, given  $\rho \ll \delta$ 
and a measure $\mu_0 \in \calU_{k,k_0}(\Theta;r,\delta)$, 
the ball $\calU_{k}(\Theta;\mu_0,\rho)$
consists of $k$-atomic measures
whose atoms are arranged such that $r_j$ of them 
cluster around $\theta_{0j}$, for each $j=1,\dots,k_0$,
and such that the clusters are roughly pairwise $\delta$-separated.

 We also adopt a loss function which,
unlike the Wasserstein distance, is adapted to the local separation structure of $\mu_0$. For its definition, let $\mu_0 = (1/k) \sum_{j=1}^{k_0} r_j \delta_{\theta_{0j}}$ 
and define the Voronoi partition of~$\RR^d$ generated by the support points of~$\mu_0$:
\begin{equation}
\label{eq:voronoi}
V_j \coloneqq V_j(\mu_0) \coloneqq  \big\{ \theta \in \RR^d \;\colon \|\theta - \theta_{0j}\| \leq \|\theta - \theta_{0l}\|,~ \forall l \neq j
\big\}\textstyle\backslash \bigcup_{h\leq j-1} V_h,\quad j=1,\dots, k_0.
\end{equation}
For any $j=1,\dots,k_0$, and any $\mu\in \calU_k(\Theta)$, we define the conditional measure 
$\mu_{V_j} = \mu(\cdot\cap V_j) / \mu(V_j)$ when $V_j$ is nonempty, and $\mu_{V_j}=0$ otherwise.\footnote{When
taking Wasserstein distances between conditional measures, 
we adopt the conventions $W_1(0,0) = 0$ and $W_1(\mu_{V_j},0) = \infty$ if $V_j$ is nonempty.} 
The local Wasserstein divergence with respect to $\mu_0$ between two probability measures $\mu, \nu\in \calU_k(\Theta)$ is then defined as 
\begin{align*}%
	\widebar \calD_{\mu_0}(\mu,\nu) = 1 \wedge \sum_{j=1}^{k_0}  W_1^{r_j}(\mu_{V_j}, \nu_{V_j}).
\end{align*}
Our first local minimax bound considers the case where the  separation parameter $\delta > 0$
between clusters 
is held fixed. For technical purposes, we will constrain the centering measure $\mu_0$
to be supported in any reduced domain $\Theta_\delta$ such that $\Theta_\delta + B(0,\delta) \subseteq \Theta$. 
 
\begin{theorem}[Local minimax risk for separated clusters]
\label{thm:local_minimax_risk_separated}
Assume the same conditions as in Theorem~\ref{thm:global_minimax_risk}. 
 Given
$1 \leq k_0 \leq k$, let $r \in \bbN^{k_0}$ be such that $|r|=k$. 
Given $\delta \in (0,\diam(\Theta)/2k)$, let 
  $\mu_0 \in \calU_{k,k_0}(\Theta_\delta;r,\delta)$ and $\epsilon\in (t^{-\frac 1 {2k}}, \delta/4]$.
  Then, 
 \begin{align}\label{eq:localMinimaxRisk}
	\inf_{\hat\mu_{t,m} \in \calE_{t,m}(\domain)} 
	 \sup_{\mu\in \calU_k(\domain;\mu_0,\epsilon)} \bbE_\mu \widebar  \calD_{\mu_0}(\hat\mu_{t,m},\mu)
\asymp t^{-\frac 1 {2}},
\end{align}
and, writing $r^* = \max_j r_j$, 
 \begin{align}\label{eq:localMinimaxRisk_W}
	\inf_{\hat\mu_{t,m} \in \calE_{t,m}(\domain)} 
	 \sup_{\mu\in \calU_k(\domain;\mu_0,\epsilon)} \bbE_\mu W_1(\hat\mu_{t,m},\mu)
\asymp t^{-\frac 1 {2r^*}},
\end{align}
where the implicit constants in the above displays depend only 
on $\Theta$, $\Omega$, $\delta$, $\gamma$, $c$, $d$, $K$, and $k$.
\end{theorem}

The upper bounds of Theorem~\ref{thm:local_minimax_risk_separated}
are again  attained under Assumption \ref{ass:kernel}$(i)$ by the method of moments and under Assumption \ref{ass:kernel}$(ii)$
by the maximum likelihood estimator.
In both cases, the estimators are adaptive to the local structure of the problem, 
in the sense that they do not require knowledge of the local model parameters $k_0, r, \mu_0,$
or $\delta$.  

Under the assumptions of
Theorem~\ref{thm:local_minimax_risk_separated}, the balls $\calU_k(\Theta;\mu_0,\epsilon)$
  consist of measures $\mu$ whose $k$ atoms form $k_0$ well-separated clusters. 
Recall that the integers $r_1,\dots,r_{k_0}$ denote the number of atoms of $\mu$ lying
in each of these clusters, and $r^*$ denotes the size of the largest
such cluster. Under this setting, the second assertion of 
Theorem~\ref{thm:local_minimax_risk_separated} implies
that the local minimax rate of estimation under the $W_1$ distance is $t^{-1/2r^*}$. 
When $k_0=1$, so that no clustering structure is present, 
one must have $r^*=k$, and 
  Theorem~\ref{thm:local_minimax_risk_separated} recovers the global minimax rate
of Theorem~\ref{thm:global_minimax_risk}. When $k_0=k$, so that
 all atoms of $\mu$ are separated, one must have $r^*=1$, and the 
local minimax rate reduces to the parametric rate $t^{-1/2}$.
The general rate in equation~\eqref{eq:localMinimaxRisk_W} 
interpolates between these two extremes; in particular, it
implies that there exists an estimator
$\hat\mu_{t,m}$ whose atoms converge to those of $\mu$ at the rate $t^{-1/2r^*}$,
up to relabeling. 

More strongly, the first assertion of Theorem~\ref{thm:local_minimax_risk_separated}
implies  that there exists an estimator
$\hat\mu_{t,m}$ such that for any
cluster $j=1,\dots,k_0$, and for any atom $\theta_i \in V_j$ of $\mu$
lying in that cluster, there exists an atom $\hat\theta_\ell$ of $\hat\mu_{t,m}$ 
such that
$$\bbE\|\hat\theta_\ell - \theta_i\| \lesssim t^{-1/2r_j}.$$
This implies that many of the atoms
of $\hat\mu_{t,m}$
may converge
faster than the rate $t^{-1/2 r^*}$ implied by the 1-Wasserstein distance. 
In particular, the convergence
behavior of the various atoms of $\hat\mu_{t,m}$ is {\it heterogeneous}, 
and depends on the local structure of $\mu$. This behavior could not have been anticipated from the Wasserstein minimax bound~\eqref{eq:localMinimaxRisk_W}. 
An example of these heterogeneous convergence
rates is given in Figure~\ref{fig:clustering56}(a).
 
\begin{figure}[t]
	\centering
	\includegraphics[width = \textwidth]{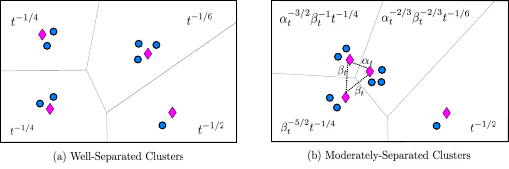}
	\caption{Illustrations of Theorem~\ref{thm:local_minimax_risk_separated} (left)
	and Theorem~\ref{thm:local_minimax_risk_nosep} (right).
	Pink diamonds represent the atoms of a measure $\mu_0$, blue atoms represent the atoms
	of a measure $\mu \in \calU_k(\Theta;\mu_0,c_2\epsilon_t)$, and grey lines
	represent the Voronoi cells generated by the atoms of $\mu_0$. 
	The distance between the pink atoms is either indicated
	by rates $\alpha_t,\beta_t \geq \delta_t$, or is constant.  
	Theorems~\ref{thm:local_minimax_risk_separated}--\ref{thm:local_minimax_risk_nosep}
	 prove that there exists
	an estimator of $\mu$ whose atoms converge to those of $\mu$ at the rates indicated
	in each Voronoi cell, whereas Theorem~\ref{thm:global_minimax_risk}
	would have implied the rate $t^{-1/16}$ for each of these atoms. 
	Notice that the rates in the right-hand figure are never slower 
	than $t^{-1/16}$ under the stated conditions on $\alpha_t$ and $\beta_t$.
}
	\label{fig:clustering56}
\end{figure}

One deficiency of Theorem~\ref{thm:local_minimax_risk_separated} is the
fact that it forces the $k_0$ clusters to be separated
by a constant distance $\delta>0$. 
As an example, it is natural to hope that  a measure $\mu$ whose atoms
are all separated by a slowly-decaying rate, such as $(\log t)^{-1}$, 
is nearly estimable at  the parametric rate, yet
Theorem~\ref{thm:local_minimax_risk_separated} does not imply a rate better
than $t^{-1/2k}$ in this regime. Our next result
will provide a sharper bound for such situations.
In this case, the appropriate generalization of the local
Wasserstein divergence is given by:
\begin{align}\label{eq:local_Wasserstein}
 \calD_{\mu_0}(\mu,\nu) = 1 \wedge \sum_{j=1}^{k_0}  \delta_j(\mu_0)
	\cdot  W_1^{r_j}(\mu_{V_j}, \nu_{V_j}),
	\quad \text{where } 
	\delta_j(\mu_0) = \prod_{\substack{1 \leq i \leq k_0 \\ i \neq j}} \|\theta_{0i} - \theta_{0j}\|^{r_i}.
\end{align}
Notice that $\widebar \calD_{\mu_0} \asymp \calD_{\mu_0}$
when the atoms of $\mu_0$ are separated by a constant, thus
Theorem~\ref{thm:local_minimax_risk_separated} continues
to hold when $\widebar\calD_{\mu_0}$ is replaced by $\calD_{\mu_0}$. 
Hence, there will be no loss of  generality in replacing $\widebar\calD_{\mu_0}$
by~$\calD_{\mu_0}$ throughout the remainder of the manuscript. 
Our final main result is stated as follows.
\begin{theorem}[Local minimax risk for approaching clusters]
\label{thm:local_minimax_risk_nosep}
Under the same conditions as Theorem~\ref{thm:local_minimax_risk_separated},
there exist constants $c_1,c_2 > 0$ such that 
if $\epsilon_t = c_1t^{-1/2k}$,
$\delta_t = c_2\epsilon_t^{ 1 / {2(k+1)}}$, then
\begin{align} 
	\inf_{\hat\mu_{t,m} \in \calE_{t,m}(\domain)} 
	\sup_{\mu_0\in \calU_{k,k_0}(\Theta; r,\delta_t)}
    \sup_{\mu\in \calU_k(\domain;\mu_0,\epsilon_t)} \bbE_\mu \calD_{\mu_0}(\hat\mu_{t,m},\mu)
\asymp t^{-\frac 1 {2}},
\end{align}
where $c_1,c_2$, and the implicit constants in the above display depend  
on $\Theta$, $\Omega$,  $\gamma$,  $d$, $K$, $k$.    
\end{theorem}

Theorem~\ref{thm:local_minimax_risk_nosep} extends the upper
bounds of Theorem~\ref{thm:local_minimax_risk_separated}
to the case where the atoms of $\mu_0$ can approach 
each other at a rate   slower than $\delta_t \asymp t^{-1/4k(k+1)}$.  
This rate is motivated by technical considerations, 
and we do not rule out the possibility
of extending
Theorem~\ref{thm:local_minimax_risk_nosep} to smaller values of~$\delta_t$. 
Theorem~\ref{thm:local_minimax_risk_nosep} also provides 
a lower bound on the minimax estimation risk, but
unlike Theorem~\ref{thm:local_minimax_risk_separated}, 
the lower bound is now taken over the {\it worst-case}  ball
$\calU_k(\domain;\mu_0,\epsilon_t )$
as $\mu_0$ ranges over $\calU_{k,k_0}(\Theta; r,\delta_t)$. 
We discuss this limitation further in Remark~\ref{rmk:MinimaxLowerBoundsDiscussion}. %

To interpret the convergence rate stated in Theorem~\ref{thm:local_minimax_risk_nosep}, 
let us provide an explicit example. 
Consider a two-point source with one ``double loading'',
$$\mu_0 = \frac 1 3 \delta_{-\theta_{01}} + \frac 2 3 \delta_{\theta_{02}},
\quad \text{with } k_0=2, ~~ k = 3, ~~\theta_{01}=-\frac 1 {\log t},~~\theta_{02} = \frac 1 {\log t}.$$
Then, any $\mu \in \calU_k(\Theta; \mu_0,\epsilon_t)$ 
will have one atom, say $\theta_1$, close to $
\theta_{01}$, and two atoms, say $\theta_2,\theta_3$, close to $\theta_{02}$. 
Theorem~\ref{thm:local_minimax_risk_nosep} implies that there exists an estimator
$\hat\mu_{t,m} = (1/k) \sum_{i=1}^k \delta_{\htheta_i}$ such that,
up to relabeling its atoms,
$$\bbE\|\hat \theta_1 - \theta_1\| \lesssim (\log t)^2t^{-1/2},
\quad \bbE\|\hat\theta_i -\theta_i\| \lesssim  \sqrt{\log t}\cdot t^{-1/4},
\quad i=2,3.$$
Meanwhile, Theorems~\ref{thm:global_minimax_risk}--\ref{thm:local_minimax_risk_separated}
cannot imply a convergence rate faster than $t^{-1/6}$ for estimating each atom of $\mu$. 
We provide a second example in Figure~\ref{fig:clustering56}(b).

\subsection{Estimators and Data Analysis}\label{subsec:estimators}

The upper bounds of our main results are achieved adaptively by simple
moment- and likelihood-based methods.
Neither of these estimators requires any explicit  regularization,
which  is a reflection 
of the parametric nature of the problem at hand, at least when $d$ and $k$
are fixed.
We briefly describe the two estimators here, deferring  a thorough description to Sections~\ref{sec:moments}--\ref{sec:mle}.  

Our moment estimator is based on a strategy 
which is widely-used in the mixture modeling literature~\citep{lindsay1989moment,wu2020}. 
Given observations $X_1,\dots,X_m$, the method consists of (i)
constructing a consistent estimator of a collection of moments of $K\star\mu$, 
(ii) deducing consistent estimators for a collection of moments of $\mu$, and finally, 
(iii) using these moment estimates to reconstruct a fitted mixing measure $\hat \mu_n$. 
This general approach has previously been used successfully in parametric
mixture models for which the kernel  admits an orthogonal basis of polynomials, 
such as the Gaussian density
 or other 
exponential families with quadratic variance function~\citep{wei2023}. One of our contributions
is to extend these ideas to generic compactly-supported kernels. 
Furthermore, our step (iii) involves a new reconstruction method
based on elementary symmetric polynomials, which leverages the uniformity
of the underlying mixing measures.  

Our second estimator is the traditional
MLE, which is given by any solution
to the problem
\begin{equation}
\label{eq:MLE}
\argmax_{\mu\in \calU_k(\Theta)} \sum_{i=1}^m \Big(X_i \log(K\star\mu(B_i)) - tK\star\mu(B_i)\Big).
\end{equation}
Although this optimization problem is convex in the intensity parameters $K\star\mu(B_i)$, it is nonconvex
when parametrized by the atoms of $\mu$. We adopt the EM algorithm as a heuristic
for approximating the above maximum, though we do not pursue a theoretical
analysis of its convergence properties. 

A \texttt{python} implementation 
of the EM algorithm and moment estimator for image data is publicly available.\footnote{\href{https://github.com/dlitskevich/poisson-deconvolution}{https://github.com/dlitskevich/poisson-deconvolution}} 
Based on our numerical evidence (cf.\,Appendix \ref{sec:simulations}),  
the EM algorithm  initialized with the method of moments estimator achieves the best performance. 
In \Cref{fig:origami_intro}, we illustrate an application of this method
to an experimental STED dataset of DNA origami samples by \cite{proksch2018multiscale}, and recover the locations of $k = 70$   fluorophores, i.e., light-emitting sources in the probe.
The data is preprocessed according to a procedure 
described in Appendix \ref{app:application}. 
As we discuss further in Section~\ref{sec:applications}, 
the approximately even spacing of fitted atoms in this figure is consistent
with the underlying geometric arrangement of DNA origami structures,
suggesting that our simple estimators can already be effective for such imaging
tasks when the number of atoms $k$ is moderate.
 
\begin{figure}[H]
	\centering
	 \includegraphics[width=0.8 \textwidth, trim={0 0 0 0}, clip]{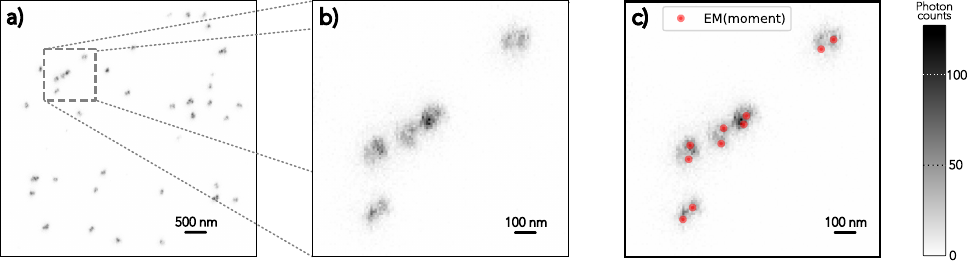}
	 \caption{Experimental STED data of a DNA  sample from \cite{proksch2018multiscale}.  \textbf{a)} Raw STED data measurement ($600\times 600$ pixel) of total width and height $6000$nm ($1$ pixel width and height corresponding to $10$nm) and \textbf{b)} $5$-fold zoom-in.  \textbf{c)} Maximum likelihood estimator computed for $k = 70$ by an EM algorithm (see \Cref{app:EM} for the specific algorithm) initialized with the method of moments estimator for the full image using some image partition and denoising step (red dots, detailed in \Cref{sec:applications} and Appendix \ref{app:application}).}%
	\label{fig:origami_intro}
\end{figure}

\subsection{Connections to Gaussian Mixture Models and Further Implications}
\label{sec:mixtures}
Our Poisson deconvolution  model~\eqref{eq:model} 
bears many connections to the literature on finite Gaussian mixture models, which we now discuss. 
In a $k$-component location Gaussian mixture model, one 
observes i.i.d. random variables
 \begin{equation}
\label{eq:mixture_model}
Y_1,\dots,Y_n \sim K\star \mu,~~\text{for some } \mu\in \calP_k(\Theta),
\end{equation}
where $\calP_k(\Theta)$ is the set of $k$-atomic probability measures
on $\Theta$, and $K$ is the $N(0,\Sigma)$ density, for some known strictly positive definite matrix
$\Sigma \in \bbR^{d\times d}$. The unknown parameter of the model
is the measure $\mu$, which  is typically referred to as a {\it 
finite mixing measure}. 

Gaussian mixture models have a long history in statistics and theoretical 
computer science, and we refer to~\cite{moitra2018,fruhwirth-schnatter2019,chen2023} 
for recent surveys. Closest 
to our work is the literature on {\it parameter estimation}
in finite mixture models,  
where the main goal is to estimate the mixing measure $\mu$. 
The study of optimal convergence rates for this problem
dates back to work of~\cite{chen1995optimal}, and was recently resolved by~\cite{heinrich2018}, 
who established that for any fixed $\mu_0\in \calP_{k_0}(\Theta)$,
and $\epsilon_0 > 0$ sufficiently small, it holds that
\begin{align}
\label{eq:heinrich_kahn}
\inf_{\hat\mu_{n}} %
 \sup_{\substack{\mu\in \calP_k(\domain) \\ W_1(\mu,\mu_0)\leq \epsilon_0}} \bbE_\mu   W_1(\hat\mu_n,\mu)
\asymp n^{-\frac 1 {4(k-k_0-1)-2}},
\end{align}
where the infimum is taken over all estimators of $\mu$ under model~\eqref{eq:mixture_model}. 
\cite{heinrich2018} proved that the above rate is achieved by a
Kolmogorov-Smirnov minimum-distance estimator.
This method has the downside of being  challenging to compute, since it involves a nonconvex optimization
problem, however
\cite{wu2020}   showed that a computationally simple estimator based
on
the method of moments is also minimax optimal. 
These two results were limited to one-dimensional Gaussian mixture models, 
but were extended to general fixed dimension by~\cite{wei2023}, 
and to the high-dimensional regime by~\cite{doss2023}.
\cite{manole2022refined} further showed that, in the one-dimensional case, 
the minimax rate of the above display continues to hold if the Wasserstein distance
is replaced by a stronger loss function, which is of similar nature
as the local Wasserstein divergence~$\calD_{\mu_0}$ introduced in the previous subsection.
In parallel to these developments, a large body of work 
in the theoretical computer science literature, initiated by~\cite{dasgupta1999}, 
was devoted to the study of 
moment-based estimators for Gaussian mixture models in varying degrees of generality;
for instance, see~\cite{belkin2010,moitra2010,kalai2010,anandkumar2012,hsu2013,ge2015}, and references therein.
 The goal in many of these works is not to sharply characterize 
 the minimax rate of estimation itself, but to identify conditions under which
 the minimax rate depends polynomially on the various problem parameters.

There are three structural differences between
our discrete Poisson deconvolution model~\eqref{eq:model} and the mixture
model~\eqref{eq:mixture_model}. %
First, our model 
requires the mixing measure $\mu$ to be uniform, an assumption which
is too restricrive for generic mixture models but is perfectly suited to the imaging
applications we have in mind. 
To the best of our knowledge, 
the minimax rate for parameter estimation in   Gaussian mixture models with uniform weights
has only been derived in the special case $k=2$, where it is known to 
scale as $n^{-1/4}$~\citep{wu2021,ho2022_fellerpaper}.
This rate is polynomially faster 
than the minimax rate for unknown weights, which is $n^{-1/6}$ by equation~\eqref{eq:heinrich_kahn}.
At first glance, this rate improvement 
is not surprising, since the uniformity assumption reduces
the number of parameters of the model. What is surprising, however, 
is that the same conclusion does not carry over when the mixture weights
are known but non-uniform. Indeed,~\cite{ho2022_fellerpaper} showed that the 
minimax rate of estimating a two-component Gaussian mixture model with 
{\it known} mixture weights is
generally $n^{-1/6}$, except in the special case where the weights are equal,
when the aforementioned rate $n^{-1/4}$ holds.
Thus, the improvements enabled by uniformity cannot be explained by a reduction
in the degrees-of-freedom of the model, and are rather due to finer structural
properties of uniform mixture models. %
We return to this point
in Remark~\ref{rem:moments_known_weights}. A similar discrepancy between the 
uniform and non-uniform settings is known to occur 
for Gaussian mixtures when the covariance matrix $\Sigma$
is unknown, both for minimax estimation~\citep{manole2020,ho2022_fellerpaper} 
and minimax detection~\citep{verzelen2017}, but once again, these results are limited
to the case of two-component mixtures.

One of the main technical contributions in this work is to
develop a toolbox for analyzing mixing measures with uniform weights,
even when the number of components is greater than two. 
Our results hinge upon a connection to 
the theory of elementary symmetric polynomials, which 
we develop in Section~\ref{sec:moments};
for other recent uses of such polynomials in statistical
applications, we refer to~\cite{gao2018} and \cite{han2024}. 
Though our results are motivated by the Poisson deconvolution model~\eqref{eq:model}, 
they have immediate implications for mixture modeling, which we state~next.
\begin{corollary}[Minimax rate of estimating uniform Gaussian mixtures]
\label{cor:mixtures}
Fix $d,k\in \NN$. Let the kernel $K$ satisfy Assumption~\ref{ass:kernel}(ii), and let
$\Theta \subseteq \bbR^d$ be a compact set with nonempty interior. 
Let $\calE_n(\Theta)$ be the set of estimators of $\mu\in \calU_k(\Theta)$ under
model~\eqref{eq:mixture_model}. Then, the following assertions hold.
\begin{enumerate}
\item (Global Minimax Risk) It holds that
$$\inf_{\hat\mu_n\in \calE_n(\Theta)} \sup_{\mu\in \calU_k(\Theta)} \bbE_\mu W_1(\hat\mu_n,\mu) 
\asymp n^{-\frac 1 {2k}},$$
where the implicit constants depend on $\Theta, \Sigma,d,k$.
\item (Local Minimax Risk) 
Let $k,k_0,r$ be defined as in Theorem~\ref{thm:local_minimax_risk_nosep} and $\calD_{\mu_0}$ as in \eqref{eq:local_Wasserstein}.  
Then, there exist constants $c_1,c_2 > 0$ such that 
if $\epsilon_n = c_1n^{-1/2k}$ and
$\delta_n = c_2\epsilon_n^{ 1 / {2(k+1)}}$, then it holds that 
$$\inf_{\hat\mu_{n} \in \calE_{n}(\Theta)} 
  \sup_{\mu_0\in \calU_{k,k_0}(\Theta; r, \delta_n)} 
  \sup_{\mu\in \calU_k(\Theta;\mu_0,\epsilon_n)} 
  \bbE_\mu \calD_{\mu_0}(\hat\mu_{n},\mu)
\asymp n^{-\frac 1 {2}},$$
where $c_1,c_2,$ and the implicit constants in the above display
depend on $\Theta, \Sigma,d,k$.  
\end{enumerate}
\end{corollary}
We provide a proof sketch of Corollary~\ref{cor:mixtures} in Appendix~\ref{subsec:proof:intro_mixtures}. 
Corollary~\ref{cor:mixtures}$(i)$ establishes the global minimax estimation
rate for uniform Gaussian mixture models with any fixed number
of components, recovering the results of~\cite{ho2022_fellerpaper,wu2021}
in the case $k=2$. In general, we find that the rate of convergence is nearly
quadratically faster than that of equation~\eqref{eq:heinrich_kahn} with $k_0=1$ for 
unknown weights. We also find that the rate coincides
with that of Theorem~\ref{thm:global_minimax_risk}, 
when making the natural identification between
the sample size $n$ and the illumination time $t$.

Turning now to the second assertion of Corollary~\ref{cor:mixtures}, consider
the special case where $\mu_0$ is taken to be a fixed
element of $\calU_{k,k_0}(\Theta;r,\delta)$ with
$\delta > 0$ fixed. 
In this case, one has $W_1^{r^*} \lesssim \calD_{\mu_0}$
with $r^* = \max_j r_j$, 
and Corollary~\ref{cor:mixtures}$(ii)$ then implies 
that 
\begin{equation}
\label{eq:step_compare_to_HK}
 \inf_{\hat\mu_{n} \in \calE_{n}(\Theta)}  
  \sup_{\mu\in \calU_k(\Theta;\mu_0,\epsilon_n)}  \bbE_\mu  \big[W_1(\hat\mu_n,\mu) \big]
  \lesssim n^{-\frac 1 {2r^*}}.
\end{equation}
Notice that for a given number $k_0$ of clusters, 
the maximal possible value of $r^*$ is $k-k_0+1$. Therefore, 
when $\delta $ is held fixed, the following upper bound always holds true: 
\begin{equation}
\label{eq:compare_to_HK}
\inf_{\hat\mu_{n} \in \calE_{n}(\Theta)}  
  \sup_{\mu\in \calU_k(\Theta;\mu_0,\epsilon_n)} 
  \bbE_\mu \big[W_1(\hat\mu_{n},\mu)\big]
\lesssim n^{-\frac 1 {2(k-k_0+1)}}.
\end{equation}
When written in this form, equation~\eqref{eq:compare_to_HK}
is directly comparable to the minimax rate
for non-uniform Gaussian mixtures in equation~\eqref{eq:heinrich_kahn}, 
which has the same qualitative behavior but is essentially quadratically
slower. As in the global minimax bounds discussed earlier, this quadratic
improvement is enabled by the uniformity restriction. 
However, our results   paint a   finer picture. First,
equation~\eqref{eq:step_compare_to_HK} shows that the upper bound
$n^{-\frac 1 {2(k-k_0+1)}}$ can be replaced by the generally faster rate  
$n^{-\frac 1 {2r^*}}$. Second, Corollary~\ref{cor:mixtures}$(ii)$
is stated under the divergence $\calD_{\mu_0}$, which implies
heterogeneous rates of convergence for the atoms of $\hat\mu_n$ which
cannot be deduced from the Wasserstein distance. Finally, 
Corollary~\ref{cor:mixtures}$(ii)$ allows for the $k_0$ clusters to approach each other
at a vanishing rate, a setting which is not captured
by equation~\eqref{eq:heinrich_kahn}. 

A second distinction between the signal recovery model~\eqref{eq:model}
and the mixture model~\eqref{eq:mixture_model} lies in the fact that the density $K\star\mu$
is only sampled (with Poisson noise) over a compact domain $\Omega$ 
in the former.
Inference for $\mu$ thus has to be performed based on a noisy realization of the 
truncated intensity function $K\star\mu|_\Omega$, a situation which is reflective
of realistic applications. While this truncation has no effect
for compact kernels satisfying Assumption~\ref{ass:kernel}(i), it 
has a nontrivial effect for the Gaussian kernel. 
We will indeed see that the presence of truncation causes the usual method of moments
to be  inconsistent for the signal recovery model with Gaussian kernel, 
in sharp contrast to Gaussian finite mixture models, where the method of moments
is well-known to be minimax optimal~\citep{wu2020}. 
While it may be possible to derive a modified moment
estimator tailored to the Gaussian kernel, we take a different
approach and analyze the maximum likelihood estimator (MLE). 
Our analysis of the MLE  will rely on 
stability bounds which relate the Wasserstein distance
over the space of  mixing measures to the $L^2(\Omega)$ 
distance between the corresponding mixture densities, which
in turn can be controlled using empirical process theory arguments. 
To obtain such a stability bound, one of our main technical observations will be that 
the $L^p(\Omega)$ and $L^p(\bbR^d)$ distances are  equivalent over the space 
finite of Gaussian mixture densities---a fact which appears to be new, and possibly of independent
interest. Concretely, we will show in Proposition~\ref{cor:Lp_Omega_Rd} for any $p, q \in [1,\infty]$ 
and any set $\Omega \subseteq \bbR^d$ with non-empty interior and bounded $\Theta \subseteq \RR^d$ that %
\begin{equation}
\label{eq:Lp_equivalence}
\|K\star(\mu-\nu)\|_{L^p(\bbR^d)} \asymp \|K\star(\mu-\nu)\|_{L^q(\Omega)},
~~\text{for all } \mu,\nu\in \calP_k(\Theta), 
\end{equation}
where $K$ is a Gaussian kernel satisfying Assumption~\ref{ass:kernel}$(ii)$, 
and the implicit constants in the above display depends on $\Omega$, $\Theta$, $K$ and, in particular, on $k$.

Our proof of equation~\eqref{eq:Lp_equivalence} also implies several other identities
between statistical divergences over the space of Gaussian mixtures densities, which
are not all needed for our analysis but have been sought-after in recent literature. 
For instance,~\cite{jia2023} recently showed
that the squared Hellinger, Kullback-Leibler, and $\chi^2$ distances
are equivalent over the space of Gaussian mixture densities (even for mixing
measures which are not finite\footnote{See also~\cite{doss2023} for analogous results
in the case
of finite mixing measures}), and their work posed the open question
of whether the Total Variation (TV) distance and Hellinger distances are also equivalent. 
Our work turns out to show that this is indeed the case, 
at least for finite mixing measures.
In fact, these metrics are even equivalent to the $L^\infty(\bbR^d)$ distance,
as shown next.
\begin{corollary}[Equivalence of TV and Hellinger for finite Gaussian mixtures]
\label{cor:TV_Hellinger_equiv}
Let $k,d\in \NN$, and let $\Theta \subseteq \bbR^d$ be a compact set. %
Let $K$ be the $N(0,\Sigma)$ density, for a positive definite matrix
$\Sigma \in \bbR^{d\times d}$. Then, 
for all $\mu,\nu\in \calP_k(\Theta)$,
$$\mathrm{TV}(K\star\mu,K\star\nu) \asymp H(K\star \mu,K\star\nu) 
\asymp \|K\star(\mu-\nu)\|_{L^\infty(\bbR^d)},
$$
where the implicit constants depend only on $\Theta,\Sigma,d,k$. 
\end{corollary}
Corollary~\ref{cor:TV_Hellinger_equiv} resolves the open problem of~\cite{jia2023}
for finite mixing measures, though we  emphasize
that extending this result to infinite mixing measures 
would likely require arguments of a different nature.
The proof appears in Appendix~\ref{app:pfs_intro}, building upon results
which will be developed in Section~\ref{sec:mle}. 

Finally, we discuss a third aspect of the Poisson deconvolution model~\eqref{eq:model}
which differs from mixture models: the presence of binning. When the number of bins $m$ is sufficiently large with respect
to the illumination time $t$, 
the discretization effect due to binning is of low order, and can effectively be ignored
in our analysis; this is the perspective that we have taken in 
Theorems~\ref{thm:global_minimax_risk}--\ref{thm:local_minimax_risk_nosep}, 
where we assumed that $m$ is greater than a polynomial of $t$. 
On the other hand, we provide an analysis of the MLE
in Section~\ref{sec:mle_bounded} when $m$ is fixed, 
in which case the effect of discretization plays a significant role.

\subsection{Outline} 

The manuscript is structured as follows. In \Cref{sec:moments} we formulate and analyze the method of moments estimator. \Cref{sec:mle} is concerned with the statistical analysis of the maximum likelihood estimator (MLE) for $m\to \infty$.  \Cref{sec:minimaxLowerBound} proves minimax lower bounds for our statistical model for regular kernels. In \Cref{sec:mle_bounded} we establish novel identifiability results based for finitely many bins, i.e., for $m$ bounded and confirm consistency of the MLE. In \Cref{sec:applications} we apply our methodology to STED microscopy data. %
We close in \Cref{sec:discussion}  with a discussion of possible extensions of our work. Nearly all proofs, additional numerical simulations and additional technical results are relegated to different appendices for a more streamlined exposition.

\subsection{Notation}\label{subsec:notation}

Throughout the manuscript we employ the following notation. We write $\lambda$ to denote the Lebesgue measure on $(\RR^d, \calB^d)$. We denote the collection of (resp.\ probability) measures on a Borel measurable subset $\calX\subseteq \RR^d$ by $\calM(\calX)$ (resp.\ $\calP(\calX)$), Further, we use the non-standard notation $\calM_k(\calX)$ (resp.\ $\calP_{k}(\calX)$) to denote (resp.\ probability) measures with at most $k\in \NN$ support points in $\calX$. Given a measurable map $f\colon \calX\to \calY$ between measurable spaces, the push-forward of a measure $\mu \in \calP(\calX)$ under $f$ is defined as $f_{\#}\mu\coloneqq \mu(f^{-1}(\cdot))$. 
We denote by $B(x,r)$ and $B_\infty(x,r)$ the open ball with center $x\in \RR^d$ and radius $r>0$ under the Euclidean and maximum norm, respectively. 
For a measure $\mu\in \calM(\KK)$ for $\KK\in \{\RR,\CC\}$ we denote the $\alpha$-th moment by $m_\alpha(\mu)\coloneqq  \int x^\alpha \dif \mu(x)$ for $\alpha\in \NN_0$ if it exists. Similarly, for a multivariate measure $\mu \in \calM(\RR^d)$ for $d\geq 1$ and a multi-index $\alpha=(\alpha_1, \dots, \alpha_d)\in \NN_0^d$ we denote the $\alpha$-th moment of $\mu$ by $m_\alpha(\mu)\coloneqq \int x^\alpha \dif \mu(x) = \int \prod_{i = 1}^{d} x^{\alpha_i}\dif \mu(x)$ provided it exists. 
For a random variable $X\sim \mu$ on $\RR^d$ we also interchangeably use the notation
$\EE[X], \EE_{\mu}[X],$ or $\EE_{X\sim \mu}[X]$ to denote its expectation (vector). 
Moreover, given two probability measures $\mu, \nu$ on $\RR^d$, we define
the $p$-Wasserstein distances for $p \in [1, \infty]$ by
\begin{align}\label{eq:def_Wasserstein}
	W_p(\mu, \nu) \coloneqq \begin{cases}\inf_{\pi\in\Pi(\mu, \nu)}%
	\left(\int \|x-y\|^p \dif \pi(x,y)\right)^{1/p} & \text{ if } p \in [1,\infty),\\
	\inf_{\pi\in \Pi(\mu, \nu)} \mathrm{ess\, sup}_{(x,y)\in \supp(\pi)}\|x-y\| &\text{ if } p = \infty,
	\end{cases} 
\end{align}
where $\Pi(\mu, \nu)$ denotes the set of couplings between $\mu$ and $\nu$.  
The Hausdorff distance between $A,B\subseteq \RR^d$ is defined as 
$d_H(A,B) \coloneqq \max\left\{ \textstyle \sup_{x \in A} \inf_{y\in B} \|x-y\|, \sup_{y\in B} \inf_{x \in A} \|x-y\| \displaystyle \right\}$, 
which we may use to quantify the difference between $\supp(\mu)$ and $\supp(\nu)$ for $\mu, \nu \in \calU_k(\RR^d)$. If $\mu$ and $\nu$ are
 also absolutely continuous with respect to some reference measure $\rho$ (often $\lambda$), with respective densities $f = d\mu/d\rho$ and 
$g = d\nu/d\rho$, and $\Omega\subseteq \RR^d$ with $\xi\in \calM(\Omega)$ we denote the $L^p(\Omega, \rho, \xi)$-distance $\|\mu - \nu\|_{L^p(\Omega, \rho, \xi)}  = (\int_{\Omega} |f - g|^p \dif \xi)^{1/p}$ for $p\in [1,\infty)$ and $\|\mu - \nu\|_{L^\infty(\Omega, \rho, \xi)}  = \textup{ess sup}_{x\in \Omega} |f(x) - g(x)|$ where null-sets are taken with respect to $\xi$. We abbreviate $L^p(\Omega, \rho, \xi)$ by  $L^p(\Omega)$ (resp.\ $\ell^p(\Omega)$) when $\rho = \lambda$ and $\xi = \lambda|_{\Omega}$ under $\lambda(\Omega)>0$ (resp.\ $\xi = \sum_{\omega \in \Omega} \delta_{\omega}$ when $|\Omega|<\infty$). Further, the squared Hellinger distance and 
$\chi^2$-divergence respectively are defined by 
$H^2(\mu, \nu) = \frac{1}{2}\int (\sqrt f -\sqrt g)^2\, \dif \rho$ and $
\chi^2(\mu,\nu) = \int (f-g)^2/g\,\dif\rho.$ 
By abuse of notation, may also plug $f$ and $g$ into the respective distances above. 
Given a field $\bbK$, we denote the ring of $m$-variable polynomials with coefficients in $\bbK$, with respect to indeterminates
$x_1,\dots,x_m$, by $\bbK[x_1,\dots,x_m]$. We also denote by $\bbK[x_1,\dots,x_m]_k$ the set of polynomials in $\bbK[x_1,\dots,x_m]$ of degree at most $k$. The collection of permutations on $\{1, \dots, l\}$ for $l\in \NN$ is denoted by $\calS(l)$. For a metric space $(X,d)$, a subset $A \subseteq X$, and $u > 0$, 
we denote by $N(u,A,d)$ the $u$-covering number of $A$ with respect to the metric
$d$, which is defined as the smallest integer $N$ for which there
exists $x_1,\dots,x_N \in X$ such that for all $a \in A$, there exists $i\in\{1,\dots,N\}$
such that $d(a,x_i) \leq u$. %
When $\Omega\subseteq \RR^d$ with $\lambda(\Omega)>0$ and $X = L^p(\Omega)$ for some $p\in [1,\infty]$ and $A$ is a Borel-measurable subset of $X$, 
we denote by
$N_{[]}(u,A,L^p(\Omega))$ the $u$-bracketing number of $A$ under the $L^p(\Omega)$-norm, 
which is defined as the smallest integer $N$ for which there exist functions
$\ell_1,u_1,\dots,\ell_N,u_N\in L^p(\Omega)$
 such that $ 0 \leq u_i - \ell_i$ with $\|u_i - \ell_i\|_{L^p(\Omega)} < u$ for all $i$, and such that
  for any $f\in A$, there exists $i\in\{1,\dots,N\}$
for which $\ell_i\leq f \leq u_i$ over $\Omega$. Finally, the probability distribution of the observed vector $(X_1,\dots,X_m)$ for our model \eqref{eq:model} is denoted by 
$\bP_\mu^t = \otimes_{i=1}^m P^{X_i}$.

\section{The Method of Moments Estimator}
\label{sec:moments}

We begin with the analysis of a moment-based estimator. 
As we discussed in Section~\ref{sec:intro},
moment-based methods are widely-used in the mixture modeling literature
due to their computational feasability and minimax optimality. We will see that 
these properties are shared by a suitable moment method for the Poisson deconvolution model~\eqref{eq:model}.

In short, the method of moments consists of three steps: 
First, estimating finitely many moments of the convolution $K\star\mu$ based on the random data, 
second, using these estimates to form unbiased estimators
of the underlying mixing distribution $\mu$,
and third, recovering the measure $\mu$ based on the estimated moments. 
Crucial for the last step is that the discrete mixing measure can be uniquely identified based on finitely many moments. To this end, we establish in \Cref{subsec:idenMoment} 
a series of moment identifiability results
 for the particular setting where $\mu$ is assumed to be a uniform distribution on $k\in \NN$ points. Further, we provide in \Cref{subsec:method_of_moments} a procedure to recover moments of the mixing distribution $\mu$ based on moments of the mixture $K\ast \mu$ and the kernel $K$. In conjunction with a novel quantitative moment comparison bound (\Cref{subsec:momentCompBound}), we then characterize the statistical convergence rate of our method of moment estimators in \Cref{subsec:method_of_moments}. 

Since our work is motivated by imaging applications, 
we will be particularly interested in the situation where the underlying observation domain $\Omega$
is two-dimensional. In this case, we will prefer to view
$\Omega$ as a subset of the complex plane $\bbC$ rather than a subset of $\bbR^2$. 
This perspective will turn out to be fruitful for computational purposes,
 as the method of moment estimator will then take on a particularly concise form.  

\subsection{Identifiability from Moments}\label{subsec:idenMoment}

We first recall a well-known moment identifiability result
for   univariate distributions supported on finitely many points, which is due to \cite{lindsay1989moment};
see also~\cite{wu2020}.

\begin{lemma}[Moment identifiability]\label{lem:momentIdentifiability_RealCompl}
Given $k\in \NN$, two probability measures $\mu, \nu \in \calP_k(\CC)$ are identical if and only if $m_\alpha(\mu) = m_\alpha(\nu)$ for each $\alpha\in \{1, \dots, 2k-1\}$. 
\end{lemma}

A proof of this result over $\bbR$ can be found in Lemma 4.1 of \cite{wu2020}, 
using polynomial interpolation techniques which readily extend to $\CC$. %
The quantity $2k-1$ can naturally be interpreted as the number
of degrees of freedom of a $k$-atomic measure,
consisting of  $k$  support points
and $k-1$ unconstrained probability masses.
In the following, we provide a refinement of Lemma \ref{lem:momentIdentifiability_RealCompl} for measures which are uniformly distributed on $k$ points, and show that $k\in \NN$ moments suffice for identifiability. %

\begin{lemma}[Moment identifiability for uniforms]\label{lem:moment_identifiability}
Given $k \in \NN$, two probability measures $\mu,\nu\in \calU_{k}(\CC)$ are identical if and only if $m_\alpha(\mu)= m_\alpha(\nu)$ for each $\alpha \in \{1, \dots, k\}$. 
\end{lemma}

\begin{remark}
\label{rem:moments_known_weights}
Once again, it is tempting to interpret Lemma~\ref{lem:moment_identifiability} 
through a degrees-of-freedom argument, by noting that a  uniform
measure on $k$ support points only has $k$ degrees of freedom, and should hence
be identifiable from $k$ moments. 
Interestingly, however, uniformity of the weights plays a
crucial for the above identifiability result: If the weights were non-uniform but fixed, 
 the space of measures would also admit $k$ degrees of freedom, but the first $k$ moments 
 would not sufficient for identifiability. For instance,   the measures $(3/4)\delta_{-1} + (1/4) \delta_{3}$ and $(3/4)\delta_{1} + (1/4)\delta_{-3}$ on $\RR$ admit identical first and second moments but are clearly different. 
\end{remark}

Lemma \ref{lem:moment_identifiability} is implicit in the work of~\cite{gao2018}, 
and was already shown by \citet[Lemma 4]{zaheer2017deep} for the setting where $\mu, \nu \in \calU_k([0,1])$. We provide a simple proof which naturally extends to $k$-uniform measures on  $\CC$. 

\begin{proof}[Proof of Lemma \ref{lem:moment_identifiability}]
First note for $\alpha>k$ that $m_\alpha(\mu) =\frac{1}{k} \sum_{i = 1}^k x_i^\alpha$ is a symmetric\footnote{This means that the sum $\sum_{i = 1}^k x_i^p$ is invariant with respect to permutations of its variables.} polynomial in $\{x_1, \dots, x_k\}$. Since the ring of symmetric polynomials with rational coefficients is equal to the rational polynomial ring $\QQ[\sum_{i = 1}^k x_i^1, \dots, \sum_{i = 1}^k x_i^k]$ \citep[pp.24]{macdonald2004}, there exists an algebraic representation of  $m_\alpha(\mu)$ in terms of the moments $m_1(\mu), \dots, m_k(\mu)$. Hence, matching moments  up to order $k$  of $\mu, \nu\in \calU_k(\CC)$  is equivalent to matching moments up to order $2k-1$. By Lemma \ref{lem:momentIdentifiability_RealCompl} the latter is equivalent to equality of $\mu$ and $\nu$.  
\end{proof}

Extending upon the previous result for the real line and the complex plane, we also state an identifiability result for uniform distributions in the multivariate setting in $\RR^d$ for $d\in \NN$. 

\begin{lemma}[Multivariate moment identifiability for uniforms]\label{lem:multivariateIdentifiability}
 For $d,k\in\NN$, two probability measures $\mu, \nu\in \calU_{k}(\RR^d)$ are identical if $m_\alpha(\mu) = m_{\alpha}(\nu)$ for each $\alpha \in \NN^d_0$ with $|\alpha|\leq k$. 
	\end{lemma}

	In \Cref{app:pf_thm_moment_comparison}, 
	we provide a proof which is inspired by \citet[Lemma A.1(i)]{wei2023} that employs the Cramer-Wold device to reduce the multivariate analysis to the real-line case. We also provide an alternative proof based on a general representational theorem for multivariate symmetric functions \cite[p.2 and Theorem 2.1]{chen2022representation}.

\subsection{Moment Comparison Inequalities for Uniform Measures}\label{subsec:momentCompBound}

The previous subsection provides qualitative conditions under
which a $k$-atomic uniform measure can be identified from finitely many moments. 
In the following we \emph{quantitatively} relate the (local) Wasserstein distance between two measures $\mu, \nu\in\calU_k(\Theta)$, where $\Theta \subseteq \RR^d$ for $d \in \NN$ or $\Theta\subseteq \CC$, to their the moment difference
\begin{align}\label{eq:def_moment_difference}
	M_k(\mu,\nu) \coloneqq \begin{cases}
		\sum_{\alpha=1}^{k} |m_\alpha(\mu) - m_\alpha(\nu)| &\text{ if } \Theta \subseteq \RR \text{ or }
		\Theta \subseteq \CC,\\
		\sum_{\substack{\alpha \in \NN_0^d\backslash\{0\}\\1\leq \|\alpha\|_1\leq k}} |m_\alpha(\mu) - m_\alpha(\nu)| &\text{ if } \Theta \subseteq \RR^d \text{ with } d\geq 2.
	\end{cases}
\end{align}
Our main comparison bound is stated as follows.

\begin{theorem}
\label{thm:moment_comparison}
Let $\Theta\subseteq \RR^d$ for $d \in \NN$ or $\Theta \subseteq \CC$ be a bounded subset and let $k \in \NN$. Then, the following two assertions hold. 
\begin{enumerate}
\item[(i)] (Global comparison bound) There exists a  constant $C = C(\domain,d,k)> 0$ such that for all 
$\mu,\nu \in \calU_k(\domain)$ it holds 
$$W_1^k(\mu, \nu) \leq C M_k(\mu,\nu).$$
\item[(ii)] (Local comparison bound) 
Let $k_0 \in \{1, \dots, k\}$ and $r\in \NN^{k_0}$ with $|r|= k$. Then, there exist  constants $C,c,\delta_0 > 0$ which depend on $\Theta, d, k$ such that for all $\delta \in (0,\delta_0)$, $\mu_0 \in \calU_{k,k_0}(\Theta; r, \delta)$, and $\mu, \nu \in \calU_{k}(\Theta; \mu_0, c\delta^{k+1})$, it holds that
\begin{equation}\label{eq:sharp_poly_stability}
\calD_{\mu_0} (\mu,\nu) \leq C M_k(\mu,\nu).
\end{equation}
\end{enumerate}
\end{theorem}

Let us now provide the main ideas behind the proof of Theorem~\ref{thm:moment_comparison}, deferring most technical
details to Appendix~\ref{app:pf_thm_moment_comparison}. For the regime where $\Theta \subseteq \RR$ or $\Theta \subseteq \CC$ we will show that the two claims of Theorem~\ref{thm:moment_comparison} can  
be reduced to the problem of deriving {\it stability bounds} for the roots 
of complex polynomials under perturbations of their coefficients, which in turn 
are well-studied. The multivariate real setting for $d\geq 2$ then follows by reducing the problem via projections to the univariate setting, using an approach inspired by \cite{wei2023}
and~\cite{doss2023}. 

To elaborate on the  complex case, let
$\mu = ( 1 / k) \sum_{i=1}^k \delta_{\theta_i}$ %
be an arbitrary measure in $\calU_k(\domain)$
with $\domain\subseteq\bbC$. 
We associate to $\mu$ the unique monic polynomial $f_\mu \in \bbC[z]$ 
 whose roots are given by the atoms of~$\mu$:
\begin{align}\label{eq:polys_pf_moment_comparison}
f_\mu(z) = \prod_{i=1}^k (z-\theta_i),\quad z \in \bbC.
\end{align}
The coefficients of $f_\mu$ are closely-related to the complex moments of $\mu$, 
since, by   {\it Vieta's formula}, one has the expansion
\begin{align}
\label{eq:vieta_formula}
f_\mu(z) = z^k + \sum_{j=1}^k a_j(\mu) z^{k-j}, \quad \text{with } a_j(\mu) = (-1)^j e_j(\theta_1, \dots, \theta_k),
  \end{align}
where we define the so-called {\it elementary symmetric polynomials}
\begin{align}\label{eq:elementary_symmetric_polynomials}
	e_0(\theta_1, \dots, \theta_k) \coloneqq 1 \quad \text{ and } 
	\quad  e_j(\theta_1, \dots, \theta_k) &\coloneqq \sum_{1\leq i_1< i_2 < \dots <  i_j \leq k} \;\; \prod_{\ell =1}^j \theta_{i_\ell}.
\end{align}
Since the polynomials $e_j$ are symmetric, they admit an algebraic representation in terms of the 
moments $m_1(\mu), \dots, m_k(\mu)$. This representation can be made explicit using
 {\it Newton's identities} %
which assert for any $l \in \{1, \dots, k\}$ that
 \begin{align}\label{eq:Newton_identity}
 	e_l(\theta_1, \dots, \theta_k) 
 	&= \frac{k}{l} \sum_{j=1}^{l} (-1)^{j-1} e_{l-j}(\theta_1, \dots, \theta_k)  m_j(\mu).
 \end{align}
Combining the preceding displays, we deduce that the coefficients of $f_\mu$ can be 
represented algebraically by the first $k$ moments of $\mu$. 
This observation leads to the following simple Lemma. 
In what follows, we denote by %
$\|\cdot\|_*$ the norm on $\bbC[z]$ defined by
$$\|g\|_* = \sum_{j=0}^k |b_j|,\quad \text{for any polynomial } g(z) = \sum_{k=0}^k b_j z^{k-j}\in \bbC[z].$$
\begin{lemma}\label{lem:coeff_lip}Let $\Theta \subseteq \CC$  be a bounded set and let $k\in \NN$. Then, there exists a constant $C = C(\domain,k) > 0$ such that for any $\mu,\nu\in \calU_k(\domain)$, 
$$\|f_\mu-f_\nu\|_* \leq C M_k(\mu,\nu).$$
\end{lemma}

The proof of Lemma~\ref{lem:coeff_lip} is a simple consequence of 
equations~\eqref{eq:polys_pf_moment_comparison}--\eqref{eq:Newton_identity} 
and is deferred to Appendix~\ref{app:pf_lem_coeff_lip}. 
In view of this result, the proof of Theorem~\ref{thm:moment_comparison}(i) reduces to bounding the (local) Wasserstein distance between
$\mu$ and $\nu$ by the corresponding coefficient distance $\|f_\mu-f_\nu\|_*$. 
This amounts to a {\it stability bound}, wherein we require the roots of a polynomial to be 
stable under variations of its coefficients.
Stability bounds of this type are classical in the numerical analysis 
literature, dating back at least to the work of~\cite{ostrowski1940}; 
see for instance \cite{galantai2008} and references therein for a survey. 
We state Ostrowski's result below, in its form appearing in~\citet[p.\,276]{ostrowski1973}. 
\begin{lemma}[\citealt{ostrowski1973}]\label{lem:ostrowski}
	Let $\Theta \subseteq \bbC$  be a bounded set and let $k\in \NN$. Then, there exists a constant $C = C(\domain,k) > 0$ such that for any $\mu,\nu\in \calU_k(\domain)$, 
\begin{equation}
\label{eq:ostrowki}
W_1^k(\mu,\nu) \leq W_k^k(\mu, \nu) \leq C \|f_\mu-f_\nu\|_*.
\end{equation}
\end{lemma} 
Theorem~\ref{thm:moment_comparison}(i) is now a direct consequence of Lemmas~\ref{lem:coeff_lip}--\ref{lem:ostrowski}. 
In order to prove claim~(ii), we require a local version of Lemma~\ref{lem:ostrowski}, 
which will allow us to reduce the exponent $k$ of equation~\eqref{eq:ostrowki} whenever the atoms of $\mu$ and $\nu$ are clustered.
Various local versions of Ostrowski's result have been derived in the literature, 
such as the bounds of~\cite{ostrowski1970}, \cite{beauzamy1999}, and~\cite{galantai2008}, which improve the exponent $k$ whenever the roots have multiplicities greater than one. 
However, we are not aware of any  existing local bounds which merely make the a priori assumption that the roots are
clustered. 
We establish such a result in the following Proposition, which may be of 
independent interest.  
\begin{proposition}\label{prop:sharp_poly_stability}
Let $\Theta \subseteq \bbC$ be a bounded set,  let
  $1\leq k_0\leq k$ and choose $r\in \NN^{k_0}$ with $|r| = k$. 
Then, there exist positive constants 
$C,c,\delta_0 > 0$  depending on $\Theta, k$ such that for all $\delta \in (0,\delta_0)$, $\mu_0 \in \calU_{k,k_0}(\Theta; r, \delta)$, and $\mu, \nu \in \calU_{k}(\Theta; \mu_0, c\delta^{k+1})$, it holds that 
\begin{equation}\label{eq:sharp_poly_stability}
\calD_{\mu_0} (\mu,\nu) \leq C \|f_\mu-f_\nu\|_*.
\end{equation}
\end{proposition}
When $k_0 = 1$, Proposition~\ref{prop:sharp_poly_stability} reduces to
Ostrowski's global result in Lemma~\ref{lem:ostrowski}, however
it provides a strictly sharper bound when the underlying roots 
of $f_\mu$ and $f_\nu$ are clustered around $k_0 > 1$ points. When the different clusters are separated by a fixed constant $\delta>0$, the exponent $k$ in Ostrowski's result improves to $r_j$ where $r_j/k$ is the mass assigned to the $j$-th cluster of $\mu_0$. Moreover, even when the clusters are collapsing, i.e., $\delta \searrow 0$, Proposition~\ref{prop:sharp_poly_stability}
still provides a sharper upper bound than Lemma~\ref{lem:ostrowski} 
as long as $\mu$ and $\nu$ tend to $\mu_0$ 
at a faster rate than  $\delta^{k+1}$. The proof of Proposition~\ref{prop:sharp_poly_stability} appears
in Appendix~\ref{app:pf_prop_sharp_poly_stability}, and is inspired by~\cite{ostrowski1973} and \cite{bhatia2007}.

We close this subsection with a black-box bound for generic estimators of $k$-atomic measures which relates the expected moment difference between an estimator and its population counterpart to the respective expected (local) Wasserstein distance. This enables us to quantify the statistical error for the method of moments estimator and the maximum likelihood estimator for our deconvolution model \eqref{eq:model} without additional case distinctions for the global and local rate. The proof is stated in \Cref{app:pf:blackbox_bound} and utilizes our global and local stability bounds from Theorem \ref{thm:moment_comparison}.

\begin{corollary}\label{cor:blackbox_bound}
	Let $d,k\in \NN$, and let $\Theta\subseteq \RR^d$  or $\Theta \subseteq \CC$ be a bounded subset. Further, let $\hat \mu \in\calU_k(\RR^d)$ or $\hat \mu \in\calU_k(\CC)$ be a corresponding estimator for $\mu\in \calU_k(\Theta)$. Let 
	$\epsilon^2  \geq	\EE\left[M_{k}^2(\hat \mu , \mu)\right]$, and 
	assume that $\hat \mu\in \calU_k(\Theta)$ almost surely or that there exists $\rho > 0$ such that $\EE\left[W_1^{2k}(\hat \mu , \mu)\mathds{1}(M_{k}(\hat \mu, \mu)\geq \rho)\right]
	\leq \epsilon^2.$
		Then, the following two assertions hold. 
		\begin{enumerate}
			\item[(i)] (Global error) There exists a  constant $C = C(\Theta,  \rho, d,k)> 0$ such that  
			\begin{align*}
				\EE\left[W_1^k(\hat \mu , \mu)\right]\leq C \epsilon. %
			\end{align*}
			\item[(ii)] (Local error) Let $1\leq k_0 \leq k$ and $r\in \NN^{k_0}$ with $|r| = k$, and define $r^* \coloneqq \max_{i = 1, \dots, k_0} r_i$ and $r_*\coloneqq \min_{i=1, \dots, k_0}r_i$. Then, there exist constants $C, c, \delta_0>0$ depending on $\Theta, d, k, \rho$ such that for all $\delta \in (\epsilon^{1/(2k(k+1))},\delta_0)$ and $\mu_0\in \calU_{k,k_0}(\Theta; r, \delta)$ it holds for $\mu \in \calU_{k}(\Theta; \mu_0, c\delta^{k+1})$ that %
			\begin{align*}
				 \delta^{k-r_*}\EE\left[ W_1^{r^*}(\hat \mu, \mu) \right]+\EE\left[\calD_{\mu_0}(\hat\mu, \mu)\right] &\leq C \epsilon. 
			\end{align*}
		\end{enumerate}
	\end{corollary}

\begin{remark}\label{rmk:Hausdorff_bound}
	The Hausdorff distance between the atoms of two $k$-atomic uniform measures is strictly weaker than the $W_1$ distance between the respective uniform measures (Lemma~\ref{lem:RelationHausdorffAndWasserstein}). This is a consequence of the fact that the Hausdorff distance does not take into account the correct multiplicity of atoms per cluster. 
	Hence, our global and local bound from Corollary \ref{cor:blackbox_bound} also remains valid when $W_k(\mu, \nu)$ is replaced by $d_{H}(\supp(\mu), \supp(\nu))$, and the global and local minimax rates in Wasserstein risk (Theorems~\ref{thm:global_minimax_risk}--\ref{thm:local_minimax_risk_nosep}) immediately yield upper bounds for the Hausdorff risk. The validity of the corresponding minimax lower bound is explained in Remark~\ref{rmk:MinimaxLowerBoundsDiscussion}$(iii)$ and proved in \Cref{app:pf:rmk:MinimaxLowerBoundsDiscussion}.
\end{remark}

Building on the error bound established for estimators of $k$-atomic uniform measures, we now proceed with the statistical analysis of the method of moments estimator.

\subsection{Statistical Performance of the Method of Moments Estimator}
\label{sec:method_of_moments}\label{subsec:method_of_moments}
Let $(X_1,\dots,X_m) \sim \bP_\mu^t$ be an observation
from model~\eqref{eq:model}. 
Our aim in this section is to derive an estimator for $\mu$ 
based on estimated moments, and to show that it 
achieves the minimax risk described in Theorems~\ref{thm:global_minimax_risk}--\ref{thm:local_minimax_risk_nosep}, for compact kernels satisfying Assumption~\ref{ass:kernel}(i).

By Lemma~\ref{lem:moment_identifiability}, the measure $\mu$ is uniquely identified by its moments up to order $k$. 
Building upon this observation, we construct an estimator of $\mu$ 
in two steps.   Using~\cref{thm:moment_polynomials} below, we first show that, under mild moment assumptions on the kernel $K$,
there exist consistent estimators based on $(X_1, \dots, X_m)$  for the moments of $\mu$ in the regime $t, m\to \infty$. Given these moments estimators, we then construct our method of moment estimator for $\mu$, inspired by the denoising approach introduced by~\cite{wu2020,wei2023}.

\begin{theorem}
\label{thm:moment_polynomials}
Let $\bbK = \RR^d$ for $d\in \NN$ or $\bbK = \CC$ (for which we set $d = 1$) and consider a probability kernel $K: \bbK \to [0, \infty)$ admitting finite multivariate
moments of all orders. Then the following assertions hold.
\begin{enumerate}
	\item For each $\alpha\in \NN^d_0$, there exists a unique monic 
	 polynomial $\psi_\alpha\colon \bbK\to \bbC$, which is real-valued
	 when $\bbK = \bbR^d$, and
	 with leading term $x^\alpha$, such that for every $\mu \in \calP(\bbK)$ with 
	 finite moment of order $\alpha$, it holds that 
	\begin{align*}
	\EE_{V\sim K\ast \mu}[\psi_\alpha(V)] = \EE_{U\sim  \mu}[U^\alpha]. %
	\end{align*}
	\item If $\bbK = \bbC$ and $K$ is rotationally symmetric around the origin, then
	for each $\alpha\in \bbN_0^d$, it 
	holds that $\psi_\alpha(z) = z^\alpha$ for each $z\in \CC$ and in particular $\EE_{V\sim K\ast \mu}[V^{\alpha}] = \EE_{U\sim  \mu}[U^\alpha]$. 
\end{enumerate}
\end{theorem}

Our proof of~\cref{thm:moment_polynomials} 
appears in Appendix~\ref{app:pf_thm_moment_polynomials}, and 
provides a constructive procedure for obtaining the polynomial $\psi_\alpha$ 
based on all moments  of $K$ up to order~$|\alpha|$. 
In Appendix~\ref{subsec:implementationPolynomials}, we describe a numerical implementation of this procedure for the settings $\KK \in \{\RR, \CC\}$,
which can be used to derive the polynomials
$\{\psi_i\}$ corresponding to any sufficiently concentrated kernel $K$.

Special cases of this result are well-known
in the literature. For instance, on the real line when~$K$ is a standard Gaussian kernel, the 
 (probabilist's) Hermite polynomials 
\begin{equation}
\label{eq:hermite_system}
H_j(z) = (-1)^j e^{z^2/2} \frac{d^j}{dz^j} e^{-z^2/2}, \quad j=1,2,\dots,
\end{equation}
form an orthogonal basis of the space~$L^2(Kd\calL)$, 
and are the unique collection of polynomials satisfying  the properties of 
\cref{thm:moment_polynomials}, see \cite{wu2020_book}.  
Similar polynomial systems have been derived when~$K$ belongs to a natural exponential family
with quadratic variance~\citep{morris1982,morris1983}; see~\cite{wei2023}
for an overview.
A  result akin to \cref{thm:moment_polynomials} for general kernels $K$
may already exist  in the literature, but we are unaware of a reference.
This level of generality allows us to apply our methodology 
to general classes of kernels, which is important for microscopy applications. %

Theorem~\ref{thm:moment_polynomials} implies that, given a random 
vector $(X_1,\dots,X_m) \sim \bP_\mu^t$ arising from model~\eqref{eq:model}, 
the random variables
\begin{align}\label{eq:moment_estimator}
\hat m_\alpha =  \sum_{i=1}^m \psi_\alpha(\gamma_i)\frac{X_i}{t}
, \quad \alpha \in \NN_0^d, |\alpha|\leq k,
\end{align}
with $\gamma_i\in B_i\subseteq \RR^d$ fixed but otherwise arbitrary,
are suitable estimators for the moments up to order $k$ of $\mu$. Our \emph{general method of moment estimator} for $\mu$ is then defined as
	\begin{align}\label{eq:method_of_moments_general}
		 \hat \mu_{t,m,\Theta}^{\mathrm{MM}} \in \argmin_{\tilde \mu \in\calU_k(\Theta)}\sum_{\substack{\alpha\in \NN_0^d, |\alpha|\leq k}}|m_\alpha(\tilde \mu) - \hat m_{\alpha}|^2  \quad \text{ if } \Theta \subseteq \RR^d, d \in \NN.
	\end{align}
By compactness of $\Theta$, the collection $\calU_k(\Theta)$ of $k$-uniform measures is compact in the $W_1$-topology and, hence, by continuity of the objective function, a minimizer in equation~\eqref{eq:method_of_moments_general} always exists. Furthermore, by identifiability of $k$-atomic uniform measures in terms of moments up to order \(k\) (Lemma \ref{lem:multivariateIdentifiability}), it follows that if $\hat{m}_{\alpha}$ closely approximates the moment $m_{\alpha}(\mu)$ for all
$|\alpha| \leq k$, then $\hat{\mu}^{\mathrm{MM}}$ is close to $\mu$. 
This justifies our construction. 
Nevertheless, under additional separation conditions on the support points of $\mu$ along a (known) principal axis, fewer moment estimators would suffice to identify a $k$-atomic mixture 
(cf.\,\citet[Lemma 1]{bing2024} and \cite{lindsay1993multivariate}).

Moreover, in the one- and two-dimensional setting, the latter being particularly relevant in our microscopy context, we may instead interpret the set $\Theta$ as a subset of $\CC$. The benefit of this perspective is that it allows us to use exactly $k$ complex moment estimators in equation~\eqref{eq:method_of_moments_general}, for which a unique solution always exists. Specifically, the following observation reveals that the complex method of moment estimator is computable by finding the roots of a suitable monic polynomial of degree $k$,
without needing to solve the minimization problem~\eqref{eq:method_of_moments_general}.

\begin{lemma}\label{lem:moment_to_measure}
Let $k\in \NN$, and $m=(m_1,\dots,m_k) \in \bbC^k$. Then, there exists a unique
measure $\mu\in \calU_k(\bbC)$ such that $\int z^j d\mu(z) = m_j$ for all $j\in \{1, \dots, k\}$. 
The atoms of $\mu$ are the $k$ roots of the polynomial 
\begin{align*}%
	\poly_m(z) =  z^k + \sum_{j=1}^k %
	(-1)^{j}\symm_{j} z^{k-j},\quad z\in \bbC,
\end{align*}
where $\{\symm_j\}_{j=0}^k$ are recursively defined by $\symm_0 = 1$ and $\symm_l = \symm_l(m)
= \frac{k}{l} \sum_{j=1}^{l} (-1)^{j-1} \symm_{l-j}(m)   m_j$. 
 Furthermore,
 there exists a positive constant $C=C(k)>0$ such that 
  $\supp(\mu)\subseteq B(0,C\|m\|_\infty^k)$. 
\end{lemma}
Hence, for $\Theta \subseteq \RR^2 \simeq \CC$ and moment estimators $(\hat m_1, \dots, \hat m_k)\in \CC^k$ defined in \eqref{eq:moment_estimator} with suitable $\gamma_i \in B_i$ for $i = 1, \dots, m$, we denote by $(\hat \theta_1,\dots,\hat\theta_k)\in \CC^k \simeq (\RR^2)^k$ the collection of $\RR^2$-embedded complex roots of $\poly_{\hat m}$. The \emph{complex method of moments estimator} of $\mu$  is then defined~as  
\begin{align}\label{eq:method_of_moments_complex}
	 \hat\mu_{t,m,\CC}^{\mathrm{MM}} \coloneqq \frac 1 k \sum_{i=1}^k \delta_{\hat\theta_i}.
\end{align}
Likewise, for the real line setting $\Theta \subseteq \RR\subseteq \CC$ we may define the \emph{real method of moments estimator} as the push-forward under the projection of the complex method of moment estimator onto the reals, 
\begin{align}
	\label{eq:method_of_moments_real_line}
	\hat\mu_{t,m,\RR}^{\mathrm{MM}} \coloneqq \frac 1 k \sum_{i=1}^k \delta_{\Re(\hat\theta_i)}.
\end{align} 
We point out that the real and complex method of moments estimators from \eqref{eq:method_of_moments_complex}-\eqref{eq:method_of_moments_real_line} may differ from the general estimator in \eqref{eq:method_of_moments_general} as their atoms may not be contained within $\Theta$. Nonetheless, they all achieve similar statistical error bounds, as shown next. 
\begin{proposition}\label{prop:mm_rate}
	Let the sets $\Omega,\Theta$ satisfy Assumption~\ref{ass:bins}. Let $K$ be a compactly-supported kernel satisfying Assumption \ref{ass:kernel}$(i)$. 
	Then, there are constants $C = C(\Theta,\Omega,d,K,k) > 0$,  $\rho = \rho(\Theta,d,K,k) > 0$ such that the method of moment estimator $\hat \mu_{t,m}=\hat \mu^{\mathrm{MM}}_{t,m, \Theta}$ and the complex method of moment estimator $\hat \mu_{t,m}=\hat \mu^{\mathrm{MM}}_{t,m, \CC}$ (for $ d= 2$)
	satisfies
	\begin{align}\label{eq:mm_rate}
		\sup_{\mu\in \calU_k(\Theta)}\EE\left[ M_k^2(\hat\mu_{t,m},\mu)  + W_k^{2k}(\hat \mu_{t,m}, \mu)\mathds{1}(M_{k}(\hat \mu_{t,m}, \mu)\geq \rho)\right] \leq C 
	\Big(t^{-1/2} + m^{-1/d}\Big)^2.
	\end{align}
\end{proposition}

The proof of Proposition \ref{prop:mm_rate} is stated in Appendix \ref{subsec:proof:mm_rate} and relies on quantitative error bounds for the moment estimators. We discuss some aspects to our analysis in the following remark. 

\begin{remark}\label{rmk:discussionMoM}
	\begin{enumerate}
		\item Together with our black-box bound (Corollary \ref{cor:blackbox_bound}),
		we deduce from Proposition~\ref{prop:mm_rate}
that the method of moments achieves the upper bounds of Theorems~\ref{thm:global_minimax_risk}--\ref{thm:local_minimax_risk_separated} for compactly-supported
kernels satisfying Assumption~\ref{ass:kernel}$(i)$. In addition, we also show in Appendix \ref{subsec:proof:mm_rate} that the real method of moment estimator, defined as a real projection of its complex counterpart, fulfills an analogous upper bound with $d = 1$.  
Thus, the method of moment estimators are minimax optimal for regular kernels
when $m\gtrsim t^{d/2}$. Notably, 
the condition of $h$-regularity in Assumption~\ref{ass:kernel}$(i)$
is not needed for the upper bound, though faster estimation rates could be possible for these~cases. 
		\item The central computational challenge for the method of moment estimator for the univariate and bivariate setting is to find the complex roots of the polynomial $\poly_{\hat m}$ defined in Lemma~\ref{lem:moment_to_measure}. %
		 This is a well-studied subject in numerical optimization, e.g., \cite{McNamee2007Numerical, mcnamee2013numerical} with iterative algorithms by \cite{jenkins1972algorithm} in Fortran and \cite{zeng2004algorithm} in Matlab or in the \texttt{scipy}-Python-package. %
		For the multivariate setting, we suggest employing general gradient descent procedures with different initializers to solve \eqref{eq:method_of_moments_general}.
		 In simulations we observe that the method of moments estimator 
		 can be computed extremely fast, see \Cref{sec:simulations}.
		\item Our analysis could be extended to sub-exponential  kernels with unbounded support by enlarging the observation domain $\Omega$ proportional to $\log(t+1)$. This analysis would involve a slight reformulation of Assumption \ref{ass:bins} to $\max_{i = 1, \dots, m}\diam(B_i)\leq C(\Omega_m) \log(t+1)m^{-1/d}$ to capture the enlarged domain. However, this comes at the cost of poly-logarithmic inflation of the  discretization error $m^{-1/d}$. Without domain enlargement, simple moment estimators as in \eqref{eq:moment_estimator} will be  inconsistent, causing the complex method of moments estimator to be inconsistent as well according to Lemma \ref{lem:moment_to_measure}. %
	\end{enumerate}
\end{remark}

\section{Maximum Likelihood Estimation}
\label{sec:mle}

Thus far, we have derived upper bounds on the risk of the method of moments estimator
for compactly-supported  kernels $K$,  assuming
that the atoms of $\mu$ are well-separated from the boundary of the observation
domain~$\Omega$. Under this condition, the convolution
$K\star\mu$ is itself supported  
within~$\Omega$, a property which allowed us to construct unbiased estimators
of the moments of $K\star\mu$, and hence of $\mu$. 
Although this analysis captures many situations of interest, 
it does not
extend to cases where~$K\star\mu$ is not supported within the image domain
$\Omega$, and hence precludes kernels with unbounded support.
Our goal in this section is to show, 
perhaps surprisingly, that optimal estimation of $\mu$
is still possible. We focus specifically on the Gaussian kernel~$K$, and
show that the maximum likelihood estimator (MLE) of $\mu$ is minimax optimal. 
Our upper bounds  also extend to compactly-supported kernels
of the type considered in the previous section, at the expense of polylogarithmic
factors. 

Given a probe space $\Theta \subseteq \bbR^d$, 
an image space $\Omega \subseteq \bbR^d$, 
and a uniform measure $\mu\in \calU_k(\Omega)$, 
let $(X_1,\dots,X_m) \sim \bP_\mu^t$ be an observation drawn from model~\eqref{eq:model}. 
The MLE $\hat\mu_{t,m}^{\mathrm{MLE}} =: \hat\mu_{t,m}$ 
of $\mu$ is defined as any solution to the optimization problem 
\begin{align}\label{eq:MLE}
	\argmax_{\tilde \mu\in \calU_k(\domain)}  \sum_{i=1}^m \Big(X_i \log \big(tK\star\tilde\mu(B_i)\big) - t K\star\tilde\mu (B_i)\Big).
\end{align} 
Notice that the log-likelihood function in the above display indeed admits a maximizer, since
it is continuous, and $\calU_k(\Omega)$ is compact, in the~$W_1$ topology. 
If it admits more than one maximizer, 
$\hat\mu_{t,m}$ may be chosen arbitrarily among them.
The main result of this section is stated as follows. 
\begin{proposition} 
	\label{prop:mle_gaussian}
	Let $k\in \NN$, and let the sets $\Omega,\Theta$ satisfy Assumption~\ref{ass:bins}. 
Then, the following holds.
\begin{enumerate}
\item (Gaussian kernel) 
	Let $K$ be a Gaussian kernel satisfying Assumption~\ref{ass:kernel}(ii).
	Then, there exists a constant $C = C(\Theta,\Omega,d, K,k) > 0$ such that
$$\sup_{\mu\in \calU_k(\Theta)} \bbE_\mu M_k^2(\hat\mu_{t,m},\mu) \leq C 
\Big(t^{-1/2} + m^{-1/d}\sqrt{\log m}\Big)^2.$$
\item (Compactly-supported kernels) 
	Let $K$ be a compactly-supported kernel satisfying Assumption~\ref{ass:kernel}(i).
	Then, there exists a constant $C = C(\Theta,\Omega,d,K,k) > 0$ such that
$$\sup_{\mu\in \calU_k(\Theta)} \bbE_\mu M_k^2(\hat\mu_{t,m},\mu) \leq C 
\Big((\log t/t)^{1/2} + m^{-1/d}\Big)^2.$$
\end{enumerate}
\end{proposition}
By combining Proposition~\ref{prop:mle_gaussian}(i) with
the black-box bound of Corollary~\ref{cor:blackbox_bound},
we deduce that the MLE $\hat\mu_{t,m}$ achieves the upper bounds stated
in Theorems~\ref{thm:global_minimax_risk}--\ref{thm:local_minimax_risk_nosep} for Gaussian kernels,
under the Wasserstein distance $W_1$ or the local Wasserstein divergence $\calD_{\mu_0}$,
provided $m\geq t^{1/d + \gamma}$ for an arbitrarily
small constant $\gamma > 0$.
By Proposition~\ref{prop:mle_gaussian}(ii), the same conclusion holds
for compactly-supported kernels, at the expense of polylogarithmic (in $t$)
factors, though we emphasize that the method of moments is also
applicable to these kernels (and does not involve redundant polylogarithmic
factors).

In the remainder of this section, we describe
the main elements involved in the proof of Proposition~\ref{prop:mle_gaussian}$(i)$.
We also discuss some elements of Proposition~\ref{prop:mle_gaussian}$(ii)$, deferring
most details to Appendix~\ref{app:pfs_mle}.
The proof consists of three steps, which are organized as follows.
\begin{itemize}
\item In Section~\ref{sec:mle_moment_L2}, we prove that the moment distance $M_k$ 
over the space of finite mixing measures is bounded from
above by the $L^2(\bbR^d)$ distance between the corresponding densities.
In particular, it will follow that
$$M_k(\hmu_{t,m},\mu) \lesssim \|K\star(\hmu_{t,m}-\mu)\|_{L^2(\bbR^d)},$$
for both types of kernels under consideration.
\item In Section~\ref{sec:mle_Lp_equiv}, we show that for Gaussian kernels, the $L^2(\bbR^d)$
distance on the right-hand side of the above display can in fact be replaced
by the $L^2(\Omega)$ norm. This fact is useful because the 
Poisson deconvolution model~\eqref{eq:model}
involves observations of the {\it truncated} density $K\star\mu |_\Omega$, 
rather than the full density $K\star\mu$. We additionally show that
all $L^p(\Omega)$ 
norms are equivalent over the space of Gaussian mixtures with at most $k$ components across values of $p\in[1,\infty]$.
\item In Section~\ref{sec:mle_bracketing}, focusing again on Gaussian kernels, we  
show that the $L^2(\Omega)$ risk of $K\star\hat\mu_{t,m}$ is 
characterized by the solution to a Le Cam-type equation, involving
the local $L^2(\Omega)$ bracketing integral of the class
of Gaussian mixture densities. We  then 
sharply bound this bracketing integral,
using the $L^p(\Omega)$ norm equivalence
proven in the previous item.
\end{itemize}
The uniformity of $\mu$ does not play a
role in any of these steps, thus the majority of our intermediary results
will be stated for arbitrary finite mixing measures $\mu\in\calP_k(\Omega)$
in the interest of generality. Uniformity
merely plays a role in translating the moment bound of Proposition~\ref{prop:mle_gaussian}
into the (local) Wasserstein bounds of Theorems~\ref{thm:global_minimax_risk}--\ref{thm:local_minimax_risk_separated}, via the stability bound of Theorem~\ref{thm:moment_comparison}.

\subsection{Moment-$L^2(\bbR^d)$ Comparison Bounds} 
\label{sec:mle_moment_L2}
The following result relates the moment distance $M_k$ over the space
of mixing measures to the corresponding $L^2(\bbR^d)$ distance
between the mixture densities.  
\begin{theorem}\label{thm:momentL2-bounds_all}
Let $k\in \bbN$, and let  $\Theta$ be a bounded set contained in  
$B_\infty(0,r)$ for some $r\geq 1$. Then, the following assertions hold.
	\begin{enumerate}
		\item (Gaussian kernels) If  $K$ is a Gaussian density satisfying
		Assumption~\ref{ass:kernel}(ii), then
		 there exists a constant $C= C(\Sigma,d,k,r)>0$ such that
		for any $\mu,\nu\in\calP_k(\Theta)$, 
		\begin{align*}
			M_k(\mu, \nu) \leq M_{2k-1}(\mu,\nu) \leq C\|K\ast (\mu - \nu)\|_{L^2(\RR^d)}.
		\end{align*}
		\item (Compactly-supported kernels) If $K\in L^2(\bbR^d)$ is compactly-supported in $B_\infty(0,1)$, 
		then there exists a positive constant $C=C(d,K,k)>0$ such that 
		for all $\mu,\nu\in\calP_k(\Theta)$, 
		\begin{align*}
M_k(\mu, \nu) \leq M_{2k-1}(\mu,\nu) \leq C  r^{2k-1+d/2} \|K\ast (\mu - \nu)\|_{L^2(\bbR^d)}.
		\end{align*}
	\end{enumerate}
\end{theorem}

The proof of both assertions
 relies on expressing the 
respective $L^2$-norm on the right-hand side in terms 
of suitable orthonormal systems. 
For assertion~$(i)$, we use as 
our orthonormal system the collection of Hermite polynomials 
weighted with a Gaussian kernel, a technique 
which has been used since the works 
of~\cite{lepski1999,ingster2001,cai2011} for   bounding
statistical distances between Gaussian mixtures. 
We also emphasize that~\citet[Lemma~4.4]{doss2023} establish a related inequality,
which coincides with assertion~$(i)$ in the univariate case. 
To prove assertion $(ii)$, we instead consider the Legendre 
polynomials, which form the classical orthogonal polynomial family with respect
to the indicator function of a compact interval (see \Cref{app:orthoPolynomials}). 
This polynomial family was also recently used by~\cite{kim2024a} 
to establish lower bounds on the Hellinger distance between
scale mixtures of continuous uniform distributions. 
A detailed proof of Theorem~\ref{thm:momentL2-bounds_all}
appears in Appendix~\ref{app:pf_thm_momentL2-bounds_all}.

By combining Theorem~\ref{thm:momentL2-bounds_all} with 
the moment comparison bound of Theorem~\ref{thm:moment_comparison}
for uniform finite mixing measures,
and its analogue for general finite mixing measures~\citep{wu2020}, we 
deduce that there exists a constant $C' > 0$ such that
\begin{align}
\label{eq:W_L2_unif}
W_1^k(\mu,\nu) &\leq C' \|K\star(\mu-\nu)\|_{L^2(\bbR^d)},\quad \text{for all }
\mu,\nu\in\calU_k(\Theta),\\*  
\label{eq:W_L2_general}
W_1^{2k-1}(\mu,\nu) &\leq C' \|K\star(\mu-\nu)\|_{L^2(\bbR^d)},\quad \text{for all }
\mu,\nu\in\calP_k(\Theta).
\end{align}
While equation~\eqref{eq:W_L2_unif} is new, 
bounds such as equation~\eqref{eq:W_L2_general}
are at the heart of many past analyses of maximum likelihood
estimation in finite mixture models.
In particular,~\cite{heinrich2018} proved a version of equation~\eqref{eq:W_L2_general} 
when $d=1$, with the $L^2(\bbR^d)$ distance replaced by the Kolmogorov-Smirnov distance.
\cite{wei2023} also proved a similar bound for general dimension $d$ 
with the $L^2(\bbR^d)$ distance replaced by the Total Variation distance.
We will see in the following subsection that the $L^2(\bbR^d)$ and Total Variation distance
are in fact equivalent
 for finite Gaussian mixture models, thus equation~\eqref{eq:W_L2_general}
recovers the result of \cite{wei2023}
for the 
important case of the Gaussian kernel. 
On the other hand, their results do not
apply to the compactly-supported kernels considered in Theorem~\ref{thm:momentL2-bounds_all}(ii), since they preclude mixture families
whose  mixture components can have non-overlapping support.

We also emphasize that
 the works of~\cite{heinrich2018,wei2023} are based on subtle asymptotic arguments, involving 
 Taylor expansions of the mixture parameters ordered according to
 a so-called ``coarse-graining'' ultrametric tree. 
 While such arguments have the advantage of being applicable to rather general
 mixture density families 
 (known as {\it strongly identifiable} mixtures, 
 which need not be of convolution type), their asymptotic nature prohibits 
 an explicit dependence of the constant~$C'$ on the various problem parameters.
 In contrast, our proof of the above equations is comparatively simple, and is based on entirely non-asymptotic
 arguments. The scaling of~$C'$ can therefore be made explicit, in principle. 
 For instance, when $d = 1$, a direct computation (which we provide in Appendix~\ref{app:pf_remark_constants}) 
 yields that the constant~$C$ in the statement of 
Theorem~\ref{thm:momentL2-bounds_all} is bounded from above by 
 $C(k)(1+ \int|x|^{2k+1} K(x) \dif x)^{2k}$. 
 Further, for $d = 2$ and a rotationally invariant kernel that is compactly-supported within $B(0,\tau)$, we also confirm in Appendix~\ref{app:pf_remark_constants} that,
  for $\mu, \nu \in \calP(\CC)$ supported within $B(0,r)$,  the $\ell_2$-moment distance is bounded above~as: 
		\begin{align*}
			M_{2k-1}(\mu, \nu)\leq \sqrt{\pi(2k-1)} \left((r+\tau) \vee (r+\tau)^{2k}\right) \|K\ast (\mu - \nu)\|_{L^2(\RR^2)}.
		\end{align*}
In particular, such explicit constants in terms of the kernel enable the use of truncation arguments to extend~\Cref{thm:momentL2-bounds_all}$(ii)$ to general sub-exponential  kernels at the cost of an additional logarithmic term on the right-hand side. 
  
\subsection{$L^p$ Equivalences for Finite Gaussian Mixtures}
\label{sec:mle_Lp_equiv}

We now focus our attention on Gaussian kernels $K$
satisfying Assumption~\ref{ass:kernel}(ii).
Our first goal in this subsection is to show that the $L^2(\bbR^d)$ 
norm appearing in Theorem~\ref{thm:momentL2-bounds_all}(i)
can in fact be controlled by 
the restricted norm $L^2(\Omega)$, 
over the compact observation domain $\Omega\subseteq \bbR^d$.
We will then also deduce some additional equivalences
between divergences on the space of finite Gaussian mixtures, which 
are of independent interest.
Our arguments will stem from the following Proposition.
\begin{proposition}
\label{prop:Lp_to_lp}
	Let $d,k\in \NN$ and let $K$ be a Gaussian kernel satisfying Assumption~\ref{ass:kernel}(ii). 
	Furthermore, given a bounded set $\calA\subseteq \RR^d$,
	consider the regular grid with anchor point $a\in \calA$ and scaling $s>0$, 
	\begin{align}\label{eq:grid}
		\calG_a\coloneqq  \calG_a(\Sigma, d,k,s) \coloneqq \left\{ a + \frac{s}{2k}\Sigma i \;\colon  i\in \{0, \dots, 2k-1\}^d \right\}.
	\end{align}
	Then, for every $p,q\in [1,\infty]$ and compact $\domain \subseteq \RR^d$  there is a constant $C=C(\domain, \Sigma, \calA, d, k, p,q,s)>0$ such that for all $\mu,\nu \in \calP_k(\domain)$ and $a\in \calA$ it holds 
	\begin{align*}
		\|K\ast (\mu - \nu)\|_{L^p(\RR^d)} \leq C \|K\ast (\mu - \nu)\|_{\ell^q(\calG_a)}.
	\end{align*} 
\end{proposition}
Proposition~\ref{prop:Lp_to_lp} shows, somewhat surprisingly, that
the $L^p(\bbR^d)$ metrics over the space of finite Gaussian mixture densities
are bounded by the distance between finitely many density evaluations, a~fact which is fundamentally tied to the structure of the Gaussian kernel.
The proof of Proposition~\ref{prop:Lp_to_lp} appears in \Cref{app:pf_prop_Lp_to_lp}, where
our main insight is the fact that the evaluations of 
Gaussian mixture densities, where $\Sigma$ is the covariance of the Gaussian kernel, can be expressed as follows
$$K\star\mu(\Sigma x) = m_x(\tilde\mu),\quad
\text{where }  \quad 
	\tilde \mu \coloneqq \sum_{l = 1}^{k} w_l\exp\left(  -\frac{1}{2} \theta_l^\top \Sigma^{-1}\theta_l\right) \delta_{\exp( \theta_l)},$$
for any $k$-atomic measure $\mu=  \sum_{l=1}^k w_l \delta_{\theta_i}$, 
where $m_x(\tilde\mu)$ denotes the real moment of order~$x$ of the 
measure $\tilde\mu$.
It follows that the $L^p(\bbR^d)$ norm of the difference $K\star(\mu-\nu)$
depends only on the real moments  of the $k$-atomic
measures $\tilde \mu$ and $\tilde \nu$, which we show are all 
quantitatively controlled by 
their first  $2k-1$ integer moments. 
In turn, these integer moments are controlled by finitely many evaluations 
of $K\star(\mu-\nu)$, by again leveraging the above identity. 
 
As a first implication of 
Proposition~\ref{prop:Lp_to_lp}, 
notice carefully that the constant $C$ in Proposition~\ref{prop:Lp_to_lp}
does not depend on the anchor point $a\in \calA$. The bound thus continues
to hold for a {\it random} anchor point $a$, uniformly distributed in $\calA$. 
By choosing $\calA$ appropriately, and taking an expectation over such a random anchor point, we obtain the following. 
\begin{proposition}
\label{cor:Lp_Omega_Rd}
Assume the setting of Proposition \ref{prop:Lp_to_lp}. 
Further, let $\Omega\subseteq \RR^d$ be a set with non-empty interior. Then, for any $p,q\in [1,\infty]$, there is a constant $C = C(\Theta, \Omega, \Sigma,d,k,p,q) > 0$ such that %
for any $\mu,\nu\in\calP_k(\Theta)$ it holds  
$$\|K\star(\mu-\nu)\|_{L^p(\bbR^d)}\leq C \|K\star(\mu-\nu)\|_{L^q(\Omega)}.$$
\end{proposition}
The proof appears in Appendix~\ref{app:pf_cor_Lp_Omega_Rd}. 
Combining this result with Theorem~\ref{thm:momentL2-bounds_all},
we deduce that the moment distance $M_{2k-1}$ between mixing measures
is bounded from above by the 
$L^2(\Omega)$ distance between the corresponding mixture densities, 
a fact which will play an important role in obtaining the convergence
rate of the MLE in the following subsection.
Before turning to this, however, we discuss another implication of 
Proposition~\ref{prop:Lp_to_lp}. 
Proposition~\ref{cor:Lp_Omega_Rd} shows that for any $\mu,\nu\in\calP_k(\Theta)$
and $p,q \in [1,\infty]$, one has 
$$\|K\star(\mu-\nu)\|_{L^p(\bbR^d)} \lesssim 
\|K\star(\mu-\nu)\|_{L^q(\Omega)} \leq \|K\star(\mu-\nu)\|_{L^q(\bbR^d)},$$
which implies that the $L^p(\bbR^d)$ metrics are equivalent over the space
of finite Gaussian mixtures. By leveraging 
this identity together with Theorem~\ref{thm:momentL2-bounds_all}, 
we arrive at the following chain of equivalences, 
some of which were already announced in 
Corollary~\ref{cor:TV_Hellinger_equiv} of the introduction.
\begin{theorem}[Equivalence of divergences between Gaussian mixtures]
\label{thm:equivalences}
Let $k\in \NN$, and let $K$ be a Gaussian kernel satisfying Assumption~\ref{ass:kernel}(ii). 
Let $\Theta \subseteq \bbR^d$ be a compact set with nonempty interior. 
Then, for any $p \in [1,\infty]$, it holds that
\begin{equation}
\|K\star(\mu-\nu)\|_{L^p(\bbR^d)} \asymp H(K\star\mu,K\star\nu) \asymp M_{2k-1}(\mu,\nu),
\quad \text{for all } \mu,\nu\in\calP_k(\Theta).
\end{equation}
In particular, it holds that
\begin{equation}
\mathrm{TV}(K\star\mu,K\star\nu) \asymp H(K\star\mu,K\star\nu),
\quad \text{for all } \mu,\nu\in\calP_k(\Theta).
\end{equation}
In each of the above displays, the implicit constants depend only on $\Sigma, \Theta,k,p$.
\end{theorem}
The proof appears in Appendix~\ref{app:pf_thm_equivalences}. 
As we already discussed in the introduction,
 Theorem~\ref{thm:equivalences} complements the work of~\cite{jia2023}, 
 who establish the equivalences
\begin{equation}
\label{eq:equiv_zeyu}
H^2(K\star\mu,K\star\nu) \asymp \KL(K\star\mu,K\star\nu) 
\asymp \chi^2(K\star\mu,K\star\nu),
\quad \text{for all } \mu,\nu \in \calP(\Theta).
\end{equation}
Furthermore,~\cite{doss2023} showed that the above divergences
are equivalent to $M_{2k-1}^2(\mu,\nu)$  under the additional
restriction that $\mu$ and $\nu$ are $k$-atomic. 
\cite{jia2023} posed the open question of whether the squared $\TV$ distance 
is also equivalent to the above divergences. Theorem~\ref{thm:equivalences}
resolves this problem for the case of finite Gaussian mixtures, and additionally
shows that all squared $L^p(\bbR^d)$ distances are  on this same scale. 
Let us emphasize, however, that our proofs rely heavily on the finiteness
of the supports of $\mu,\nu$, and the implicit constants in Theorem~\ref{thm:equivalences}
may depend exponentially on $k$. Thus, the result cannot easily be extended
to infinite Gaussian mixtures. In fact,   it cannot hold true
in such generality: it was pointed out by~\cite{jia2023}
that the $L^2(\bbR^d)$ and Hellinger distances
are not equivalent  over the space of infinite Gaussian mixture models.

\subsection{Convergence Rate of the MLE: Proof of Proposition~\ref{prop:mle_gaussian}}
\label{sec:mle_bracketing} 
We are now in a position to complete the proof of Proposition~\ref{prop:mle_gaussian}. 
We will prove assertion~$(i)$, concerning Gaussian kernels, and we defer the proof of assertion~$(ii)$ to Appendix~\ref{app:pf_mle_ii}. 
In view of Theorem~\ref{thm:momentL2-bounds_all} and Proposition~\ref{cor:Lp_Omega_Rd},
the MLE $\hat\mu_{t,m}$ satisfies the bound
\begin{equation}
\label{eq:pf_prop_mle_step}
\bbE M_{2k-1}^2(\hat\mu_{t,m}, \mu) \lesssim 
\bbE\|K\star (\hat\mu_{t,m}-\mu)\|^2_{L^2(\bbR^d)}
\lesssim \bbE\|K\star (\hat\mu_{t,m}-\mu)\|^2_{L^2(\Omega)},
\end{equation}
thus we are left with bounding the convergence rate of the 
fitted mixture density $K\star \hat\mu_{t,m}$ under the   $L^2(\Omega)$ norm. 
We prove such an upper bound next, beginning with several definitions. 
Given a function class $A \subseteq L^2(\Omega)$, 
 define the bracketing integral
$$\calJ_{[]}(\delta,A,L^2(\Omega)) = \int_0^\delta \left(\sqrt{\log N_{[]}(u,A,L^2(\Omega))}\vee 1\right)du,\quad \text{for all } \delta > 0.$$ 
Recall that $B_1,\dots,B_m$ denotes a partition of size $m$
 of the domain $\Omega$ satisfying Assumption \ref{ass:bins}, and define the histogram operator
$\Pi_m f(x) = \sum_{i=1}^m f(B_i)I_{B_i}(x)/\lambda(B_i),$ for any $f\in L^2(\bbR^d)$. 
Finally, define the function classes
\begin{align*}
\calF &= \{K\star\mu:\mu\in\calU_k(\Theta)\},\\
\calF_m(\gamma;f_0) &= \left\{ \Pi_m f: f\in\calF, \|\Pi_m(f-f_0)\|_{L^2(\Omega)}\leq \gamma\right\},
\quad \text{for any } \gamma > 0, f_0 \in\calF.
\end{align*}
We then have the following result.
\begin{lemma}
\label{lem:chaining_ours}
Let $d,k\in \NN$, let $\Theta, \Omega$ be compact sets satisfying Assumption \ref{ass:bins}, and let $K$ be a Gaussian kernel satisfying Assumption~\ref{ass:kernel}(ii).
Then, there exist constants $C,C_0,c_0 > 0$ 
depending only on $\Omega,\Theta,\Sigma, d,k$ such that the following
assertion holds: If
there exists $\gamma_{t,m} > 0$ such that for all $\gamma \geq \gamma_{t,m}$,   
\begin{align}
\label{eq:le_cam_bracketing_condition_ours} 
\sqrt t \gamma^2 \geq 
C_0 \sup_{f_0\in\calF}\calJ_{[]}(c_0\gamma, \calF_m(\gamma;f_0),L^2(\Omega)),
\end{align}
then,
$$\sup_{\mu\in\calU_k(\Theta)}\bbE_\mu \big\|  K\star(\hat\mu_{t,m}-\mu) \big\|_{L^2(\Omega)}^2 \leq C \left(\frac 1 {\sqrt t} + \gamma_{t,m} + m^{-1/d}\right)^2.$$
\end{lemma}
Lemma~\ref{lem:chaining_ours} shows that the $L^2(\Omega)$ convergence rate of the 
estimator $K\star\hmu_{t,m}$ is governed by the $L^2(\Omega)$
bracketing entropy of the class of finite Gaussian mixture densities. 
The proof is given in Appendix~\ref{app:pf_lem_chaining_ours},
and uses the well-known fact that estimating the intensity function
$\Pi_m[K\star\mu]$ of our inhomogeneous 
Poisson point process model~\eqref{eq:model} 
is closely connected to the task of estimating the probability density
$\Pi_m[K\star\mu]/\Pi_m[K\star\mu](\Omega)$ from i.i.d.~observations~\citep{hohage2016inverse}. 
Indeed, our proof is inspired by classical work on nonparametric 
maximum likelihood density estimation~\citep{wong1995,vandegeer2000}, 
with appropriate formalism to handle the presence of
Poisson fluctuations and the lack of normalization of the
intensity functions involved. Notice that these classical density estimation
results are formulated with respect to the Hellinger loss function, 
which is equivalent to the $L^2(\Omega)$ norm in our setting (cf. Proposition~\ref{cor:Lp_Omega_Rd}
and Theorem~\ref{thm:equivalences}).

To obtain a convergence rate from Lemma~\ref{lem:chaining_ours}, it remains to   
bound the {\it local} bracketing integral appearing in the Le Cam-type
 equation~\eqref{eq:le_cam_bracketing_condition_ours}.
Although the {local} bracketing
entropy is bounded
above by the {\it global} 
bracketing entropy, which is well-studied for Gaussian mixtures~\citep{ghosal2001}, 
such upper bounds typically lead to rates containing redundant logarithmic factors.
For further discussion of this point, we refer to~\cite{doss2023},
who were the first to provide sharp bounds on the {\it local} Hellinger metric entropy
(without bracketing)
of finite Gaussian mixture families. 
Their approach was to use the equivalence $H(K\star\mu,K\star\nu)\asymp M_{2k-1}(\mu,\nu)$
discussed in the previous subsection,
which reduces the problem of computing Hellinger
{\it covering} numbers for Gaussian mixture densities to that of computing 
Euclidean covering numbers over the moment space, which in turn
can be done by elementary means (at least when $d$ is fixed). 
In principle, such an equivalence is  insufficient for bounding {\it bracketing} 
numbers, which require some level of pointwise control
(e.g.~\citet[Theorem 2.7.17]{vandervaart2023}). 
Fortunately, however, we have shown in Theorem~\ref{thm:equivalences}
that the $L^2(\Omega)$ distance is equivalent to the $L^\infty(\bbR^d)$ distance
for finite Gaussian mixtures, thus the local $L^2(\Omega)$-bracketing
entropies involved in Lemma~\ref{lem:chaining_ours} are actually equivalent 
to $L^\infty(\bbR^d)$ bracketing entropies, which are simply
$L^\infty(\bbR^d)$ metric entropies. 
Since the $L^\infty(\bbR^d)$ distance is further equivalent to the moment distance
$M_{2k-1}$, we can bound these $L^\infty(\bbR^d)$ metric entropies using 
a similar strategy as~\cite{doss2023}, with appropriate modifications to account
for the discretization induced by the operator $\Pi_m$.
These considerations ultimately lead to the following. 
\begin{lemma}
\label{lem:local_bracketing}
Let  $\Sigma\in\bbR^{d\times d}$
be a positive definite matrix, and let $K$ be the $N(0,\Sigma)$ density. 
Then, there exists constants $c_1, C_1>0$ that only depend on $\Omega,\Theta, \Sigma, d,k$ such that
for any $u,\gamma > 0$, 
$$\sup_{f_0\in\calF}\log N_{[]}(u, \calF_m(\gamma;f_0),L^2(\Omega)) \leq C_1
\begin{cases}
\gamma/u,  & \gamma > u\geq c_1 m^{-1/d},\\
\log(1/u), & \mathrm{otherwise}.
\end{cases}$$
\end{lemma}
The proof appears in Appendix~\ref{app:pf_lem_local_bracketing}.
In particular, it follows from here that%
\begin{align*}
\sup_{f_0\in\calF}\calJ_{[]}(c_0\gamma,\calF_m(\gamma;f_0),L^2(\Omega))  
 \lesssim m^{-1/d}\sqrt{\log m} + \gamma,~~\text{for all } \gamma >0,
\end{align*}
thus, the Le Cam-type equation~\eqref{eq:le_cam_bracketing_condition_ours}
is solved for all 
$$\gamma \gtrsim \gamma_{t,m}:= t^{-\frac 1 2} + \sqrt{t^{-\frac 1 2}m^{-1/d}\sqrt{\log m}}.$$
Continuing from equation~\eqref{eq:pf_prop_mle_step}, and applying Lemma~\ref{lem:chaining_ours}
together with the above display, we finally obtain
$$\bbE M_k^2(\hat\mu_{t,m}, \mu)
\lesssim \bbE\|K\star(\hat\mu_{t,m}-\mu)\|_{L^2(\Omega)}^2
\lesssim (\gamma_{t,m} + t^{-1/2} + m^{-1/d})^2
\lesssim (t^{-\frac 1 2} + m^{-1/d}\sqrt{\log m})^2.$$
Proposition~\ref{prop:mle_gaussian}$(i)$ thus follows.
The proof of Proposition~\ref{prop:mle_gaussian}$(ii)$
follows  along similar lines, however we do not
know of an analogue of Theorem~\ref{thm:equivalences}
for generic compact kernels $K$, thus we have not found a way to 
sharply bound the local bracketing entropies that arise in this case.
Our proof instead bounds the global bracketing entropies, and
incurs polylogarithmic factors in $t$ (which are likely redundant, and 
do not arise in our analysis of the method of moments). We defer a proof sketch to
 Appendix~\ref{app:pf_mle_ii}.

\section{Minimax Lower Bounds}\label{sec:minimaxLowerBound}

Complementary to our upper bounds on the statistical performance of the method of moments and the MLE, we now establish matching minimax lower bounds in terms of parameter $t$. More precisely, we derive global and local minimax lower bounds for the set of estimators of $\mu$, denoted by $\calE_{t,m}(\domain)$, 
i.e., the set of Borel-measurable functions $\hat\mu_{t,m}$ of 
an observation $(X_1,\dots,X_m) \sim \bP_\mu^t$  drawn from model~\eqref{eq:model}. Our main result (Proposition \ref{prop:minimax_lower_bound}) in this context is stated for $k$-regular kernels.

\begin{definition}[Regular kernels]\label{def:regularKernel}
	A Lipschitz probability kernel $K\colon \RR^d \to \RR$ is called $s$-regular for $s \in \NN$
	if there exists  $\sigma \in \mathbb{S}^{d-1}$ such that $K(x)= \tilde K((\textup{id}-\sigma \sigma^\top)x) \overline K(\sigma^\top x)$ for all $x\in \RR^d$ where $\tilde K\colon \textup{span}(\sigma)^{\perp}\subseteq \RR^d \to [0,\infty]$ and $\overline K\colon \RR\to [0,\infty)$ are probability kernels, $\overline K$ is $s$-times continuously differentiable, and  for any collection $\theta_\iota\leq \dots \leq \theta_s\in \RR$ with $\iota = \mathds{1}(s\neq 1)$, $\theta_\iota< 0$, and $\theta_s>0$ there is $\tau>0$ with 
	\begin{align*}
				 \max_{\substack{i \in \{1,\dots, s\}}}\sup_{\substack{\epsilon \in (0,\tau)}}\sup_{t\in [0,1]}\int_{\RR} \frac{\left(D^{s}\overline  K(y - t \epsilon \theta_i) \right)^2}{\sum_{j = \iota}^{s} \overline K(y - \epsilon \theta_j)}\dif y < \infty,
	\end{align*}
	where we define $0/0 = 0$ and $c/0 =+\infty$ for $c>0$. 
\end{definition}
Our notion of regular kernels requires that the kernel $K$ is sufficiently smooth and decays appropriately near the boundary of its support. In particular, it is a strictly weaker condition than the formalism of $(s,2)$-smoothness introduced by \cite{heinrich2018}, which is imposed as part of their Assumption $A(s,\theta_0)$, or Assumption 2.1 by \cite{wei2023}, as it does not rule out compact kernels. 

The following lemma provides sufficient conditions for the kernel to be $k$-regular. Its proof is detailed in Appendix~\ref{app:prf:lem:regularKernelSufficientCondition}.

\begin{lemma}\label{lem:regularKernelSufficientCondition}
	Let $\sigma \in \SS^{d-1}$ and consider probability kernels $\overline K\colon \RR\to [0,\infty)$ and $\tilde K\colon \textup{span}(\sigma)^{\perp}\subseteq \RR^d \to [0,\infty)$. Then, $K(x)\coloneqq \tilde K((\textup{id}-\sigma \sigma^\top)x) \overline K(\sigma^\top x)$ is an $s$-regular probability kernel for $s\in \NN$ if~one of the following settings are met.
	\begin{enumerate}
		\item The kernel $\overline K$ is $s$-times continuously differentiable, compactly-supported with $\supp(\overline K) = [-\gamma, \gamma]$, there exists some $\delta >0$ such that  $\min_{x\in [-\gamma + \delta, \gamma-\delta]} K(x) >0$ and on $[0,\delta)$ the functions $\overline K_-(x) \coloneqq \overline K(-\gamma + x)$ and $\overline K_+(x)\coloneqq K(\gamma-x)$ are proportional to a function in 
		\begin{align*}
			\mathcal{F}_\delta\coloneqq \left\{ f \colon [0, \delta)\to [0,\infty) \;\colon\; f(x) \propto \exp\left( -\alpha/x^\beta \right)x^\rho, \begin{array}{l}
				(\alpha,\beta,\rho)\in (0,\infty)^2\times \RR \text{ or } \\
				(\alpha,\beta,\rho)\in  \{0\}^2 \times (2s-1,\infty)
			  \end{array}\right\}.
		\end{align*}
		\item The kernel $\overline K$ is given by a centered normal density with positive variance.
	\end{enumerate}
\end{lemma} 
Based on this formalism of regular kernels, we now state our  minimax lower bounds,
whose proofs are given in~\Cref{app:pf:prop:minimax_lower_bounds}. 

\begin{proposition}
\label{prop:minimax_lower_bound}
Let 
$d,k\in \NN$ and $t> 1$. 
Let $\Theta \subseteq \RR^d$   be a set with non-empty   interior. 
Further, consider a probability kernel $K\colon \RR^d \to \RR$  
which is $h$-regular for $h=1,\dots,k$. 
Given $m\in \NN$, let $B_1, \dots, B_m$ be a partition of the detection domain $\Omega\subseteq \RR^d$. Then, the following assertions~hold.
\begin{enumerate}
	\item (Global lower bound) There exists a constant $C= C(\Theta, K, k)>0$ such that 
\begin{align*}
	\inf_{\hat \mu_{t,m}\in \calE_{t,m}(\Theta)} \sup_{\substack{\mu \in \calU_{k}(\domain)}} \EE W_1(\hat \mu_{t,m}, \mu) \geq C t^{-\frac{1}{2k}}.
\end{align*} 
	\item (Local lower bound) Let $1 \leq k_0 \leq k$, $r\in\NN^{k_0}$ with $|r| = k$, and  $c>0$. 
	Then, for any $\delta\in (t^{-1/2k},\diam(\Theta)/2k)$ 
	and every $\mu_0= \frac{1}{k}\sum_{j = 1}^{k_0} r_j\delta_{\theta_{0j}}\in \calU_{k,k_0}(\Theta; r, \delta)$ with $\supp(\mu_0)+B(0,\delta/2)\subseteq \Theta$,
	there exists a positive constant $C = C(c, \diam(\Theta),  K, k)>0$ such that 
	for all $j=1,\dots,k_0$, 
	\begin{align*}
		\inf_{\hat \mu_{t,m}\in \calE_{t,m}(\Theta)}& \sup_{\substack{\mu \in \calU_{k}(\Theta; \mu_0,c t^{-1/{2r_j}})}} \EE \calD_{\mu_0}(\hat \mu_{t,m}, \mu) \geq C  \delta_j(\mu_0)  t^{-\frac{1}{2}}.
	\end{align*}
	Furthermore, there exists a fixed choice
	of $\delta \in (0,\diam(\Theta)/2k)$, and a constant $c' > 0$, 
	both of which depend only on $\Theta,k,k_0$, for which
	there exists $\mu_0\in\calU_{k,k_0}(\Theta; r,  \delta)$ with $\supp(\mu_0)+B(0,\delta/2)\subseteq \Theta$ and $\max_j \delta_j(\mu_0) \geq c'$.  
\end{enumerate}
\end{proposition}

\begin{remark}\label{rmk:MinimaxLowerBoundsDiscussion}A few comments on the minimax lower bounds are in order. 
	\begin{enumerate} 
	\item Proposition~\ref{prop:minimax_lower_bound}$(i)$
	  implies the lower bound of Theorem~\ref{thm:global_minimax_risk}, whereas
	Proposition~\ref{prop:minimax_lower_bound}$(ii)$, with $\delta$ held constant,
	 implies the lower bounds of Theorems~\ref{thm:local_minimax_risk_separated}--\ref{thm:local_minimax_risk_nosep}.
	
	\item Contrary to the lower bound of Theorem~\ref{thm:local_minimax_risk_nosep},
	which is merely stated for the {\it worst case} choice of 
	$\mu_0 \in \calU_{k,k_0}(\Theta;r,\delta)$,
	Proposition~\ref{prop:minimax_lower_bound}$(ii)$ yields
	lower bounds for many instances of $\mu_0$. 
    The main condition in this result is for the quantity
     $\max_{j=1, \dots, k_0}\delta_j(\mu)$ to remain
      strictly positive, meaning that $\mu_0$ is required to have at least one atom which is
     distinctly separated from the others, even if the remaining atoms are allowed to converge. 
     Together with the upper bounds established in Propositions \ref{prop:mm_rate} and \ref{prop:mle_gaussian}, Proposition~\ref{prop:minimax_lower_bound}
     implies local and instance-specific    minimax risk bounds.  
     The only regime left open by this result is one in which
     all clusters are approaching each other, i.e., $\delta_j(\mu)\to 0$ for all $j=1,\dots, k_0$.  
		\item The proof of Proposition \ref{prop:minimax_lower_bound} is based on Le Cam's two-point method (see, e.g., \citealt{tsybakov2009nonparametric}). More precisely, for the local lower bound we construct two families of mixture distributions $K\ast \mu_\epsilon$ and $K\ast \nu_\epsilon$ where both $\mu_\epsilon, \nu_\epsilon$ are $\epsilon$-close to $\mu_0$ with respect to $W_{\infty}$ while fulfilling $\calD_{\mu_0}(\mu_\epsilon, \nu_\epsilon)\gtrsim \delta_j\epsilon^{r_j}$ and all (multivariate) moments of $\mu$ and $\nu$ of order up to $r_j-1$ coincide.
Using the $r_j$-regularity of the kernel $K$,
it then follows that $H^2(K\ast \mu_\epsilon, K\ast \nu\epsilon)\lesssim \epsilon^{r_j}$. The global analysis then follows by choosing $\mu_0$ as a single Dirac measure. 
Related strategies have been adopted in the previous studies of finite
mixture models~\citep{heinrich2018,wu2020}.
		\item The construction of the lower bound also yields a local minimax lower bound for the Wasserstein and the Hausdorff distance. Specifically, under the same set of assumptions as in Proposition \ref{prop:minimax_lower_bound} $(ii)$ we obtain for $r^* = \max_{j = 1, \dots, k_0}r_j$ with $C = C(c,\diam(\Theta), K,k)$ the lower bound  
		\begin{align*}
			\inf_{\hat \mu_{t,m}\in \calE_{t,m}(\Theta)} \sup_{\substack{\mu \in \calU_{k}(\Theta; \mu_0,c t^{-1/{2r^*}})}} \EE W_1(\hat \mu_{t,m}, \mu) \geq C t^{\frac{1}{2r^*}},
		\end{align*} 
and it also remains true if $W_1$ is replaced by the Hausdorff distance $d_H$ between the supports of $\hat \mu_{t,m}$ and $\mu$. 
		We provide a formal proof for both risks in  \Cref{app:pf:rmk:MinimaxLowerBoundsDiscussion}. Combined with Lemma~\ref{lem:hausdorff_to_wasserstein} this confirms that the (local)  Hausdorff minimax rate in $t$ aligns with (local) Wasserstein minimax rate.  
\item Our technique also implies minimax lower bounds for i.i.d.\ random variables $X_1, \dots, X_n$ based on $K\ast \mu$, asserting the same convergence rate as in Proposition \ref{prop:minimax_lower_bound} but with $t$ replaced by $n$, i.e., global convergence rates of atoms of order $n^{-1/2k}$ and local rates for atoms of order $n^{-1/2r}$ where $r$ denotes the maximum number of mass assigned to a cluster of $\mu_0$, see \Cref{app:pf:rmk:MinimaxLowerBoundsDiscussion}. We are only aware of the works by \cite{ho2022_fellerpaper} and \cite{wu2021} who treat the (high-dimensional) local analysis of two-component uniform Gaussian mixtures $(k = 2)$ with $k_0=1$  and $r = 2$, asserting a minimax lower bound of order $n^{-1/4}$. %
		Our Proposition \ref{prop:minimax_lower_bound} extends this to general $k$ but with fixed dimension. Notably, without the uniformity constraint on the underlying weights, estimation of a $k$-atomic mixing measure $\mu$ cannot be realized faster than $n^{-1/(4k-2)}$ and locally nearby a (fixed)  $k_0$-atomic measure $\mu_0$ no faster than $n^{-1/(4k-4k_0+2)}$ \citep{heinrich2018,wu2020,wei2023}. We also point out that \cite{ho2022_fellerpaper} established for $k= 2$ and $k_0=1$ that if the weights were known but non-uniform, then the convergence the local minimax rate is given by $n^{-1/6}$, matching the rate as if no knowledge on the weight was available. This further emphasizes the statistical benefits of the uniformity constraint. 
		\item The requirement that the kernel is sufficiently smooth is necessary in order to establish such lower bounds. For non-smooth kernels, e.g., a uniform distribution, it is known from the deconvolution literature \citep{caillerie2013deconvolution, dedecker2013minimax, dedecker2015improved} that faster rates, potentially even parametric rates, can be realized, e.g., using Fourier inversion estimators. We leave the sharp analysis of the minimax risk in terms of the number $k$ of atoms and the degree of smoothness for the kernel to future work. 

	\end{enumerate}
\end{remark}

\section{Identifiability and Consistency  under Finite Discretization}
\label{sec:mle_bounded}

Up to this point, our focus has been on statistical procedures that achieve the minimax estimation rate in the regime of increasing bin precision, i.e., when the bin count $m$ grows relative to the illumination time $t$. In this section, we examine the regime where $m$ remains fixed as $t$ increases. Specifically, we establish the consistency of the maximum likelihood estimator for specific kernels (Proposition \ref{prop:consistency_LSE_MLE_finite_m}).
This is achieved through new
 identifiability results (\Cref{thm:identifiabilityFromFunctionals}). Central to our analysis is a new condition which we call \emph{root-regularity} (see Definition \ref{def:root_regularity}), which is currently restricted to the univariate setting $d = 1$.  Extensions to higher dimensions are discussed in Remark~\ref{rmk:extensions_finiteM}. 
\begin{proposition}[Consistency]\label{prop:consistency_LSE_MLE_finite_m}
	Let $\delta>0$ be fixed. Let $\Theta, \Omega\subseteq \RR$ be sets satisfying Assumption \ref{ass:bins} and let $K\colon \RR\to (0,\infty)$ be a probability kernel such that the function $x\in \RR\mapsto \int_{x}^{x+\delta} K(t)\dif t$ is $(2k,l)$-root-regular (see Definition \ref{def:root_regularity} below) for $(k,l)\in \NN\times \NN_0$. Then, for a compact class $\tilde \calP_k(\Theta)\subseteq \calP_k(\Theta)$, every $\mu \in \tilde \calP_k(\Theta)$, and a fixed number $m\geq 2k+l$ of distinct bins of the shape $B_j=[s_j,s_j+\delta)$ for $s_1< \dots < s_{m}$ it~follows that the  maximum likelihood estimator for the model \eqref{eq:model}, 
	\begin{align*}
		\hat \mu_{t,m} \coloneqq \hat \mu_{t,m}^{\textup{MLE}} = \argmax_{\mu \in \tilde \calP_k(\Theta)}\ell_t(\mu), \quad   \ell_t(\mu) = \sum_{j=1}^m \log p_{t  \int_{B_j}K\star \mu(x)\dif x}(X_j),\quad \mu\in \calM_k(\domain)
	\end{align*}
	is a.s.\ consistent in estimating $\mu$ as $t\to \infty$, i.e., $W_1(\hat \mu_{t,m}, \mu) \to 0$ a.s.	
\end{proposition}

The proof of Proposition \ref{prop:consistency_LSE_MLE_finite_m} is provided in \Cref{app:subsec:consistency_MLE_finite_m} and relies on compactness of the probe space in conjunction with identifiability of the mixing measure in terms of $m$ distinct bin evaluations. 

This result confirms that for certain classes of kernels, the underlying mixing measure can be indeed recovered even when only finitely many bin evaluations are available. The key conceptual insight here is that the class $\mathcal{M}_k(\Theta)$ has only finite many degrees of freedom, and that the bins encode enough information about the mixing distributions. This informational sufficiency is characterized by the order $(k, l) \in \mathbb{N} \times \mathbb{N}_0$ of the root-regularity of the kernel. For example, we show $(k,l)$ root-regularity under a Gaussian kernel with $l = k+1$ (Proposition \ref{prop:rootRegularityGaussian}) while for other analytic kernels such as the Cauchy kernel that $l = 2(2k-1)+1$ (Example \ref{expl:rationalKernels}). In the following we develop our arguments concerning root-regularity and its connection to identifiability.

\begin{definition}[Root regularity]\label{def:root_regularity}
	A function $f\colon \RR\to \RR$ is called \emph{$(k,l)$-root-regular} with $k\in \NN$, $l\in \NN_0$ if for any $k$ distinct points $x_1, \dots, x_{k}\in \RR$ and coefficients $(a_1, \dots, a_{k})\in\RR^k\backslash\{0\}$ the function $t\in \RR \mapsto \sum_{i = 1}^{k} a_i f(t-x_i)$ admits at most $k+l-1$ zeros.
\end{definition}

The notion of $(k,l)$-root-regularity for $l\in \NN$ can be interpreted as a relaxation of the formalism of Tchebycheff-systems, see \cite{karlin1966tchebycheff}, and as such it exhibits strong connections to the concept of \emph{strict totally positivity} \citep{karlin1968total,gasca2013total}, \emph{P\'olya frequency functions} \citep{karlin1961moment,polya1925aufgaben}, and \emph{sign regularity} \citep{lehmann2005testing} which have extensive applications in numerics, statistics and probability theory. Indeed, according to \citet[Chapter 1, Theorem 4.1]{karlin1966tchebycheff}, if a function is $\Psi$ is $(k,0)$-root-regular, then it follows for any collection of $k$ distinct points $x_1 < \dots < x_{k}\in \RR$ and $t_1<\dots < t_{k}\in\RR$ for any $k'\in\{1, \dots, k\}$ that 
\begin{align*}
	\det\begin{pmatrix}
		\Psi(t_1 - x_1) & \cdots & \Psi(t_{k'} - x_{1})\\
		\vdots & \ddots & \vdots \\
		\Psi(t_{1} - x_{k'}) & \cdots & \Psi(t_{k'} - x_{k'})
	\end{pmatrix}\neq 0,
\end{align*}
which is precisely the definition for $\Psi$ to be a Tchebycheff-system on $\RR$. 
The following result, extends this insight to general $(k,l)$-root-regular functions for $l \in \NN_0$.

\begin{lemma}\label{lem:fullRankRRFunction}
	Let $\Psi\colon \RR\to \RR$ be a continuous, $(k,l)$-root-regular function with $k\in \NN$ and $l\in \NN_0$. Then, for any collection of distinct points $x_1,  \dots, x_{k'}\in \RR$ and distinct $t_1,  \dots ,t_{k +l}\in \RR$ with $k'\leq k$ it follows that 
	\begin{align}\label{eq:FullRankProperty}
	\textup{Rank}\begin{pmatrix}
		\Psi(t_1-x_1) & \cdots & \Psi(t_{k+l}-x_1)\\
		\vdots & \ddots & \vdots \\
		\Psi(t_1-x_{k'}) & \cdots & \Psi(t_{k+l}-x_{k'})
	\end{pmatrix}=k'.
	\end{align}
\end{lemma}

Based on this result on the rank of the matrix consisting of translated evaluations of the function~$\Psi$, we formulate a general identifiability result for measures in $\calM_k(\RR)$. 

\begin{theorem}[Identifiability from root-regular functions]\label{thm:identifiabilityFromFunctionals}
	Let $\mu, \nu$ be two measures on $\RR$ with at most $k$ support points and let $\Psi\colon \RR\to \RR$ be a $(2k,l)$-root-regular function with $l\in \NN_0$. If~for distinct $t_1, \dots, t_{2k+l}\in \RR$ the equality 
	\begin{align}\label{eq:idenfiableCondition}
		\int \Psi(t_j-x) \dif\mu(x) = \int \Psi(t_j-y) \dif\nu(y)
	\end{align} is met for all $j \in \{1, \dots, 2k+l\}$, then $\mu = \nu$. 
\end{theorem}

We provide the proof in \Cref{app:pf:thm:identifiabilityFromFunctionals}. A few comments are in order.

\begin{remark}\begin{enumerate}
	\item Recalling our MLE consistency result in Proposition \ref{prop:consistency_LSE_MLE_finite_m}, its proof involves identifying two $k$-atomic measures $\mu, \nu$ based on $\int_{B_j} K\ast \mu(x)\dif x = \int_{B_j} K \ast \nu(x) \dif x$ for $j = 1, \dots, m$ where $m\geq 2k+l$ is the number of distinct bins. To this end we employ Theorem \ref{thm:identifiabilityFromFunctionals} with the function $\Psi(x)= \int_x^{x+\delta}K(t)\dif t$. With this choice, the above equality is equivalent to \eqref{eq:idenfiableCondition} with $t_j \coloneqq s_j$ for all $j=\{1, \dots, m\}$, and the identifiability follows. 	
	\item The notion of root-regularity prohibits that distinct translations of a function are linearly dependent and thus demands that a function is sufficiently heterogenous. For instance, the constant function $\Psi(x)= 1$ for $x\in \RR$ is not root-regular, and the integral of translations of $\Psi$ only provides information about the total mass of the underlying measure $\mu$ but no information about its support points. 
	\item \Cref{thm:identifiabilityFromFunctionals} quantifies a trade-off in heterogeneity of the function $\Psi$ and the number of integrals of translations of $\Psi$ which are required to be known. If the functional $\Psi$ is $(2k,0)$-root-regular, then  exactly $2k$  integrals are required to identify a $k$-atomic measure $\mu$. A $k$-atomic measure is characterized by the location of $k$-support points on $\RR$ and $k$ weights, hence the space of $k$-atomic measure on $\RR$ admits $2k$ degrees of freedom (over $\RR$). Insofar, for a $(2k,0)$-root-regular function \Cref{thm:identifiabilityFromFunctionals} provides the smallest possible number of integrals required. %
	Meanwhile, if $\Psi$ is $(2k,l)$-root-regular for $l\in \NN$, then exactly $l$ more integrals of translations are required to overcome the lack of heterogeneity of the function $\Psi$. 
\end{enumerate}
\end{remark}

In the following we detail some example classes for the formalism of root-regularity.  We begin with a characterization of root-regularity for functions whose derivative is root-regular. 

\begin{lemma}\label{lem:rootregViaDerivative}
	Let $\Psi\colon \RR\to (0,\infty)$ be a differentiable function such that its derivative $\Psi'$ is $(k,l)$-root-regular on $\RR$ for $k\in \NN$, $l\in \NN_0$. Then, it follows that $\Psi$ is $(k,l+1)$-root-regular. 
\end{lemma}

The previous lemma permits us to infer identifiability of mixtures with a root-regular kernel based on finitely many bin evaluations. Indeed, if the function $x\mapsto K(x+\delta) - K(x)$ for fixed $\delta>0$ is root-regular, then Lemma \ref{lem:rootregViaDerivative} it follows that the kernel-binning $x\mapsto \int_{x}^{x+\delta} K(t)\dif t$ is also root-regular. 
We now proceed with relevant classes of  root-regular kernels.

\begin{proposition}[Root regularity of Gaussian kernel]\label{prop:rootRegularityGaussian}
	Let $K\colon \RR\to (0,\infty)$ %
	 be the Gaussian kernel with variance $\sigma^2>0$.
	Then, the bin integral function $x\mapsto\int_{x}^{x+\delta}K(t)\dif t$ for $\delta>0$ is $(k,k+1)$-root-regular for every $k\in \NN$.
\end{proposition}

\begin{proposition}[Root regularity for rational kernels]\label{prop:rationalFunctionRootRegular}
	Let $K\colon \RR\to \RR, x\mapsto P(x)/Q(x)$ be a rational probability kernel where $P$ and $Q$ are polynomials of degree $p$ and $q$, respectively.
	Then, the bin integral function $x\mapsto \int_x^{x+\delta} K(t)\dif t$ for $\delta>0$ is $(k,(2k-1)q + p+1)$-root-regular for every $k\in \NN$. 
\end{proposition}

\begin{example}\label{expl:rationalKernels}Based on Proposition \ref{prop:rationalFunctionRootRegular} the following two classes of kernels are root-regular. %
	\begin{enumerate}
		\item For the Cauchy kernel $K(x) \propto (1+\sigma x^2)^{-1}$ with $\sigma>0$ it follows for $\delta>0$ that the bin integral function $x \mapsto \int_x^{x+\delta} K(t)\dif t$ is $(k,2(2k-1)+1)$-root-regular for every $k\in\NN$. 
		\item For the student-$t$ kernel $K(x) \propto (1 + x^2/n)^{-\frac{n+1}{2}}$ with an odd number $n\in \NN$ of degrees of freedom the bin integral function is $x \mapsto \int_x^{x+\delta} K(t)\dif t$ is $(k,(2k-1)(n+1)+1)$-root-regular for every $k \in \NN$. 
	\end{enumerate}
\end{example}

Notably, it is not clear whether the above derived orders of root-regularity are optimal and if potential restrictions of the points $x_1, \dots, x_k$ in the definition of root-regularity lead to improvements. For instance, based on prior work by \cite{dette1996sign} it is known that the kernel $K(x) \propto (1+x)^{-r}$ for $r>0$ forms a Tchebycheff system on $\RR_+$, i.e., it is $(2k,0)$-root-regular for every $k\in \NN$. Now, if the densities of $K\ast \mu$ and $K\ast \nu$ were to match at finitely many distinct points $t_1, \dots, t_{2k}>0$, then we could infer that $\mu = \nu$. However, based on our model \eqref{eq:model} we only have access to the bin evaluations of $K\ast \mu$ and $K\ast \nu$.

We close this section with a short discussion on potential extensions. 
\begin{remark}\label{rmk:extensions_finiteM}
	\begin{enumerate}
		\item Extending upon our consistency result for the maximum likelihood, it would be worthwhile to investigate the precise convergence rate. Based on numerical evidence we conjecture that the MLE achieves for Gaussian kernels the same estimation rate as for the regime where $m \to \infty$. However, such an analysis remains out of scope due to a lack of quantitative stability bounds, say for $\mu, \nu \in \calP_k(\Theta)$ of the type
		\begin{align}\label{eq:quantitativeBinsIneq}
			\|K\ast (\mu - \nu)\|_{L^2(\RR)} \lesssim \sum_{i = 1}^{m} \left| \int_{B_i} K\ast \mu(x) - K\ast\nu(x)\dif x \right|.
		\end{align}
		Deducing this from the notion of root-regularity is not possible because the latter represents a qualitative property. Moreover, the non-explicit nature of the Gaussian c.d.f.\ hinders an algebraic approach similar to Proposition \ref{prop:Lp_to_lp}. %

		\item All results of this section are stated for the univariate setting since the concept of root-regularity, i.e., quantitative properties for zero-sets of polynomials or sums of exponential functions is best understood in dimension one.  For the multivariate setting, zero sets of analytic functions are not just single points but analytic varieties. For a sensible notion of root-regularity for a function $f\colon \RR^d \to \RR$ we anticipate that a good understanding of zero sets of weighted functions $t \mapsto \sum_{i = 1}^{k} a_i f(t- x_i)$ is necessary. We leave this for future work. 

	\end{enumerate}
\end{remark}

\section{Application}\label{sec:applications}

We apply our methods to determine the location and orientation of DNA origami samples from experimental STED data as reported in \cite{proksch2018multiscale} (see \Cref{fig:origami_data_2}a--c). These samples are precisely designed DNA probes that form a parallelogram structure equipped with fluorophores located along lines at two opposite ends. The distance between these lines is known to be $71$\,nm. Hence, they form an ideal test bed for our purposes as the ground truth is approximately known.

\begin{figure}[b!]
	\centering
	\quad\;\includegraphics[width=\textwidth, trim={0 0 0 0}, clip]{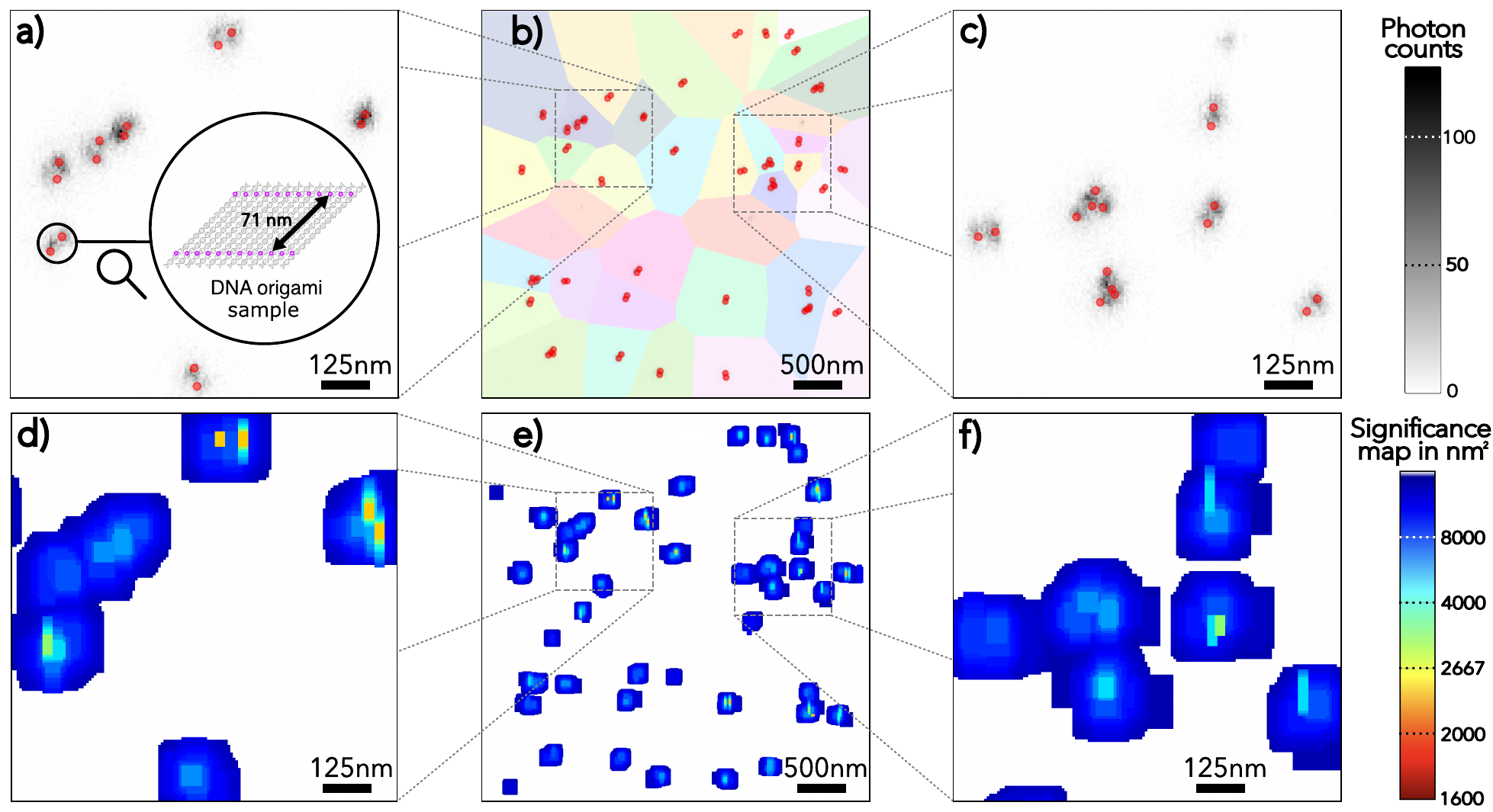}
	\caption{MLE computed for DNA origami data alongside the significance map obtained from the MISCAT method introduced by \cite{proksch2018multiscale}. \textbf{a)-c)} Reconstruction of the mixing measure $\mu$ using MLE (red dots) for the full image (middle, $600\times 600$ pixels) and zoomed-in sections (left and right, each $150\times 150$ pixels). The MLE was computed with approximately $k = 70$, employing a partitioning strategy (indicated by different colors) and a denoising step; both steps utilized a mode-hunting approach (see \Cref{app:application} for details). Due to rounding the number of components per partition cell, the final estimate does not admit $k=70$ atoms, in our case it equals $84$. 
    Figure \textbf{a)} provides a schematic representation of the DNA origami structure. \textbf{d)-f)} Corresponding MISCAT significance map by \cite{proksch2018multiscale}, which performs statistical significance tests across boxes of varying sizes. Each pixel in the significance map is color-coded to indicate the smallest scale (box volume in nm$^2$) at which significance is detected. See the main text for further details. }	\label{fig:origami_data_2}
\end{figure}

Details to our used methodology are described in Appendix \ref{app:application}. 
Specifically, to recover the spatial arrangement of these DNA structures, we first partition and denoise the data (see Appendices \ref{app:apl:partition} and \ref{app:apl:estimation}) and then employ the maximum likelihood estimator (MLE) from \eqref{eq:MLE} computed via an expectation--maximization (EM) algorithm from \Cref{app:EM}, tailored to our statistical model \eqref{eq:model} (details in \Cref{app:apl:estimation}) which we initialize using the complex method of moments estimator \eqref{eq:method_of_moments_complex}. 
The data partitioning step is done for numerical reasons. Indeed, related numerical experiments in Appendices \ref{sec:simulations} and \ref{app:application} confirm that a high-number of components leads to numerical instabilities for the method of moments estimator and the EM-based MLE. As the underlying kernel for our STED image data we consider an isotropic Gaussian kernel with standard deviation of $42$\,nm. This is chosen to fit well the (noisy) kernel matrix from \citet[Supplement]{proksch2018multiscale}.  Specifically, to account for the denoising procedure we apply, which leads to clusters roughly exhibiting diameter 400nm, we consider the Gaussian kernel which is closest in TV norm to the kernel restricted to a disk of diameter 400nm. 
Further, we set $k = 70$ and round for each partition cell, depending on the underlying mass assigned to the region, the number of components to be even. Here we utilize the fact that each DNA origami sample is symmetric and exhibits two clusters of fluorophores at opposite ends. Notably, due to the rounding step, the actual number of components in the final estimate is not necessarily equal to $k$. We additionally investigate in \Cref{app:apl:estimation} how the recovery changes for different choices of $k\in \{40, 50, 60, 70\}$ and find no pronounced change in the recovered location of atoms. In Appendix \ref{app:apl:robustness} we do observe, however, that a coarse partitioning leads to a degradation of the performance of the method of moments and the maximum likelihood estimator.

To assess the accuracy of our reconstruction, we compare it with the MISCAT significance map, a nonparametric multiscale scanning test specifically designed for inverse problems
of this type, introduced by \cite{proksch2018multiscale} (see Figure \ref{fig:origami_data_2}d-f). The MISCAT method identifies significant structures within the data by applying a statistical test that controls the family-wise error rate (FWER). Each pixel in the MISCAT significance map is color-coded according to the smallest scale (volume of the box in nm$^2$) at which a signal is detected. Overall, MISCAT performs $2,\!125,\!764$ tests on this dataset, rejecting $94,\!824$ local hypotheses. The control of the FWER ensures with asymptotic probability of at least $90\%$ that the selected regions do not contain false detections. Notably, our MLE now provides estimations of the exact locations of the DNA origamis. This reconstruction aligns closely with the MISCAT significance map, demonstrating that our method accurately captures the locations of the fluorophores within the DNA origami sample. Importantly, the EM based MLE initialized with the complex method of moments estimator as well as all relevant tuning parameters do not rely on any information from the MISCAT significance map.

\section{Outlook and Discussion}\label{sec:discussion}

The core contribution of our work has been to develop a toolbox 
for the statistical analysis of deconvolution problems
when the signal of interest is discrete and uniform.
We have argued that there are a variety of applications in the physical
sciences where, as the resolution
of imaging  devices continues to improve, one is increasingly
faced with inverse problems involving signals of this nature. 
In contrast to traditional smooth deconvolution problems, which
are well-known to be limited by logarithmic minimax rates of convergence, 
our work emphasizes that the discrete setting is amenable to a fine local minimax analysis.
This perspective reveals that the global minimax rate is rarely achieved, and significantly
improves in realistic situations where there is clustering among the atoms of the discrete signal. 
In fact, with an appropriate choice of loss function, we have shown that
different subsets of the discrete signal can be estimated at differing rates of convergence, 
depending on their individual clustering structure. 
Although our model is new to the best of our knowledge, our analysis drew on many parallels
with the literature on finite mixture models, where local minimax analyses of this type
have been studied since the work of~\cite{heinrich2018}.

We close this paper with a discussion of open questions and potential venues for future research.

\paragraph{Beyond support point estimation: Optimal matchings.}
\label{subsec:functional}

Our work has focused on the Wasserstein risk of estimating
the uniform measure $\mu$ under a Poisson deconvolution model. 
Such results can also be used to deduce
upper bounds on the minimax 
rate of estimating {\it functionals} of $\mu$, provided
that they are smooth under the  Wasserstein distance. 
It turns out that a variety of   optimal transport functionals
are smooth in this sense, and our
results sometimes lead to optimal
estimators for them. We provide
one example here for the problem of estimating optimal
transport matchings; 
a more detailed investigation of related problems
will appear in a separate paper. 

We consider a two-sample deconvolution problem, in which
a practitioner has access to two sets of independent observations, of the form
\begin{equation}
\label{eq:two_sample}
X_i \sim \mathrm{Po}(t K\star \mu(B_i)),\quad Y_i \sim \mathrm{Po}(t K\star \nu(B_i)),
\quad i=1,\dots,m,
\end{equation}
where $\mu=(1/k)\sum_{i=1}^k\delta_{\theta_i}$ and 
$\nu = (1/k)\sum_{i=1}^k \delta_{\eta_i}$ are unknown uniform measures in $\calU_k(\Theta)$. The goal is to estimate
the quadratic {\it optimal transport matching} from $\mu$ and $\nu$, which is defined
as the solution to the optimization problem
\begin{equation}
\label{eq:ot_matching}
\tau_0 = \argmin_{\tau\in \calS_k} \frac 1 k \sum_{i=1}^k \|\theta_{\tau(i)} - \eta_i\|^2,
\end{equation}
provided that such a solution exists, and is unique. 
The problem of estimating optimal transport
matchings under our deconvolution model is motivated by the recent
work of~\cite{tameling2021}, which leveraged such objects to define
new measures of spatial proximity, or {\it colocalization}, for biophysical
applications. More broadly, the problem
of estimating optimal  
transport maps has been the subject of intensive
research in the statistical optimal transport literature
in recent years, though typically not in a deconvolution setting;
we refer to~\cite{chewi2025} for a recent survey. 

In order to make the parameter $\tau_0$ well-defined and identifiable, we require
appropriate regularity conditions on the underlying measures $\mu$ and $\nu$. 
Concretely, we will adopt a smoothness assumption
originating from the works of~\cite{gigli2011,hutter2021}
for the study of optimal transport maps, which was later adopted
in the discrete setting by~\cite{ding2022}. 
We assume that
the underlying measures lie in the class $\calT_\lambda(\Theta)\subseteq \calU_k(\Theta) \times\calU_k(\Theta)$ consisting of all pairs of measures $(\mu,\nu)$
for which there exists $\lambda > 0$ and 
a strictly convex and twice continuously differentiable
function $\varphi_0:\Theta \to \bbR$
such that $\nu = (\nabla\varphi_0)_{\#}\mu$, and 
$$\frac{I_d}{\lambda} \preceq \nabla^2 \varphi_0(x) \preceq \lambda I_d,\quad \text{for all } x \in \Theta.$$
It can be deduced from~\cite{ding2022}
that $(\mu,\nu)\in \calT_\lambda(\Theta)$
for some $\lambda > 0$ if and only if
there exists a unique optimal coupling $\tau_0$ from $\mu$ to $\nu$,
whose  inverse $\tau_0^{-1}$ which is the unique optimal matching from $\nu$ to $\mu$
(see Lemma~\ref{lem:ding} for a formal statement).
Thus, the class $\calT_\lambda(\Theta)$ can be viewed
as the collection of measures $(\mu,\nu)$ which are quantitatively
bounded away from those pairs of measures which are not related
by a unique optimal matching.

Having identified sufficient conditions for the identifiability
of optimal matchings, it remains to define a loss function
in order to state our convergence rates. Notice
that any matching $\tau_0$ defines a coupling $\pi_0 \in \calU_k(\Theta\times\Theta)$
between $\mu$ and $\nu$, given by
$\pi_0 = \frac 1 k \sum_{i=1}^k \delta_{(\theta_{\tau_0(i)}, \eta_i)}$. 
Given any estimator $\hat\pi \in \calP(\Theta\times\Theta)$ of $\pi_0$, 
we will measure its risk using the 2-Wasserstein distance over $\calP(\Theta\times\Theta)$. 
\begin{corollary}\label{cor:matchings}
Given $\lambda > 0$, under the same conditions as Theorem~\ref{thm:global_minimax_risk}, 
it holds that
\begin{equation}
\inf_{\hat\pi_t} \sup_{(\mu,\nu)\in \calT_\lambda(\Theta)} 
\bbE_{\mu,\nu} W_2(\hat\pi_t,\pi_0) \asymp t^{-\frac 1 {2k}},
\end{equation}
where the infimum is taken over all estimators
$\hat\pi_t \in \calP(\Theta\times\Theta)$, and the expectation is 
taken over the independent samples~\eqref{eq:two_sample}. 
\end{corollary}

The claim is an immediate consequence of Theorem~\ref{thm:global_minimax_risk}
and   recent stability bounds for 
the quadratic optimal transport problem~\citep{balakrishnan2025}, which relate
the Wasserstein risk of the estimator $\hat\pi_t$
to the Wasserstein risk of its marginal distributions
(see Lemma~\ref{lem:stab_w2} for a self-contained summary). 
These stability bounds allow us to reduce the problem of estimating~$\pi_0$
to the problem of estimating $\mu$ and $\nu$ individually, which 
in turn was the content of our main results in Section~\ref{sec:intro}.
It is also straightforward to see that the convergence rate in Corollary~\ref{cor:matchings}
can be improved if the measures are $k_0$-clusterable, 
as in Theorem~\ref{thm:local_minimax_risk_separated}. We omit a formal statement for brevity. 
A short proof of Corollary~\ref{cor:matchings}
is given in Appendix~\ref{app:pf_cor_matchings}.

\paragraph{Estimation of the number of mixture atoms.}
\label{subsec:Dis:EstimationK}

Our work made the assumption that the number of  atoms, $k$, is known, which may generally not be the case
in practice. This raises the question of performing data-driven estimation of $k$ under our
discrete signal recovery model. The related question
of estimating the number of components in finite mixture
models has been intensively studied in the literature, with
solutions including information criteria~\citep{keribin2000,drton2017}, iterative
hypothesis testing methods~\citep{li2010,chakravarti2019}, penalized
likelihood methods~\citep{huang2017,manole2021estimating}, and Bayesian methods~\citep{richardson1997bayesian,miller2018mixture}. 
We refer to~\cite{chen2023}
for a recent survey. 
As future research, we deem it interesting to investigate how these approaches perform for our binned Poisson convolution model \eqref{eq:model}, in particular under uniform weights.

\paragraph{Extensions and refinements in terms of kernels.}
\label{subsec:Dis:Extensions}

Our current analysis covers compactly-supported and Gaussian kernels. This leaves open the analysis of general analytic kernels
like the Airy disk. A notable observation for analytic kernels is that the convolution $K \ast \mu$ is identifiable from a set of positive Lebesgue measure \citep[p.\ 240]{federer2014geometric}, which guarantees that $\mu$ can be uniquely recovered. However, analytic kernels, being non-compactly supported, introduce boundary effects that render the method of moments inconsistent within our model \eqref{eq:model} without additionally enlarging the domain. In this regime, we expect the maximum likelihood estimator to remain consistent under increasing precision ($m \to \infty$) and potentially even when $m$ 
fixed, but larger than a constant that depends on the number of components $k$. Extending our analysis to this setting necessitates the derivation of moment-$L^2$ comparison bounds, which we have so far only been established for the Gaussian kernel (\Cref{thm:equivalences}).

Another interesting extension involves analyzing the regime where the kernel is only assumed to belong to a suitable (parametric or nonparametric) class and needs to be estimated without additional measurements, a problem known as \emph{blind deconvolution}, \citep{hall2007blind, hall2007, akhavan2019}, or using separate measurements \citep{diggle1993fourier,johannes2009deconvolution}. Prior work has left open the regime where the underlying signal is $k$-atomic. An important aspect to settle in this regard is the identifiability of kernel. For instance, identifiability fails when $K$ is only assumed to belong to a family of uniform box kernels \citep{kim2024a}. Addressing these questions would potentially result in estimators which better adapt to model misspecification.

\paragraph{Alternative Statistical Models.}

Our statistical model only involves Poisson noise, which implies that the expected number of photons and their variance per bin are identical. 
Our leading motivation for this setting has been fluorescence microscopy, where photon emission is known to follow
a Poisson distribution to first order. However, this assumption may not be realistic in 
some experiments due to the presence of additional background noise, or due to detector dead times -- brief intervals during which they cannot register subsequent photons after detecting one~\citep{aspelmeier2015modern}. 
Extending the statistical analysis of our two estimators
to account for more general error models is therefore a worthwhile direction for future research. This could include  Bernoulli, binomial or Gaussian error models, as introduced in \citep{munk2020statistical}. A key technical challenge lies in ensuring that the distribution of observables remains somewhat comparable as $m$ changes. In our Poisson model \eqref{eq:model}, this is achieved naturally, as the model can also be interpreted as the binned version of an inhomogenous Poisson point process \citep{reiss2012course} with intensity measure $t \cdot K \ast \mu$. For more general regression models this is not clear without restricting the class of error distributions and provides an interesting avenue for further work.

\paragraph{Acknowledgments.} 
The authors would like to thank Victor Panaretos, Sivaraman Balakrishnan, Mikael Kuusela, and 
Larry Wasserman for discussions related to this work,  Frank Werner for providing the STED microscopy data set, and 
Philippe Rigollet for communicating an alternate
proof of Lemma~\ref{lem:moment_independence} on the existence of $k$-atomic uniform measures on $\RR$ with $k-1$ matching moments.  
SH would like thank Thomas Staudt for helpful discussions related to microscopy.

\paragraph{Funding.} 
SH and AM acknowledge support from the DFG RTG 2088 \emph{Discovering structure in complex data:
Statistics meets optimization and inverse problems} and the DFG CRC 1456 \emph{Mathematics of the Experiment A04, C06.}
SH additionally gratefully  acknowledges funding from the German National Academy of Sciences Leopoldina under grant number LPDS 2024-11.
TM is funded by a Norbert Wiener postdoctoral fellowship.
AM further gratefully acknowledges support by  DFG RU 5381 \emph{Mathematical Statistics in the information age – Statistical efficiency and computational
tractability}, and AM and SH the DFG Cluster of Excellence 2067 MBExC \emph{Multiscale bioimaging–from
molecular machines to networks of excitable cells}.

 \addtocontents{toc}{\protect\setcounter{tocdepth}{1}}
\appendix

\section{Proofs for Section~\ref{sec:intro}}
	\label{app:pfs_intro}

	\subsection{Proof of Theorem \ref{thm:global_minimax_risk}}

	The lower bound on the global minimax risk under the stated set of assumptions follows from Proposition~\ref{prop:minimax_lower_bound}$(i)$ which applies to $h$-regular kernels (see Definition \ref{def:regularKernel})
	for $h=1,\dots,k$. In particular, $h$-regularity of kernels is met for compactly supported kernels under Assumption \ref{ass:kernel}$(i)$, and under Assumption \ref{ass:kernel}$(ii)$, i.e., under a Gaussian kernel, where $h$-regularity is always met (Lemma \ref{lem:regularKernelSufficientCondition}).

	The upper bound follows by combining our global estimate in our black-box bound (Corollary \ref{cor:blackbox_bound}$(i)$) with our convergence statements on the expected moment difference for the method of moments estimator under Assumptions \ref{ass:bins} and \ref{ass:kernel}$(i)$, see Proposition \ref{prop:mm_rate}, and the maximum likelihood estimator under Assumptions \ref{ass:bins} and \ref{ass:kernel}$(ii)$, see Proposition \ref{prop:mle_gaussian}$(i)$. Notably, in the bounds on the expected (squared) moment difference, the $m$-dependent discretization error is negligible compared to the $t$-dependent statistical error due to our lower bound on $m$ in terms of $t$. \qed

	\subsection{Proof of Theorem \ref{thm:local_minimax_risk_separated}}

	The minimax lower bound for the local Wasserstein distance follows from Proposition \ref{prop:minimax_lower_bound}$(ii)$, under the stated condition on $\epsilon$,
	 and the observation that the two local Wasserstein distances $\overline D_{\mu_0}$ and $\calD_{\mu_0}$ from \eqref{eq:local_Wasserstein} are related by 
	  \begin{align*}
	  	\widebar{\calD}_{\mu_0}(\mu, \nu) = 1 \wedge \sum_{j=1}^{k_0}  W_1^{r_j}(\mu_{V_j}, \nu_{V_j}) &\geq (1\vee \max_{j =1, \dots, k_0}\delta_j(\mu_0))^{-1}  \cdot \left(1 \wedge \sum_{j=1}^{k_0}  \delta_j(\mu_0)
	\cdot  W_1^{r_j}(\mu_{V_j}, \nu_{V_j})\right)\\
	&= (1\vee \max_{j =1, \dots, k_0}\delta_j(\mu_0))^{-1}{\calD}_{\mu_0}(\mu, \nu)\\
	&\geq (1\vee \diam(\Theta)^{k})^{-1}{\calD}_{\mu_0}(\mu, \nu).
	  \end{align*}
	The upper bound on the local minimax risk for the local Wasserstein distance follows from Theorem~\ref{thm:local_minimax_risk_nosep} combined with the fact that 
		  \begin{align*}
	  	\widebar{\calD}_{\mu_0}(\mu, \nu) = 1 \wedge \sum_{j=1}^{k_0}  W_1^{r_j}(\mu_{V_j}, \nu_{V_j}) &\leq (1\wedge \min_{j =1, \dots, k_0}\delta_j(\mu_0))^{-1}  \cdot \left(1 \wedge \sum_{j=1}^{k_0}  \delta_j(\mu_0)
	\cdot  W_1^{r_j}(\mu_{V_j}, \nu_{V_j})\right)\\
	&= (1\wedge \min_{j =1, \dots, k_0}\delta_j(\mu_0))^{-1}{\calD}_{\mu_0}(\mu, \nu)\\
	&\leq (1\wedge \delta^{k})^{-1}{\calD}_{\mu_0}(\mu, \nu).
	  \end{align*}
	  
	Moreover, the minimax lower bound for the Wasserstein distance is shown in Remark \ref{rmk:MinimaxLowerBoundsDiscussion}$(iv)$. The upper bound follows analogously to Theorem \ref{thm:local_minimax_risk_nosep} with the key difference that we utilize the Wasserstein bound from Corollary \ref{cor:blackbox_bound} instead of the local Wasserstein divergence. \qed 
	\subsection{Proof of Theorem \ref{thm:local_minimax_risk_nosep}}
	The proof is similar to the preceding two results. 
	On the one hand, for the lower bound, one can 
	use the same construction as in Theorem~\ref{thm:local_minimax_risk_separated},
	together with the fact that the divergences
	$\widebar \calD_{\mu_0}$ and $\calD_{\mu_0}$ are equivalent
	when the atoms of $\mu_0$ are well-separated.
	Note that the choice $\epsilon_t \asymp t^{-1/2k}$ is sufficient
	for these lower bounds to apply.
	For the upper bound, we combine the local estimate in our black-box bound (Corollary \ref{cor:blackbox_bound}$(ii)$) with our convergence results for the expected moment difference. Specifically, by our lower bound on $m$ in terms of $t$, the expected moment difference for the method of moments estimator under Assumptions \ref{ass:bins} and \ref{ass:kernel}$(i)$, see Proposition \ref{prop:mm_rate}, and the maximum likelihood estimator under Assumptions \ref{ass:bins} and \ref{ass:kernel}$(ii)$, see Proposition \ref{prop:mle_gaussian}$(i)$, is upper bounded by 
	$\tilde\epsilon_t =  C(\Theta, \Omega, d, K, k, \gamma) t^{-1/2}$. Hence, we observe that the 
 condition $\delta_t \geq  \tilde \epsilon_t^{1/(2k(k+1))}$ is met by choosing
 the constant $c_2$ in the definition of $\delta_t$ sufficiently large, and we conclude that $\EE[\calD_{\mu_0}(\hat \mu_{t,m}, \mu)] \leq  C(\Theta, \Omega, d, K, k, \gamma) t^{-\frac{1}{2}}$. This finishes the proof. \qed

	\subsection{Proof of Corollary \ref{cor:mixtures}}\label{subsec:proof:intro_mixtures}
	The local and global minimax lower bounds are discussed in Remark \ref{rmk:MinimaxLowerBoundsDiscussion}$(iii)$ and are shown in \Cref{app:pf:rmk:MinimaxLowerBoundsDiscussion}. In particular, they apply because the Gaussian kernel is $h$-regular of any order $h\in \NN$ (Lemma \ref{lem:regularKernelSufficientCondition}).

	For our upper bounds we show that the global and local minimax lower bounds are realized by the general method of moments estimator, as defined in \eqref{eq:method_of_moments_general}. To elaborate, for i.i.d.\ random variables $Y_1, \dots, Y_n \sim K\ast \mu$ we consider for polynomials $\{\psi_\alpha\}_{\alpha \in \NN_0^d}$ from \Cref{thm:moment_polynomials} the moment estimators 
	\begin{align*}
		\hat m_{\alpha} = \frac{1}{n}\sum_{i = 1}^{n} \psi_{\alpha}(X_i), \quad \alpha \in \NN_0^d, \|\alpha\|_1\leq k.
	\end{align*}
	Our general method of moments estimator is then defined as 
	\begin{align*}
		\hat \mu_n\coloneqq\hat \mu_{n,\Theta}^{\textup{MM}}\in \argmin_{\tilde \mu \in\calU_k(\Theta)}\sum_{\substack{\alpha\in \NN_0^d, \|\alpha\|_1\leq k}}|m_\alpha(\tilde \mu) - \hat m_{\alpha}|^2.
	\end{align*}
	Following along step 2.\ of the proof of Proposition \ref{prop:mm_rate}, it follows that 
	\begin{align*}
		\EE \left[M_{k}^2(\hat \mu_n, \mu)\right] \lesssim_k \sum_{\substack{\alpha\in \NN_0^d\backslash\{0\}\\1\leq \|\alpha\|_1 \leq k}} \EE\left|\hat m_{\alpha}-m_\alpha(\mu) \right|^2.
	\end{align*}
	Hence, once we show that the right-hand side in the above display is dominated by $C(\Theta, \Sigma, d,k )n^{-1}$ the assertion follows	
	from Corollary \ref{cor:blackbox_bound}. To this end, first note by \Cref{thm:moment_polynomials} that $m_{\alpha}(\mu) =  \EE_{X\sim \mu}[X^\alpha]= \EE_{Y\sim K\ast \mu}\left[\psi_\alpha(Y)\right]$. This implies by Jensen's inequality, combined with sub-Gaussianity of $K\ast \mu$ with parameters depending on $\Theta$ and $\Sigma$ and since the polynomial $\psi_\alpha$ only depends on $\Sigma, d,k$ that 
	\begin{align*}
		\EE\left|\hat m_\alpha  -m_{\alpha}\right|^2 \leq n^{-1}(\Var_{Y\sim K\ast \mu}\left[\psi_\alpha(Y)\right]) \leq n^{-1} (\EE_{Y\sim K\ast \mu}[\psi_\alpha(Y)^2])\lesssim_{\Theta, \Sigma, k,d} n^{-1}. 
	\end{align*}
	which proves the claim. \qed

	\subsection{Proof of Corollary \ref{cor:TV_Hellinger_equiv}}
	This assertion immediately follows from our equivalence statement for finite Gaussian mixtures (\Cref{thm:equivalences}) in Section \ref{sec:mle_Lp_equiv}.\qed

\section{Proofs for Section~\ref{sec:moments}}
\label{app:pf_thm_moment_comparison}

\subsection{Proof of Lemma \ref{lem:multivariateIdentifiability}}
The proof is strongly inspired by \cite[Lemma A.1(i)]{wei2023}. Consider random vectors $X\sim \mu$ and $Y\sim \nu$ and note for  $\eta\in \SS^{d-1} = \{\eta \in \RR^d \,\colon \|\eta\|= 1\}$ and $p\in \{1, \dots, k\}$ by assumption that
	\begin{align*}
		\EE \langle \eta, X\rangle^p &=\sum_{\alpha \in \NN_0^d, |\alpha| = p} \binom{p}{\alpha} \eta^\alpha \EE X^\alpha = \sum_{\alpha \in \NN_0^d, |\alpha| = p} \binom{p}{\alpha} \eta^\alpha m_\alpha(\mu)\\
		&= \sum_{\alpha \in \NN_0^d, |\alpha| = p} \binom{p}{\alpha} \eta^\alpha m_\alpha(\nu) =\sum_{\alpha \in \NN_0^d, |\alpha| = p} \binom{p}{\alpha} \eta^\alpha \EE Y^\alpha = \EE \langle \eta, Y\rangle^p.
	\end{align*} 
	Since $\langle \eta, X\rangle$ and $\langle \eta, Y\rangle$ are real-valued random variables whose underlying distributions lie in $\calU_k(\RR)$, Lemma \ref{lem:moment_identifiability} yields that $\langle \eta, X\rangle \stackrel{\mathcal{D}}{=}  \langle \eta, Y\rangle$. Further, since $\eta\in \SS^{d-1}$ was arbitrary, the Cram\'er-Wold device implies that $\mu$ and $\nu$ are identical.  
\qed

	\subsection{Alternative Proof of Lemma \ref{lem:multivariateIdentifiability}}
	It is known that  $\mu = (1/k)\sum_{i = 1}^{k}\delta_{\theta_i}$ and $\nu = (1/k)\sum_{i = 1}^{k}\delta_{\eta_i}$ are identical if and only if $\sum_{i = 1}^{n}f(\theta_i) =\sum_{i = 1}^{n}f(\eta_i)$ for every continuous $f\colon \RR^d \to \RR$. To show the latter, define the function $F\colon (\RR^{d})^k \to \RR,\quad (x_1, \dots, x_k)\mapsto (1/k)\sum_{i=1}^{k}f(x_i)$, which
is invariant under permutations, i.e., $F(x_1, \dots, x_k) = F(x_{\sigma 1}, \dots, x_{\sigma k})$ for any permutation $\sigma \in \calS(k)$. Hence, the representational theorem for multivariate symmetric functions \cite[p.2 and Theorem 2.1]{chen2022representation} yields that $F(x_1, \dots, x_k) = G\left((\sum_{i = 1}^{k} x_i^\alpha)_{\alpha \in \NN_0^d, \|\alpha\|_1\leq k}\right)$ for some function $G$ depending on $f$. As the moments of $\mu$ and $\nu$ up to order $k$ coincide, we infer
	$$\sum_{i = 1}^{n}f(\theta_i) = F(\theta_1, \dots, \theta_k) = F(\eta_1, \dots, \eta_k)= \sum_{i = 1}^{n}f(\eta_i),$$ which implies equality of $\mu$ and $\nu$ since $f$ was arbitrary.\qed

\subsection{Proof of Theorem~\ref{thm:moment_comparison}}\label{app:pf_thm_moment_comparison_bound}

For $\Theta \in \KK$ with $\KK \in \{\RR, \CC\}$ the global (resp.\ local) comparison bound follows by combining Lemma~\ref{lem:coeff_lip} with Lemma~\ref{lem:ostrowski} (resp.\ Proposition \ref{prop:sharp_poly_stability}). 

For the multivariate setting $\Theta \subseteq \RR^d$ with $d\geq 3$ we employ the following auxiliary results which reduce the problem to the real-line case. 
To this end, we denote for a probability measure $\mu\in \calP(\RR^d)$  the push-forward under the one-dimensional projections  along $\eta\in \mathbb{S}^{d-1}$ as $\mu^\eta \coloneqq \langle \eta, \cdot \rangle_{\#}\mu$. Proofs for these auxiliary results are provided at the end of this subsection. 

The first result relates the Wasserstein distance between two $k$-atomic measures to their max-sliced Wasserstein distance, which is due to \cite{doss2023}.

\begin{lemma}[Lemma 3.1, \citealt{doss2023}]\label{lem:EquivalenceWassersteinSlicedWasserstein}
	For every pair $\mu, \nu \in \calP_k(\RR^d)$ it holds 
	\begin{align*}
		W_1(\mu, \nu) \leq k^2 \sqrt{d}\sup_{\eta \in \SS^{d-1}} W_1(\mu^\eta, \nu^\eta).
	\end{align*}
\end{lemma}
While Lemma \ref{lem:EquivalenceWassersteinSlicedWasserstein} is sufficient for the moment global comparison bound, it does not enable the derivation of a moment local comparison bound since the clustering structure of a neighboring measure $\mu_0$ does not retain its clustering structure under every one-dimensional projection. The following result delimits the collection of directions $\eta$ in the max-sliced Wasserstein distance at the price of a larger constant to retain the clustering structure.  %

\begin{lemma}\label{lem:EquivalenceWassersteinSlicedWasserstein_Local}
	For $1\leq k_0 \leq k$ and $r\in \NN^{k_0}$ with $|r| = k$ let $\mu_0\in \calU_{k,k_0}(\RR^d; r, \delta)$ with $\delta>0$. Define set  $A(\mu_0) \coloneqq \{\eta\in \SS^{d-1} \,\colon  \,|\langle \eta, y-y'\rangle| \geq(2k^2 \sqrt{d})^{-1}\|y-y'\|\,\forall y,y' \in \supp(\mu_0)  \}$. Then, the following assertions hold. 
\begin{enumerate}
	\item For each pair of $k$-atomic probability measures $\mu, \nu \in\calU_k(\RR^d)$ it holds for each $r\geq 1$ that
	\begin{align*}
		W_1(\mu, \nu) \leq 2k^2\sqrt{d}\sup_{\eta \in A(\mu_0)} W_1(\mu^\eta,\nu^\eta). 
	\end{align*}
	\item For each $\eta \in A(\mu_0)$ it holds $\mu_0^\eta \in\calU_{k,k_0}(\RR; r, \delta/(2k\sqrt{d}))$. Further, upon denoting $\mu_0 = \frac{1}{k} \sum_{i =1}^{k_0}r_j \delta_{\theta_{0j}}$ and $\mu_0^\eta = \frac{1}{k} \sum_{i =1}^{k_0}r_j \delta_{\theta_{0j}^\eta}$ with $\theta_{0j}^\eta\coloneqq \langle \eta, \theta_{0j}\rangle$ it holds for each $j\in \{1, \dots, k_0\}$,   
	\begin{align}\label{eq:boundOnDelta_j}
		 \prod_{\substack{1\leq i \leq k_0\\ i\neq j}}\|\theta_{0i} - \theta_{0j}\|^{r_i} \leq (2k^2\sqrt{d})^{k-r_j}  \prod_{\substack{1\leq i \leq k_0\\ i\neq j}}|\theta^\eta_{0i} - \theta^\eta_{0j} |^{r_i}.
	\end{align} %
	\item If $\mu, \nu \in \calU_{k}(\RR^d;  \mu_0, \epsilon)$ for $\epsilon>0$, then for each $\eta \in A(\mu_0)$ it holds $\mu^\eta, \nu^\eta\in \calU_{k}(\RR; \mu_0^\eta, \epsilon)$. 
	\item For $\mu, \nu \in \calU_k(\RR^d; \mu_0, \delta/(5k^2\sqrt{d}))$ it holds 
	\begin{align*}
		\calD_{\mu_0}(\mu, \nu)\leq (2k^2\sqrt{d})^k \sup_{\eta \in A(\mu_0)} \calD_{\mu_0^\eta}(\mu^\eta, \nu^\eta).
	\end{align*}
\end{enumerate}
\end{lemma}

In addition to the bound between Wasserstein and max-sliced Wasserstein distance for $k$-atomic measures, we utilize the bound from Lemma \ref{lem:sliced_multivariate_moment_bound} which relates the moments of projected measures to the moment of underlying measures.

With these technical tools at our disposal, we can state the proof for the multivariate global and local moment comparison bound. For the global bound we combine Lemma \ref{lem:EquivalenceWassersteinSlicedWasserstein} with the univariate global moment comparison bound and Lemma \ref{lem:sliced_multivariate_moment_bound}, which yield after denoting $\tilde \Theta \coloneqq \bigcup_{\eta \in \SS^{d-1}} \langle \eta, \Theta\rangle$ and since $\mu^\eta, \nu^\eta \in \calU_k(\tilde \Theta)$ for every $\eta \in \SS^{d-1}$ that
\begin{align*}
	W_1^k(\mu,\nu)  &\leq (k^2 \sqrt{d})^k \sup_{\eta\in\SS^{d-1}}W_1^k(\mu^\eta, \nu^\eta)\\ 
	&\leq  (k^2 \sqrt{d})^k \sup_{\eta\in \SS^{d-1}} C(\tilde \Theta, k) M_k(\mu^\eta, \nu^\eta)\\
	&\leq k^{2k} d^{k} C(\tilde \Theta, k) M_k(\mu, \nu).
\end{align*}

For the local comparison bound we utilize Lemma \ref{lem:EquivalenceWassersteinSlicedWasserstein_Local}$(iv)$ which guarantees the existence of a non-empty set $A(\mu_0)\subseteq \SS^{d-1}$ such that for $\mu, \nu \in \calU_{k}(\Theta; \mu_0, \delta/(5k^2\sqrt{d}))$ it holds
\begin{align}\label{eq:localWS-to-sliced}
	\calD_{\mu_0}(\mu, \nu)\leq (2k^2\sqrt{d})^k \sup_{\eta \in A(\mu_0)} \calD_{\mu_0^\eta}(\mu^\eta, \nu^\eta).
\end{align}
To relate the right-hand side to the univariate moment stability bound, let us consider via Proposition \ref{prop:sharp_poly_stability} the constants $\tilde C, \tilde c, \tilde \delta_0>0$ depending on $\tilde \Theta \coloneqq \bigcup_{\eta \in A(\mu_0)} \langle \eta, \Theta\rangle$ and $k$ such that for all $\tilde \delta\in (0,\tilde \delta_0)$ with $\tilde \mu_0 \in \calU_{k,k_0}(\tilde \Theta; r, \tilde \delta)$ and $\tilde \mu, \tilde \nu \in \calU(\tilde \Theta; \tilde \mu_0, \tilde c \tilde \delta^{k+1})$ it holds 
\begin{align}\label{eq:univariate_moment_comparison_bound_nice}
	\calD_{\tilde \mu_0}(\tilde \mu, \tilde\nu) \leq C M_k(\tilde \mu, \tilde \nu).
\end{align}
Hence, upon setting $\delta_0 \coloneqq 2 k^2 \sqrt{d} \min(\tilde \delta_0, (2/5 \tilde c)^{1/k})$ it follows for all $\delta\in (0, \delta_0)$ that 
\begin{align*}
	\frac{\delta}{2k^2\sqrt{d}}\leq \tilde \delta_0 \quad \text{ and }\quad  \tilde c\left(\frac{\delta}{2k^2 \sqrt{d}}\right)^{k+1} \leq \frac{\delta}{5k^2\sqrt{d}}.
\end{align*}
Further, from  Lemma \ref{lem:EquivalenceWassersteinSlicedWasserstein_Local}(ii) it follows from $\mu_0 \in \calU_{k,k_0}(\Theta; r, \delta)$ for every $\eta \in A(\mu_0)$ that $\mu_0^\eta\in \calU_{k,k_0}(\Theta; r, \delta/(2k^2\sqrt{d}))$. Moreover, by choosing $c \coloneqq \tilde c/(2k^2\sqrt{d})^{k+1}$, it follows for $\mu, \nu \in \calU_{k}(\Theta; \mu_0, c \delta^{k+1})$ via  Lemma \ref{lem:EquivalenceWassersteinSlicedWasserstein_Local}(iii) for all $\eta \in A(\mu_0)$ that $\mu^\eta, \nu^\eta\in \calU_{k}(\tilde \Theta; \mu_0^\eta, \tilde c (\delta/(2k^2\sqrt{d}))^{k+1})\subseteq \calU_{k}(\tilde \Theta; \mu_0^\eta, \delta/(5k^2\sqrt{d}))$. Combining \eqref{eq:localWS-to-sliced} with the univariate moment comparison bound~\eqref{eq:univariate_moment_comparison_bound_nice} and Lemma~\ref{lem:sliced_multivariate_moment_bound} therefore implies 
\begin{align*}
	\calD_{\mu_0}(\mu, \nu)&\leq (2k^2\sqrt{d})^k \sup_{\eta \in A(\mu_0)} \calD_{\mu_0^\eta}(\mu^\eta, \nu^\eta)\\
	&\leq (2k^2\sqrt{d})^k \sup_{\eta \in A(\mu_0)} \tilde C M_k(\mu^\eta, \nu^\eta)\\
	&\leq 2^kk^{2k}d^k \tilde C  M_k(\mu, \nu).
\end{align*}
The assertion thus follows by selecting $C\coloneqq 2^kk^{2k}d^k \tilde C$. \qed
 
\subsubsection{Proof of Lemma \ref{lem:EquivalenceWassersteinSlicedWasserstein_Local}}
	The proof of the first assertion is inspired by the construction pursued in the proof by \cite{doss2023} for the bound relating Wasserstein and max-sliced Wasserstein distance for $k$-atomic measures  (Lemma \ref{lem:EquivalenceWassersteinSlicedWasserstein}) with the key difference that we do not take the supremum over $\eta \in \SS^{d-1}$, as it would affect the cluster structure under projection. Instead we consider the supremum over $\eta\in A(\mu_0)$. In the following, we first show that for all $k$-atomic probability measures $\mu, \nu$ on $\RR^d$ there exists $\eta\in A(\mu_0)$ such that 
	\begin{enumerate}
		\item[$(a)$] the map $y\mapsto \langle \eta, y\rangle$ is bijective on $\supp(\mu)\cup \supp(\nu)$
		\item[$(b)$] and for all $y\in \supp(\mu), y'\in \supp(\nu)$ it holds  $\|y - y'\| \leq 2k^2\sqrt{d} |\langle \eta, y-y'\rangle|$.
	\end{enumerate}
	To this end, note for given $k$-atomic probability measures $\mu, \nu$ that the set of directions $\eta\in \SS^{d-1}$ for which property $(a)$ is not met forms a null-set with respect to the surface measure on $\SS^{d-1}$. Hence, once we show that the collection $\eta\in A(\mu_0)$ for which condition $(b)$ is met admits positive surface measure in $\SS^d$, the existence of $\eta$ satisfying conditions $(a)$ and $(b)$ follows. We prove this by some probabilistic argument. Let $\eta$ be uniformly distributed on $\SS^{d-1}$ and note by inequality (2.2) in the supplement of \cite{doss2023} for $x\in\RR^d$ and $t>0$ that $\PP\left(|\langle \eta, x\rangle| < t \|x\|_2 \right)< t\sqrt{d}$. Invoking union bound it follows from $|\supp(\mu_0)|\leq k_0\leq k$ that 
	\begin{align}\label{eq:unionBoundSlicedBound}
		\PP\left( \exists y,y' \in \supp(\mu_0) \;\colon |\; |\langle \eta, y-y'\rangle| < t\|y-y'\|\right)< k_0^2 t\sqrt{d}\leq k^2 t\sqrt{d}.
	\end{align}
	Choosing $t^*= (2k^2 \sqrt{d})^{-1}$ we infer that $A(\mu_0)$ has positive surface measure on $\SS^{d-1}$,  
	\begin{align*}
		\PP\left( \eta \in A(\mu_0)\right)> 1- k^2 t^*\sqrt{d}= 1/2.
	\end{align*}
	Arguing as for \eqref{eq:unionBoundSlicedBound} we obtain via union bound 
	 \begin{align*}
		\PP\left( \exists y \in \supp(\mu), \exists y\in \supp(\nu)  \;\colon |\; |\langle \eta, y-y'\rangle| < t^*\|y-y'\|\right)< k^2 t^*\sqrt{d}= 1/2,
	 \end{align*}
	which yields 
	\begin{align*}
		\PP\left( \eta \in A(\mu_0)\;\colon|  \;|\langle \eta, y-y'\rangle| \geq t^*\|y-y'\|\,\forall y,y' \in \supp(\mu_0)\right) > 1- 2 k^2 t^*\sqrt{d} =0,
	\end{align*}
	and overall implies the existence of $\eta\in A(\mu_0)$ such that conditions $(a)$ and $(b)$ are met. For such $\eta\in A(\mu_0)$ the map $\langle \eta, \cdot \rangle \colon \supp(\mu)\cup \supp(\nu) \to \RR$ admits an inverse $g_\eta$ on its range. Therefore, any coupling $\pi_\eta$ between $\langle \eta, \cdot \rangle_{\#}\mu$  and $\langle \eta, \cdot \rangle_{\#}\nu$ yields a coupling $(g_\eta, g_\eta)_{\#}\pi$ between $\mu$ and $\nu$. Taking $\pi_\eta$ as the optimal coupling between $\langle \eta, \cdot \rangle_{\#}\mu$  and $\langle \eta, \cdot \rangle_{\#}\nu$ with respect to $W_1$ we infer by condition $(b)$ that
	\begin{align*}
		W_1(\mu, \nu) &\leq \int_{\RR^{2d}} \|y-y'\| \dif((g_\eta, g_\eta)_{\#}\pi_\eta)(y,y')\\&= \int_{\RR^{2}} \| g_\eta(x) - g_\eta(x')\| \dif \pi_\eta(x,x')\\ &\leq 2k^2\sqrt{d} \int_{\RR^{2}} |x-x'|\dif \pi_{\eta}(x,x')\\
		 &= 2k^2\sqrt{d} \cdot W_1(\mu^\eta, \nu^\eta). 
	\end{align*} 
	Taking supremum on the right-hand side over $\eta\in A(\mu_0)$ and the $r$-th root on both sides implies Assertion $(i)$. 

	To show Assertion $(ii)$ note by definition of $A(\mu_0)$ for each $\eta \in A(\mu_0)$ and $y\neq y' \in \supp(\mu_0)$ that $\delta \leq \|y-y'\|\leq 2k^2 \sqrt{d}|\langle \eta, y-y'\rangle|$, which implies that $\mu^\eta_0 \in\calU_{k,k_0}(\RR; r, \delta/(2k^2 \sqrt{d}))$. In combination with the fact $\sum_{1\leq i \leq k_0, i\neq j} r_i = k - r_j$,  inequality \eqref{eq:boundOnDelta_j} also follows.

	Assertion $(iii)$ is a consequence of the fact  $W_\infty(\mu^\eta, \mu_0^\eta)\leq 	W_\infty(\mu, \mu_0)$ for every $\mu, \mu_0\in \calP(\RR^d)$ since $|\langle \eta, x-x'\rangle|\leq \|x-x'\|$ for every $\eta \in \SS^{d-1}$.

	For Assertion $(iv)$ let $\mu, \nu\in\calU_k(\RR^d; \mu_0, \delta/(5k^2 \sqrt{d}))$. Then, for each $x\in \supp(\mu)\cup \supp(\nu)$ there exists an element $x_0\in \supp(\mu_0)$ with $\|x-x_0\|< \delta/(5k^2 \sqrt{d})$ and further $\langle \eta, x_0\rangle \in \supp(\mu_0^\eta)$ fulfills 
	$|\langle \eta, x\rangle- \langle \eta, x_0\rangle|\leq \delta/(5k^2 \sqrt{d})$. 
	Further, for any $\tilde x_0 \in \supp(\mu_0)\backslash\{x_0\}$ it holds since all atoms of $\mu_0$ are separated by $\delta$ and by definition of $A(\mu_0)$ that
	\begin{align*}
		\|x - \tilde x_0\| &\geq \|\tilde x_0 - x_0\|  - \|x - x_0\| \geq  \delta - \frac{\delta}{5k^2\sqrt{d}} > \frac{\delta}{5k^2\sqrt{d}},\\ 
		|\langle \eta, x - \tilde x_0\rangle|&\geq |\langle \eta,  \tilde x_0 - x_0\rangle|-|\langle \eta, x - x_0\rangle|\geq \frac{\delta}{2k^2\sqrt{d}} - \frac{\delta}{5k^2\sqrt{d}} > \frac{\delta}{5k^2\sqrt{d}}.
	\end{align*}
	This implies that $x_0\in \supp(\mu_0)$ is the unique closest element to $x$, and $\langle \eta, x_0\rangle \in \supp(\mu_0^\eta)$ is the unique closest element to $\langle \eta, x\rangle$. Hence, upon specifying the Voronoi-cells induced by the supports of $\mu_0$ and $\mu_0^\eta$ indexed by $ j=1,\dots, k_0$, 
	\begin{align*}
		V_j(\mu_0) \equiv V_j &= \big\{ \theta \in \RR^d : |\theta - \theta_{0j}| \leq |\theta - \theta_{0k}|,~ \forall k \neq j
	\big\},\\
	V_j(\mu_0^\eta) \equiv V_{j}^{\eta} &= \big\{ \theta \in \RR : |\langle \eta, \theta - \theta_{0j}\rangle | \leq |\langle \eta, \theta - \theta_{0k}\rangle|,~ \forall k \neq j
	\big\},
	\end{align*}
	we have $\min_{j = 1, \dots, k_0}\mu(V_j)\wedge \nu(V_j)>0$ and $\min_{j = 1, \dots, k_0}\mu^\eta(V_j^\eta)\wedge \nu^\eta(V_j^\eta)>0$. Further, it follows that $$\langle \eta, \cdot \rangle_{\#}(\mu|_{V_j})= (\langle \eta, \cdot \rangle_{\#}\mu)|_{V_j} \quad \text{ and } \quad \langle \eta, \cdot \rangle_{\#}(\nu|_{V_j})= (\langle \eta, \cdot \rangle_{\#}\nu)|_{V_j}.$$ 

	Hence, by utilizing the notation from Assertion $(ii)$, $\mu_0 = \frac{1}{k} \sum_{i =1}^{k_0}r_j \delta_{\theta_{0j}}$ and $\mu_0^\eta = \frac{1}{k} \sum_{i =1}^{k_0}r_j \delta_{\theta_{0j}^\eta}$, we obtain by existence of $\eta \in A(\mu_0)$ which fulfills conditions $(a)$ and $(b)$ through the same proof approach as for Assertion $(i)$ combined with the bound on the product terms in Assertion $(ii)$ that
	\begin{align*}
		\calD_{\mu_0}(\mu, \nu) &= 1 \wedge \sum_{j = 1}^{k_0} \prod_{\substack{1\leq i \leq k_0\\ i\neq j}}\|\theta_{0i} - \theta_{0j}\|^{r_i} W_{1}^{r_j}(\mu_{V_j}, \nu_{V_j})\\
		&\leq 1 \wedge \sum_{j = 1}^{k_0} (2k^2\sqrt{d})^k \prod_{\substack{1\leq i \leq k_0\\ i\neq j}}|\theta_{0i}^\eta - \theta_{0j}^\eta|^{r_i} W_{1}^{r_j}(\mu^\eta_{V_j}, \nu^\eta_{V_j})\\
		&\leq (2k^2\sqrt{d})^k  \calD_{\mu_0^\eta}(\mu^\eta, \nu^\eta).
	\end{align*}	
	Taking the supremum over $\eta \in A(\mu_0)$  yields the claim. \qed

\subsection{Proof of Lemma~\ref{lem:coeff_lip}}
\label{app:pf_lem_coeff_lip}
Let $\mu=  (1/k) \sum_{i=1}^k \delta_{\theta_i}$ and $\nu=(1/k)\sum_{i=1}^k\delta_{\eta_i}$
be elements in $\calU_k(\domain)$. 
In view of Vieta's formula~\eqref{eq:vieta_formula}, it 
suffices to prove that the elementary symmetric polynomials
are Lipschitz with respect to the underlying moment vectors, in the sense that
there exists $C=C(\domain,k) > 0$ such that 
for every $l=0,\dots,k$, 
\begin{equation}
\label{eq:elementary_symm_poly_lipschitz}
|e_l(\theta_1,\dots,\theta_k) - e_l(\eta_1,\dots,\eta_k)| 
\leq C M_k(\mu,\nu).
\end{equation}
This bound is immediate for $l=0$. Assuming by induction that it holds for some $ l\in \{0,\dots,k-1\}$, 
we further have, 
\begin{align*}
 \big| e_{l+1}&(\theta_1, \dots, \theta_k)   -  e_{l+1}(\eta_1, \dots, \eta_k)  \big| \\
 &= \frac{k}{l} \left| \sum_{j=1}^{l} 
    (-1)^{j-1}\Big[ e_{l-j}(\theta_1, \dots, \theta_k)  m_j(\mu) -  
                    e_{l-j}(\eta_1, \dots, \eta_k)  m_j(\nu)\Big]\right| \\ 
 &\lesssim\sum_{j=1}^{l}  \Big| e_{l-j}(\theta_1, \dots, \theta_k)  m_j(\mu) -  
                    e_{l-j}(\eta_1, \dots, \eta_k)  m_j(\nu)\Big| \\ 
 &\lesssim \sum_{\ell=1}^k   \big| e_{l-j}(\theta_1, \dots, \theta_k)  -  e_{l-j}(\eta_1, \dots, \eta_k)||m_j(\mu)| +  \big| m_j(\mu)-m_j(\nu)||e_{l-j}(\eta_1, \dots, \eta_k)|. 
\end{align*}
By compactness of $\domain$, it is clear that $m_j(\mu)$ and $e_{j}(\eta_1, \dots, \eta_k)$
are bounded by a constant depending only on $\domain,k$, for all $j=1,\dots,k$, thus 
\begin{align*}
 \big| e_{l+1}&(\theta_1, \dots, \theta_k)   -  e_{l+1}(\eta_1, \dots, \eta_k)  \big| \\
 &\lesssim \sum_{j=1}^k  \Big[ \big| e_{l-j}(\theta_1, \dots, \theta_k)  -  e_{l-j}(\eta_1, \dots, \eta_k)|+\big| m_j(\mu)-m_j(\nu)|\Big]. 
\end{align*}
By the induction hypothesis, we deduce that 
\begin{align*}
 \big| e_{l+1}(\theta_1, \dots, \theta_k)   -  e_{l+1}(\eta_1, \dots, \eta_k)  \big| 
 \lesssim \sum_{j=1}^k   \big| m_j(\mu)-m_j(\nu)| = M_k(\mu,\nu),
\end{align*}
where the implicit constants do not depend on $l$. 
We thus deduce by induction that equation~\eqref{eq:elementary_symm_poly_lipschitz} holds, and the claim follows.\qed 

\subsection{Proof of Proposition~\ref{prop:sharp_poly_stability}}
\label{app:pf_prop_sharp_poly_stability}

Our proof is inspired by~\citet[Chapter 5]{bhatia2007}. 
Specifically, our proof begins  by deriving a version of inequality~\eqref{eq:sharp_poly_stability} 
in which the divergence $\calD_{\mu_0}(\mu,\nu)$ is replaced by the following local Hausdorff-type divergences between  
the atoms of $\mu$ and $\nu$, 
\begin{equation}\begin{aligned}
\label{eq:local_hausdorff}
d_{\mu_0}(\mu,\nu) &= 1 \wedge \max_{1 \leq j \leq k_0}
  \delta_{j}(\mu_0) \cdot \max \left\{\max_{\theta_i \in  V_j} \min_{\eta_\ell\in   V_j}   |\theta_i-\eta_\ell|^{r_j}, 
  \max_{\eta_\ell\in V_j} \min_{\theta_i\in V_j} |\theta_i-\eta_\ell|^{r_j}\right\},\\
\bar d_{\mu_0}(\mu,\nu) &= 1 \wedge \max_{1 \leq j \leq k_0}\delta_{j}(\mu_0) \cdot 
  \max \bigg\{\max_{\theta_i \in  V_j}  \prod_{\eta_\ell \in V_j} |\theta_i-\eta_\ell|, 
  \max_{\eta_\ell\in V_j}   \prod_{\theta_i \in V_j} |\theta_i-\eta_\ell|\bigg\},  
\end{aligned}
\end{equation}
where we recall that for any measure $\mu_0 = (1/k) \sum_{i = 1}^{k_0} r_i \delta_{\theta_{0i}} \in \calU_{k,k_0}(\Theta;r, \delta)$,
we write 
$$\delta_j :=\delta_j(\mu_0) = \prod_{i\neq j} |\theta_{0i}-\theta_{0j}|^{r_j},~~V_j \equiv V_j(\mu_0) = \big\{ \theta \in \domain : |\theta - \theta_{0j}| \leq |\theta - \theta_{0\ell}|,~ \forall \ell \neq j
\big\},\quad j=1,\dots, k_0.$$        
\begin{lemma}\label{lem:hausdorff_stability}
Let $\Theta \subseteq \CC$ be a bounded set, let $1\leq k_0\leq k$, and consider $r\in \NN^{k_0}$ with $|r| = k$. 
Let $\delta \geq 4\rho > 0$. Then there exists a positive constant
$C_0 > 0$ depending on $\Theta, k$ such that for all 
$\mu_0\in \calU_{k,k_0}(\Theta; r, \delta)$ and $\mu, \nu \in \calU_{k}(\Theta; \mu_0, \rho)$, 
\begin{equation*}%
d_{\mu_0} (\mu,\nu) \leq \bar d_{\mu_0}(\mu,\nu) \leq C_0 \|f_\mu-f_\nu\|_*.
\end{equation*}
\end{lemma}

\begin{proof}
Denote by $\mu = (1 / k) \sum_{i=1}^k \delta_{\theta_i}$
and $\nu = (1/k) \sum_{i=1}^k\delta_{\eta_i}$
two measures in $\calU_k(\Theta;\mu_0,\rho)$. 
Let $j_0\in \{1,\dots,k_0\}$ and $\theta_{i_0} \in V_{j_0}$. 
Then, since $\theta_{i_0}$ is a root of $f_\mu$, we have
$$|f_\nu(\theta_{i_0})-f_\mu(\theta_{i_0})| = |f_\nu(\theta_{i_0})| 
= \prod_{j =1}^{k_0} \prod_{i:\eta_i\in V_j} |\theta_{i_0} - \eta_i|.$$
Notice that for all $j \neq j_0$ and $\eta_i \in V_j$, we have 
$$|\theta_{i_0} - \eta_i| \geq 
   |\theta_{0j_0} - \theta_{0j}| - |\theta_{i_0} - \theta_{0j_0}| 
   - |\eta_{i} - \theta_{0j}| \geq |\theta_{0j_0} - \theta_{0j}| - 2\rho\geq 
 \frac 1 2  |\theta_{0j_0} - \theta_{0j}|,$$
   thus we obtain 
$$|f_\nu(\theta_{i_0})-f_\mu(\theta_{i_0})| 
\geq \left(\prod_{j\neq j_0}\frac 1 {2^{r_j}}|\theta_{0j_0} - \theta_{0j}|^{r_j}\right) \prod_{i:\eta_i\in V_{j_0}} |\theta_{i_0} - \eta_i|
 = \frac {\delta_{j_0}} {2^{k-r_{j_0}}}\prod_{i:\eta_i\in V_{j_0}} |\theta_{i_0} - \eta_i|.$$
On the other hand, the compactness of $\domain$ implies
that there is a constant $C_{j_0} = C_{j_0}(\domain,k) > 0$ such that
\begin{align}
\label{eq:poly_stab_step11}
|f_\nu(\theta_{i_0})-f_\mu(\theta_{i_0})| \leq C_{j_0} \|f_\nu - f_\mu\|_*.
\end{align}
  Combining the preceding displays, we deduce that 
  \begin{align}
\label{eq:poly_stab_to_cite}
\prod_{i:\eta_i\in V_{j_0}} |\theta_{i_0} - \eta_i|
  \leq \frac{ C_{j_0}}{\delta_{j_0}} \|f_\mu-f_\nu\|_*,
  \end{align}
for a large enough constant $C_{j_0} > 0$, and thus,
  there must exist $\eta_i \in V_{j_0}$ such that
\begin{align}
\label{eq:poly_stab_step12}
 |\theta_{i_0} - \eta_i|^{r_{j_0}} 
  \leq \frac{ C_{j_0}}{\delta_{j_0}} \|f_\mu-f_\nu\|_*.
  \end{align}
  We have thus shown that for any $j_0=1,\dots,k_0$,
  $$\max_{i_0: \theta_{i_0} \in V_{j_0}} \min_{i:\eta_i\in V_{j_0}} \delta_{j_0} |\theta_{i_0} - \eta_i|^{r_{j_0}} 
  \lesssim  \|f_\mu-f_\nu\|_*.$$
  Repeating a symmetric argument, we readily deduce  that there exists $C_0 = C_0(\Theta,k) > 0$ such that
  $$d_{\mu_0}(\mu, \nu) \leq C_0\|f_\mu-f_\nu\|_*.$$
This completes the  proof.   
\end{proof}

It is straightforward to see that $d_{\mu_0} \lesssim \calD_{\mu_0}$ in general (Lemma \ref{lem:RelationHausdorffAndWasserstein}), however
the reverse inequality may not hold, as illustrated in Figure~\ref{fig:differenceWassersteinHausdorff} in \Cref{app:auxiliary}. 
In order to deduce Proposition~\ref{prop:sharp_poly_stability} from Lemma~\ref{lem:hausdorff_stability}, 
we will argue that, in fact,
the situation illustrated in Figure~\ref{fig:differenceWassersteinHausdorff} is the only possible obstruction to 
the equivalence $d_{\mu_0} \asymp \calD_{\mu_0}$, and that even under this situation, 
the upper bound $\calD_{\mu_0}(\mu,\nu) \lesssim \|f_\mu-f_\nu\|_*$ holds true. 

\begin{proof}[Proof of Proposition~\ref{prop:sharp_poly_stability}]
Let $C_0,C_1,C_2 > 0$ be the constants respectively
appearing in the statements of Lemmas~\ref{lem:hausdorff_stability}, \ref{lem:forward_bound}, \ref{lem:ostrowski},
which depend only on $\Theta$ and $k$.
Let $c,\delta_0$ be the largest positive constants in $(0,1)$ for which  the following conditions are satisfied:
\begin{equation}
\label{eq:constant_constraints}
c\delta_0^{k+1} < \delta_0/8,\quad (2C_0C_1c)^{1/k}  \leq 1/8
\end{equation} 
Let $\delta \in (0,\delta_0)$, and 
let the measures $\mu = (1 / k) \sum_{i=1}^k \delta_{\theta_i}$
and $\nu = (1/k) \sum_{i=1}^k\delta_{\eta_i}$ lie in the set $\calU_{k}(\Theta; \mu_0, c\delta^{k+1})$. 
When $k=2$, it is straightforward to see that $d_{\mu_0}(\mu,\nu) \asymp \calD_{\mu_0}(\mu,\nu)$, 
and the claim then follows from step 1. 
We may thus   assume $k \geq 3$ in what follows.

Fix $j_0 \in \{1,\dots,k_0\}$.  It will suffice to prove that 
$W_1^{r_{j_0}}(\mu_{V_{j_0}},\nu_{V_{j_0}}) \lesssim_{\Theta, k} \delta_{j_0}^{-1}\|f_\mu-f_\nu\|_*$. 
Define
$$\epsilon=(C_0\|w\|_*/\delta_{j_0})^{1/r_{j_0}},\quad \text{where } 
w = f_\nu-f_\mu$$
Notice that, by Lemma~\ref{lem:forward_bound}, %
\begin{align}
\label{eq:epsilon_bound_poly}
\nonumber 
\epsilon 
 &\leq (C_0 C_1W_1(\mu,\nu)/\delta_{j_0})^{1/r_{j_0}} \\
\nonumber 
 &\leq (2C_0C_1c\delta^{k+1}/\delta_{j_0})^{1/r_{j_0}} \\
 &\leq (2C_0C_1c)^{1/r_{j_0}}
  \delta \leq \delta/8, 
\end{align}
where we used the fact that $(\delta^{k+1}/\delta_{j_0})\leq \delta^{r_{j_0}+1} \leq \delta^{r_{j_0}}$ and condition~\eqref{eq:constant_constraints} on the constant $c$. 

We begin by proving the claim under the condition
\begin{equation}
\label{eq:assm_21}
\begin{aligned}
\supp(\mu_{V_{j_0}}) &\cap B(\eta_i,\epsilon)\neq \emptyset, \quad \text{for all } \eta_i\in V_{j_0},~\text{and}\\  
\supp(\nu_{V_{j_0}}) &\cap B(\theta_i,\epsilon) \neq \emptyset,\quad \text{for all } \theta_i\in V_{j_0}.
\end{aligned}
\end{equation} 
Under this assumption, we will make use of the following result due to~\citet[Theorem 1.5]{bhatia2007}.
\begin{lemma}[\cite{bhatia2007}]\label{lem:hausdorff_to_wasserstein}
Let $\alpha_1,\beta_1,\dots,\alpha_\ell,\beta_\ell\in \domain$ for some $\ell\in \bbN$.
Suppose that  any connected component $C$  of either one of the unions of open balls 
$$\bigcup_{j=1}^\ell B(\alpha_j,\epsilon) \quad \text{and}\quad   
  \bigcup_{j=1}^\ell B(\beta_j,\epsilon)$$
contains the same number of $\alpha_j$s as it does $\beta_j$s. Then, 
$$W_1\left(\frac 1 \ell \sum_{j=1}^\ell \delta_{\alpha_j}, \frac 1 \ell \sum_{j=1}^\ell \delta_{\beta_j}\right) \leq \epsilon.$$
\end{lemma}
 
Now, let
$S = \bigcup_{\theta_i\in V_{j_0}} B(\theta_i,\epsilon) \subseteq \bbC$. Let $M$ be a connected component
of $S$. Note that $M$ is the union of some number $L$ of the $k$ disks whose union makes up $S$.  
Let 
$$h_t= (1-t)f_\mu + tf_\nu = f_\mu + t w,\quad t\in[0,1].$$
We make use of the following fact.

\begin{lemma}
\label{lem:poly_bdry}
For all $t \in [0,1]$, the polynomial $h_t$ does not vanish on  $\Gamma:=\partial M$. 
\end{lemma}
\begin{proof}[Proof of Lemma~\ref{lem:poly_bdry}]
By assumption~\eqref{eq:assm_21}, the claim holds for $h_1$, %
 thus let $t \in [0,1)$. 
Let $\xi_{t,1},\dots,\xi_{t,k}$ denote the roots of $h_t$, and let 
$\zeta_t  = (1/k) \sum_{i=1}^k \delta_{\xi_{t,i}}$. We need to show
that the roots $\xi_{t,i}$ do not lie on the boundary $\Gamma$. 
By Lemmas~\ref{lem:ostrowski} and \ref{lem:forward_bound}, it holds that
$$W_\infty^k(\zeta_t,\mu) \leq C_2 \|f_{\zeta_t} - f_\mu\|_* 
= C_2 t \|f_\nu - f_\mu\|_* \leq C_1C_2W_1(\mu,\nu) %
\leq 2 C_1C_2c\delta^{k+1}.$$
Thus, using equation~\eqref{eq:constant_constraints}, 
 we deduce that  
$$W_\infty(\zeta_t,\mu_0) \leq W_\infty(\zeta_t,\mu) + W_\infty(\mu,\mu_0) \leq 
(2C_1C_2c\delta^{k+1})^{1/k}+ c\delta^{k+1}\leq \delta/4,$$
and hence $\zeta_t \in \calU_k(\Theta;\mu_0,\delta/4)$. 
From here,
we may apply Lemma~\ref{lem:hausdorff_stability} with $\rho = \delta/4$
to deduce that
$$d_{\mu_0}(\zeta_t,\mu) \leq C_0 \|f_{\zeta_t}-f_\mu\|_* = C_0 t \|f_\mu-f_\nu\|_* < C_0 \|w\|_*.$$
By definition of $d_{\mu_0}$, we deduce that
for every $i=1,\dots,k$, there exists $\ell_0=1,\dots,k_0$ and $\theta_j \in V_{\ell_0}$ such that
$$|\xi_{t,i} - \theta_j| < (C_0\|w\|_* / \delta_{\ell_0})^{1/r_{\ell_0}}.$$
If $\ell_0 = j_0$, the above display implies that $\xi_{t,i}$ lies in the open ball $B(\theta_j,\epsilon)$, 
and thus cannot lie in the set $\Gamma$. On the other hand, if $\ell_0 \neq j_0$, 
then our conditions imply that $\xi_{t,i}$ must lie more than a distance $\epsilon$ from any 
element $\theta_r \in V_{j_0}$. Indeed, for any such $\theta_r$, we have:
\begin{align*}
|\theta_r - \xi_{t,i}|
 &\geq |\theta_r - \theta_j| - \epsilon \\ 
 &\geq |\theta_{0j_0} - \theta_{0\ell_0}|  - 2c\delta^{k+1} - \epsilon\\ 
 &\geq \delta - 2c\delta^{k+1} - \epsilon \\
 &\geq 3\delta/4 - \epsilon \\
 &\geq 4\epsilon,
\end{align*}
where we used equation~\eqref{eq:epsilon_bound_poly}.
This implies that any such atom $\xi_{t,i}$ lies at distance at least $\epsilon$ from the set $M$, and thus
cannot lie on the boundary $\Gamma$. 
The stated Lemma thus follows.\end{proof}

Since $h_t$ does not vanish on $\Gamma$, the map $|h_{t}(z)|$ is bounded from 
below by a positive constant over $\Gamma$. Therefore, for any small
enough $\gamma > 0$, it holds that 
$$|\gamma w(z)| < |h_{t}(z)|\quad \text{for all } z \in \Gamma, t \in [0,1]$$
which implies, by Rouch\'e's Theorem, that the number of roots $N(t)$ of $h_{t}$ inside
$M$ is equal to the number of roots of $h_{t}+\gamma w = h_{t+\gamma}$. That is, we have
for that for all $t \in [0,1]$, there exists $\gamma_0 > 0$ such that for all $\gamma \leq \gamma_0$, 
$$N(t) = N(t+\gamma).$$
It follows that $N$ is continuous over $[0,1]$. Since it is integer-valued, it must be
constant. Thus, we have found that the number of roots of $h_t$ in the interior of $M$ is constant. 
In particular, this implies that $f_\mu$ and $f_\nu$ have the same number $L$ of roots inside $M$.
This can be repeated with the set $S' = \bigcup_{i\in V_{j_0}}^k B(\eta_i,\epsilon)$.
We then deduce from Lemma~\ref{lem:hausdorff_to_wasserstein} that 
$$W_1 (\mu_{V_{j_0}},\nu_{V_{j_0}}) \leq \epsilon \lesssim_{\Theta,}  \left(\frac{ \|f_\mu-f_\nu\|_* }{\delta_{j_0}}\right)^{\frac 1 {r_{j_0}}},$$
as was to be shown. 

Finally, it remains to prove the claim when equation~\eqref{eq:assm_21} does not hold.
In this case, without loss of generality,
there exists some $j_0\in\{1,\dots,k_0\}$ and some $\theta_{i_0} \in V_{j_0}$ such that
$$\inf_{\eta_i \in V_{j_0}} |\theta_{i_0} - \eta_i| \geq \epsilon.$$
On the other hand by Lemma~\ref{lem:hausdorff_stability}, we have
$$\bar d_{\mu_0}(\mu,\nu) \leq \epsilon,$$
and the only way in which the preceding two displays can both hold is 
if $|\theta_{i_0} - \eta_i| = \epsilon$ for all $\eta_i\in V_j$. 
But then, it is a simple observation that the optimal matching distance 
between $\mu_{V_{j_0}}$ and $\nu_{V_{j_0}}$ cannot exceed a constant multiple of $\epsilon$, and the claim once again follows. 
\end{proof}

\subsection{Proof of Corollary \ref{cor:blackbox_bound}}\label{app:pf:blackbox_bound}

If $\hat \mu \in \calU_k(\Theta)$ almost surely, we infer for Assertion $(i)$ from \Cref{thm:moment_comparison}$(i)$ that there exists a positive constant $C = C(\Theta, d, k)>0$ such that
\begin{align}\label{eq:WassersteinBound_MomentSquareError}
	\EE\left[W_1^{2k}(\hat \mu, \mu)\right]&\leq C^2\EE\left[M_k^2(\hat \mu, \mu)\right]\leq C^2\epsilon^2.
\end{align}
Jensen's inequality then implies that 
\begin{align*}
	\EE\left[W_1^{k}(\hat \mu, \mu)\right]\leq \EE\left[W_1^{2k}(\hat \mu, \mu)\right]^{1/2}  \leq C \epsilon. 
\end{align*}
Meanwhile, if $M_k(\hat \mu, \mu)\leq \rho$ for a Euclidean subset $\Theta\subseteq \RR^d$ it follows by Lemma \ref{lem:momentdifferenceZeroMeasureUpperBound} below that $\hat \mu, \mu\in \calU_k(\check\Theta)$ where $\check \Theta \coloneqq B(0, 1+ k\rho + k\sum_{\substack{\alpha\in \NN_0^d, \|\alpha\|_1\leq k}}|m_\alpha(\mu)|)$. Likewise, for a complex subset $\Theta\subseteq \CC$ it follows under $M_k(\hat \mu, \mu)\leq \rho$ from Lemma~\ref{lem:moment_to_measure} and  inequality \eqref{eq:momentdifferenceZeroMeasureUpperBound} below that $\hat \mu, \mu \in \calU_k(\check \Theta)$ where $\check \Theta\coloneqq B(0,C(k)(\rho + \sum_{\alpha = 1}^{k}|m_\alpha(\mu)|)^{k})$ for some positive constant $C(k)>0$. In particular, the set $\check \Theta$ is bounded and only depends on $\Theta, \rho, k, d$. We thus conclude from \Cref{thm:moment_comparison}$(i)$ and our imposed assumption that 
\begin{align*}
	\EE\left[W_1^{2k}(\hat \mu, \mu)\right]&\leq  \EE\left[W_1^{2k}(\hat \mu, \mu) \mathds{1}(M_k(\hat \mu, \mu)< \rho)\right]+  \EE\left[W_1^{2k}(\hat \mu, \mu) \mathds{1}(M_k(\hat \mu, \mu)\geq \rho)\right]\\
	&\leq C(\check\Theta, \rho, d,k)\epsilon^2 +  \epsilon^2 = C(\Theta, \rho, d,k)\epsilon^2.
\end{align*}

To show Assertion $(ii)$, observe by \Cref{thm:moment_comparison}$(ii)$, combined with  $W_\infty(\hat \mu, \mu_0)\leq k W_1(\hat \mu, \mu_0)$ by Lemma \ref{lem:inequalityWasserstein1Infty} and inequality \eqref{eq:WassersteinBound_MomentSquareError} using the lower bound on $\delta$ that
\begin{align*}
	\EE\left[\calD_{\mu_0}(\hat \mu, \mu)\right] &\leq \EE\left[\calD_{\mu_0}(\hat \mu, \mu)\mathds{1}(W_\infty(\hat \mu, \mu)<c \delta^{k+1}/2)\right] + \EE\left[\calD_{\mu_0}(\hat \mu, \mu)\mathds{1}(W_\infty(\hat \mu, \mu)\geq c \delta^{k+1}/2)\right]\\
	&\leq \EE\left[\calD_{\mu_0}(\hat \mu, \mu)\mathds{1}(W_\infty(\hat \mu, \mu_0)<c \delta^{k+1})\right] + \mathbb{P}(W_\infty(\hat \mu, \mu)\geq c \delta^{k+1})\\
	&\leq C \EE\left[M_k(\hat \mu, \mu)\right] + \mathbb{P}(W_\infty^{2k}(\hat \mu, \mu)\geq c^{2k} \delta^{2k(k+1)})\\
	&\leq C \epsilon + C^2k^{2k}\epsilon^2 c^{-2k} \delta^{-2k(k+1)}\\
	&\leq C(\Theta, d, k, c)\epsilon.
\end{align*}

Moreover, for the bound on the Wasserstein distance to the power $r^*$ we observe by combining Lemma \ref{lem:inequalityWasserstein1Infty} with Lemma \ref{lem:inequalityWassersteinLocal} that 
\begin{align*}
	\EE\left[W_1^{r^*}(\hat\mu, \mu)\right] &= \EE\left[W_1^{r^*}(\hat\mu, \mu)\mathds{1}(W_\infty(\hat\mu, \mu)<\delta/4)\right] + \EE\left[W_1^{r^*}(\hat\mu, \mu)\mathds{1}(W_\infty(\hat\mu, \mu)\geq \delta/4)\right]\\
	&\leq C(k)\EE\left[\delta^{-k+r_*}\calD_{\mu_0}(\hat\mu, \mu)\mathds{1}(W_\infty(\hat\mu, \mu)<c \delta^{k+1})\right]+ C(k)\EE\left[\delta^{-k+r^*} W_1^{k}(\hat\mu, \mu)\right]\\
	&\leq C(\Theta, d,k, c)\delta^{-k+r^*} \epsilon + C(k)\delta^{-k+r^*} \epsilon =  C(\Theta, d,k, c) \delta^{-k+r^*}\epsilon. 
\end{align*}
This completes the proof of our error bounds for estimators. \qed

\begin{lemma}\label{lem:momentdifferenceZeroMeasureUpperBound}
	For $k,d\in \NN$ let $\Theta\subseteq \RR^d$  be a bounded set. Then, for
	any $\mu\in \calU_k(\Theta)$, $\rho>0$, and $\nu\in \calU_k(\RR^d)$ with
	$M_k(\mu, \nu) < \rho$, it holds that $\nu\in \calU_k(B(0, 1+k\rho +k\sum_{\substack{\alpha\in \NN_0^d, \|\alpha\|_1\leq k}}|m_\alpha(\mu)| ))$. 
\end{lemma}

\subsubsection{Proof of Lemma \ref{lem:momentdifferenceZeroMeasureUpperBound}}

	First, observe for $\nu \in \calU_k(\RR^d)$ and $\mu\in \calU_k(\Theta)$ that 
	\begin{align}
		M_k(\nu, \delta_{0})\leq M_k(\mu, \nu) + M_k(\mu, \delta_{0}) \leq \rho+ M_k(\mu, \delta_{0}) \leq \rho + \sup_{\mu\in \calU_k(\Theta)}\sum_{\substack{\alpha\in \NN_0^d, \|\alpha\|_1\leq k}}|m_\alpha(\mu)|.\label{eq:momentdifferenceZeroMeasureUpperBound}
	\end{align}
	For $k = 1$ it follows that $\nu = \delta_{\eta}$ for some $\eta\in \RR^d$, and we obtain 	
	\begin{align*}
		M_k(\nu, \delta_0) = \sum_{j = 1}^{d} |\eta_j| = \|\eta\|_1\geq \|\eta\|.
	\end{align*} 
	Moreover, for $k\geq 2$ it follows that $\nu = (1/k)\sum_{i = 1}^{k}\delta_{\eta_i}$ for some $\eta_1, \dots, \eta_k \in\RR^d$
	\begin{align*}
		M_{k}(\nu, \delta_0) \geq \sum_{j=1}^{d} \left|\EE_{X\sim \nu}[X_j^2]\right| =  \EE_{X\sim \nu}[\|X\|^2] = \frac{1}{k}\sum_{i = 1}^{k} \|\eta_i\|^2 \geq \frac{1}{k}\|\eta_h\|^2,
	\end{align*}
	for each $h\in \{1, \dots, h\}$. We thus obtain from the previous two displays and \eqref{eq:momentdifferenceZeroMeasureUpperBound} for general $k\in \NN$ and $\eta \in \supp(\nu)$ that 
	\begin{align*}
		\|\eta\|\leq 1+k\rho + k\sup_{\mu\in \calU_k(\Theta)}\sum_{\substack{\alpha\in \NN_0^d, \|\alpha\|_1\leq k}}|m_\alpha(\mu)|,
	\end{align*}
	which finishes the proof. \qed

\subsection{Proof of Theorem~\ref{thm:moment_polynomials}}
\label{app:pf_thm_moment_polynomials}
Let $\bbK'=\bbR$ if $\bbK=\bbR^d$, and $\bbK'=\bbC$ if $\bbK=\bbC$.
	We first prove for each $\alpha \in \NN_0^d$ with $\|\alpha|\leq \ell\in \NN$ that $\tilde \psi_\alpha\colon \bbK \to \bbK', z \mapsto \int_\bbK x^\alpha K(x-z)\dif x$ is a monic
polynomial of degree $\alpha$. Indeed, by substituting $y = x-z$ we obtain for every $z\in \bbK$ that \begin{align*}
		\tilde \psi_\alpha(z) &= \int_\bbK (y+z)^\alpha K(y)\dif y \\
		&=  \int \sum_{\substack{\beta \in \NN_0^d\\\beta \leq \alpha}}   \textstyle  \binom{\alpha}{\beta}y^{\beta} z^{\alpha-\beta}  K(y) \dif y\\
		&=   z^{\alpha} + \sum_{\substack{\beta \in \NN_0^d\\\beta \leq \alpha}}   z^{\alpha - \beta} \textstyle \binom{\alpha}{\alpha} \displaystyle\int y^{\beta}  K(y) \dif y.
	\end{align*}
It thus follows that $\{\tilde \psi_{\beta}\}_{{\beta \leq \alpha}}$ forms a (linearly independent) basis of $\bbK[z]_{\alpha}$ the ring of polynomials up order $\alpha$  with coefficients in $\bbK'$, i.e., where the exponents of the respective multivariate monomials are dominated by $\alpha$ with respect to total ordering. Hence, 
there exist unique coefficients $(a_{\alpha, \beta})_{\beta\in \NN_0^d, \beta <\alpha}\in \bbK$ (here we write $\beta < \alpha$ to denote that $\beta_i\leq \alpha_i$ for all $i \in \{1, \dots, d\}$ with at least one inequality being strict) so that 
$$ z^\alpha = \tilde \psi_\alpha(z)  + \sum_{\substack{\beta\in \NN_0^d\\\beta <\alpha}} a_{\alpha,\beta}\tilde \psi_{\beta}(z) \quad \text{ for all } z \in \bbK.$$
	With these coefficients, we define the polynomial $\psi_{\alpha}\colon \bbK \to \bbK'$, $z \mapsto z^\alpha + \sum_{\substack{\beta\in \NN_0^d, \beta < \alpha}} a_{\alpha,\beta}z^\beta.$
Since $\int |z|^\alpha K(z) \dif z< \infty$ for all $\alpha \in \NN_0^d$ with $\|\alpha\|_1\leq \ell$ it holds by Fubini's theorem that
for any $\mu\in \calP(\bbK)$, 
	\begin{align*}
		\EE_{Y\sim K\ast \mu}[\psi_\alpha(Y)] &= \int_{\bbK} \psi_\alpha(y) \left[\int_{\bbK} K(y-x) \dif \mu(x) \right]\dif y\\
		&= \int_{\bbK} \left[ \int_{\bbK} \psi_\alpha(y) K(y-x) \dif y\right]\dif \mu(x)\\
		&= \int_{\bbK} \bigg[ \tilde \psi_\alpha(x)  + \sum_{\substack{\beta \in \NN_0^d\\\beta < \alpha}}a_{\alpha,\beta}\tilde \psi_{\beta}(x) \bigg]\dif \mu(x) = \int_\bbK x^\alpha \dif \mu(x) = \EE_{X\sim  \mu}[X^\alpha],
	\end{align*}
	which proves the first claim. 
	
	For the second claim, we first show for $\bbK = \bbC$ under rotational symmetry of the kernel $K$, i.e., if $K(y) = K(|y|)$ for every $y\in \bbC$ that $\int y^\beta K(y) \dif y = 0$ for every $\beta\in \{1, \dots, \alpha-1\}$. To this end, we employ the polar transformation and note by Fubini's theorem that 
	\begin{align*}
		\int_{\CC} y^\beta K(y) \dif y &= \int_{0}^\infty \int_{0}^{2\pi} \left( r\exp(i\theta)\right)^\beta r K(r)\dif \theta \dif r \\
		&= \int_{0}^{\infty} r^{\beta+1} K(r)\dif r\int_{0}^{2\pi} \exp(\beta i \theta) \dif \theta \\
		&= \int_{\bbC} |y|^{\beta} K(y) \dif x\left[\frac{1}{\beta i}\exp(\beta i \theta)\right]_{0}^{2\pi} = 0,
	\end{align*}
	where equality to zero follows by periodicity of $\theta\mapsto \exp(\beta i\theta)$ with period $2\pi$. 
	Hence, for $\alpha\in \{1, \dots, \ell\}$ it holds $\tilde \psi_\alpha(z) = z^\alpha$ and thus $a_{\alpha,\beta} = 0$ for each $0\leq \beta< \alpha\leq \ell$, asserting $\psi_\alpha(z) = z^\alpha$. \qed

	\subsection{Proof of Lemma \ref{lem:moment_to_measure}}
		The polynomial $\poly_m$ is of degree $k$, and thus admits
		$k$  roots $\theta_1,\dots,\theta_k \in \bbC$. 
		Let $\mu = (1/k) \sum_{i=1}^k \delta_{\theta_i}$. We then have $\poly_m = f_\mu$,
		where  $f_\mu$ is the unique monic polynomial~\eqref{eq:vieta_formula} whose
		roots are $\theta_1,\dots,\theta_k$ including multiplicity. 
		Since $\poly_m$ and $f_\mu$ are both monic, we deduce that $\poly_m$ and $f_\mu$ share the same coefficients.
		By Newton's identities \eqref{eq:Newton_identity} for the coefficients of $f_\mu$, the recursion on $\{\epsilon_j\}_{j = 0}^{k}$ follows. 
		
		For the second assertion we first show for fixed $k\in \NN$ under $|m_1|, \dots, |m_k|\leq M$ that $|\symm_{l}|\leq C(k,l)M^l$ for some positive constant $C(k,l)>0$ that only depends on $l,k$. We prove this by induction over $l\in \{0, \dots, k\}$. For $l =0$, the assertion is trivially met for any $C(k,0) =1$. Now assume the assertion is met for $0\leq l-1\leq k-1$, then it remains to show it for $l$. To this end, note by induction hypothesis that $|\symm_{h}|\leq C(k,h)M^h$ for each $h\in \{0, \dots, l-1\}$ and arbitrary $M\geq 1$, and we observe since $M\geq 1$ that 
		\begin{align*}
			\left|\symm_l\right| &\leq \frac{k}{l}\sum_{j = 1}^{l} |\symm_{l-j} m_j|\leq \frac{k}{l}   \sum_{j=1}^{l}C(k,l-j)M^{l-j} M \leq \left(k\max_{j = 1, \dots, l} C(k,l-j)\right) M^l.
		\end{align*}
		Setting $C(k,l) = \left(k \max_{j = 1, \dots, l} C(k,l-j)\right)$ yields by induction that  $|\symm_{l}|\leq C(k,l)M^l$ for all $M\geq 1$. Hence, Cauchy's bound on complex roots of polynomials in terms of its coefficients implies that all roots of $\poly_m$ are contained in $B\left(0,R\right)$ for $R = 1+ \max_{l = 1,\dots, k} |\symm_l|\leq 1+ C(k)M^k$ for $C(k) = \max_{l =1, \dots, k} C(k,l)$.\qed

\subsection{Proof of Proposition \ref{prop:mm_rate}}\label{subsec:proof:mm_rate} 

The differing definitions of the three moment estimators require slightly different proof techniques. We therefore divide the proof into multiple steps. In step~1, we derive quantitative error bounds for the moment estimator (i.e., not the method of moment estimator) in \eqref{eq:moment_estimator}. Then, in steps 2 and 3 we detail the specific analysis for the general and  complex method of moments estimators, respectively. Finally, in step 4 we explain how the real method of moments estimator can be analyzed using the analysis from the complex estimator, and thereby provide a proof for Remark \ref{rmk:discussionMoM}$(i)$.

\subsubsection*{Step 1. Discretization and estimation error of moment estimators}

We first analyze the error of the moment estimators in estimating their population counterpart. We do this in a unified framework, and consider $\hat m_{\alpha}$ for $\alpha \in \NN_0^d$ with $\|\alpha\|_1\leq k$ when we take (multivariate) real moment estimator ($\Theta \subseteq \RR^d$) or $\alpha \in \{1, \dots, k\}$ when we consider complex moment estimators ($\Theta \subseteq \CC$, corresponding to $d = 2$). Specifically, we show that 
\begin{align}
   \EE\left[\left|\hat m_{\alpha} - m_{\alpha}(\mu)\right|^2\right] \leq C(\Omega, d, K,k)\left(t^{-1/2} + m^{-1/d}\right)^2. \label{eq:moment_estimator_error}
\end{align}

To this end, we employ \Cref{thm:moment_polynomials} and observe since $\supp(K)+\Theta \subseteq \Omega$ (Assumption \ref{ass:kernel}$(i)$) for some unique suitable monic polynomial $\psi_\alpha$ of order $\alpha$ that
\begin{align*}
	\hat m_\alpha- m_\alpha(\mu)%
	&= \sum_{i = 1}^{m} \psi_\alpha(\gamma_i) \frac{X_i}{t} - \int_{\Omega}\psi_\alpha(x) (K\ast \mu)(x) \dif x,
\end{align*}
where $X_i \sim \textup{Poi}(t K\ast \mu(B_i))$. %
It thus follows  that 
\begin{align}
	\left|\hat m_\alpha- m_\alpha(\mu)\right|^2 &= \left|\sum_{i = 1}^{m} \psi_\alpha(\gamma_i) \frac{X_i}{t} - \int_{\Omega}\psi_\alpha(x) (K\ast \mu)(x) \dif x\right|^2\notag\\
	&\leq 2\left|\sum_{i = 1}^{m} \psi_\alpha(\gamma_i) \frac{X_i}{t} - \sum_{i = 1}^{m} \psi_\alpha(\gamma_i) K\ast \mu(B_i) \right|^2\label{eq:stochasticErrorMomentEstimator}\\
	&+ 2\left|\sum_{i = 1}^{m} \psi_\alpha(\gamma_i) K\ast \mu(B_i) - \int_{\Omega}\psi_\alpha(x) (K\ast \mu)(x) \dif x \right|^2.\label{eq:discretizationErrorMomentEstimator}
\end{align}
We upper bound the terms in \eqref{eq:stochasticErrorMomentEstimator} and \eqref{eq:discretizationErrorMomentEstimator} separately. For the second term it holds
\begin{align}
	\left|\sum_{i = 1}^{m} \psi_\alpha(\gamma_i) \frac{K\ast \mu(B_i)}{m} - \int_{\Omega}\psi_\alpha(x) (K\ast \mu)(x) \dif x \right|^2
	&\leq  \left(\sum_{i = 1}^{m} \int_{B_i}\left|\psi_\alpha(\gamma_i) - \psi_\alpha(x) \right| (K\ast \mu)(x) \dif x\right)^2\notag \\
	&\leq C(\Omega, \psi_\alpha,d) \max_{i = 1, \dots, m}\diam(B_i)^2\notag\\
	&\leq C(\Omega, \psi_\alpha,d) m^{-2/d},\label{eq:error1}
\end{align}
where we used in the second inequality that $\psi_\alpha$ is a polynomial and therefore Lipschitz on the detection domain $\Omega$ and in the third inequality we employed Assumption \ref{ass:bins}. Recall that the Lipschitz constant only depends on the coefficients of $\psi_\alpha$ which are in turn are determined by $K$ (recall \Cref{thm:moment_polynomials}).  %
This treats the discretization error. 

We continue with bounding the first term \eqref{eq:stochasticErrorMomentEstimator}. Here, we first note that 
\begin{align*}
	\EE\left[\sum_{i = 1}^{m} \psi_\alpha(\gamma_i) \frac{X_i}{t}\right] &= \sum_{i =1}^{m}\frac{\psi_\alpha(\gamma_i)}{t}\EE\left[ X_i\right] = \sum_{i =1}^{m}\psi_\alpha(\gamma_i) K\ast \mu(B_i).
\end{align*} 
Hence, by Jensen's inequality, since $\psi_\alpha$ is bounded on $\Omega$ by some constant $C(\Omega, d, K, k)>0$, we obtain that the statistical estimation error is bounded by 
\begin{align}
	 \EE \left|\sum_{i = 1}^{m} \psi_\alpha(\gamma_i) \frac{X_i}{t} - \sum_{i = 1}^{m} \psi_\alpha(\gamma_i) K\ast \mu(B_i) \right|^2
	&= \textup{Var}\left(\sum_{i = 1}^{m} \psi_\alpha(\gamma_i) \frac{X_i}{t} \right)\notag \\*
	&=\sum_{i = 1}^{m}  \left(\frac{\psi_\alpha(\gamma_i)}{t}\right) ^2 t K\ast \mu(B_i) \notag \\*
	&\leq C(\Omega,\psi_\alpha,d) t^{-1}.\label{eq:error2_momentError}
\end{align}
Combining the bounds from \eqref{eq:stochasticErrorMomentEstimator}--\eqref{eq:error2_momentError} implies \eqref{eq:moment_estimator_error}.

\subsubsection*{Step 2. Statistical analysis of general method of moment estimator}

We now put our focus on the analysis of the general method of moments estimator $\hat \mu_{t,m} = \hat \mu_{t,m, \Theta}^{\textup{MM}}$ defined in \eqref{eq:method_of_moments_general}. Here we obtain by a straight-forward computation that
\begin{align*}
   M_k^2(\hat \mu_{t,m}, \mu) &= \Bigg(\sum_{\substack{\alpha\in \NN_0^d\backslash\{0\}\\ 1\leq \|\alpha\|_1\leq k}} |m_\alpha(\hat \mu_{t,m}) - m_\alpha(\mu)|\Bigg)^2\\
   &\leq \Bigg(\sum_{\substack{\alpha\in \NN_0^d\backslash\{0\}\\ 1\leq \|\alpha\|_1\leq k}} |m_\alpha(\hat \mu_{t,m}) - \hat m_{\alpha}| + |\hat m_{\alpha}-m_\alpha(\mu)|\Bigg)^2\\
   &\leq 2|\{\alpha\in \NN_0^d\colon \|\alpha\|_1\leq k\}| \bigg(\sum_{\substack{\alpha\in \NN_0^d\backslash\{0\}\\ 1\leq \|\alpha\|_1\leq k}} |m_\alpha(\hat \mu_{t,m}) - \hat m_{\alpha}|^2 + |\hat m_{\alpha}-m_\alpha(\mu)|^2\bigg)\\
   &\leq 2|\{\alpha\in \NN_0^d\colon \|\alpha\|_1\leq k\}| \bigg( 2\sum_{\substack{\alpha\in \NN_0^d\backslash\{0\}\\ 1\leq \|\alpha\|_1\leq k}} |m_\alpha(\mu) - \hat m_{\alpha}|^2  \bigg), 
\end{align*}
where the second-to-last inequality follows from the Cauchy-Schwarz inequality and the last inequality is a consequence of the definition of the method of moment estimator as minimizer. Combining the above computation with \eqref{eq:moment_estimator_error} we obtain 
\begin{align*}
   \EE\left[M_k^2(\hat \mu_{t,m}, \mu)\right] \leq C(\Omega, d, K,k)\left(t^{-1/2} + m^{-1/d}\right)^2.
\end{align*}
It remains to specify some $\rho$ to ensure that the second term on the left-hand side in \eqref{eq:mm_rate} is similarly bounded. To this end, note that the moment difference $M_k$ is continuous and that $\calU_k(\Theta)$ is by compactness of $\Theta$ (Assumption \ref{ass:bins}) compact with respect to the weak topology. Hence, it follows that specifying $\rho \coloneqq 1+ \max_{\tilde \mu, \mu\in \calU_k(\Theta)} M_k(\tilde \mu, \mu) = C(\Theta,d, k)$ yields that 
\begin{align*}
   \EE\left[W_k^{2k}(\hat \mu_{t,m}, \mu)\mathds{1}(M_{k}(\hat \mu_{t,m}, \mu)\geq \rho)\right] =0, 
\end{align*}
and the validity of \eqref{eq:mm_rate} for the general method of moments estimator follows. 

\subsubsection*{Step 3. Statistical analysis of complex method of moment estimator}
In the following, we consider the complex method of moment estimator $\hat \mu_{t,m}= \hat\mu_{t,m,\CC}^{\textup{MM}}$, defined in  \eqref{eq:method_of_moments_complex}. Specifically, we have $d = 2$. Using  Lemma \ref{lem:moment_to_measure} we observe that $m_{\alpha}(\hat \mu_{t,m}) = \hat m_{\alpha}$ for each $\alpha \in \{1, \dots, k\}$. In combination with Cauchy-Schwarz inequality  and \eqref{eq:moment_estimator_error}, it is evident that
\begin{align*}
   \EE[M_k^2(\hat \mu_{t,m}, \mu)] &\leq k \EE\left[\left(\sum_{\substack{\alpha=1}}^{k} |m_\alpha(\mu) - \hat m_{\alpha}|^2\right)\right] \leq C(\Omega, K,k)\left(t^{-1/2} + m^{-1/d}\right)^2.
\end{align*}
This confirms an upper bound for the moment difference. It remains to specify some $\rho>0$ such that the second term in \eqref{eq:mm_rate} admits a similar upper bound. To this end, note by Cauchy-Schwarz inequality that 
\begin{align}
   & \EE\left[W_k^{2k}(\hat \mu_{t,m}, \mu)\mathds{1}(M_{k}(\hat \mu_{t,m}, \mu)\geq \rho)\right]\leq  \EE\left[W_k^{4k}(\hat \mu_{t,m}, \mu)\right]^{1/2} \PP\left(M_{k}(\hat \mu_{t,m}, \mu)\geq \rho\right)^{1/2}.\label{eq:bound_WassersteinMoment}
\end{align}
Observe that the probability term on the right-hand side is upper bounded by 
\begin{align*}
   \PP\left( M_{k}(\hat \mu_{t,m}, \mu)\geq \rho\right)&\leq \PP\left(\exists \alpha \in \{1, \dots, k\} \colon |m_{\alpha}(\hat \mu_{t,m}) - m_{\alpha}(\mu)| \geq \rho/k\right)\\
   &\leq \sum_{\alpha = 1}^{k} \PP\left( |m_{\alpha}(\hat \mu_{t,m}) - m_{\alpha}(\mu)| \geq \rho/k\right)\\
   &\leq \sum_{\alpha = 1}^{k} \PP\left( m_{\alpha}(\hat \mu_{t,m})\geq m_{\alpha}(\mu)+ \rho/k\right) + \PP\left( m_{\alpha}(\hat \mu_{t,m})\leq m_{\alpha}(\mu) - \rho/k\right).
\end{align*}
Hence, if $\rho> \tilde \rho_1 \coloneqq k\max_{\mu\in \calU_k(\Theta), \alpha = 1, \dots, k}m_{\alpha}(\mu)$, it follows that $m_{\alpha}(\mu) - \rho/k<0$ and above display reduces to  
\begin{align*}
   \PP\left( M_{k}(\hat \mu_{t,m}, \mu)\geq \rho\right)&\leq \sum_{\alpha = 1}^{k} \PP\left( m_{\alpha}(\hat \mu_{t,m})\geq m_{\alpha}(\mu)+ \rho/k\right)\leq \sum_{\alpha = 1}^{k} \PP\left( m_{\alpha}(\hat \mu_{t,m})\geq  \rho/k\right).
\end{align*}
To control the right-hand side, recall that 
\begin{align}
   t m_{\alpha}(\hat \mu_{t,m}) =  t\hat m_{\alpha} \sim \textup{Poi}\left(t \sum_{i = 1}^{m}\psi_{\alpha}(\gamma_i)K\ast \mu(B_i)\right),\label{eq:PoissonConcentrationBound}
\end{align}
which implies for $\rho> \tilde \rho_2 \coloneqq 2k\max_{\mu\in \calU_k(\Theta), \alpha = 1, \dots, k}\sum_{i = 1}^{m}\psi_{\alpha}(\gamma_i)K\ast \mu(B_i)$ via standard exponential tail bounds for Poisson distributions that 
\begin{align*}
   \PP\left( m_{\alpha}(\hat \mu_{t,m})\geq  \rho/k\right)%
   &\leq  \exp\left(-t \frac{(\rho/k - \sum_{i = 1}^{m}\psi_{\alpha}(\gamma_i)K\ast \mu(B_i))^2}{\rho/k}\right)\\
   &\leq  \exp\left(-t \left(\rho/k - 2\sum_{i = 1}^{m}\psi_{\alpha}(\gamma_i)K\ast \mu(B_i)\right)\right)\\
   &\leq \exp\left(-t(\rho  -\tilde \rho_2)/k\right).
\end{align*}
In particular, we conclude for $\rho\coloneqq 2\max(\tilde \rho_1, \tilde \rho_2) = C(\Theta, d, K,k)$ that 
\begin{align}\label{eq:momemt_tail_bound_prob}
   \PP\left( M_{k}(\hat \mu_{t,m}, \mu)\geq \rho\right)^{1/2} \leq \sqrt{k}\exp(-t \cdot C(\Omega, \Theta, K, k)) \leq C(\Omega, \Theta, K,k)t^{-1}.
\end{align}
To finish the proof of the assertion for the complex method of moments, we need to show that the first expectation in \eqref{eq:bound_WassersteinMoment} is bounded by a constant that does not depend on $t$ and $m$. To confirm this, note by our bound on the atoms of a measure in terms of its complex moment  (Lemma~\ref{lem:moment_to_measure}) and our concentration bound in \eqref{eq:PoissonConcentrationBound} that 
\begin{align}
   &\EE\left[W_k^{4k}(\hat \mu_{t,m},  \mu) \right]\notag \\
   = \;&\EE\left[W_k^{4k}(\hat \mu_{t,m},  \mu)\cdot \mathds{1}(\max(|\hat m_1|, \dots, |\hat m_k|)\leq \rho/k ) \right]\notag \\*
   &+%
	\sum_{l = 1}^\infty \EE\left[W_k^{4k}(\hat \mu_{t,m},  \mu)\cdot \mathds{1}(\max(|\hat m_1|, \dots, |\hat m_k|)\in  (l\rho/k,(l+1)\rho/k])\right]\notag \\
   \leq \;& \EE\left[2^{4k-1}\left(W_k^{4k}(\hat \mu_{t,m}, \delta_0) + W_k^{4k}(\delta_0, \mu)\right)\mathds{1}(\max(|\hat m_1|, \dots, |\hat m_k|)\leq \rho/k)\right]+\notag\\*
   & +\sum_{l = 1}^\infty \EE\left[ 2^{4k-1}\left(W_k^{4k}(\hat \mu_{t,m}, \delta_0) + W_k^{4k}(\delta_0, \mu)\right)\cdot \mathds{1}(\max(|\hat m_1|, \dots, |\hat m_k|)\in  (l\rho/k,(l+1)\rho/k]) \right]\notag \\
   \leq \;& C(k) \sum_{l = 1}^{\infty} \left(\rho^{4k^2} (l+1)^{4k^2} + C(\Theta,k)\right)k \exp\left(-t\left[\frac{l \rho}{k} - \frac{\tilde\rho_2}{k}\right]\right) \notag \\ 
   \leq \;& C(\Omega, k)\sum_{l =2}^\infty \left((l+1)^{4k^2} + C(\Theta,k)\right)\exp\left(-\left[\frac{l \rho}{k} - \frac{\tilde\rho_2}{k}\right]\right)= C(\Omega,\Theta, k,\rho) = C(\Omega,\Theta, K,k)<\infty.\notag
\end{align}
Combining the above display with \eqref{eq:bound_WassersteinMoment} and \eqref{eq:momemt_tail_bound_prob} yields the assertion for the complex method of moments estimator.

\subsubsection*{Step 4. Statistical analysis of real method of moment estimator}
In this last step we analyze the real method of moment estimator $\hat \mu_{t,m}= \hat \mu_{t,m,\RR}^{\textup{MM}}$, which we introduced in \eqref{eq:method_of_moments_real_line}. For the analysis recall that the real estimator is constructed by  first computing the complex method of moments estimator, $\hat \mu_{t,m,\CC}^{\textup{MM}}$, based on real-line based moment estimates $(\hat m_1, \dots, \hat m_k)$. Following along the lines of step 3.\ we note that all bounds immediately carry over, with the slight difference that the squared discretization error is order $m^{-2}$ due to the univariate setting ($d = 1$). 

Based on our black-box bound for estimators (Corollary \ref{cor:blackbox_bound}), we thus infer that the complex method of moment $\hat \mu_{t,m,\CC}^{\textup{MM}}$ achieves for $\epsilon^2\coloneqq C(t^{-1/2} + m^{-1})^2$ a certain bound in recovering the underlying measure $\mu$ with respect to the Wasserstein distance and its local counterpart from \eqref{eq:local_Wasserstein}. Since $\mu$ is concentrated on the real line it follows that projecting the atoms of $\hat \mu_{t,m,\CC}^{\textup{MM}}$ never enlarges the (local) Wasserstein risk, see Lemma \ref{lem:WassersteinProjectionRealLine} below. In consequence, the real method of moment estimator achieves a (local) Wasserstein risk that is no worse than the complex method of moments estimator based on moment estimators for the real line, and  thus finishes the proof. \qed 

\begin{lemma}\label{lem:WassersteinProjectionRealLine}
	For any pair of probability measures $\mu, \nu \in \calP(\CC)$ it holds 
	\begin{align*}
		W_1(\Re_{\#}\mu, \Re_{\#}\nu) \leq W_1(\mu, \nu)
	\end{align*}
	\end{lemma}
	\begin{proof}
		This directly follows from \citet[Lemma A.1]{nies2021transport} in conjunction with the fact that $|\Re(x) - \Re(y)|\leq |x-y|$ for any $x,y\in \CC$. 
	\end{proof}

\section{Proofs for Section~\ref{sec:mle}}
\label{app:pfs_mle}

\subsection{Proof of Theorem \ref{thm:momentL2-bounds_all}}
\label{app:pf_thm_momentL2-bounds_all}
	We prove the two assertions for Gaussian and general compact kernels by means of representing the $L^2$-difference in terms of suitable orthonormal basis which are intimately linked to the underlying moments.  
	This approach is inspired by the proof of Lemma 4.4 of \cite{doss2023}, which proves the statement for the Gaussian kernel on the real line. 

	\noindent 
{\bf{Part $(i)$. Gaussian Kernels. }}	The proof for the Gaussian kernel is divided into three steps. We first show the Assertion for the setting where  $\Sigma$ is the identity matrix, then we extend it to diagonal matrices, and finally we prove the assertion for general symmetric positive semi-definite matrices. %

{\bf{Step 1.1. Standard Normal Kernel. }}
The first step is concerned with the standard Gaussian kernel, i.e., $K(x) = (2\pi)^{d/2}\exp(-\|x\|^2/2)$, for which we lift the argument by \citet[Lemma 4.4]{doss2023} up to the multivariate setting. Specifically, for $j \in \NN_0^d$ we set $\Gamma_j(y) \coloneqq \prod_{i = 1}^{d} \gamma_{j_i}(y)$ where we denote for $n\in \NN_0$ the univariate Gaussian weighted Hermite polynomial by $\gamma_n(y) \coloneqq (\sqrt{2}\phi_1(\sqrt{2}y)/{n!})^{1/2} H_n(\sqrt{2}y)$, here $H_n$ denotes the probabilists Hermite polynomial and $\phi_1$ is the univariate standard Gaussian kernel. In particular, since $\{\gamma_n\}_{n\in \NN_0}$ forms an orthonormal basis for $L^2(\RR)$ (see, e.g., \citealt[Equation 22.2.15]{abramowitz1948handbook} or \citealt{johnston2014weighted}) it follows by Lemma \ref{lem:productONB} that $\{\Gamma_j\}_{j\in \NN_0^d}$ is an orthonormal basis for $L^2(\RR^d)$. Hence, by squared integrability of  $K\ast \xi$ over $\RR^d$, it follows for $\xi \in \{\mu, \nu\}$ that $K\ast \xi(\cdot) = \sum_{j\in \NN_0^d} a_j(\xi) \Gamma_j(\cdot)$, where the summation is understood in the $L^2$-sense and 
\begin{align*}
	\Gamma_j(\xi)&\coloneqq \langle \Gamma_j, K\ast \xi\rangle_{L^2(\RR^d)} = \EE_{Y\sim K\ast\xi}\left[ \Gamma_j(Y)\right] = \EE_{X\sim \xi}\left[ \Gamma_j\ast K(X)\right]. 
\end{align*}
Further, by Lemma \ref{lem:convolutionProduct}, the product representations of $\Gamma_j$ and $K$, and \citet[7.374.6, p. 803]{gradshteyn2014table} it holds  $$\Gamma_j\ast K(x) = \prod_{i = 1}^{d} \gamma_{j_i}\ast \phi_1(x_i) =  \prod_{i = 1}^{d} \frac{1}{2^{\frac{j_i+1}{2}}\pi^{1/4}}x_i^{j_i} \exp(-x_i^2/4).$$
This implies that 
\begin{align*}
	a_j(\xi) = \frac{1}{2^{\frac{\|j\|+d}{2}}\pi^{d/4}} \EE_{X\sim \xi}\left[X^j \exp(-\|X\|^2/4)\right].
\end{align*}
Altogether, we obtain that 
\begin{align*}
	&\left\|K\ast(\mu - \nu)\right\|_{L^2(\RR^d)}^2\\
	=& \sum_{j \in \NN_0^d} \frac{1}{2^{{\|j\|+d}}\pi^{d/2}} \left| \EE_{X\sim \mu}\left[X^j \exp(-\|X\|^2/4)\right] -  \EE_{X\sim \nu}\left[X^j \exp(-\|X\|^2/4)\right]\right|^2\\
	\geq& \sum_{j \in \{0, \dots, 2k-1\}^d} \frac{1}{2^{{\|j\|_1+d}}\pi^{d/2}} \left| \EE_{X\sim \mu}\left[X^j \exp(-\|X\|^2/4)\right] -  \EE_{X\sim \nu}\left[X^j \exp(-\|X\|^2/4)\right]\right|^2
\end{align*}
It remains to relate the right-hand side to the moment difference of order $2k-1$ between the $k$-atomic probability measures $\mu$ and $\nu$. To this end, we employ the following result. 

\begin{lemma}%
	\label{lem:prod_interpolation_Gaussians}
	For any pair of $k$-atomic probability measures $\mu, \nu \in \calP_k(B_\infty(0,r))$ and $j \in \NN_0^d$ there exist polynomials $P_{j_1}, \dots, P_{j_d}\colon \RR \to \RR$ of degree at most $2k-1$ whose coefficients are upper bounded by a constant $C(j, k,r)>0$ such that for $X\sim \xi$ with $\xi \in \{\mu, \nu\}$ it holds
	\begin{align*}
		\EE\left[ X^j\right] = \EE\left[\exp\left(-\frac{1}{4}\|X\|^2\right) \prod_{i = 1}^{d}P_{j_i}(X_i)\right]. 
	\end{align*}
	\end{lemma}
Consequently, using Lemma \ref{lem:prod_interpolation_Gaussians} there exists a positive constant $C= C(d,j,k,r)>0$, where we recall that $\Theta\subseteq B_\infty(0,r)$, such that 
\begin{align*}
	&\quad \left|m_j(\mu)- m_j(\nu)\right|^2 \\
	&= \left| \EE_{X\sim \mu}\left[P_j(X)\exp(-\|X\|^2/4)\right] -  \EE_{X\sim \nu}\left[P_j(X)\exp(-\|X\|^2/4)\right]\right|^2\\
	&\leq C \sum_{j' \in \{0, \dots, 2k-1\}^d} \left|\EE_{X\sim \mu}\left[X^{j'}\exp(-\|X\|^2/4)\right] -  \EE_{X\sim \nu}\left[X^{j'}\exp(-\|X\|^2/4)\right]\right|^2\\
	&\leq C 2^{{\|j\|_1+d}}\pi^{d/2}\left\|K\ast (\mu - \nu)\right\|^2_{L^2(\RR^d)}.
\end{align*}
Taking the square-root on both sides and summing over $j\in \NN_0^d$ with $\|j\|_1\leq 2k-1$ yields the claim for the multivariate standard Gaussian kernel. 
\textbf{Step 1.2. Diagonal Covariance Matrices.}
We now prove the assertion for $K_{\Sigma}(x) \coloneqq (2\pi)^{-d/2}\det(\Sigma)^{-1/2}\exp(-\frac{1}{2}x^\top \Sigma^{-1} x)$ where $\Sigma = \textup{diag}(\sigma_1, \dots, \sigma_d)$ for $\sigma_1, \dots, \sigma_d>0$ using a reduction onto the identity matrix case. For clarity of the argument  we emphasize the covariance $\Sigma$ in the notation of the kernel. 
We start by observing that 
\begin{align*}
	&\quad \|K_\Sigma \ast (\mu - \nu)\|_{L^2(\RR^d)}^2\\
	&= \int_{\RR^d} \left|\int_{B_{\infty}(0,r)}{(2\pi)^{-d/2}|\det(\Sigma)|^{-1/2}}\exp\left(-\frac{1}{2}(x-\theta)^\top\Sigma^{-1}(x-\theta)\right)\dif(\mu - \nu)(\theta)\right|^2\dif x\\
	&= \int_{\RR^d} \left|\int_{\Sigma^{-1/2}B_{\infty}(0,r)} (2\pi)^{-d/2}\exp\left(-\frac{1}{2}\|y - \eta\|^2\right)\dif(\Sigma^{-1/2}_{\#}\mu - \Sigma^{-1/2}_{\#}\nu)\right|^2\dif y\\
	&= \|K_{\textup{Id}}\ast (\Sigma^{-1/2}_{\#}\mu - \Sigma^{-1/2}_{\#}\nu)\|^2_{L^2(\RR^d)}
\end{align*}
where we used in the second to last equality the substitution $y = \Sigma^{-1/2}x$ and denote by $\Sigma^{-1/2}_{\#}\xi = \sum_{i =1}^{k} w_i \delta_{\Sigma^{-1/2}\theta_i}$ the push-forward of $\xi = \sum_{i =1}^{k} w_i \delta_{\theta_i}$ under $\Sigma^{-1/2}$. Hence, based on step 1.1.\ it follows that there exists a positive constant $C =C(\Sigma, d,k,r) = C(\max_{\eta \in \Sigma^{-1/2}B_{\infty}(0,r)}\|\eta\|_\infty, d,k)>0$ such that 
\begin{align*}
	M_{2k-1}(\Sigma^{-1/2}_{\#}\mu, \Sigma^{-1/2}_{\#}\nu) \leq C \|K_\Sigma \ast (\mu - \nu)\|_{L^2(\RR^d)}.
\end{align*}
To relate the left-hand side with the moment difference between $\mu$ and $\nu$ we note for $j\in \NN_0^d$ with $\|j\|_1\leq 2k-1$ and $\xi = \sum_{i = 1}^{k}w_i \delta_{\theta_i}\in \{\mu, \nu\}$  that 
\begin{align*}
	m_j(\Sigma^{-1/2}_{\#}\xi) &= \sum_{i =1}^{k} w_i \prod_{h =1}^{d} (\sigma_h^{-1/2}\theta_{i,h})^{j_h} = \prod_{h = 1}^{d}\sigma_h^{-j_h/2} m_j(\xi) = \sigma^{-j/2} m_j(\xi),  
\end{align*}
where we set $\sigma\coloneqq (\sigma_1, \dots, \sigma_d)$. 
Hence, for the squared moment difference between $\Sigma^{-1/2}_{\#}\mu$ and $\Sigma^{-1/2}_{\#}\nu$ it follows that 
\begin{align*}
	M_{2k-1}(\Sigma^{-1/2}_{\#}\mu, \Sigma^{-1/2}_{\#}\nu) = &\sum_{\substack{j\in \NN_0^d, \|j\|_1\leq 2k-1}} \left|m_j(\Sigma^{-1/2}_{\#}\mu) - m_{j}(\Sigma^{-1/2}_{\#}\nu)\right|\\
	=&\sum_{j \in \NN_0, \|j\|_1\leq 2k-1} \sigma^{-j/2}\left|m_{j}(\mu) - m_j(\nu)\right| \\
	\geq \, &(\|\sigma\|_\infty^{1}\vee \|\sigma\|_\infty^{2k-1})^{-1/2} M_{2k-1}(\mu, \nu),%
\end{align*}
where we used in the inequality that $m_0(\mu)  =m_0(\nu)$. Overall, this asserts that 
\begin{align}\label{eq:GaussianMomentComparison_Diagonal}
	M_{2k-1}(\mu, \nu) \leq C \cdot (\|\sigma\|_\infty^{1}\vee \|\sigma\|_\infty^{2k-1})^{1/2}\|K_\Sigma \ast (\mu - \nu)\|_{L^2(\RR^d)}.
\end{align}
\textbf{Step 1.3. General Positive Definite Covariances.}
To extend the claim to general positive definite covariances $\Sigma$ we consider the spectral decomposition $\Sigma = Q^\top \tilde \Sigma Q$ 
where $Q$ is orthonormal and $\tilde \Sigma = \textup{diag}(\sigma_1^2, \dots, \sigma_d^2)$
 is diagonal. The assertion now follows from the following derivation, for which we provide the proof to all individual step below, 
\begin{align}
	\|K_{\Sigma}\ast (\mu - \nu)\|_{L^2(\RR^d)} &= \|K_{\tilde \Sigma}\ast (Q_{\#}\mu - Q_{\#}\nu)\|_{L^2(\RR^d)} \label{eq:GaussianMomentComparisonProofBound1}\\
	&\gtrsim_{\Sigma, d,k,r} M_{2k-1}(Q_{\#}\mu, Q_{\#}\nu)\label{eq:GaussianMomentComparisonProofBound2}\\
	&\gtrsim_{k,d} \sup_{\eta \in \SS^d} M_{2k-1}((Q_{\#}\mu)^{\eta}, (Q_{\#}\nu)^{\eta})\label{eq:GaussianMomentComparisonProofBound3}\\
	&= \sup_{\eta \in \SS^d} M_{2k-1}(\mu^{\eta}, \nu^{\eta})\label{eq:GaussianMomentComparisonProofBound4}\\
	&\gtrsim_{k} M_{2k-1}(\mu, \nu).\label{eq:GaussianMomentComparisonProofBound5}
\end{align}
To show the first equality we first note for $\xi= \sum_{i = 1}^{k} w_i \delta_{\theta_i}\in \{\mu, \nu\}$ by $\Sigma^{-1} =  Q^\top \tilde \Sigma^{-1} Q$ that 
\begin{align}
	K_\Sigma\ast \xi(x) &= \sum_{i  =1}^{k} p_k \exp\left( -\frac{1}{2}(x-\theta_i)^\top\Sigma^{-1}(x - \theta_i) \right)\left((2\pi)^{d} \det(\Sigma)\right)^{-1/2}\notag \\
	&= \sum_{i  =1}^{k} p_k \exp\left( -\frac{1}{2}(Qx-Q\theta_i)^\top\tilde \Sigma^{-1}(Qx - Q\theta_i) \right)\left((2\pi)^{d} \det(\tilde \Sigma)\right)^{-1/2}\notag \\
	&= (K_{\tilde \Sigma}\ast (Q_{\#}\xi))(Qx)\label{eq:KernelOrthonalTrafo_Formula}
\end{align}
Consequently, the Equality in \eqref{eq:GaussianMomentComparisonProofBound1} follows from 
\begin{align*}
	\|K\ast (\mu - \nu)\|^2_{L^2(\RR^d)} &= \int_{\RR^d} \left| \int K(x - y)\dif (\mu - \nu)(y)\right|^2\dif x\\
	&= \int_{\RR^d}\left| \int K(Q \tilde x - y) \dif(\mu - \nu)(y) \right|^2\dif \tilde x\\
	&= \int_{\RR^d}\left| \int K(Q \tilde x - Q y) \dif(\mu - \nu)(Q^{-1} y) \right|^2\dif \tilde x\\
	&= \int_{\RR^d}\left| \int K(\tilde x - \tilde y) \dif(Q_{\#}\mu - Q_{\#}\nu)(y) \right|^2\dif \tilde x \\
	&= \|K\ast (Q_{\#}\mu - Q_{\#}\nu)\|^2_{L^2(\RR^d)}.
\end{align*}
Inequality \eqref{eq:GaussianMomentComparisonProofBound2} follows from step 1.2. Inequalities \eqref{eq:GaussianMomentComparisonProofBound3} and \eqref{eq:GaussianMomentComparisonProofBound5}  follow moment difference bounds between measures and sliced measures from Lemma \ref{lem:sliced_multivariate_moment_bound} and \ref{lem:moment_sliced_moment}, respectively. Finally, for Equality \eqref{eq:GaussianMomentComparisonProofBound4} note for $\eta \in \SS^{d}$ and $\xi \in \{\mu, \nu\}$ for $\ell \in \NN$ that 
\begin{align*}
	m_{\ell}(\xi^\eta)  = \EE_{X\sim \xi}[\langle X, \eta \rangle ^\ell] = \EE_{X\sim \xi}[\langle Q X, Q\eta \rangle ^\ell] = \EE_{Y\sim Q_{\#}\xi}[\langle Y, Q\eta \rangle ^\ell] = m_{\ell}((Q_{\#}\xi)^{Q\eta}),
\end{align*}
which yields that 
\begin{align}
	\sup_{\eta \in \SS^d} M_{2k-1}(\mu^{\eta}, \nu^{\eta}) =\sup_{\eta \in \SS^d} M_{2k-1}((Q_{\#}\mu)^{Q\eta}, (Q_{\#}\nu)^{Q\eta}) =\sup_{\eta \in \SS^d}M_{2k-1}((Q_{\#}\mu)^{\eta}, (Q_{\#}\nu)^{\eta}).\label{eq:slicedMomentDifferenceInvarianceOrthonormal}
\end{align}
This proves the moment comparison bound for Gaussian kernels.

\noindent
{\bf{Part $(ii)$. Compact Kernels.}}
We will prove the following more general bound: 
If $\Theta$ is a bounded set contained in $B_\infty(0,r)$ for $r > 0$ 
and the kernel $K$ is supported in $B_\infty(0,\tau)$ for $\tau>0$,  
then there exists a constant $C = C(d, K,k) > 0$ such that:
$$	M_{2k-1}(\mu,\nu) \leq C \Big((r+\tau)^{d/2}\vee (r+\tau)^{k+d/2}\Big)\|K\ast (\mu - \nu)\|_{L^2(B_\infty(0,r+\tau))}.
$$
	Throughout the proof of this assertion, we write $ R= (r+\tau)$. Take for $L^2([-R,R]^d)$ the orthonormal basis $\{\phi_j^R\}_{j\in \NN_0^d}$ of renormalized Legendre polynomials (Lemma \ref{lem:LegendrePolynomialOrthonormal}) and observe that 
	\begin{align*}
		  \norm{K\ast (\mu-\nu)}_{L^2([-R,R]^d)}^2\geq \sum_{\substack{j \in \NN_0^d, \|j\|_1\leq 2k-1}} \langle \phi_j^R, K\ast (\mu-\nu)\rangle_{L^2([-R,R]^d)}^2,
	\end{align*}
	where the summand associated to the index $j\in\NN_0^d$ is given by 
	\begin{align*}
		  &\,\langle \phi_j^R, K\ast (\mu-\nu)\rangle_{L^2([-R,R]^d)}\\
		  =&\,\int_{[-R,R]^d}2^{-|j|} \prod_{i = 1}^{d}\left( \sqrt{j_i+1/2} \sum_{\substack{l_i = 0\\ j_i-l_i \text{ even}}}^{j_i} R^{-l_i-1/2} (-1)^{(j_i-l_i)/2} \binom{j_i}{(j_i-l_i)/2}\binom{j_i+l_i}{j_i}x_i^{l_i} \right) \\
		  & \quad \quad \times(K\ast \mu(x) - K\ast \nu(x))\dif x\\
		  =&\, 2^{-|j|} \!\! \! \!  \sum_{\substack{l\in \NN_0^d, l\leq j\\ j-l\in 2\NN_0^d}} \! \! R^{-|l|-d/2}(-1)^{|j-l|/2}\prod_{i = 1}^{d} \left[\sqrt{j_i+1/2}\binom{j_i}{(j_i-l_i)/2}\binom{j_i+l_i}{j_i}\right] (m_l(K\ast \mu) - m_l(K\ast \nu)),
	\end{align*}
	where we used for the last equality that $\supp(K\ast \mu+K\ast \nu)\subseteq [-R,R]^d$. 
Next enumerate the vectors $l\in \NN^d_0$ with $\|l\|_1\leq 2k-1$ by $v_1, \dots, v_{h(d,k)}$ for appropriate $h(d,k)\in \NN$ such that $\|v_p\|_1\geq \|v_q\|_1$ if $p\leq q$. Further, define the square-matrix $A = (a_{pq})_{p,q=1}^{h(d,k)}$ of size equal to the number of vectors $l\in \NN^d_0$ with $\|l\|_1\leq k$ with entries 
\begin{align*}
	a_{pq} &= \begin{cases}
	0 & \text{ if } p<q \text{ or } v_p - v_q\notin 2\NN_0^d,\\
	2^{-\|v_p\|_1}  (-1)^{\|v_p-v_q\|_1/2}\prod_{i = 1}^{d} \left[\sqrt{v_{p,i}+1/2}\binom{v_{p,i}}{(v_{p,i}-v_{q,i})/2}\binom{v_{p,i}+v_{q,i}}{v_{p,i}}\right]   & \text{ if } p\geq q \text{ and } v_p - v_q \in 2\NN_0^d,
	\end{cases}
\end{align*}

Note that $A$ is an upper triangle matrix with positive entries along the diagonal and that 
\begin{align*}
	&\quad \sum_{\substack{j \in \NN_0^d\\ |j|\leq 2k-1}} \langle \phi_j^R, K\ast(\mu-\nu)\rangle_{L^2([-R,R]^d)}^2\\
	&= \|A (R^{-\|v_{1}\|_1-d/2} m_{v_1}(K\ast(\mu-\nu)), \dots, R^{-\|v_{h(d,k)}\|_1-d/2}m_{v_{h(d,k)}}(K\ast(\mu-\nu)))\|_2^2\\
	&\geq \sigma_1(A)\|(R^{-\|v_{1}\|_1-d/2}m_{v_1}(K\ast(\mu-\nu)), \dots, R^{-\|v_{h(d,k)}\|_1-d/2}m_{v_{h(d,k)}}(K\ast(\mu-\nu)))\|_2^2\\
	&\geq \sigma_1(A)(R^{-d} \wedge R^{-2(2k-1)-d})\|(m_{v_{1}}(K\ast(\mu-\nu)), \dots, m_{v_{h(d,k)}}(K\ast(\mu-\nu)))\|_2^2.
\end{align*}
To reduce the right-hand side to the moment difference between the probability measures $\mu$ and $\nu$ observe for $j\in \NN_0^d$ with $\|j\|_1\leq 2k-1$ using Fubini's theorem that
\begin{align*}
	m_j(K\ast \mu) = \int x^j (K\ast \mu)(x)\dif x &=\int\int  x^j K(x-\theta)\dif x \dif\mu(\theta)\\
	&= \int \int (x+\theta)^j  K(x)\dif x \dif\mu(\theta)\\
	&= \sum_{\substack{l\in \NN_0^d\\ l\leq j}} \binom{j}{l}m_{j-l}(K) m_{l}(\mu).%
\end{align*}
Hence, upon defining the upper triangle matrix $M = (m_{p,q})_{p,q = 1}^{h(d,k)}$ with entries given by $m_{p,q} = \mathds{1}(v_p\geq v_q)\binom{v_p}{v_q}m_{v_p-v_q}(K)$ it follows that 
\begin{align*}
	\|(m_{v_{1}}(K\ast (\mu-\nu)), \dots, m_{v_{h(d,k)}}(K\ast (\mu-\nu)))\|_2^2  &= \|M(m_{v_{1}}(\mu-\nu), \dots, m_{v_{h(d,k)}}(\mu-\nu))\|_2^2\\
	& \geq \sigma_1(M) \|(m_{v_{1}}(\mu-\nu), \dots, m_{v_{h(d,k)}}(\mu-\nu))\|_2^2.
\end{align*}
Combining all inequalities, we infer from the relation between $\ell^1$ and $\ell^2$ norm on $\RR^{h(d,k)}$ that 
\begin{align*}
	M_{2k-1}(\mu, \nu) &=\|(m_{v_{1}}(\mu-\nu), \dots, m_{v_{h(d,k)}}(\mu-\nu))\|_1 \\&=\sqrt{h(d,k)}\|(m_{v_{1}}(\mu-\nu), \dots, m_{v_{h(d,k)}}(\mu-\nu))\|_2 \\
	& \leq\sqrt{h(d,k)\sigma_1^{-1}(M) \sigma_1^{-1}(A)}(R^{d/2} \vee R^{2k-1+d/2}) \|K\ast(\mu-\nu)\|_{L^2([-R,R]^d)}\\
	&\lesssim_{d,K,k} (R^{d/2} \vee R^{2k-1+d/2}) \|K\ast(\mu-\nu)\|_{L^2([-R,R]^d)}.
\end{align*}
The precise claim in the assertion now follows by choosing $\tau = 1$ and assuming that  $r\geq 1$. Then it holds $(R^{d/2} \vee R^{2k-1+d/2}) =(1+r)^{2k-1+d/2} \leq (2r)^{2k-1+d/2}\lesssim_{k,d} r^{2k-1+d/2}$.\qed

\subsubsection{Proof of Lemma \ref{lem:prod_interpolation_Gaussians}}
	Let $\mu_i \coloneqq (P_{i})_{\#}\mu$ and $\nu_i\coloneqq (P_{i})_{\#}\nu$ be the projection onto the $i$-th marginal of $\mu$ and let $S_i \coloneqq \supp(\mu_i)\cup\supp(\nu_k)\subseteq [-r,r]$ be their joint support. Note that $s_i \coloneqq |S_i|\leq 2k$. Then, from the proof of Lemma 4.4 of \cite{doss2023} it follows that there exist a polynomial $P_{j_i}$ of degree $s_i-1$ whose coefficients are upper bounded by a constant $C(j_i, k,r)>0$ such that $t_i^{j_i} = P_{j_i}(t_i)\exp(-\frac{1}{4}t_i^2)$ for all $t_i\in S_i$. Hence, for $t\in S\coloneqq \otimes_{i = 1}^{d} S_i$ it holds that $$t^j = \prod_{i = 1}^{d} P_{j_i}(t_i)\exp\left(-\frac{1}{4}t^2\right)= \exp\left(-\frac{1}{4}\|t\|^2\right) \prod_{i =1}^{d} P_{j_i}(t_i),$$
	which yields since $\supp(\mu)\cup \supp(\nu)\subseteq S$ for $X\sim \xi$ with $\xi\in \{\mu, \nu\}$ the desired equality. \qed 

\subsection{Explicit Constants in Theorem~\ref{thm:momentL2-bounds_all}}
\label{app:pf_remark_constants}
Similar to the proof of \Cref{thm:momentL2-bounds_all}$(ii)$, we consider the setting where $\Theta$ is a bounded set contained in $B_\infty(0,r)$ for $r > 0$ 
and the kernel $K$ is supported in $B_\infty(0,\tau)$ for $\tau>0$ and $L^2$-integrable. For an upper bound on suppressed constant in terms of $K$ for $d = 1$ we note according to Lemma \ref{lem:bound_smallest_singular_value_M} below that $\sigma_1(M)^{-1} \lesssim_{k} \left(1 + \int |x|^k K(x) \dif x\right)^{k+1}$. 
This yields 
\begin{align*}
	M_{2k-1}(\mu, \nu) &\leq_k \sqrt{\frac{2k-1}{\sigma_1(A)\sigma_1(M)}}   \left((r+\tau)^{1/2}\vee (r+\tau)^{2k-1/2}\right) \|K\ast(\mu- \nu)\|_{L^2(\RR)}\\
	&\lesssim_k  \left(1 + \int |x|^{2k-1} K(x) \dif x\right)^{2k}\left((r+\tau)^{1/2}\vee (r+\tau)^{2k-1/2}\right)\|K\ast(\mu- \nu)\|_{L^2(\RR)}.%
\end{align*}

\begin{lemma}\label{lem:bound_smallest_singular_value_M}
	Let $K\colon \RR\to [0,\infty)$ be a probability kernel such that 
		$\int |x|^kK(x)\dif x<\infty$ for $k\in \NN$. Then, the smallest singular value of the matrix $M= (m_{l,h})_{l,h =0}^{k}$ with entries $m_{l,h} = \mathds{1}(l\geq h)\binom{l}{h}m_{l-h}(K)$ fulfills 
		\begin{align*}
			\sigma_1(M)^{-1} \leq \left(1+ \int |x|^kK(x)\dif x\right)^{k+1}2^{(k+1)^2}.
		\end{align*}
\end{lemma}

Moreover, for the setting where $\Theta \subseteq B(0,r)\subseteq \CC$ and the probability kernel is supported within the Euclidean ball $B(0,\tau)$ for $r,\tau>0$ we will show in the following for $\mu, \nu\in \calU_k(\Theta)$ that 
\begin{align*}
	M_{2k-1}(\mu, \nu) \leq \pi \sqrt{k} \left[(r+\tau)\vee (r+\tau)^{2k+1}\right]^{1/2}\|K\ast(\mu- \nu)\|_{L^2(\RR^2)}.
\end{align*}
The key technical insight to show the above inequality is that the complex monomials form an orthogonal system (Lemma \ref{lem:ComplexMonomial}) and if $K$ is isotropic, the moment matrix $M$ essentially reduces to a diagonal matrix (\Cref{thm:moment_polynomials}).  

To formalize this argument, define $R = (r+\tau)$ and observe by Lemma \ref{lem:ComplexMonomial}$(ii)$ that 
\begin{align*}
	\|K\ast (\mu - \nu)\|_{L^2(B(0,R))}^2 \geq \sum_{j = 0}^{2k-1} \frac{j+1}{\pi}R^{-2j-2}\left|\int_{B(0,R)} K\ast (\mu - \nu) z^j \dif z\right|^2.
\end{align*}
The integral term can be rewritten as 
\begin{align*}
	\left|\int_{B(0,R)} K\ast (\mu - \nu) z^j \dif z\right|^2 = |m_j(\mu) - m_j(\nu)|^2
\end{align*}
since for $\xi \in \{\mu, \nu\}$ it holds
\begin{align*}
	\int_{B(0,R)} z^j K\ast \mu(z)\dif z &=\int_{\CC} z^j K\ast \mu(z)\dif z \\&=
	 \int_\CC \int_\CC z^j K(z - \theta)\dif z \, \dif\mu(\theta) \\
	&= \int_\CC\int_\CC (z+\theta)^jK(z)\dif z \,\dif \mu(\theta) \\
	&= \sum_{l = 0}^{j}\binom{j}{l}\int_\CC z^{j-l} K(z) \dif z \, m_{j}(\mu) = m_j(\mu),
\end{align*}
where we used in the last line that $K$ is isotropic. 
Combining the previous display yields 
\begin{align*}
	M_{2k-1}(\mu, \nu) &= \sum_{j = 1}^{2k-1}|m_j(\mu) - m_j(\nu)|\\
	&\leq \sqrt{2k-1}\left(\sum_{j = 1}^{2k-1}|m_j(\mu) - m_j(\nu)|^2\right)^{1/2}\\
	&\leq \sqrt{\pi(2k-1)} (R \vee R^{2k}) \|K\ast (\mu - \nu)\|_{L^2(B(0,R))}.
\end{align*}

\subsubsection{Proof of Lemma \ref{lem:bound_smallest_singular_value_M}}
	First note that the matrix $M$ is a triangular matrix with diagonal equal to the identity matrix and thus invertible.
	Hence, 
	\begin{align*}
		\sigma_1(M) &= \min_{x\in \RR^{k+1}\backslash\{0\}} \frac{\|Mx\|_2}{\|x\|_2}\\ &=  %
		\min_{x\in \RR^{k+1}\backslash\{0\}} \frac{\|x\|_2}{\|M^{-1}x\|_2}    = \left(\max_{x\in \RR^{k+1}\backslash\{0\}} \frac{\|M^{-1}x\|_2}{\|x\|_2}\right)^{-1}  =  \frac{1}{\sigma_k(M^{-1})}.
	\end{align*}
	Moreover, since $M = Id_{k+1} + N$ for some nilpotent matrix $N$ with $N^{k+1} = 0$ it follows,
	\begin{align*}
		M^{-1} = (I+M)^{-1} = Id_{k+1} + \sum_{j= 1}^{k}(-1)^{j} N^j.
	\end{align*}
	Consequently, by \citet[Corollary 8.6.2]{golub2013matrix}, we infer 
	\begin{align*}
		\sigma_{k}(M^{-1})  &\leq 1+ |\sigma_{k}(M^{-1}) - \sigma_k(I_{k+1})|  \\
		&\leq 1+ \left\|\sum_{j= 1}^{k}(-1)^{j} N^j\right\|_2 \leq \sum_{j = 0}^{k} \|N\|_2^j = \frac{\|N\|_2^{k+1}-1}{\|N\|_2-1}, 
	\end{align*}
	which can be analytically extended to $\|N\|_2 = 1$ with the value $k+1$. Overall, this implies  $\sigma_1(M)  \geq \frac{\|N\|_2-1}{\|N\|_2^{k+1}-1}$ or equivalently $\sigma_1(M)^{-1} \leq \frac{\|N\|_2^{k+1}-1}{\|N\|_2-1}$. To upper bound $\|N\|_2^2$, we note, 
	\begin{align*}
		\|N\|_2^2\leq \|N\|_F^2 = \sum_{l,h= 0}^{k} N_{l,h}^2 &= \sum_{h = 0}^{k-1} \sum_{l = h+1}^{k} \binom{l}{h}^2m_{l-h}(K)^2\\
		&= \max_{\substack{h = 0, ..., k-1\\ l = h+1, \dots, k}}\left|\binom{l}{h} m_{l-h}(K)\right|\times \sum_{h = 0}^{k-1} \sum_{l = h+1}^{k} \binom{l}{h}|m_{l-h}(K)|\\
		&= \max_{r = 1, \dots, k} \left|\binom{k}{r}m_{r}(K)\right|\times \sum_{h = 0}^{k-1} \sum_{l = h+1}^{k} \binom{l}{h}|m_{l-h}(K)|
	\end{align*}
where we take $r = l-h$ and note that $\binom{l}{h}=\binom{r+h}{h} = \binom{r+h}{r}$ which is maximized when $h$ is chosen as large as possible, i.e., when $r+h = l = k$. 
We further observe that 
\begin{align*}
	\sum_{h = 0}^{k-1} \sum_{l = h+1}^{k} \binom{l}{h}|m_{l-h}(K)|&= 
	\sum_{h = 0}^{k-1} \sum_{l = h+1}^{k} \binom{l}{l-h}|m_{l-h}(K)| \\
	&=\sum_{l = 1}^{k} \sum_{h=0}^{l-1}\binom{l}{l-h}|m_{l-h}(K)| \\
	&= \sum_{l = 1}^{k} \sum_{r = 1}^{l}\binom{l}{r}|m_{r}(K)|\\
	&= \sum_{r = 1}^{k} |m_{r}(K)| \sum_{l = r}^{k} \binom{l}{r}\\
	&= \sum_{r = 1}^{k} |m_{r}(K)|\binom{k+1}{r+1}\\
	&\leq \left(1+\int |x^k| K(x)\dif x\right) \sum_{r = 1}^{k}\binom{k+1}{r+1}\\
	&\leq \left(1+\int |x^k| K(x)\dif x\right)2^{k+1}
\end{align*}
where the first inequality follows from Jensen's inequality, 
\begin{align*}
	m_r(K)&\leq \left(\int |x^r| K(x)\dif x \right) \\*&\leq \left(\int |x^p| K(x)\dif x\right)^{r/k} \\* &\leq \left(1+\int |x^k| K(x)\dif x\right)^{r/k}\leq \left(1+\int |x^k| K(x)\dif x\right).
\end{align*}
Further, it follows that 
\begin{align*}
	\max_{r = 1, \dots, k} \left|\binom{k}{r}m_{r}(K)\right|\leq 2^k \left(1+ \int |x|^kK(x)\dif x\right).
\end{align*} 
Combining these inequalities we thus obtain for $\|N\|_2$ the upper bound 
\begin{align*}
	\|N\|_2 \leq \left(1+ \int |x|^kK(x)\dif x\right)2^{(k+1)},
\end{align*}
which overall asserts 
\begin{align*}
	\sigma_1(M)^{-1} \leq \left(1+ \int |x|^kK(x)\dif x\right)^{k+1}2^{(k+1)^2},
\end{align*}
and proves the claim. \qed

\subsection{Proof of Proposition~\ref{prop:Lp_to_lp}}
\label{app:pf_prop_Lp_to_lp}
 
First observe for $\mu = \sum_{l = 1}^{k} w_l \delta_{\theta_i}\in \calP_k(\domain)$, $a\in \calA$ and $x\in \RR^d$ that 
\begin{align*}
	&\quad K\ast \mu\left(a+\frac{s}{2k}\Sigma x\right) \\
	&= \frac{1}{(2\pi)^{d/2} \det(\Sigma)^{1/2}} \sum_{l = 1}^{k} w_l \exp\left( - \frac{1}{2}\left(a+\frac{s}{2k}\Sigma x-\theta_l\right)^\top \Sigma^{-1}\left(a+\frac{s}{2k}\Sigma x- \theta_l\right)\right).
\end{align*}
It therefore follows that 
\begin{align*}
	&K\ast \mu\left(a+\frac{s}{2k}\Sigma x\right)  
	(2\pi)^{d/2} \det(\Sigma)^{1/2}\exp\left(\frac{1}{2}\left(a+\frac{s}{2k}\Sigma x\right)^\top \Sigma^{-1} \left(a+\frac{s}{2k}\Sigma x\right)\right)\\
	=\;&  \sum_{l = 1}^{k} w_l\exp\left( \theta_l^\top \Sigma^{-1} \left(a+\frac{s}{2k}\Sigma x\right)  -\frac{1}{2} \theta_l^\top \Sigma^{-1}\theta_l\right)\\
	=\;&  \sum_{l = 1}^{k} w_l\exp\left( \theta_l^\top \Sigma^{-1} a  -\frac{1}{2} \theta_l^\top \Sigma^{-1}\theta_l\right)\prod_{j = 1}^{d}\exp\left(\frac{s}{2k}\theta_{l,j}\right)^{x_j},
\end{align*}
which equals the moment of real order $x$ for the $k$-atomic (non-normalized) measure 
\begin{align*}
	\tilde \mu \coloneqq \sum_{l = 1}^{k} w_l\exp\left( \theta_l^\top \Sigma^{-1} a  -\frac{1}{2} \theta_l^\top \Sigma^{-1}\theta_l\right) \delta_{\exp(\frac{s}{2k} \theta_l)}.
\end{align*}
Arguing analogously, we infer for $\nu = \sum_{l = 1}^{k} v_l \delta_{\eta_l}$ that 
\begin{align*}
	&K\ast \nu\left(a+\frac{s}{2k}\Sigma x\right) 
	(2\pi)^{d/2} \det(\Sigma)^{1/2}\exp\left(\frac{1}{2} \left(a+\frac{s}{2k}\Sigma x\right)^\top \Sigma^{-1} \left(a+\frac{s}{2k}\Sigma x\right)\right)
\end{align*}
equals the moment of real multivariate order $x$ for 
\begin{align*}
	\tilde \nu \coloneqq \sum_{l = 1}^{k} v_l\exp\left( \eta_l^\top \Sigma^{-1} a  -\frac{1}{2} \eta_l^\top \Sigma^{-1}\eta_l\right) \delta_{\exp(\frac{s}{2k} \eta_l)}.
\end{align*}
Based on our bound on real-order moment differences in terms of low order moment differences (Lemma \ref{lem:realMoment_lowIntegerMoments}$(ii)$ below) there exists a positive constant $C = C(\Theta, d,k,s)>0$ such that for all $x\in \RR^d$, 
\begin{align*}
	\left|m_x(\tilde  \mu) - m_x(\tilde \nu) \right|\leq C^{1+\|x\|_1} \max_{\beta\in \{0, \dots, 2k-1\}^d} \left|m_\beta(\tilde \mu) - m_\beta(\tilde \nu)\right|,
\end{align*}
which implies by our considerations above for all $x\in\RR^d$ that 
\begin{align*}
	&\quad \left|K\ast (\mu - \nu)\left(a+\frac{s}{2k}\Sigma x\right) \right| \exp\left(\frac{1}{2} \left(a+\frac{s}{2k}\Sigma^{-1} x\right)^\top \Sigma^{-1} \left(a+\frac{s}{2k}\Sigma x\right)\right) \\
	&\leq C^{\|x\|_1} \max_{g\in \calG_a} \left|K\ast (\mu - \nu)\left(g\right) \right|  \exp\left(\frac{1}{2} g^\top \Sigma^{-1} g\right).
\end{align*}
Performing a change of variables $y \coloneqq a+\frac{s}{2k}\Sigma x$ gives 
\begin{align*}
	\left|K\ast (\mu - \nu)\left(y\right) \right| \leq C^{\|\frac{2k}{s}\Sigma(y-a)\|_1}\exp\left(-\frac{1}{2}y^\top\Sigma^{-1}y\right)\max_{g\in \calG_a} \left|K\ast (\mu - \nu)\left(g\right) \right| \exp\left(\frac{1}{2} g^\top \Sigma^{-1} g\right). 
\end{align*}
Note that $y\mapsto C^{\frac{2k}{s}\|\Sigma^{-1} y\|_1}\exp\left(-\frac{1}{2}y^\top\Sigma^{-1}y\right)$ is $L^p$-integrable for every $p\in [1,\infty)$ and bounded. Taking on both sides of above display the integral in $y$ of the $p$-th power  for $p \in[1,\infty)$ or the maximum over $y\in\RR^d$ (for $p= \infty$), we conclude since $\calA$ is bounded the existence of a positive constant $C=C(\domain, \Sigma, \calA, d, k, p,s)>0$ such that for all $a \in \calA$ and $\mu, \nu \in \calP_k(\Theta)$, 
\begin{align*}
	\|K\ast (\mu - \nu)\|_{L^p(\RR^d)} \leq C \|K\ast (\mu - \nu)\|_{\ell^\infty(\calG_a)}.
\end{align*}
The assertion now follows from the equivalence of $\|\cdot\|_{\ell^\infty(\calG_a)}$ and $\|\cdot\|_{\ell^q(\calG_a)}$, where the suppressed constant  only depends on $q$ and the cardinality of $\calG_a$, which equals $(2k)^d$. 
 \qed 

\begin{lemma}\label{lem:realMoment_lowIntegerMoments}\label{lem:ContinuousMomentBound}
	Let $\Theta\subseteq \RR^d$ be compact and let $k\in \NN$. Then, there is a positive constant $C = C(\Theta, d,k)>0$ which fulfills the following. 
	\begin{enumerate}
		\item For all $\alpha \in \NN_0^d$ and all (at most) $k$-atomic probability measures $\mu, \nu\in \calP_k(\mu)$ it holds 
		\begin{align*}
			\left|m_\alpha(\mu) - m_\alpha(\nu) \right|\leq C^{1+\|\alpha\|_1} M_{2k-1}(\mu, \nu). 
		\end{align*}
		\item If $\Theta\subseteq (0,\infty)^d$, then for all $\alpha \in \RR^d$ and all (at most)  $k$-atomic measures $\mu, \nu \in \calM_k(\Theta)$,
		\begin{align*}
			\left|m_\alpha(\mu) - m_\alpha(\nu) \right|\leq C^{1+\|\alpha\|_1} M_{2k-1, \infty}(\mu, \nu),
		\end{align*}
		where we define $M_{2k-1, \infty}(\mu, \nu) \coloneqq \max_{\beta \in \{0, \dots, 2k-1\}^d} \left|m_\beta(\mu) - m_\beta(\nu)\right|$. In particular, note by $\Theta\subseteq (0,\infty)^d$ that the moments on the left-hand side are well-defined. 
	\end{enumerate}

\end{lemma}

\subsubsection{Proof of Lemma \ref{lem:ContinuousMomentBound}}
	
	In our proof for Assertion $(i)$ (resp.\ $(ii)$) under $d =1$ we first interpolate the integer (resp.\ real) order polynomial 
	\begin{align}\label{eq:polynomialForInterpolation}
		f_\alpha \colon \Theta \to \RR, \quad x\mapsto x^{\alpha}
	\end{align}
	for $\alpha \in \NN_0$ (resp.\ $\alpha \in \RR$) in terms of polynomials of order $\beta\in \{1, \dots, 2k-1\}$. 
	The multivariate setting $d\geq 2$ for Assertion $(i)$ then follows by some slicing argument which reduces the analysis to the univariate setting. This reduction does not directly apply for real order moments in Assertion $(ii)$, and we instead develop a multivariate interpolation argument which leads to $M_{2k-1, \infty}(\mu, \nu)$ in the upper bound instead of $M_{2k-1}(\mu, \nu)$. 

	{\bf{Step 1. Preliminaries on interpolation.}}
	The interpolation argument for univariate setting is inspired by the proof of \citet[Lemma~10]{wu2020}, see also their Section 5.1 for a short review of polynomial interpolation methods. To set some notation, given a function $f$ on $\Theta$ and a finite collection of distinct interpolation points $x_1, \dots, x_l\in \Theta$ the interpolation polynomial $P_f$, i.e., the unique polynomial of degree $l-1$ which coincides on $\{x_1, \dots, x_l\}$ with $f$, can be represented as 
	\begin{align}\label{eq:UnivariateInterpolationPolynomial}
		P_f(x) &\coloneqq P(x \mid f, x_1, \dots, x_{l}) \coloneqq \sum_{i = 1}^{l} f[x_1, \dots, x_i]g_{i-1}(x),
	\end{align}where the functions $g_i$ and the divided differences $f[x_1, \dots, x_i]$ of $f$ are defined as \begin{align*}
		g_0(x)& \coloneqq g_0(x| x_1, \dots, x_{l})\coloneqq 1\quad\quad g_r(x) \coloneqq g_r(x| x_1, \dots, x_{l}) \coloneqq\prod_{j =1}^{r}(x-x_j),\\
		f[x_i]&\coloneqq f(x_i), \quad \quad  f[x_i, \dots, x_{i+r}]= \frac{f[x_{i+1}, \dots, x_{i+r}] - f[x_{i}, \dots, x_{i+r-1}]}{x_{i+r} - x_i}.
	\end{align*}
	In particular, under smoothness of $f$ on $\RR$ for Assertion $(i)$ (resp.\ on $(0,\infty)$ for Assertion $(ii)$) 
	 there exist by the intermediate theorem some $\xi_i \in [x_1, x_i]$ such that 
	\begin{align}\label{eq:UnivariateInterpolationPolynomial_representation}
		P_f(x) = \sum_{i = 1}^{l} \frac{f^{(i-1)}(\xi_i)}{(i-1)!}g_{i-1}(x).
	\end{align}
	With that notation at our disposal, we proceed with the proofs of Assertions $(i)$ and $(ii)$. 

	{\bf{Step 2. Integer and real moment difference bound for $d = 1$.}} The proof of this assertion closely follows along the lines of \citet[Lemma 10]{wu2020} with the difference that in Assertion $(i)$ the set  $\Theta$ is not restricted to be contained in $[-1,1]$ and in Assertion $(ii)$ we permit real order polynomials. 

	Let $\alpha \in \NN_0$ for Assertion $(i)$ and $\alpha \in \RR$ for Assertion $(ii)$. Then, $f_\alpha$ from \eqref{eq:polynomialForInterpolation} is well-defined and smooth on an open set containing $\Theta$. Further, denote by $x_1, \dots, x_l$ the distinct support points of $\supp(\mu)\cup\supp(\nu)$, hence $l\in \{1, \dots, 2k\}$ and consider the corresponding interpolation polynomial $P_{f_\alpha}$ from \eqref{eq:UnivariateInterpolationPolynomial}. In particular, by \citet[Lemma~26]{wu2020} we obtain since every $x_j$ fulfills $|x_j|\leq \|\Theta\| \coloneqq \max_{\theta\in \Theta}|\theta|$ that $g_{i-1}(x\,|\,x_1, \dots, x_l) = \sum_{j = 0}^{i-1}a_j x^j$ with $|a_j| \leq \binom{i-1}{j}\|\Theta\|^{i-1-j}$ for all $j\in \{1, \dots, i-1\}$. Moreover, for $\xi_i\in \mathrm{convHull}(\Theta)$ it holds for Assertion $(i)$ that $|f_\alpha^{(i-1)}(\xi_i)| \leq (1+\|\Theta\|)^{|\alpha| + i-1} \prod_{j = 1}^{i-1}|\alpha - j+1|$ and for Assertion $(ii)$ that $|f_\alpha^{(i-1)}(\xi_i)| \leq (1+\max_{\theta\in \Theta}\theta + 1/\min_{\theta\in \Theta} \theta)^{|\alpha| + i-1} \prod_{j = 1}^{i-1}|\alpha - j+1|$. We thus obtain that 
	\begin{align*}
		\left|m_\alpha(\mu) - m_\alpha(\nu) \right|&= \left|\int P_{f_\alpha}(x)\dif (\mu-\nu)(x)\right|\\
		&\leq \sum_{i =1}^{l}\frac{f_\alpha^{(i-1)}(\xi_i)}{(i-1)!} \sum_{j =0}^{i-1} |a_j||m_{j}(\mu) - m_{j}(\nu)|\\
		&\leq \tilde C^{1+|\alpha|} (|\alpha|+l)^{l} M_{l}(\mu, \nu)\leq C^{1+|\alpha|} M_{2k-1}(\mu, \nu),
	\end{align*}
	where the constants $C, \tilde C$ in the last line are strictly positive and depend on $\Theta$ and $k$ but not on $\alpha$, $\mu$ and $\nu$. 

	{\bf{Step 3. Integer moment difference bound for $d\geq 2$. }} Let $\alpha \in \NN_0^d$, then it follows by Lemma \ref{lem:moment_sliced_moment} combined with step 2.\ and Lemma \ref{lem:sliced_multivariate_moment_bound} that 
	\begin{align*}
		\left|m_\alpha(\mu) - m_\alpha(\nu) \right| &\leq k^{(\|\alpha\|_1-1)/2} \sup_{\eta\in \SS^d} |m_{\|\alpha\|_1}(\mu^\eta) - m_{\|\alpha\|_1}(\nu^\eta)|\\
		&\leq k^{(\|\alpha\|_1-1)/2} C^{1 + \|\alpha\|_1}\sup_{\eta\in \SS^d} M_{2k-1}(\mu^\eta, \nu^\eta)\\
		&\leq k^{(\|\alpha\|_1-1)/2} C^{1 + \|\alpha\|_1} d^{(2k-1)/2} M_{2k-1}(\mu, \nu),
	\end{align*}
	where the positive constant $C>0$ in the second line is determined by the constant from step 2.\ for $\bigcup_{\eta\in \SS^d} \langle \Theta, \eta \rangle$ and $k$. This confirms the Assertion $(i)$ for $d\geq 2$. 

	{\bf{Step 4. Real order moment difference bound for $d\geq 2$. }}
For the proof of the multivariate statement for Assertion $(ii)$ assume without loss of generality that $\Theta$ is a rectangular set of the form $\Theta = \prod_{i = 1}^{d} [\theta_i^-, \theta_i^+]$. Similar to the technique in step 2.\ we interpolate the multivariate real-order polynomial $f_\alpha\colon (0, \infty)^d \to \RR, x\mapsto \prod_{j = 1}^{d} x_j^{\alpha_j}$ on the set $\supp(\mu)\cup\supp(\nu)$. We pursue this approach because the two crucial Lemma \ref{lem:sliced_multivariate_moment_bound} and \ref{lem:moment_sliced_moment} which enabled us in step 3.\ to reduce the analysis to the univariate setting do not appear to be easily generalizable to real order moments. For the interpolation, denote by $\tilde x_{1, j} < \dots < \tilde x_{l_j,j}$ for $j\in \{1, \dots, d\}$ and  $l_j\in \{1, \dots, 2k\}$ the ordered support points of the push-forward of $\mu+\nu$ under the projection onto the $j$-th component. Further, for each $j\in \{1, \dots, d\}$ define the univariate polynomial $P_{f_{\alpha_j}}(t)\coloneqq P(t \,|\, f_{\alpha_j}, \tilde x_{1,j}, \dots, \tilde x_{l_j,j})$ of degree $l_j-1$ from \eqref{eq:UnivariateInterpolationPolynomial} and note via \eqref{eq:UnivariateInterpolationPolynomial_representation} for suitable $\tilde \xi_{i,j} \in [\tilde x_{1,j}, \tilde x_{i,j}]$ that
\begin{align*}
	P_{f_\alpha}(x)\coloneqq \prod_{j = 1}^{d} P_{f_{\alpha_j}}(x_j) = \prod_{j = 1}^{d} \sum_{i_j = 1}^{l_j} \frac{f_{\alpha_j}^{(i_j-1)}(\tilde \xi_{i,j})}{(i_j-1)!}g_{i_j -1}(x_j\,|\, \tilde x_{1,j}, \dots, \tilde x_{l_j,j})
\end{align*}
interpolates $f_\alpha$ on $\supp(\mu)\cup\supp(\nu)$. Arguing as in step 2.\ we obtain by \citet[Lemma 26]{wu2020} since $0<\tilde x_{r_j,j}\leq \theta_j^+$ that $g_{i_j -1}(x_j\,|\, \tilde x_{1,j}, \dots, \tilde x_{l_j,j}) = \sum_{h_j = 0}^{i_j-1}a_{h_j,j} x_j^{h_j}$ with $|a_{h_j,j}| \leq \binom{i-1}{j}(\theta^+_j + 1/\theta^-_j)^{i_j-1-h_j}$ and $|f_{\alpha_j}^{(i_j-1)}(\tilde \xi_{i,j})|\leq (1+ \theta_j^+ +1/\theta_j^-)^{|\alpha_j| + i_j-1}\prod_{h_j=1}^{i_j-1}|\alpha_j - h_j+1|$ since $\tilde \xi_{i,j}\in [\theta_j^-,\theta_j^+]$. Combining these bounds we arrive at
\begin{align*}
	|m_\alpha(\mu) - m_\alpha(\nu)|&= \left|\int P_{f_\alpha}(x)\dif (\mu-\nu)(x)\right|\\
	&\leq\sum_{i \in \bigtimes_{j = 1}^{d} \{0, \dots, l_j\}} \sum_{h \in \bigtimes_{j = 1}^{d} \{0, \dots, i_j\}} \left|m_h(\mu) - m_h(\nu)\right| \prod_{j = 1}^{d}\left| \frac{f_{\alpha_j}^{(i_j-1)}(\tilde \xi_{i,j})}{(i_j-1)!} a_{h_j,j}\right|\\
	&\leq \tilde C^{1+\|\alpha\|_1}\sum_{i \in \bigtimes_{j = 1}^{d} \{0, \dots, l_j\}} \prod_{j = 1}^{d}(|\alpha_j|+i_j)^{i_j} \max_{h\in \bigtimes_{j = 1}^{d} \{0, \dots, l_j\}}\left|m_h(\mu) - m_h(\nu)\right| \\
	&\leq C^{1+\|\alpha\|_1} M_{2k,\infty}(\mu, \nu),
\end{align*}
where the constants $\tilde C, C$ in the last two lines are strictly positive and depend on $\Theta$, $d$, $k$ but not on $\alpha$, $\mu$ and $\nu$. This completes the proof. \qedhere

\subsection{Proof of Proposition~\ref{cor:Lp_Omega_Rd}}
\label{app:pf_cor_Lp_Omega_Rd}
Let $N \coloneqq 2k$ and consider for anchor point $a\in \RR^d$ and scaling $s>0$ the grid
$\calG_{a} = \calG_{a}(\Sigma, d,k,s)$ from \eqref{eq:grid}. 
Since $\Omega$ has non-empty interior, there exists some $s>0$ and an offset $a'\in \RR^d$ such that for each $a\in \calA\coloneqq \{a' + \Sigma \xi \,\colon \xi\in [0,s/N)^d\}$ the grid $\calG_a$ is contained in $\Omega$. Note that this is equivalent to $\tilde \Omega \coloneqq\{ a' + \Sigma x \,\colon x\in [0,s)^d\}\subseteq \Omega$. 
Further, by Proposition~\ref{prop:Lp_to_lp}, there exists a positive constant $C > 0$ such that for all $\mu, \nu \in \calP_k(\Theta)$ and $a\in \calA$ it holds 
\begin{align*}
	\|K\star(\mu-\nu)\|_{L^p(\bbR^d)}   \leq C \|K\star(\mu-\nu)\|_{\ell^q(\calG_a)},
\end{align*}
If $q = \infty$, we now immediately obtain 
\begin{align*}
	\|K\star(\mu-\nu)\|_{L^p(\bbR^d)}  \leq C \max_{a\in \calA}  \|K\star(\mu-\nu)\|_{\ell^\infty(\calG_a)} \leq C \|K\star(\mu-\nu)\|_{L^\infty(\Omega)},
\end{align*}
where we used monotonicity of the $L^\infty$-norm under domain enlargements and continuity of the kernel. 
Otherwise, if $q<\infty$, we integrate the $q$-th power of both sides in the penultimate display over $a\in \calA$ which yields 
\begin{align*}
	\|K\star(\mu-\nu)\|_{L^p(\bbR^d)}^q&\leq \frac{C}{\mathcal{L}(\calA)}  \int_{\calA} \|K\star(\mu-\nu)\|_{\ell^q(\calG_a)}^qda\\
	&= \frac{C}{\mathcal{L}(\calA)}  \sum_{i_1,\dots,i_d=1}^N \int_{\calA + \frac{s}{N}\Sigma (i_1, \dots, i_d)^\top} |K \ast (\mu - \nu)(u)|^q \dif u\\
	&= \frac{C}{\mathcal{L}(\calA)} \int_{\tilde \Omega} |K \ast (\mu - \nu)(u)|^q \dif u\\
	&=  \frac{C}{\mathcal{L}(\calA)} \|K \ast (\mu - \nu)\|_{L^q(\tilde \Omega)}^q \\
	&\leq \frac{C}{\mathcal{L}(\calA)} \|K \ast (\mu - \nu)\|_{L^q(\Omega)}^q,
\end{align*}
where we used in the final inequality that the $L^p$-norm increases under enlargements of the integration domain. The assertion for $q<\infty$ now follows by noticing that the constant on the right-hand side is minimal when $\calL(\calA)$ is maximized, and that the maximal value depends on $\Omega, \Sigma, d, k$. \qed

\subsection{Proof of Theorem~\ref{thm:equivalences}}
\label{app:pf_thm_equivalences} 
To emphasize the dependency on $\Sigma$ in the kernel we write $K_\Sigma$. The relations
$$ M_{2k-1}(\mu,\nu)\lesssim_{\Theta, \Sigma, d,k}
\|K_\Sigma\star(\mu-\nu)\|_{L^2(\bbR^d)} \asymp_{\Theta, \Sigma, d,k,p} \|K_\Sigma\star(\mu-\nu)\|_{L^p(\bbR^d)},\quad \mu,\nu\in\calP_k(\Omega)$$
are a direct consequence of Theorem \ref{thm:momentL2-bounds_all}$(i)$ and Proposition~\ref{cor:Lp_Omega_Rd}. 
Furthermore, since the Gaussian mixture densities (with positive definite
covariance matrix $\Sigma$) 
are bounded
above over $\bbR^d$ (by a constant depending on $\Sigma$), 
it is straightforward observation that 
$$\|K_\Sigma\star(\mu-\nu)\|_{L^2(\bbR^d)} \lesssim_{\Sigma} H(K_\Sigma\star\mu,K\star\nu).$$
The assertion follows once we show that the Hellinger distance is dominated by $M_{2k-1}(\mu, \nu)$, for which we will rely on \citet[Theorem 4.2]{doss2023} which however was only formulated for $\Sigma = \textup{Id}$. To extend their assertion, first define for a $k$-atomic measure $\xi\in \calU_k(\RR^d)$ and $\ell\in \{1, \dots, 2k-1\}$ the order-$\ell$ moment tensor $\bM_\ell(\xi)$ according to \eqref{eq:MomentTensor1} from within the proof of Lemma \ref{lem:moment_sliced_moment} in Appendix~\ref{app:momentDifferenceInequalities}. In that proof we also introduce the Frobenius norm and the operator norm for these moment operators by $\|\cdot\|_F$ and $\|\cdot\|_{op}$, respectively. 
With these tools, the following inequalities hold, which we explain below, 
\begin{align*}
	H(K_\Sigma\star\mu,K_\Sigma\star\nu) & =
	H(K_\textup{Id}\star(\Sigma_{\#}^{-1/2}\mu),K_\textup{Id}\star(\Sigma_{\#}^{-1/2}\nu))
	\\
	&\lesssim_{\Sigma, \Theta,d, k}\max_{l\in 1, \dots, 2k-1}\|\bM_\ell(\Sigma_{\#}^{-1/2}\mu) - \bM_\ell(\Sigma_{\#}^{-1/2}\nu)\|_F\\
	&\lesssim_{k}\max_{l\in 1, \dots, 2k-1}\|\bM_\ell(\Sigma_{\#}^{-1/2}\mu) - \bM_\ell(\Sigma_{\#}^{-1/2}\nu)\|_{op}\\
	&= \max_{l\in 1, \dots, 2k-1} \sup_{\eta\in \mathbb{S}^{d-1}}\left|m_\ell((\Sigma_{\#}^{-1/2}\mu)^\eta) - m_\ell((\Sigma_{\#}^{-1/2}\nu)^\eta)\right|\\
	&\lesssim_{\Sigma, d,k} \max_{l\in 1, \dots, 2k-1} \sup_{\eta\in \mathbb{S}^{d-1}}\left|m_\ell(\mu^\eta) - m_\ell(\nu^\eta)\right|\\
	&\lesssim_{d,k} M_{2k-1}(\mu, \nu).
\end{align*}
In the above display, the first equality follows from an integral transformation $y = \Sigma^{-1/2}x$,  the second line is due to \citet[Theorem~4.2]{doss2023}, the third and fourth line is a consequence of \eqref{eq:MomentTensor2} and \eqref{eq:MomentTensor3}, respectively, the fifth line follows using the spectral representation of $\Sigma$ with steps 1.2 and 1.3 of the proof of \Cref{thm:momentL2-bounds_all}, 
and the final inequality follows from Lemma \ref{lem:moment_sliced_moment}. \qed

\subsection{Proof of Lemma~\ref{lem:chaining_ours}}
\label{app:pf_lem_chaining_ours}

We will prove the following more general statement. 
\begin{lemma}
\label{lem:chaining}
Let $\Omega\subseteq \RR^d$ be a domain such that Assumption \ref{ass:bins} is met for a partition $\{B_i\}_{i = 1,\dots, m}$. 
Let $\calF$ be a class of Borel-measurable functions of the form
$f:\Omega\to\bbR$ which admits a dense subset with respect to the uniform norm on $\RR^d$ and assume that there exist a constant
 $b \geq 1 $ such that $b^{-1} \leq f \leq b$ over $\Omega$ for each $f \in \calF$.
Consider independent random variables, 
$$Y_i \sim \mathrm{Poi}(tf_0(B_i)),\quad i=1,\dots,m,$$
for some $f_0 \in \calF$ and assume there exists a Borel-measurable maximum likelihood estimator of $f_0$, 
$$\hat f_{t,m} \in \argmax_{f \in \calF} \sum_{i=1}^m \Big( Y_i \log(tf(B_i)) - tf(B_i)\Big).$$
Define the function class
$$\calF_m(\gamma;f_0) = \left\{ \Pi_m f: f\in\calF, \|f-f_0\|_{L^2(\Omega)}\leq \gamma\right\},
\quad \text{for any } \gamma > 0,$$
and, given constants $C_0,c_0 > 0$ depending only on $b$, assume
there exists $\gamma_{t,m} > 0$ such that for all $\gamma \geq \gamma_{t,m}$,  
the following Le Cam-type equation is satisfied:
\begin{align}
\label{eq:le_cam_bracketing_condition} 
\sqrt t \gamma^2 \geq 
C_0 \sup_{f_0\in\calF}\calJ_{[]}(c_0\gamma, \calF_m(\gamma;f_0),L^2(\Omega)). 
\end{align}
Then, there exists a constant $C > 0$ depending only on $\Omega,b$ such that
$$\bbE \big\| \Pi_m (\hat f_{t,m}-f_0)\big\|_{L^2(\Omega)}^2 \leq C \left(\frac 1 {\sqrt t} + \gamma_{t,m}\right)^2.$$
\end{lemma}
Before proving Lemma~\ref{lem:chaining}, let us explain why it implies 
Lemma~\ref{lem:chaining_ours}. As our respective function class we consider the collection of uniform Gaussian mixtures $\calF = \{K\ast \mu \colon \mu \in \calU_k(\Theta)\}$. Each element in that class is lower and upper bounded on $\Omega$ by positive numbers due to compactness of $\Theta$ and $\Omega$, and using \Cref{thm:equivalences} it also evident that the class $\calF$ is compact in uniform norm due to compactness of $\calU_k(\Theta)$ with respect to the moment difference metric $M_k$. Compactness of $\calF$ also ensures the existence of a maximum likelihood estimator for every realization of $(Y_1, \dots, Y_m)$, and according to \citet[Proposition 7.33]{bertsekas1996stochastic} it can even be chosen to be Borel measurable in terms of $(Y_1, \dots, Y_m)$. Applying Lemma~\ref{lem:chaining}
to $\calF$, we obtain 
$$\bbE \big\| \Pi_m [K\star(\hat\mu_{t,m}-\mu)]\big\|_{L^2(\Omega)}^2 \lesssim \left(\frac 1 {\sqrt t} + \gamma_{t,m}\right)^2,$$
where we used the fact that the densities $K\star\mu$ are bounded from below 
by a positive constant over $\Omega$. 
Now, since the elements of $\calF$ are Lipschitz (with a uniform Lipschitz modulus), 
we may apply Lemma~\ref{lem:Pi_error} to deduce
$$\bbE \big\| K\star(\hat\mu_{t,m}-\mu)\big\|_{L^2(\Omega)}^2 \lesssim
 \left(\frac 1 {\sqrt t} + \gamma_{t,m} + m^{-1/d}\right)^2,$$
 where the implicit constant depends only on $\Omega,\Sigma$.
 
 It thus remains to prove Lemma~\ref{lem:chaining}. 
Throughout the proof of this Lemma, 
 $C_0,C_1,\dots > 0$ denote generic constants which depend only on $\Omega,b$.
Likewise, the symbols $\asymp$ and $\lesssim$ hide constants depending only on $\Omega,b$.

Since all elements of $\calF$ take values in the set $[b^{-1},b]$ over
  $\Omega$, it must also follow that all elements
  of the class 
  $$\calF_m := \{\Pi_m f: f\in \calF\},~~~m=1,2,\dots$$
  take values in the set $[b^{-1},b]$. 
In particular, it follows that the
real number $Z =  \int_\Theta \Pi_m f_0(x)dx$
satisfies   $Z\asymp 1$.
Now,  define  
$$T \sim \mathrm{Poi}\left( t   Z\right),\quad
\text{and,}\quad 
\left.Y_i\sim \Pi_m f_0 /Z,\quad i=1,2,\dots\right.$$
It is easy to see that these random variables may be coupled with $X_1,\dots,X_m$
in such a way that the following equality holds almost surely,
$$X_j = \sum_{i=1}^T \mathds{1}(Y_i\in B_j),\quad j=1,\dots,m.$$ 
Turning now to the proof, we know  that
the log-likelihood function value at $\hat f_{t,m}$
is greater than its value at $f_0$, so that
$$\sum_{i=1}^m \left\{  X_i  \log \frac{  f_0 (B_i)}{\hat f_{t,m}(B_i)} - t\big(f_0 (B_i) -  \hat f_{t,m}(B_i)\big)\right\} 
\leq 0,$$
or equivalently, 
\begin{align}
\label{eq:basic_ineq_reduc}
\sum_{i=1}^m \left( f_0(B_i) \log \frac{ f_0 (B_i)}{\hat  f_t(B_i)} -  f_0 (B_i) + \hat f_{t,m}(B_i)\right)\leq 
\sum_{i=1}^m \left(f_0 (B_i) - \frac{X_i}{t}\right)  \log \frac{ f_0 (B_i)}{ \hat f_{t,m}(B_i)},
\end{align}
The left-hand side of the above display is precisely the Kullback-Leibler divergence between
the laws of the inhomogeneous Poisson point processes
with intensity functions $\Pi_m f_0 |_\Omega$ and $\Pi_m \hat f_{t,m}|_\Omega$~\citep{anestis2006,hohage2016inverse},
and is bounded from below by a multiple of their squared Hellinger distance. 
To elaborate, let 
 $H_\Omega^2(f,g) =
  \int_\Omega \big( f^{1/2} -g^{1/2}\big)^2$ be
   a restriction of the Hellinger distance  between any two densities $f,g$ on $\bbR^d$
   to the domain~$\Omega$. One has the inequality~\citep{borwein1991}
$$(u-v)^2 \leq 2\left( u +  v\right) \left(u \log \frac u v - u + v\right),\quad \text{for all } u\geq 0,v > 0,$$
and hence, 
$$\left(\frac{u-v}{\sqrt u + \sqrt v}\right)^2 \leq 2  \left(u \log \frac u v - u + v\right),\quad \text{for all } u\geq 0,v > 0,$$
which, together with equation~\eqref{eq:basic_ineq_reduc}, implies that 
\begin{equation}
\label{eq:hellinger_ub_mle_step}
\begin{aligned}
H_\Omega^2(\Pi_m \hat f_{t,m},\Pi_m f_0) 
 &\leq \sum_{i=1}^m \left(f_0 (B_i) - \frac{X_i}{t}\right)  \log \frac{ f_0 (B_i)}{ \hat f_{t,m}(B_i)} \\
 &= \int_\Omega \log \frac{\Pi_m \hat f_{t,m}}{\Pi_m f_0 } d\left(\frac 1 t \sum_{i=1}^T \delta_{Y_i} - \Pi_m f_0 \right).
 \end{aligned}
\end{equation}
Now,
  since all elements of $\calF$ take values in the set $[b^{-1},b]$ over
  $\Omega$, it must also follow that all elements
  of the class 
  $$\calF_m = \{\Pi_m f: f\in \calF\}$$
  take values in the set $[b^{-1},b]$.
  In particular, these functions are bounded from below by a positive constant, thus
   one readily has the equivalence 
  $H_\Omega  \asymp \|\cdot\|_{L^2(\Omega)}$ over $\calF_m$.
Now, let $P$ be the probability distribution with density $\Pi_m f_0 / Z$,
and let $P_\tau = (1/\tau)\sum_{i=1}^\tau \delta_{Y_i}$ 
be the empirical measure
based on the first $\tau\in\bbN$ observations $Y_i$.
From equation~\eqref{eq:hellinger_ub_mle_step}, we obtain
\begin{align}
\label{eq:basic_inequality}
\nonumber 
\|\Pi_m(\hat f_{t,m}-f_0)\|_{L^2(\Omega)}^2 
 &\lesssim  \int_\Omega \log \frac{\Pi_m \hat f_{t,m}}{\Pi_m f_0 } d\left(\frac 1 t \sum_{i=1}^T \delta_{Y_i} - ZP \right) \\
 \nonumber
 &= \frac T t \int_\Omega \log \frac{\Pi_m \hat f_{t,m}}{\Pi_m f_0 } d\left( P_T - \frac{tZ}{T} P \right) \\
 &= \frac T t \bbG_T \hat f_{t,m} + \frac{Zt-T}{t} \Delta \hat f_{t,m},
\end{align}
where we abbreviate, for all $f\in \calF$, 
\begin{align*}
\Delta f &= \KL_\Omega(\Pi_m f_0\|\Pi_mf) \equiv \int_\Omega \log \frac{\Pi_m f_0}{\Pi_m f} 
     \Pi_m f_0, \\
\bbG_\tau f &= 
    \int_\Omega \log \frac{\Pi_m f}{\Pi_m f_0} d(P_{\tau} - P),\quad \tau=1,2,\dots,
\end{align*}
and we define $\bbG_0$ arbitrarily.
With these preparations in place, 
we can now begin bounding the squared $L^2(\Omega)$ risk of $\Pi_m \hat f_{t,m}$. 
Let $\tau_0 = \lfloor Zt -t^{3/4}\rfloor$ and $\tau_1=\lceil Zt+t^{3/4}\rceil$.
For any $\delta > 0$,  we have
\begin{align}
\label{eq:mle_pf_decomp}
\nonumber 
\bbP & (\|\Pi_m(\hat f_{t,m} - f_0)\|_{L^2(\Omega)}^2 > \delta) \\
\nonumber
 &\leq \bbP(|T-Zt| > t^{3/4}) + \sum_{\tau=\tau_0 }^{\tau_1} \bbP(\|\Pi_m(\hat f_{t,m} - f_0)\|_{L^2(\Omega)}^2 >  \delta 
 \,\big|\, T=\tau)\bbP(T = \tau) \\
 &\leq 2 \exp(-C_0\sqrt t) + \sum_{\tau=\tau_0 }^{\tau_1} \bbP(\|\Pi_m(\hat f_{t,m} - f_0)\|_{L^2(\Omega)}^2 >  \delta 
 \,\big|\, T=\tau)\bbP(T=\tau),
\end{align}
for a small enough constant $C_0 > 0$, where 
the final inequality follows by standard tail bounds for Poisson
random variables (cf. Lemma~\ref{lem:poisson_concentration}), together
with the fact that $Z \asymp 1$. 
Let us now bound the conditional probabilities appearing in the above display.
Given $\tau \in \{ \tau_0,\dots,\tau_1\}$, 
the basic inequality~\eqref{eq:basic_inequality} implies
\begin{align*}
  \bbP&\left(\|\Pi_m(\hat f_{t,m} - f_0)\|_{L^2(\Omega)}^2 > \delta  \,\big|\,
  T=\tau\right) \\
   &\leq \bbP\left(\sup_{f\in \calF:\|\Pi_m (f-f_0)\|_{L^2(\Omega)}^2 >  \delta } 
 \frac{\tau}{t}|\bbG_\tau f| + \left|\frac{\tau-Zt}{t}\right|\Delta f - C_1  \|\Pi_m (f-f_0)\|_{L^2(\Omega)}^2 > 0 \right) \\
   &\leq \bbP\left(\sup_{f\in \calF:\|\Pi_m (f-f_0)\|_{L^2(\Omega)}^2 > \delta } 
  |\bbG_\tau f| + t^{-\frac 1 4}\Delta f - C_2  \|\Pi_m (f-f_0)\|_{L^2(\Omega)}^2 > 0 \right),
 \end{align*}
 for $C_1>0$ denoting the hidden constant from the basic inequality \eqref{eq:basic_inequality}, and $C_2 > 0$ is a large enough constant, where we used the fact that 
 $\tau \asymp t$ and $|\tau-Zt| \lesssim  t^{3/4}$, by definition of $\tau_0$ and $\tau_1$.
We note that the above probabilities are all well-defined since the involved quantities are all Borel-measurable; indeed, since $\calF$ has a countable dense subset in uniform norm and since for each $\tau \in \{ \tau_0,\dots,\tau_1\}$ the functionals $\mathbb{G}_{\tau}(\cdot)$ and $\Delta(\cdot)$ are continuous in uniform norm, the relevant suprema also can be taken over the countable dense subset which ensures the measurability of the supremum. In all subsequent displays we tacitly utilize this fact, which enables us to overcome potential measurability issues. 

 Now, recalling that all elements of $\calF_m$ are bounded from above and below
by positive constants over $\Delta$, we must have for all $f \in \calF$, 
$$\Delta f = \mathrm{KL}_\Omega(\Pi_m f \|\Pi_m f_0)
\asymp \|\Pi_m( f - f_0)\|_{L^2(\Omega)}^2,$$ and thus, 
assuming without loss of generality that $t$ is so large that $t^{-1/4} < C_2/2$, we have, 
\begin{align*}
  \bbP&\left(\|\Pi_m(\hat f_{t,m} - f_0)\|_{L^2(\Omega)}^2 > \delta  \,\big|\,
  T=\tau\right) \\ 
   &\leq \bbP\left(\sup_{f:\|\Pi_m (f-f_0)\|_{L^2(\Omega)}^2 >  \delta } 
  |\bbG_\tau f|  - C_3 \|\Pi_m (f-f_0)\|_{L^2(\Omega)}^2 > 0 \right).
\end{align*}
We now bound the above probability using a peeling argument inspired by~\cite{vandegeer2000}.
Letting  $S = \min\{s \geq 0: 2^{s+1}\delta > 1\}$, we obtain 
\begin{align*}
\bbP&\left(\|\Pi_m(\hat f_{t,m} - f_0)\|_{L^2(\Omega)}^2 > \delta \,\big|\,
  T=\tau\right) \\ 
 &\leq \sum_{s=0}^S \bbP\left(\sup_{f: 2^{s}\delta \leq 
 \|\Pi_m (f-f_0)\|_{L^2(\Omega)}^2 \leq 2^{s+1}\delta} 
 |\bbG_\tau f| >  C_3 2^{s+1}\delta     \right) \\
 &\leq \sum_{s=0}^S \bbP\left(\sup_{f \in \calF_m(\sqrt{2^{s+1}\delta};f_0)} 
 |\bbG_\tau f|   > C_4 2^s\delta \right).
\end{align*} 
Next, we will use the following fact,
which follows from standard
expectation bounds for suprema of empirical processes, together 
with Talagrand's concentration inequality.
\begin{lemma}
\label{lem:talagrand}
Let $\calG\subseteq \calF_m - \Pi_m f_0$. Then, there exist constants $C_5,c_5 > 0$ depending only on $b$ 
such that if 
\begin{equation}
\label{eq:le_cam_eq_general}
\sqrt{\tau} \gamma^2 \geq \calJ_{[]}(c_5\gamma,\calG,L^2(\Omega))~~\text{for  } \gamma = \sup_{g\in\calG} \|g\|_{L^2(\Omega)},
\end{equation}
then for any $v > 0$,
$$\bbP\left( \sup_{g\in\calG} |\bbG_\tau g| \geq  
C_5\tau^{-1/2}\calJ_{[]}(c_5\gamma,\calG,L^2(\Omega)) + v\right) 
\leq \exp\left(-\frac{Z\tau v^2}{C_5(\gamma^2 +  v)}\right),$$
\end{lemma}
We provide a proof in Appendix~\ref{app:pf_lem_talagrand} below.  
Now, by condition~\eqref{eq:le_cam_bracketing_condition}  with an appropriate choice of $C_0,c_0 > 0$, it must hold that
\begin{align*}
C_4 2^{s}\delta  / 2 
 &> 
C_2\tau^{-\frac 1 2} \calJ_{[]}\left(c_5 \sqrt{ 2^{s+1}\delta }, \calF_m(\sqrt{2^{s+1}\delta};f_0), L^2(\Omega)\right) =: J_m,
\end{align*}
for all  $\delta \geq C_6 \gamma_{\tau,m}^2$. By Lemma~\ref{lem:talagrand}, we obtain
\begin{align*}
\bbP&\left(\|\Pi_m(\hat f_{t,m} - f_0)\|_{L^2(\Omega)}^2 > \delta \,\big|\,
  T=\tau\right) \\
&\leq \sum_{s=0}^S \bbP\left(\sup_{f\in \calF_m(\sqrt{2^{s+1}\delta};f_0)}
 |\bbG_\tau f|   > J_m+ C_4 2^{s-1}\delta \right) \\
&\leq \sum_{s=0}^S \exp\left(-\frac{\tau (2^{s-1}\delta)^2}{C_5(2^{s+1}{\delta} + C_4 2^{s-1}\delta)}\right) \\
&\leq \sum_{s=0}^S \exp\left(-C_7 \tau 2^{s}\delta \right)
\leq \sum_{s=1}^{S+1} \exp\left(-C_7 \tau s \delta \right) \lesssim \exp(-C_8 \tau\delta).
\end{align*} 
Returning to equation~\eqref{eq:mle_pf_decomp}, we have thus shown
that for all $\delta \geq C_6 \gamma_{t,m}^2$, 
\begin{align*}
\bbP\left(\|\Pi_m(\hat f_{t,m} - f_0)\|_{L^2(\Omega)}^2 > \delta \right) 
  &\lesssim   \exp(-C_0 \sqrt t) + \sum_{\tau=\tau_0}^{\tau_1}
  \bbP(T=\tau)\exp(-C_8 \tau\delta) \\
  &\leq   \exp(-C_0\sqrt t) + \exp(-C_9 \tau\delta)\sum_{\tau=\tau_0}^{\tau_1}
  \bbP(T=\tau) \\
   &\leq  \exp(-C_0\sqrt t) + \exp(-C_9 \tau\delta).
\end{align*}
Finally, since $\big\|\Pi_m (\hat f_{t,m} - f_0)\big\|_{L^2(\Omega)}^2 \leq C_{10}$, 
we deduce
\begin{align*}
\bbE \big\|\Pi_m (\hat f_{t,m} - f_0)\big\|_{L^2(\Omega)}^2
 &\leq \int_0^{C_{10}} \bbP\Big(\big\|\Pi_m (\hat f_{t,m} - f_0)\big\|_{L^2(\Omega)}^2 \geq \delta\Big)d\delta \\
 &\lesssim \gamma_{t,m}^2 + \exp(-C_0\sqrt t) + \int_{\gamma_{t,m}^2}^{C_{10}} \exp(-C_9t\delta)d\delta  
 \lesssim \gamma_{t,m}^2 + \frac 1 t.
 \end{align*}
The claim follows.\qed

\subsubsection{Proof of Lemma~\ref{lem:talagrand}}
\label{app:pf_lem_talagrand}
Write $\|\bbG_\tau\|_\calG := \sup_{g\in\calG} \big|\bbG_\tau g\big|$
and $\gamma_P = \sup_{g\in\calG} \|g\|_{L^2(P)}$. 
By Talagrand's concentration inequality for suprema
of empirical processes (in its form stated in Corollary~2.15.7 
in~\cite{vandervaart2023}), there exists a constant $C > 0$
such that
$$\bbP\left(  \|\bbG_\tau\|_{\calG}  \geq  \bbE\|\bbG_\tau\|_{\calG} + v\right) 
\leq \exp\left(-\frac{C\tau v^2}{\gamma_P^2 +  \bbE\|\bbG_\tau\|_{\calG} + v}\right).$$ 
Furthermore, the mean of $\|\bbG_\tau\|_{\calG}$ can be bounded
from above as follows~(see, for instance, Proposition~3.5.15 
of~\cite{gine2016}), for a  small enough constant $c_5 > 0$, 
$$\bbE\|\bbG_\tau\|_{\calG} \lesssim 
\frac 1 {\tau}\left( \sqrt{\tau} + \frac 1 {\gamma_P^2} \calJ_{[]}(c_5\gamma_P,\calG,L^2(P))\right)
\calJ_{[]}(c_5\gamma_P,\calG,L^2(P)).$$
Now, notice that $\|\cdot\|_{L^2(\Omega)} \asymp_b \|\cdot\|_{L^2(P)}$
and $\gamma_P \asymp \gamma$. 
Thus,  by Lemma~\ref{lem:bracketing_number_equiv},
 up to increasing $C_5,c_5> 0$, the above display implies
$$\bbE\|\bbG_\tau\|_{\calG} \lesssim 
\frac 1 {\tau}\left( \sqrt{\tau} + \frac 1 {\gamma^2} \calJ_{[]}(c_5\gamma ,\calG,L^2(\Omega))\right)
\calJ_{[]}(c_5\gamma,\calG,L^2(\Omega)) \lesssim \frac 1 {\sqrt{\tau}}   \calJ_{[]}(c_5\gamma_P,\calG,L^2(\Omega)),$$
where the final inequality follows from condition~\eqref{eq:le_cam_eq_general}. 
Combining these facts implies that, for a large enough constant $C_5 > 0$, it holds
for all $v > 0$ that
$$\bbP\left(  \|\bbG_\tau\|_{\calG}  \geq  
C_5\tau^{-1/2}\calJ_{[]}(c_5\gamma,\calG,L^2(\Omega)) + v\right) 
\leq \exp\left(-\frac{\tau v^2}{C_5(\gamma^2 +  v)}\right),$$
as claimed.
\qed

\subsection{Proof of Lemma~\ref{lem:local_bracketing}}
\label{app:pf_lem_local_bracketing}
Let $\mu\in\calU_k(\Theta)$. We write $f_0 = K\star \mu$ and 
we use the abbreviations
\begin{align*}
\calF &= \{K\star\nu: \nu\in\calU_k(\Theta)\},\\
\calF_m &= \{\Pi_m f:f\in\calF\},\\
\calF(\gamma;f_0) &= \{f\in\calF:\|f-f_0\|_{L^2(\Omega)}\leq \delta\},\\
\calF_m(\gamma;f_0) &= \{f\in\calF_m:\|f-\Pi_mf_0\|_{L^2(\Omega)}\leq \delta\}.
\end{align*}
To prove the claim, we begin with a naive bound on the bracketing entropy
of $\calF_m(\gamma;f_0)$ which holds for all  $u,\delta > 0$.
Notice that $\Pi_m:L^2(\bbR^d) \to L^2(\Omega)$ is a non-expansive mapping, 
thus
\begin{align*}
\log N_{[]}(u,\calF_m(\gamma;f_0),L^2(\Omega))
 \leq \log N_{[]}(u,\calF_m,L^2(\Omega))  
 \leq \log N_{[]}(u,\calF,L^2(\Omega)).
\end{align*}
Now, recall from Proposition~\ref{cor:Lp_Omega_Rd} and Theorem~\ref{thm:equivalences}
that the $L^2(\Omega)$ distance is equivalent to the Hellinger
distance over $\calF$. Thus, by Lemma~\ref{lem:bracketing_number_equiv}, 
there is a constant $c>0$ such that 
\begin{align*} 
\log N_{[]}(u,\calF,L^2(\Omega))
\leq \log N_{[]}(cu,\calF,H)
\leq \log N_{[]}(cu,\{K\star \nu: \nu\in\calU_k(\Theta)\},H)
\lesssim \log(1/u),
\end{align*}
where the final inequality follows from Theorem~3.1 of~\cite{ghosal2001},
when $d=1$, or Lemma~2.1 of~\cite{ho2016a}
when $d>1$. This proves one of the claimed bounds,
and we now provide a sharper bound in the regime $u,\gamma \gtrsim 1/m $.  
Notice first that the elements of $\calF$ are Lipschitz
with a uniform Lipschitz modulus, thus, by Lemma~\ref{lem:Pi_error},
there exists a constant $C$ depending only on 
$L$ and $d$ such that
for any $h \in \calF$, 
\begin{align}
\label{eq:discr_err}
\|h -\Pi_m h\|_{L^2(\Omega)}  \leq C m^{-1/d}.
\end{align}
Therefore, for any $u,\gamma \geq 2C m^{-1/d}$, one has
\begin{align*}
N_{[]}(u, \calF_m(\gamma;f_0),L^2(\Omega)) 
 &\leq N_{[]}\Big(u- Cm^{-1/d}, 
          \{f\in \calF: \|\Pi_m(f-f_0)\|_{L^2(\Omega)}\leq \gamma\} ,L^2(\Omega)\Big) \\    
 &\leq N_{[]}\big(u/2, 
        \{f\in \calF: \|f-f_0\|_{L^2(\Omega)}\leq \gamma + Cm^{-1/d}\},L^2(\Omega)\big) \\
 &\leq N_{[]}(u/2, \calF(2\gamma;f_0),L^2(\Omega)) \\
 &\leq N_{[]}(u/2, \calF(2\gamma;f_0),L^\infty(\Omega)) \\
 &\leq N(u/4,\calF(2\gamma;f_0),L^\infty(\Omega)).
\end{align*}
It thus remains to bound the covering number on the final line of the above display, 
which, in view of Theorem~\ref{thm:equivalences}, will 
be bounded from above by a 
covering number with respect to the class
 $$\calM(\epsilon) = \left\{ (m_\alpha(\nu))_{\alpha\in I}: \nu\in\calU_k(\Theta),
 M_k(\mu,\nu)\leq\epsilon\right\}\subseteq \bbR^{|\calI|},\text{for some } \epsilon > 0,$$
where we write $\calI = \{\alpha\in \bbN_0^d\setminus\{0\}: \|\alpha\|_1\leq k\}.$
Indeed, recall from Theorem~\ref{thm:equivalences} that
there exists $\lambda > 0$ such that
$$\lambda^{-1} M_k(\tilde\nu,\nu) \leq \|K\star(\tilde\nu-\nu)\|_{L^2(\Omega)}
\leq \|K\star(\tilde\nu-\nu)\|_{L^\infty(\Omega)}
 \leq \lambda M_k(\tilde\nu,\nu).$$  
Let $N = N(u/(8\lambda),\calM(2\lambda\gamma),\|\cdot\|_1)$,
and let $\mathsf{m}^{(1)},\dots,\mathsf{m}^{(N)}$ be a $(u/(8\lambda))$-cover
of $\calM(2\lambda\gamma)$ with respect to the $\|\cdot\|_1$ metric. Define
$$\nu^{(i)} = \argmin_{\tilde\nu\in\calU_k(\Theta)}
\sum_{\alpha\in\calI} |m_\alpha(\tilde \nu) - \mathsf{m}_\alpha^{(i)}|, 
\quad i=1,\dots,N.$$
We will argue that $\{K\star \nu^{(i)}:1\leq i \leq N\}$
forms a $(u/4)$-cover of $\calF(2\gamma)$.  
Indeed, let $K\star\nu\in\calF(2\gamma)$. 
Then, we must have
$$M_k(\mu,\nu) \leq \lambda \|K\star(\mu-\nu)\|_{L^2(\Omega)}
 \leq 2\lambda\gamma,$$
thus $(m_\alpha(\nu))_{\alpha\in\calI} \in\calM(2\lambda\gamma)$. It follows that
there exists $1 \leq i \leq N$ such that $\sum_{\alpha\in\calI} |m_\alpha(\nu)-\mathsf{m}_\alpha^{(i)}| \leq u/(8\lambda)$,
and hence
\begin{align*}
\|K\star(\nu-\nu^{(i)})\|_{L^\infty(\Omega)}
 &\leq \lambda M_k(\nu, \nu^{(i)}) \\
 &\leq \lambda \left(\sum_{\alpha\in\calI} |m_\alpha( \nu) - 
\mathsf{m}_\alpha^{(i)}| + \sum_{\alpha\in\calI} |
\mathsf{m}_\alpha^{(i)} - m_\alpha(\nu^{(i)})|\right) \\
 &\leq 2\lambda \sum_{\alpha\in\calI} |m_\alpha( \nu) - 
\mathsf{m}_\alpha^{(i)}| \leq u/4 . 
\end{align*}
We have thus shown that
$$N(u/4,\calF(2\gamma),\|\cdot\|_{L^\infty}(\Omega)) \leq 
N = N(u/(8\lambda),\calM(2\lambda\gamma),\|\cdot\|_1).$$
Finally, notice that $\calM(2\lambda\gamma)$ is simply an $\ell_1$-ball
of radius $2\lambda\gamma$ in $\bbR^{|\calI|}$, thus,
by Example~5.8 of~\cite{wainwright2019},  
$$N(u/(8\lambda),\calM(2\lambda\gamma),\|\cdot\|_1) \leq \left(1 + \frac{C\gamma}{u}\right)^{|\calI|},$$
for a constant $C > 0$ depending on $\lambda$. 
The claim follows from here after taking the logarithm and using the numerical inequality $\log(1+x)\leq x$ for $x\geq 0$.\qed

\subsection{Proof of Proposition~\ref{prop:mle_gaussian}$(ii)$}
\label{app:pf_mle_ii}
Proposition~\ref{prop:mle_gaussian}$(ii)$
can be proven by a similar argument as Proposition~\ref{prop:mle_gaussian}$(i)$, 
thus we only provide a brief outline of the proof. 
Given a kernel $K$ satisfying Assumption~\ref{ass:kernel}$(i)$,
define again the classes
\begin{align*}
\calF = \{K\star\mu:\mu\in\calU_k(\Theta)\},\quad 
\calF_m = \left\{ \Pi_m f: f\in\calF\right\}.
\end{align*}
Theorem~\ref{thm:momentL2-bounds_all} 
implies
\begin{equation}
\bbE \Big[M_{2k-1}^2(\hat\mu_{t,m}, \mu)\Big] \lesssim  
\bbE\|K\star (\hat\mu_{t,m}-\mu)\|^2_{L^2(\Omega)}
\lesssim \bbE \Big[H^2 (K\star \hat\mu_{t,m}, K\star\mu)\Big],
\end{equation}
where the final inequality uses the boundedness of $K$. 
Unlike in our proof of Lemma~\ref{lem:chaining}, the convolutions
$K\star \mu$ are not bounded from below over $\Omega$.
However, they do integrate to unity over~$\Omega$, 
and  in this case, the final expectation in the above display 
can   be bounded from above using essentially the same strategy
as for upper bounds on the nonparametric MLE
under the Hellinger distance (for instance, see
 Theorem~7.4 of~\cite{vandegeer2000}), with appropriate
 modifications
 to account for the presence of binning and Poisson fluctuations
  (as in Lemma~\ref{lem:chaining}). 
 Doing so leads to the upper bound
$$\bbE \Big[H^2 (K\star \hat\mu_{t,m}, K\star\mu)\Big]
\lesssim \left(\frac 1 {\sqrt t} +   \gamma_{t,m}  + m^{-1/d}\right)^2,$$
where, for appropriate constants $C_0,c_0 > 0$, 
$\gamma_{t,m}$ is defined as the smallest positive number such that for all $\gamma \geq 
\gamma_{t,m}$, one has
\begin{align} 
\label{eq:lecam_again}
\sqrt t \gamma^2 \geq 
C_0  \calJ_{[]}(c_0\gamma, \calF_m,H).
\end{align}
Furthermore, the global Hellinger metric entropy can  easily be bounded as follows,
due to the Lipschitz assumption on the kernel $K$.
\begin{lemma}
There exists a constant $C = C(\Theta,k,d) > 0$ such that for all $m\geq 1$
and $u \in (0,1)$,  
$$N(u,\calF_m, H)\leq C u^{2dk}.$$
\end{lemma}
The Lemma is elementary, and stated without proof. 
We deduce from here that the Le Cam equation~\eqref{eq:lecam_again} reduces to 
$\gamma^2 \sqrt t \gtrsim \gamma \sqrt{\log(1/\gamma)}$, which is satisfied
for all $\gamma\gtrsim \gamma_{t,m} := \sqrt{\log t/t}.$  
Thus, we obtain 
$$\bbE \Big[H^2 (K\star \hat\mu_{t,m}, K\star\mu)\Big]
\lesssim \left(\sqrt{\frac{\log t}{t}}  + m^{-1/d}\right)^2,$$
and the claim now follows.\qed 

\section{Proofs for Section \ref{sec:minimaxLowerBound}}

\subsection{Proof of Lemma \ref{lem:regularKernelSufficientCondition}}\label{app:prf:lem:regularKernelSufficientCondition}

For Assertion $(ii)$, we infer by an induction argument that $D^{s} \overline K(x) = \overline K(x) P_{s, \sigma}(x)$ for some polynomial of degree $2s$ which depends on the variance $\sigma>0$ of the normal density $\overline K$. Hence, for $\theta\neq \tilde \theta\in \RR$ it follows that 
	\begin{align*}
		\frac{(D^s \overline K(x-\theta))^2}{\overline K(x-\tilde \theta)} &= \exp\left(-  \frac{(x-\theta)^2}{\sigma^2} + \frac{(x-\tilde \theta)^2}{2\sigma^2} \right) P^2_{s, \sigma}(x-\theta)\\
		&= \exp\left(-\frac{x^2}{2\sigma^2} + \frac{(2\theta - \tilde \theta) x}{\sigma^2}  + \frac{-2\theta^2+ \tilde \theta^2}{2\sigma^2}\right) P_{s, \sigma}^2(x-\theta). 
	\end{align*}
	This representation confirms that the integral function $(\theta, \theta')\mapsto \int \frac{(D^s \overline K(x-\theta))^2}{\overline K(x-\tilde \theta)}\dif x$ is bounded and continuous. In particular, we conclude that the centered normal density is $s$-regular for any $s\in \NN$. 

	For Assertion $(i)$ we use that there exists some $\delta'\in (0,\delta]$ such that the following holds three properties hold. Our argument for their validity is detailed at the end of this proof. 
	 \begin{enumerate}
		\item[$(a)$] the kernel $\overline K$ is strictly positive on $[-\gamma+\delta', \gamma -\delta']$
		\item[$(b)$] the squared $s$-th order derivative $(D^s \overline K)^2$ is nondecreasing on $(-\infty,-\gamma+\delta')$ and nonincreasing on $(\gamma-\delta', \infty)$, and  
		\item[$(c)$] the quotient $(D^s \overline K)^2/\overline K$ has a finite integral on $(-\gamma,-\gamma+\delta')\cup (\gamma-\delta', \gamma)$.
	\end{enumerate}
	In fact, as we will show below, if $\overline K$ is $s$-times continuously differentiable and fulfills the above three properties, the $s$-regularity of the kernel $K$ follows.

	To this end, define the index $\iota \coloneqq \mathds{1}(s\neq 1)$, consider $\theta_\iota,\dots, \theta_s\in \RR$ as in the assertion and recall that $|\theta_s - \theta_\iota|>0$. Then, based on property $(a)$ it follows for $\tau>0$ small enough with $\tau(\theta_s - \theta_\iota)\leq \gamma - \delta'$ that the sum $\sum_{j = \iota}^{s}K(\cdot - \epsilon\theta_j)$ is strictly positive on $[\epsilon\theta_\iota - \gamma + \delta', \epsilon\theta_s + \gamma - \delta']$ for $\epsilon\in (0,\tau)$. This implies jointly with $s$-times continuous differentiability of $\overline K$ that
	\begin{align*}
\max_{i \in \{1, \dots, s\}}\sup_{\epsilon \in (0,\tau)}\sup_{t\in [0,1]} \int_{\epsilon\theta_\iota - \gamma + \delta'}^{\epsilon\theta_s + \gamma - \delta'} \frac{\left(D^{s} \overline K(y - t \epsilon \theta_i) \right)^2}{\sum_{j = v}^{s} \overline K(y - \epsilon \theta_j)}\dif y < \infty.
	\end{align*}
Moreover, for $t \in [0,1]$ we obtain for each $i \in \{\iota+1, \dots, s\}$ and $\epsilon\in (0,\tau)$, $t\in [0,1]$ since $\theta_\iota<0$, and thus $\theta_\iota< t\theta_i$, combined with property $(b)$ for all 
$y'\in (-\infty, -\gamma+\delta')$ and $t\in [0,1]$ that $\left(D^{s} K(y' - \epsilon (t\theta_i -\theta_\iota)) \right)^2 \leq \left(D^{s} K(y') \right)^2.$
Hence, it follows from our convention $0/0 \coloneqq 0$ that
\begin{align*}
	\int_{-\infty}^{\epsilon \theta_\iota -\gamma+\delta'}  \frac{\left(D^{s} \overline K(y - t \epsilon \theta_i) \right)^2}{\sum_{j = v}^{s} \overline K(y - \epsilon \theta_j)} \dif y &\leq \int_{-\infty}^{\epsilon\theta_\iota -\gamma+\delta'}  \frac{\left(D^{s} \overline K(y - \epsilon \theta_\iota) \right)^2}{\sum_{j = v}^{s} \overline K(y - \epsilon \theta_j)}\dif y\\
	&\leq \int_{-\infty}^{\epsilon\theta_\iota -\gamma+\delta'}  \frac{\left(D^{s} \overline K(y - \epsilon \theta_\iota) \right)^2}{\overline K(y - \epsilon \theta_\iota)}\dif y\\ 
	&= \int_{-\infty}^{-\gamma+\delta'}  \frac{\left(D^{s} \overline K(y) \right)^2}{\overline K(y)}\dif y \\
	&= \int_{-\gamma}^{-\gamma+\delta'}  \frac{\left(D^{s} \overline K(y) \right)^2}{\overline K(y)}\dif y<\infty,
\end{align*} 
where the last equality follows from the fact that the integrand vanishes on $(-\infty,-\gamma]$ and the finiteness of the integral follows from property $(c)$.

Likewise, it follows for $t \in [0,1]$,  $i\in \{1, \dots, s\}$, and $\epsilon \in (0,\tau)$ since $\theta_s>0$ that $\epsilon\theta_s> t\epsilon\theta_i$ and thus using property $(b)$ that 
\begin{align*}
		\int^{\infty}_{\epsilon \theta_s +\gamma-\delta'}  \frac{\left(D^{s} \overline K(y - t \epsilon \theta_i) \right)^2}{\sum_{j = v}^{s} \overline K(y - \epsilon \theta_j)} \dif y \leq \int^{\gamma}_{\gamma-\delta'}  \frac{\left(D^{s} \overline K(y) \right)^2}{\overline K(y)}\dif y  <\infty.
	\end{align*}
Combining the previous three displays implies the $s$-regularity of $K$.
	
	It remains to confirm properties $(a)-(c)$ for kernels for which $\overline K_-$ and $\overline  K_+$ are proportional on some interval $[0,\delta)$ to a function in $\calF_\delta$. We only confirm the properties for $\overline K_-$, as the argument for $\overline K_+$ is analogous.
	If $\overline K_{-}(x) \propto \exp(-\alpha/x^\beta)x^\rho$ on $(0,\delta)$ for $(\alpha, \beta, \rho) \in (0,\infty)^2\times \RR$, then it follows by some induction argument for $x\in (-\infty,-\gamma+\delta)$ that $$D^s \overline K(x) = \mathds{1}(x> -\gamma)\exp(-\alpha/(x+\gamma)^\rho) P_{\alpha,\beta,\rho,s,\overline K}(x+\gamma)$$ where $P_{\alpha,\beta,\rho,s,\overline K}$ is a polynomial on $(0,\infty)$ with real order exponents. In particular, this implies $$(D^s \overline K(x))^2/\overline  K(x) = \exp(-\alpha/(x+\gamma)^\beta) P_{\alpha,\beta,\rho,s,\overline K}^2(x)$$ and guarantees the finite integral condition $(c)$ on $(-\gamma,-\gamma+\delta)$. Moreover, since $D^{s+1} \overline K(x) = \mathds{1}(x> -\gamma)\exp(-\alpha/(x+\gamma)^\rho) P_{\alpha,\beta,\rho,s+1,\overline K}(x+\gamma)$ and since $P_{\alpha,\beta,\rho,s+1,\overline K}$ only admits finitely many roots on $(0,\infty)$ (see, e.g., \citealt[pp.\ 9]{karlin1966tchebycheff} or \citealt[p.\ 48, Aufgabe 75]{polya1925aufgaben}), we infer that there exists some sufficiently small  $\tilde \delta'\in (0,\delta]$ such that $(D^s \overline K(x))^2$ is non-decreasing on $(-\gamma, -\gamma+ \tilde \delta')$ which confirms the first part of property $(b)$. Repeating the argument for $\overline K_+$ yields potentially a smaller value $\delta' \in (0,\tilde\delta']$ for which the second part of property $(b)$ is true. 
	Meanwhile, in case $\overline K_{-}(x) \propto x^\rho$ on $(0,\delta)$ for $\rho \in (2k,\infty)$ it follows on $(-\infty, -\gamma+\delta)$ that $$D^s(x) \propto \mathds{1}(x>-\gamma)(x+\gamma)^{\rho-k}$$ 
	which is increasing on $(\gamma, \gamma+\delta)$ and so is $(D^s(x))^2$, confirming the first part of property $(b)$. Further, for $x\in (\gamma, \gamma+\delta)$ it follows that 
	\begin{align*}
		(D^s \overline K(x))^2/\overline K(x) \propto \mathds{1}(x>-\gamma) (x+\gamma)^{\rho-2k},
	\end{align*}
	which is integrable by $\rho>2k-1$, confirming the finite integral property $(c)$ on $(-\gamma, -\gamma+\delta)$. Finally, recall that $\overline K$ is strictly positive on $[-\gamma+\delta, \gamma-\delta]$ and by the specific shape of $\overline K$ at the boundary of its support, it also follows that $\overline K$ is strictly positive on $[-\gamma +\delta', \gamma-\delta']$, confirming property $(a)$ and completing the proof of Assertion $(i)$. \qed

	\subsection{Proof of Proposition \ref{prop:minimax_lower_bound}}\label{app:pf:prop:minimax_lower_bounds}

	For the proof of our global and local minimax lower bounds (Proposition \ref{prop:minimax_lower_bound}) we rely on the following three lemmata. The first one (Lemma \ref{lem:hellinger_bound_for_minimax}) relates the Hellinger distance of the distribution underlying the model \eqref{eq:model} with intensity $t>0$ to the Hellinger distance between the underlying mixture distributions. The second one (Lemma \ref{lem:moment_independence}) confirms the existence of distinct $k$-atomic uniform measures on the real line such that their first $k-1$ moment match. The third lemma (Lemma \ref{lem:constructionHypothesis}) details a construction for the two-point hypothesis which we employ for Le Cam's method.

	\begin{lemma}[Upper bound Hellinger distance]\label{lem:hellinger_bound_for_minimax}
		Consider the setting of Proposition \ref{prop:minimax_lower_bound}, then it follows for $\mu, \nu \in \calU_k(\domain)$ and the associated distributions $\bP_{\mu}^{t}$, $\bP_{\nu}^{t}$ on $\RR^m$ for the statistical model \eqref{eq:model} that their Hellinger distance is bounded by 
		 are bounded by
		\begin{align*}
			 H^2(\bP_{\mu}^\t, \bP_{\nu}^\t) &\leq \frac{t}{2}  \cdot H^2\left((K\ast \mu)|_{\bigcup_{i = 1}^{n} B_i}, (K\ast \nu)|_{\bigcup_{i = 1}^{n} B_i})\right) \leq \frac{t}{2} \cdot H^2(K\ast \mu, K\ast \nu).\quad 
		\end{align*}
		\end{lemma}
		The proof is detailed in \Cref{app:pf:lem:hellinger_bound_for_minimax} and utilizes an explicit expression for the Hellinger distance between Poisson distributions.

	\begin{lemma}\label{lem:moment_independence}
	For any $k\in \NN, \epsilon, \delta>0$ and $\mu \in \calU_k((-\delta, \delta))$ with $k$ support points 
	there exists a distinct measure $\nu\in \calU_k((-\delta, \delta))$ such that $m_\alpha(\mu) = m_\alpha(\nu)$ for each $\alpha \in \{0, \dots, k-1\}$ and $0<W_\infty(\mu, \nu)\leq \epsilon$.  In particular, since $\mu$ is a uniform distribution on $k$ distinct points, it holds $\supp(\mu)\neq \supp(\nu)$.  
\end{lemma}
	The proof is provided in \Cref{app:pf_lem:moment_independence} and relies on the correspondence between $k$-atomic uniform distributions and roots of polynomials of degree $k$ from \Cref{subsec:momentCompBound}. In addition, we provide in \Cref{app:pf_alternative_proof_Rigollet} an alternative proof based on representational theory which was communicated by Philippe Rigollet.

	\begin{lemma}[Construction of two-point hypothesis]\label{lem:constructionHypothesis}
		Consider the setting of Proposition \ref{prop:minimax_lower_bound}$(ii)$ with $\mu_0 =  \sum_{i = 1}^{k_0} (r_i/k) \delta_{\theta_{0i}}  \in \calU_{k,k_0}(\Theta;r, \delta)$ for $\delta\in [t^{-1/2k},1]$. Consider $r^* \coloneqq r_{k_0}$ and assume that the kernel $K$ is $r^*$-regular for the direction $\sigma \in \RR^d\backslash\{0\}$ with $\|\sigma\|=1$. %
		Then, there exist distinct $r^*$-atomic measures $\tilde \mu= (1/r^*) \sum_{j =1}^{r^*} \delta_{\theta_j}$, $\tilde \nu= (1/r^*)\sum_{j =1}^{r^*} \delta_{\eta_j}\in\calU_{r^*}([-1/4,1/4])$ such that the following conditions are met for every $\epsilon\in (0,C\wedge \delta)$ for some positive constant $C=C(K,k) \leq 1$.
		\begin{enumerate}
			\item It holds $M_{r^*-1}(\tilde \mu, \tilde \nu) = 0$. %
			\item The $k$-atomic uniform measures $\mu_\epsilon\coloneqq (1/k)\sum_{j = 1}^{r^*} \delta_{\theta_{0k_0} + \epsilon \sigma \theta_j} +  \sum_{i = 1}^{k_0-1} (r_i/k) \delta_{\theta_{0i}}$ and $\nu_\epsilon\coloneqq (1/k)\sum_{j = 1}^{r^*} \delta_{\theta_{0k_0} + \epsilon \sigma \eta_j} +  \sum_{i = 1}^{k_0-1} (r_i/k) \delta_{\theta_{0i}}$ are both contained in $\calU_{k}(\Theta; \mu_0, \epsilon/4)$. %
			\item It holds for the Wasserstein distance that $W_1(\mu_\epsilon, \nu_\epsilon) = \epsilon \frac{r^*}{k} W_1(\tilde \mu, \tilde \nu)\geq  C(k)\epsilon>0$. In case $\epsilon \in (0,\diam(\Theta)^{1-(k/r^*)})$, it holds for the local Wasserstein divergence that $\calD_{\mu_0}(\mu_\epsilon, \nu_\epsilon) = W_1^{r^*}(\tilde \mu, \tilde \nu) \delta_{k_0}(\mu_0) \epsilon^{r^*} \geq C(k)\delta_{k_0}(\mu_0) \epsilon^{r^*}$ where $\delta_{k_0}(\mu_0) = \prod_{i=1}^{k_0-1}\|\theta_{0k_0} - \theta_{0i}\|^{r_i}$. Further, for $\epsilon<1/2$ the Hausdorff distance is lower bounded by $d_H(\mu_\epsilon, \nu_\epsilon) \geq  \epsilon \cdot d_H(\tilde \mu, \tilde \nu)\geq C(k)\epsilon$. 
			\item There exists a positive constant $C=C(\overline K,k)>0$ where $K(x) = \tilde K((\mathrm{id} - \sigma \sigma^\top)x)\overline K(\sigma^\top x)$ such that $H^2(K\ast \mu_\epsilon,K\ast \nu_\epsilon)\leq C  \cdot \epsilon^{2r^*}$.%
		\end{enumerate}	
	\end{lemma}
	
	The proof of this result is stated in \Cref{app:pf:lem:constructionHypothesis} and crucially relies on the existence of two distinct $r^*$-atomic uniforms on $\RR$ whose first $r^*-1$ moments (Lemma \ref{lem:moment_independence}) match in conjunction with the $r^*$-regularity of the kernel. 
	
	With these tools at our disposal, we can now state the proof for our global and local minimax lower bound. 

	\begin{proof}[Proof of Proposition \ref{prop:minimax_lower_bound}]
		The global lower bound immediately follows from the local lower bound by taking $\mu_0$ as a Dirac measure at $x\in \interior{\Theta}$ with $\delta>0$ such that $B(x,\delta/2)\subseteq \Theta$.

		To prove the local bound we assume without loss of generality that
		$j=k_0$. Furthermore, fix the notation $r^* = r_{k_0}$. 
		For the proof we rely on Le Cam's two point method, see \citet[Sections 2.2--2.4]{tsybakov2009nonparametric}, which is applicable since $\calD_{\mu_0}^{1/r^*}$ defines by Lemma~\ref{lem:inequalityWassersteinLocal} a metric on $\calU_k(\Theta; \mu_0, \delta/4)$.
		Specifically, we choose $\epsilon_t\coloneqq \kappa t^{-1/2r^*}$ with $\kappa\in (0,1)$ small enough such that $\kappa \leq \min(C(K,k), \diam(\Theta)^{1 - (k/r^*)}, 1/4)$ and specified later such that the measures $\mu_{\epsilon_t}, \nu_{\epsilon_t}$ from Lemma \ref{lem:constructionHypothesis} fulfill all properties for all $t\geq  1$ and are contained in $\calU_k(\Theta; \mu_0, \delta/4)$. In particular, it holds $\calD_{\mu_0}^{1/r^*}(\mu_{\epsilon_t}, \nu_{\epsilon_t})\geq  \kappa (C(k)\delta_{k_0}(\mu_0) t^{-1/2})^{1/r^*}$, and also 
		$W_1(\mu_{\epsilon_t}, \nu_{\epsilon_t}) \geq C(k) \kappa t^{-1/2r^*}$.  
		Combined with Lemma \ref{lem:hellinger_bound_for_minimax} we have $$H^2(\bP_{\mu_{\epsilon_t}}^t, \bP_{\nu_{\epsilon_t}}^t)\leq t \cdot H^2(K\ast \mu_{\epsilon_t}, K\ast \nu_{\epsilon_t}) \leq C(\overline K,k) t   \epsilon^{2r^*}_{t} = C(\overline K,k)\kappa^{2r^*}.$$
		Choosing $\kappa>0$ in the definition of $\epsilon_t$ as large as possible such that the right-hand side in the above display is upper bounded by $1/2$ and such that $\delta \kappa/2 \leq c$ to ensure that $\mu_{\epsilon_t}, \nu_{\epsilon_t}\in \calU_k(\Theta;\mu_0, ct^{-1/2r})$ implies by \citet[Theorem 2.2]{tsybakov2009nonparametric} the local minimax lower bound 
		\begin{align*}
\inf_{\hat \mu_{t,m}\in \calE_{t,m}(\Theta)} \sup_{\substack{\mu \in \calU_{k}(\Theta; \mu_0,c t^{-1/{2r_j}})}} \mathbb{P}(\calD^{1/r^*}_{\mu_0}(\hat \mu_{t,m}, \mu)\gtrsim_{\Theta, c,K,k} (\delta_{k_0}(\mu_0) t^{-1/2})^{1/r^*}) \geq \frac{1}{2}\left(1 - \sqrt{\frac{1}{2}\left(1 - \frac{1}{8}\right)}\right)\geq \frac{1}{7}.
		\end{align*}
		Taking the $r^*$-th power on both sides in the probability term and applying Markov's inequality yields the local minimax lower bound in the assertion.

	 To prove the final assertion, 
	 consider two points $\overline \theta_{0}, \overline \theta_1\in \Theta^o$ with $\|\overline \theta_{0} - \overline \theta_1\|\geq \frac{5}{6}\diam(\Theta^o)$. 
Since $\overline \theta_0$ is contained in the interior of $\Theta$, there exists some $\overline \delta\in (0, \diam(\Theta^o)/3)$ such that $B(\overline \theta_0,\overline \delta) \cup B(\overline \theta_1,\overline \delta) \subseteq \Theta^o$. Now, choose $\delta\in (0, \overline \delta)$ as large as possible such that there exist $k_0-1$ elements $\theta_{01}, \dots, \theta_{0k_0-1} \in \Theta^o$ which are all $\delta$-separated and such that 
$\bigcup_{j = 1}^{k_0-1} B(\theta_{0j}, \delta/2)\in B(\overline \theta_0, \overline \delta)$. Further, define $\theta_{0k_0} \coloneqq \overline \theta_1$. Then, it holds
\begin{align*}
	\|\theta_{0k_0} - \theta_{0j}\| \geq \|\theta_{0k_0} - \overline \theta_0\| - \|\overline \theta_0 - \theta_{0j}\| \geq \frac{5}{6}\diam(\Theta^o) - \frac{1}{3}\diam(\Theta^o) = \frac{1}{2}\diam(\Theta^o)\geq \delta. 
\end{align*}
Hence, it follows for any vector $r\in \NN^{k_0}$ with $|r| = k$ that the measure $\mu_0 = \frac{1}{k}\sum_{i = 1}^{k_0} r_i \delta_{\theta_{0i}}\in \calU_{k,k_0}(\Theta; r, \delta)$ fulfills $\delta_{k_0}(\mu_0) \geq (\frac{1}{2}\diam(\Theta^o))^{k-r_{k_0}}$ and $\supp(\mu_0)+B(0, \delta/2)\subseteq \Theta$. 
	\end{proof}

	\subsection{Proof of Remark \ref{rmk:MinimaxLowerBoundsDiscussion}$(iii)$ and $(iv)$}\label{app:pf:rmk:MinimaxLowerBoundsDiscussion}

For the proof of Assertion $(iii)$  we note by our lower bounds between $\mu_\epsilon$ and $\nu_\epsilon$ in terms of the Wasserstein distance and the Hausdorff distance (Lemma \ref{lem:constructionHypothesis}$(iii)$), that the same proof approach as for Prop \ref{prop:minimax_lower_bound} for $j\in \arg\max_{h = 1, \dots, k_0}r_j$  yields the minimax lower bounds with respect to the Wasserstein and the Hausdorff distance. 
	
	Moreover, Assertion $(iv)$ on the minimax rates for $n\in \NN$ i.i.d.\ observations based on the mixture $K\ast \mu$ follows by upper bounding the Hellinger distance for the same two-point hypotheses. Based on the construction from Lemma \ref{lem:constructionHypothesis} it follows for i.i.d.\ samples $X_1, \dots, X_n \sim K\ast \mu_{\epsilon}$ and $Y_1, \dots, Y_n \sim K\ast \nu_{\epsilon}$ by the tensorization property of the Hellinger distance for  $\epsilon_n \coloneqq \kappa n^{-1/2r}$ with $\kappa\in (0,1/4)$ small enough that
	\begin{align*}
		H^2(\bP^{(X_1, \dots, X_n)}, \bP^{(Y_1, \dots, Y_n)}) &= 1 - (1-H^2(K\ast \mu_{\epsilon_n}, K\ast \nu_{\epsilon_n}))^n\\
		&\leq 1 - (1- C(\overline K,k)   \epsilon_n^{2r})^n\\
		&= 1 - (1- C(\overline K,k)  \kappa^{2r}/n)^n \\
		&\leq 1- \exp\left(-C(\overline K,k)  \kappa^{2r}\right)(1-C(\overline K,k)^2  \kappa^{4r}/n).
	\end{align*} 
	In particular, if $\kappa$ is chosen such that $C(\overline K,k)^2  \kappa^{4r}<1/2$, then the right-hand side is strictly smaller than one and the assertion follows using \citet[Theorem 2.2]{tsybakov2009nonparametric}. \qed
	
	\subsection{Proof of Lemma \ref{lem:hellinger_bound_for_minimax}}\label{app:pf:lem:hellinger_bound_for_minimax}

	For the Poisson model it holds by the tensorization property of the Hellinger distance,
	\begin{align}
		 H^2(\bP_{\mu}^\t, \bP_{\nu}^\t) &= \left(1-\prod_{i = 1}^{n}\left(1- H^{2}\left(\textup{Poi}\left(\textstyle t  K\ast \mu(B_i)\right), \textup{Poi}\left(\textstyle t  K\ast \nu(B_i)\right) \right)\right)\right)\notag \\
		 &=   \left(1-\prod_{i = 1}^{n}\left(\exp\left(-\frac{1}{2}\textstyle\left[ \sqrt{t K\ast \mu(B_i)}  - \sqrt{t  K\ast \nu(B_i)}\, \right]^2 \right)\right)\right) \notag \\
		 &=  1 - \exp\left(- \frac{1}{2}\sum_{i = 1}^{n}\textstyle \left[\sqrt{t K\ast \mu(B_i)} -  \sqrt{t K\ast \nu(B_i)}\,\right]^2\right). 
	\end{align}
	Moreover, invoking the Cauchy-Schwarz inequality it follows for each $i\in \{1, \dots, n\}$ that 
	$$ \int_{B_i} \sqrt{t^2 K\ast \mu(x)K\ast \nu(x)} \dif x \leq \sqrt{t K\ast \mu(B_i)}\sqrt{t K\ast \nu(B_i)}.$$
	From this we infer
	\begin{align*}
		&\sum_{i = 1}^{n}\textstyle \left[\sqrt{ t  K\ast \mu(B_i)} -  \sqrt{t K\ast \nu(B_i)}\,\right]^2\\
		\leq \; & \sum_{i = 1}^{n} \int_{B_i}t K\ast \mu(x) \dif(x) + \int_{B_i}t K\ast \nu(x) \dif(x) - 2 \int_{B_i} \sqrt{t^2 K\ast \mu(x)K\ast \nu(x)} \dif x\\
		= \;& \int_{\bigcup_{i = 1}^{n} B_i }\left(\sqrt{t K\ast \mu(x)} - \sqrt{t K\ast \nu(x)}\right)^2 \dif x  \\
		= \;& t\cdot  H^2((K\ast \mu)|_{\bigcup_{i = 1}^{n} B_i}, (K\ast \nu)|_{\bigcup_{i = 1}^{n} B_i}).
	\end{align*}
	Overall, this asserts that 
	\begin{align*}
		 H^2(\bP_{\mu}^\t, \bP_{\nu}^\t) \leq & \, 1-  \exp\left(- \frac{t}{2}  H^2(K\ast \mu|_{\bigcup_{i = 1}^{n} B_i}, K\ast \nu|_{\bigcup_{i = 1}^{n} B_i})\right)\\
		  \leq &\, \frac{t}{2}\, H^2((K\ast \mu)|_{\bigcup_{i = 1}^{n} B_i}, (K\ast \nu)|_{\bigcup_{i = 1}^{n} B_i})\\
		  \leq &\, \frac{t}{2}\, H^2(K\ast \mu, K\ast \nu),
	\end{align*}
	where we used for the second inequality that $1- \exp(-x) \leq x$ for $x\in \RR$. \qed

\subsection{Proof of Lemma \ref{lem:moment_independence}}\label{app:pf_lem:moment_independence}
	Consider the $k$ distinct support points $x_1, \dots, x_k\in (-\delta, \delta)$ of $\mu$ such that $\mu = \frac{1}{k} \sum_{i =1}^{k} \delta_{x_i}$ and define for $l \in \{1, \dots, k\}$ the elementary symmetric polynomials 
	\begin{align*}
		e_0(x_1, \dots, x_n) \coloneqq 1 \quad \text{ and } \quad  e_l(x_1, \dots, x_n) &\coloneqq \sum_{1\leq i_1< i_2 < \dots i_l \leq k} \;\; \prod_{j =1}^l x_{i_j}.\end{align*}
	 Newton's identities \eqref{eq:Newton_identity} assert for any $l \in \{1, \dots, k\}$ that 
	 \begin{align*}
		 e_l(x_1, \dots, x_n) &= \frac{1}{l} \sum_{j=1}^{l} (-1)^{j-1} e_{l-j}(x_1, \dots, x_n) \sum_{i=1}^{k} x_i^j\\
		 &= \frac{k}{l} \sum_{j=1}^{l} (-1)^{j-1} e_{l-j}(x_1, \dots, x_n)  m_j(\mu)
	 \end{align*}
	This implies for any $l\in \{1, \dots, k\}$ that $e_l(x_1, \dots, x_k)$ admits an algebraic representation in terms of the moments $m_j(\mu)$ for $j\in \{1, \dots, l\}$. Further, since $e_0 = 1$, it also follows that $m_l(\mu)$ has an algebraic representation in terms of $e_j(x_1, \dots, x_k)$ for $j\in \{1, \dots, l\}$. 
	
	Next, define the polynomial $T_\mu(x) \coloneqq \prod_{i = 1}^k (x-x_i)$ which uniquely determines the measure $\mu$ since its roots coincide with the atoms from $\mu$. Recalling Vieta's identity, note that $T_\mu$ can be represented as
	\begin{align*}
		T_\mu(x) &= \sum_{l = 0}^{k} (-1)^l e_l(x_1, \dots, x_k) x^{k-l} \\
		&= (-1)^k e_k(x_1, \dots, x_k) + \sum_{l = 0}^{k-1} (-1)^l e_l(x_1, \dots, x_k) x^{k-l} \\
		&= (-1)^k\frac{k}{l} m_k(\mu) + \frac{k}{l}\sum_{j=1}^{k-1}(-1)^{k+j-1} e_{k-j}(x_1, \dots, x_n) m_j(\mu) + \sum_{l = 0}^{k-1} (-1)^l e_l(x_1, \dots, x_k) x^{k-l}.
	\end{align*}
	Since all roots of $T_\mu$ are distinct and real, for $\tau \in (0, \epsilon^{1/k})$  sufficiently small the perturbed polynomial $T_\mu^\tau(x)\coloneqq T_\mu(x) + \tau$ admits $k$ distinct real roots $x_1^\tau, \dots, x_k^\tau$ for which, upon denoting $\mu^{\tau} \coloneqq \frac{1}{k}\sum_{i = 1}^{k} \delta_{x_i^\tau}$, the inequality $0<W_\infty(\mu, \mu^\tau)\leq \tau^{1/k}\leq \epsilon$ and the inclusion $\mu^\tau \in \calU_k((-\delta, \delta))$ are fulfilled, see Lemma~\ref{lem:DisctinctRootsPerturbation} below. 
	Further, since the leading coefficient of $T_\mu^\tau$ equals one, it follows that $T_\mu^\tau$ can be represented as $\prod_{i=1}^k(x - x_i^\tau)$ and by comparing the coefficients of $T_\mu$ and $T_\mu^\tau$ we infer $$ e_l(x_1, \dots, x_n) = e_l(x_1^\tau, \dots, x_k^\tau) \quad \text{ for } l \in \{1, \dots, k-1\},$$
	which implies that $m_l(\mu) = m_l(\mu^\tau)$ for $l \in \{0, \dots, k-1\}$.  Only the coefficients for $l = k$ differ which yields that $\mu \neq \mu^\tau$. In particular, this implies that their supports must differ, since $\mu$ and $\mu^\tau$ are both $k$-atomic uniforms and $\mu$ already admits $k$ support points.\qed

	\begin{lemma}\label{lem:DisctinctRootsPerturbation}
	Let $p\colon \RR\to \RR$ be a monic polynomial of degree $k\in \NN$ with real coefficients  and distinct roots $x_1, \dots, x_k$. Let $\delta \in (0,\min_{i \neq j} |x_i - x_j|/2)$  and let $q\colon \RR\to \RR$ be a polynomial of degree smaller than $k$ such that $|q(x_i)| \leq \delta^{k}$ for each $i \in \{1, \dots, k\}$. Then, $g(x) \coloneqq p(x) + q(x)$ also admits $k$ real distinct roots $z_1, \dots, z_k$ and it holds 
	\begin{align*}
		W_\infty\left(\textstyle\frac{1}{k}\sum_{i =1}^{k} \delta_{x_i}, \frac{1}{k}\sum_{i =1}^{k} \delta_{z_i}\displaystyle \right) = \min_{\sigma \in \calS(k)} \max_{i =1, \dots, k} |x_i - z_{\sigma i}| \leq \delta.
	\end{align*}
	\end{lemma}
	
	\begin{proof}
		Since $g$ is a monic polynomial of order $k\in \NN$ the fundamental theorem of algebra asserts the existence of $k$ complex roots $z_1, \dots, z_k\in \CC$ such that $g(x) = \prod_{i = 1}^{k} (x-z_i)$. Further, for each root $x_j$ of $p$ it holds by assumption on $q$ and definition of $g$ that 
		$$\prod_{i = 1}^{k} |x_j - z_i| = |g(x_j)| = |p(x_j) + q(x_j)| = |q(x_j)| \leq \delta^k.$$ 
		Hence, for each root $x_j$ of $p$ there exists a root $z_j$ of $g$ such that $|x_j-z_j|\leq \delta$. Since all roots of $p$ are separated by more than $2\delta$ and $g$ only admits $k$ roots, it is evident that $z_j$ is the only root which fulfills $|x_j-z_j|\leq \delta$ and that all roots of $g$ are distinct. To show that  all roots are real-valued first, note that if one root $z_i$ of $g$ was genuinely complex, since $g$ only exhibits real-valued coefficients it would follow that $g(\overline z_i) = \overline{g(z_i)} = 0$ and the conjugate $\overline z_i$ would also be a complex root. However, from $|x_i - z_i| \leq \delta$  we would then infer that also $\overline z_i$ fulfills $|x_i - \overline z_i| \leq \delta$ and a contradiction would ensue since $z_i$ was the unique root with this property. 
	\end{proof}

	\subsection{Alternative Proof of Lemma \ref{lem:moment_independence}}\label{app:pf_alternative_proof_Rigollet} 
	This alternative proof was communicated by Philippe Rigollet. Let $\theta_1, \dots, \theta_k\in (-\delta, \delta)$ be distinct atoms, and let
	 $\mu = \frac{1}{k}\sum_{i = 1}^{k}\delta_{\theta_k}\in \calU_k((-\delta, \delta))$. 
	 Furthermore, denote the first $k$ moments of $\mu$ by $(m_{1}(\mu), \dots, m_k(\mu))$. Define the function $f\colon \RR\times \RR^{k-1} \to \RR^{k-1}$ by 
\begin{align*}
	f_\alpha(x_1, (x_2, \dots, x_k)) = \sum_{i =1}^{k} x_i^\alpha - m_\alpha(\mu) \quad \text{ for } \alpha = 1, \dots, k-1
\end{align*}
so that $f(\theta_1, (\theta_2, \dots, \theta_k)) = 0$. Moreover, the Jacobian of $f$ 
with respect to $(x_2, \dots, x_k)$ is a Vandermonde matrix so it is nonsingular at $(\theta_2, \dots, \theta_k)$ because all atoms of $\mu$ are distinct. We can therefore apply the implicit function theorem to get that there exist two open sets $U\subseteq (-\delta, \delta)$ and $V\subseteq (-\delta, \delta)^{k-1}$, such that $\theta_1 \in U$, $\overline \theta \coloneqq (\theta_2, \dots, \theta_k)\in V$, and a (unique) continuously differentiable function $\varphi \colon U \to V$ such that 
\begin{align*}
	\left\{(x_1, \varphi(x_1)) \colon x_1 \in U \right\}= \left\{(x_1, \overline x)\in U\times V \colon f(x_1, \overline x) = 0 \right\}.
 \end{align*}  
Note that $U$ may only depend on $k$ and $\delta$, so there exists $\tilde \theta_1\in U$ such that $\min_{i = 1,\dots, k} |\tilde\theta_1 - \theta_i|>0$. To conclude, observe that $f(\tilde \theta_1, \varphi(\tilde \theta_1)) = 0$, which implies that the first $k-1$ moments of the $k$-atomic uniform measure $\nu = \frac{1}{k}\big(\delta_{\tilde \theta_1} + \sum_{i = 1}^{k-1}\delta_{\varphi_i(\tilde \theta_1)}\big) \in \calU_k((-\delta, \delta))$ coincide with those of $\mu$, however $\nu$ differs from $\mu$ because their supports differ. Finally, by continuity of $\varphi$ we can choose $\tilde \theta_1$ sufficiently close to $\theta_1$ to ensure that $W_\infty(\mu, \nu)\leq \epsilon$. \qed

\subsection{Proof of Lemma \ref{lem:constructionHypothesis}} \label{app:pf:lem:constructionHypothesis}

	For Assertion $(i)$ with $r^* = 1$ consider $\tilde \mu = \delta_{1/4}, \tilde \nu = \delta_{-1/4} \in \calU_{1}([-1/4, 1/4])$. For $r^* \geq 2$ take an $r^*$-atomic measure $\tilde \mu = (1/r^*)\sum_{i = 1}^{r^*} \delta_{\theta_i}\in \calU_{r^*}([-1/4,1/4])$ with $r^*$ support points where at least one is located on $(-\infty, 0)$ and at least one on $(0, \infty)$. By Lemma~\ref{lem:moment_independence} another $r^*$-atomic measure $\tilde \nu = (1/r^*)\sum_{i = 1}^{r} \delta_{\eta_i}\in\calU_{r^*}([-1/4,1/4])$ exists which meets the same support condition and fulfills $M_{r-1}(\tilde \mu, \tilde \nu)=0$ while  $W_1(\tilde \mu, \tilde \nu)\wedge d_H(\tilde \mu, \tilde \nu)>0$. Note that the measures $\tilde \mu, \tilde \nu$ can be chosen such that $W_1(\tilde \mu, \tilde \nu)\wedge d_H(\tilde \mu, \tilde \nu)$ only depends on $k$.  %

	Assertion $(ii)$ is now a direct consequence of Assertion $(i)$ since $\epsilon<\delta$ combined with $\supp(\tilde \mu+ \tilde \nu)\in [-1/4,1/4]$ and $\supp(\mu_0)+B(0,\delta/2)\subseteq \Theta$.

	For Assertion $(iii)$ note by the Kantorovich-Rubinstein formulation for the 1-Wasserstein distance that 
	\begin{align*}
		W_1(\mu_\epsilon, \nu_\epsilon) &= \sup_{f\in \textup{Lip}_1(\RR^d)} \int f \dif (\mu_\epsilon - \nu_\epsilon)\\
		&= \sup_{f\in \textup{Lip}_1(\RR^d)} (1/k)\sum_{ j= 1}^{r^*} \left[f(\theta_{0k_0} + \epsilon \sigma \theta_j)- f(\theta_{0k_0} + \sigma\epsilon \eta_j)\right]\\
		&= (r^*/k)\sup_{f\in \textup{Lip}_1(\RR^d)} (1/r^*)\sum_{ j= 1}^{r^*} \left[f(\epsilon \sigma\theta_j)- f(\epsilon \sigma\eta_j)\right]\\
		&= (r^*/k) W_1\left((1/r^*)\sum_{j = 1}^{r^*}\delta_{\epsilon \sigma\theta_j}, (1/r)\sum_{j = 1}^{r^*}\delta_{\epsilon \sigma\eta_j}\right) \\
		&= \epsilon \frac{r^*}{k} W_1(\tilde \mu, \tilde \nu) \geq \epsilon\frac{W_1(\tilde \mu, \tilde \nu)}{k}  = C(k)\epsilon.
	\end{align*}

	Moreover, for the local Wasserstein divergence $\calD_{\mu_0}$ it holds since $\tilde \mu, \tilde \nu \in \calU_{r^*}([-1/4, 1/4])$ and $\epsilon<\delta$ for the measures $\mu_\epsilon, \nu_\epsilon$ conditioned on the Voronoi partition $V_1, \dots, V_{k_0}$ induced by $\supp(\mu_0)$, where we use the convention $\theta_{0i}\in V_i$ for $i = 1, \dots, k_0$, that 
	\begin{align*}
		&&\mu_{\epsilon, V_i} &= \delta_{\theta_{0i}}, &\nu_{\epsilon, V_i} &= \delta_{\theta_{0i}} \quad  \text{ for } i = 1, \dots, k_0-1,  && \\
		&&\mu_{\epsilon, V_{k_0}}  &= \frac{1}{r^*}\sum_{j = 1}^{r^*}\delta_{\theta_{0k_0} + \epsilon \sigma \theta_j}, & \nu_{\epsilon, V_{k_0}}  &= \frac{1}{r^*}\sum_{j = 1}^{r^*}\delta_{\theta_{0k_0} + \epsilon \sigma \eta_j}&&
	\end{align*}
	Hence, it follows that
	\begin{align*}
		\calD_{\mu_0}(\mu_\epsilon, \nu_\epsilon) &= 1 \wedge \textstyle\sum_{j=1}^{k_0} \delta_j(\mu_0) W_1^{r_j}(\mu_{\epsilon, V_j}, \nu_{\epsilon, V_j})\\*
		&= 1\wedge \left(\delta_{k_0}(\mu_0)W_1^{r^*}\left((1/r^*)\sum_{j = 1}^{r^*}\delta_{\epsilon \sigma\theta_j'}, (1/r^*)\sum_{j = 1}^{r^*}\delta_{\epsilon \sigma\eta_j'}\right)\right)\\*
		&= 1\wedge ( \epsilon^{r^*} \delta_{k_0}(\mu_0) W_1^{r^*}(\tilde \mu, \tilde \nu)).
	\end{align*}
	Since $W_1(\tilde \mu, \tilde \nu)\leq W_1(\tilde \mu, \delta_0)+W_1(\delta_0,\tilde \nu)\leq \delta < 1$ and since $\delta_{k_0}(\mu_0)\leq \diam(\Theta)^{k-r^*}$, we conclude by $\epsilon \leq \diam(\Theta)^{1-(k/r^*)}$ that 
	\begin{align*}
		\calD_{\mu_0}(\mu_\epsilon, \nu_\epsilon)  = \epsilon^{r^*} \delta_{k_0}(\mu_0) W_1^{r^*}(\tilde \mu, \tilde \nu)\geq  C(k)\delta_{k_0}(\mu_0)\epsilon^{r^*}.
	\end{align*}

	For the characterization of the Hausdorff distance between $\mu_\epsilon$ and $\nu_\epsilon$ we first note that 
	\begin{align*}
		\supp(\mu_\epsilon)&=\{\theta_{01}, \dots, \theta_{0k_0-1}, \theta_{0k_0} +\epsilon \sigma \theta_1', \dots, \theta_{0k_0} + \epsilon\sigma \theta_{r^*}' \}\\
		\supp(\nu_\epsilon)&=\{\theta_{01}, \dots, \theta_{0k_0-1}, \theta_{0k_0} + \epsilon \sigma\eta_1', \dots, \theta_{0k_0} + \epsilon\sigma \eta_{r^*}' \}.
	\end{align*}
	Now for every $j,j'\in \{1, \dots, r^*\}$ and $l \in \{1, \dots, k_0-1\}$ it holds for $\epsilon< 1/2$ that 
	\begin{align*}
		\|\theta_{0k_0} +\epsilon\sigma \theta_{j}' - \theta_{0k_0} - \epsilon \sigma \eta_{j'}'\| &= \epsilon |\theta_{j}' - \eta_{j'}'|\\
		&\leq \delta/2 < 3\delta/4\\
		&\leq  \|\theta_{0k_0} - \theta_{0l}\| -\delta/4 \\
		&\leq  \|\theta_{0k_0} +\epsilon\sigma \theta_{j}' - \theta_{0l}\|
	\end{align*} 
	This implies that 
	\begin{align*}
		d_H(\mu_\epsilon, \nu_\epsilon) &=d_H(\supp(\mu_\epsilon), \supp(\nu_\epsilon))\\
			&= d_H(\{\theta_{0k_0} +\epsilon \sigma \theta_1', \dots, \theta_{0k_0} + \epsilon\sigma \theta_{r^*}' \}, \{\theta_{0k_0} + \epsilon \sigma\eta_1', \dots, \theta_{0k_0} + \epsilon\sigma \eta_{r^*}' \}  )\\
			&= \epsilon  d_H(\{\sigma \theta_1', \dots,\sigma \theta_{r^*}' \}, \{  \sigma\eta_1', \dots,  \sigma \eta_{r^*}' \}  )\\
			&= \epsilon  d_H(\{\theta_1', \dots,\theta_{r^*}' \}, \{  \eta_1', \dots,  \eta_{r^*}' \}  ) \\
			&= \epsilon d_H(\tilde \mu, \tilde \nu).
	\end{align*}

	Finally, for Assertion $(iv)$ we first observe by the factorizing representation of $K$ due to its $r^*$-regularity for $x\in \RR^d$ that
	\begin{align*}
		K\ast\mu_\epsilon(x) - K\ast \nu_\epsilon(x) = \tilde K((\textup{id} - \sigma\sigma^\top)x)\left(\overline K\ast \left(\frac{r^*}{k}\sum_{i = 1}^{r^*}\delta_{\sigma^\top \theta_{0k_0} + \epsilon \theta_{i}} - \delta_{\sigma^\top \theta_{0k_0} + \epsilon \eta_{i}}\right) \right)(\sigma^\top x)
	\end{align*}
	Hence, upon defining $\mu^\sigma_{\epsilon}\coloneqq \sigma^{\top}_{\#}\mu_{\epsilon}$ and $\nu^\sigma_{\epsilon}\coloneqq \sigma^{\top}_{\#}\nu_{\epsilon}$, it follows from the tensorization property of the Hellinger distance that 
	\begin{align*}
			H^2(K\ast\mu_\epsilon, 	K\ast\nu_\epsilon) &= H^2(\overline K\ast \mu_\epsilon^\sigma, \overline K\ast \nu_\epsilon^\sigma)\\
			&=\frac{1}{2} \int \frac{\left(\overline K\ast \mu_\epsilon^\sigma(x) - \overline K\ast \nu_\epsilon^\sigma(x)\right)^2}{\textstyle \left(\sqrt{\overline K\ast \mu^\sigma_{\epsilon}(x)} + \sqrt{\overline K\ast \nu^\sigma_{\epsilon}(x)}\right)^2\displaystyle }\dif x\\
			&\leq \frac{1}{2} \int \frac{\left(\overline K\ast \mu_\epsilon^\sigma(x) - \overline K\ast \nu_\epsilon^\sigma(x)\right)^2}{\overline K\ast \mu^\sigma_{\epsilon}(x)  + \overline K\ast \nu^\sigma_{\epsilon}(x) }\dif x\\
			&\leq \frac{1}{2} \int\frac{ \left(\sum_{i = 1}^{r^*} \overline K(x-\epsilon\theta_i)- \overline K(x - \epsilon\eta_i) \right)^2}{\sum_{i = 1}^{r^*} \overline K(x-\epsilon\theta_i)+ \overline K(x - \epsilon\eta_i)  }\dif x
	\end{align*}
	To bound the right-hand side, we use the following $(r^*-1)$-th-order Taylor expansion of  $\sum_{i = 1}^{r^*} K(x-\epsilon\xi_i)$ with $(\xi_1, \dots, \xi_{r^*})\in \{(\theta_1, \dots, \theta_{r^*}), (\eta_1, \dots, \eta_{r^*})\}$ at $x = 0$, 
	\begin{align*}
		\sum_{i = 1}^{r^*} \overline K(x-\epsilon\xi_i)&= \sum_{i =1}^{r^*} \sum_{\alpha = 0}^{r^*-1} \frac{1}{\alpha!} D^{\alpha}\overline K(x)\epsilon\xi^{\alpha} + \frac{1}{(r^*-1)!}R_{r^*}(x, \epsilon \xi_i)(\xi_i)^{r^*}, 
	\end{align*}
	where $R_{r^*}(x, \epsilon \xi_i) = \int_0^1 (1-t)^{r^*-1}D^{r^*} K(x - t\epsilon \xi_i)\dif t$ denotes the rest term. Hence, applying this formula for the difference $\sum_{i = 1}^{r^*} \overline K(x-\epsilon\theta_i)- \overline K(x - \epsilon\eta_i)$ yields 
	\begin{align}
		\sum_{i = 1}^{r^*} \overline K(x-\epsilon\theta_i)- \overline K(x - \epsilon\eta_i) &=  \sum_{i = 1}^{r^*} \sum_{\alpha = 0}^{r^*-1} \frac{1}{\alpha!} D^{\alpha}\overline K(x)\left(\EE_{X\sim \tilde \mu}[(\epsilon X)^{\alpha}] - \EE_{Y\sim \tilde \nu}[(\epsilon Y)^{\alpha}]\right)\label{eq:DifferenceKernelConvZeroSumTerm}\\
		&\quad + \frac{1}{(r^*-1)!} \sum_{i = 1}^{r} R_{r^*}(x, \epsilon \theta_i)(\epsilon \theta_i)^{r^*} - R_{r^*}(x, \epsilon \eta_i)(\epsilon \eta_i)^{r^*}.\notag
	\end{align}
	Since the measures $\tilde \mu$ and $\tilde \nu$ exhibit matching moments up to order $r^*-1$, we have for all $\alpha \in \{0, \dots, r^*-1\}$ that 
	\begin{align*}
		\EE_{X\sim \tilde \mu}[(\epsilon X)^{\alpha}] = \epsilon^{\alpha}m_{\alpha}(\tilde \mu) = \epsilon^{\alpha}m_{\alpha}(\tilde \nu) =  \EE_{Y\sim \tilde \nu}[(\epsilon Y)^{\alpha}],
	\end{align*}
	and the right-hand side of \eqref{eq:DifferenceKernelConvZeroSumTerm} equals zero. We thus obtain, 
	\begin{align}
		H^2(K\ast\mu_\epsilon, 	K\ast\nu_\epsilon) &\leq \epsilon^{2r^*}C(\overline K, k)\max_{i = 1, \dots, r^*}\int \frac{R_{r^*}^2(x, \epsilon \theta_i) + R_{r^*}^2(x, \epsilon \eta_i)}{\sum_{i = 1}^{r^*} \overline K(x-\epsilon\theta_i)+ \overline K(x - \epsilon\eta_i)}\dif x \label{eq:hellyBound1}\\
		&\leq  \epsilon^{2r^*}C(\overline K, k)\max_{i = 1, \dots, r^*}\int \frac{R_{r^*}^2(x, \epsilon \theta_i)}{\sum_{i = 1}^{r^*} \overline K(x-\epsilon\theta_i)} +\frac{R_{r^*}^2(x, \epsilon \eta_i)}{\sum_{i = 1}^{r^*} \overline K(x - \epsilon\eta_i)}\dif x.\label{eq:hellyBound2}
	\end{align}

The asserted bound on the squared Hellinger distance now follows once we show that integral term on the right-hand side in the top line \eqref{eq:hellyBound1} or bottom line \eqref{eq:hellyBound2} stays bounded for $\epsilon \searrow 0$. To this end, consider $i \in \{1, \dots, r^*\}$ and note by the integral form of the remainder for the setting $r^* \in\{2, \dots, k\}$ that 
\begin{align*}
	\int \frac{R^2_{r^*}(x,\theta_0, \theta_i)}{\sum_{i = 1}^{r^*} \overline K(x-\epsilon\theta_i)} \dif x &= \int \frac{ \left(\int_{0}^{1}(1-t)^{r^*-1}D^{r^*} \overline K(x- t\epsilon\theta_i)\dif t\right)^2}{\sum_{i = 1}^{r^*} \overline K(x - \epsilon\theta_i)}\dif x\\
	&\leq C(k) \int \int_0^{1} \frac{\left(D^{r^*} \overline K(x - t\epsilon\theta_i)\right)^2}{\sum_{i = 1}^{r^*} \overline K(x - \epsilon\theta_i)}\dif t \dif x\\ 
	&= C(k)\int_0^{1} \int  \frac{\left(D^{r^*} \overline K(x - t\epsilon\theta_i)\right)^2}{\sum_{i = 1}^{r^*} \overline K(x - \epsilon\theta_i)}\dif x\dif t \\ 
	&\leq C(k) \sup_{t\in [0,1]}  \int  \frac{\left(D^{r^*} \overline K(x - t\epsilon\theta_i)\right)^2}{\sum_{i = 1}^{r^*} \overline K(x - \epsilon\theta_i)}\dif x.
\end{align*}
By $r^*$-regularity of the kernel $K$ and since $\tilde \mu$ was chosen to admit support point left and right of the origin, there exists some $\tau>0$ which depends on $K$ and $\theta_1, \dots, \theta_{r^*}$ such that the right-hand side in the above display is uniformly bounded for all $\epsilon\in (0,\tau)$. Likewise, since $\tilde \nu$ can also be chosen to have support points left and right of the origin, after possible decreasing $\tau>0$ it follows that the integral $\int R^2_{r^*}(x,\theta_0, \eta_i)/\sum_{i = 1}^{r^*} \overline K(x - \epsilon\eta_i) \dif x$ is uniformly bounded for all $\epsilon\in (0,\tau)$. This yields the assertion for $r^*\geq 2$. 

For the remaining case $r^* = 1$, we utilize \eqref{eq:hellyBound1} and obtain by a similar argument that 
\begin{align*}
	\int \frac{R^2_{r^*}(x,\theta_0, \theta_1)}{\sum_{i = 1}^{r^*} \overline K(x-\epsilon\theta_i)} \dif x &= \int \frac{ \left(\int_{0}^{1}(1-t)^{r^*-1}D^{r^*} \overline K(x- t\epsilon\theta_1)\dif t\right)^2}{\overline K(x - \epsilon\theta_1) + \overline K(x - \epsilon \eta_1)}\dif x\\
	&\leq C(k) \sup_{t\in [0,1]}  \int  \frac{\left(D^{r^*} \overline K(x - t\epsilon\theta_1)\right)^2}{\overline K(x - \epsilon\theta_1) + \overline K(x - \epsilon \eta_1)}\dif x.
\end{align*}
Now since $\tilde \mu = \delta_{\theta_1}$ with $\theta_1 = 1/4$ and $\tilde \nu = \delta_{-\eta_1}$ with $\eta_1 = -1/4$ it follows from $r^*$-regularity of the kernel $K$ that there exists some $\tau>0$ such that the right-hand side in the above display is uniformly bounded for all $\epsilon\in (0,\tau)$. Repeating the argument for the other remainder term in \eqref{eq:hellyBound1} concludes the proof of the Hellinger bound. In particular, the quantity $\tau$ serves as the constant $C(K,k)>0$ in the formulation of the statement.  \qed

	\section{Proofs for Section~\ref{sec:mle_bounded}}

	\subsection{Proof of Proposition \ref{prop:consistency_LSE_MLE_finite_m}}\label{app:subsec:consistency_MLE_finite_m}
	By definition of the maximum likelihood estimator it holds $\ell_t(\hat \mu_{t,m})\geq \ell_t(\mu)$. Hence, by our Poisson model \eqref{eq:model} it follows after dividing by $t>0$ that 
	\begin{align*}
		& - \sum_{j = 1}^{m} \int_{B_j}K\ast \hat \mu_{t, m}(x)  \dif x  + \frac{X_j}{t} \log\left( \int_{B_j}K\ast \hat \mu_{t, m}(x)  \dif x\right)\\*
		\geq & - \sum_{j = 1}^{m} \int_{B_j}K\ast \mu(x)  \dif x  + \frac{X_j}{t} \log\left( \int_{B_j}K\ast \mu(x)  \dif x\right).
	\end{align*}
	By the strong law of large numbers it follows for almost every realization $\omega$ of the underlying probability space $X_j(\omega)/t \to \int_{B_j} K\ast \mu(x)\dif x$ as $t\to \infty$. Further, by compactness of $\tilde \calP_k(\Theta)$ there exists an increasing sequence $\{t_i\}_{i \in \NN}$ with $\lim_{i \to \infty} t_i = \infty$ such that $\hat \mu_{t,m}(\omega)$ converges to some element $\overline \mu\in \tilde \calP_k(\Theta)$. This implies by the above display that 
	\begin{align*}
		& - \sum_{j = 1}^{m} \int_{B_j}K\ast \overline \mu(x)  \dif x  + \int_{B_j}K\ast  \mu(x)  \dif x\cdot  \log\left(\int_{B_j}K\ast \overline \mu(x)  \dif x\right)\\
		\geq & - \sum_{j = 1}^{m} \int_{B_j}K\ast \mu(x)  \dif x  + \int_{B_j}K\ast  \mu(x)  \dif x \cdot \log\left(\int_{B_j}K\ast \mu(x)  \dif x\right).
	\end{align*}
	Since $\int_{B_j}K\ast  \mu(x)\dif x >0$ for each $j \in \{1, \dots, m \}$ and since for $t>0$ the function $x\in (0, \infty)\mapsto -x + t\cdot \log(x)$ is uniquely maximized at $x = t$, we deduce by the above display for all $j\in \{1, \dots, m\}$ that 
	\begin{align*}
		\int_{B_j}K\ast \overline \mu(x)  \dif x = \int_{B_j}K\ast \mu(x)  \dif x.
	\end{align*}

	Setting  $\Psi(x)\coloneqq \int_{x}^{x+\delta}K(r)\dif r$ it holds for all $j\in \{1, \dots, m\}$ and $\xi \in \{\mu, \overline \mu\}$ that
	\begin{align*}
			\int_{B_j} \int K(x-y) \dif \xi(y)  \dif x= \int \int_{s_j-y}^{s_j-y+\delta}K(x) \dif x \dif \xi(y) = \int  \Psi(s_j-y)\dif \xi(y).
	\end{align*}
	Since the bin integral function $\Psi$ is $(2k,l)$-root-regular, we conclude from the previous two displays and since $m\geq 2k+l$ by our identifiability result (\Cref{thm:identifiabilityFromFunctionals}) that $\overline \mu = \mu$ and the assertion follows. \qed

	\subsection{Proof of Lemma \ref{lem:fullRankRRFunction}}
		We prove the claim by induction over $k'\in\{1, \dots, k\}$. For $k'=1$ it follows since $\Psi$ is $(k,l)$-root-regular that the function $\Psi(\cdot - x_1)$ can admit at most $k+l-1$ zeros. Hence, there exists some $t\in \{t_1, \dots, t_{k+l}\}$ such that $\Psi(t - x_1) \neq 0$ and consequently $\textup{Rank}(\Psi(t_1-x_1), \dots, \Psi(t_{1+l}-x_1))= 1$. Now suppose the claim holds for $k'-1\leq k-1$, then we aim to show the assertion for $k'$. So consider distinct points $x_1< \dots < x_{k'}\in \RR$ and $t_1< \dots < t_{k+l}\in\RR$. By induction assumption, we know for $x_1< \dots < x_{k-1}\in \RR$ and $t_1< \dots < t_{k+l}\in\RR$ that equation \eqref{eq:FullRankProperty} is met. Hence, there exists a subselection $t_1'< \dots < t_{k'-1}'$ such that 
		\begin{align}\label{eq:detNonzero}
			\det\left(\left(\Psi(t_j'-x_i)\right)_{i = 1, \dots, k'-1, j = 1, \dots, k'-1}\right)\neq 0. 
		\end{align}
		Now consider the function 
		\begin{align*}
			D\colon \RR\to \RR, \quad t \mapsto \det\left(\left(\Psi(t_j'-x_i)\right)_{i = 1, \dots, k', j = 1, \dots, k'-1} \;|\; (\Psi(t-x_i))_{i = 1, \dots, k'}\right)
		\end{align*}
		which by definition of the determinant can be represented as 
		\begin{align*}
			D(t) = \sum_{i =1}^{k'} a_i \Psi(t-x_i) \quad \text{ where } \quad a_i = \det\left(\left(\Psi(t_j'-x_{i'})\right)_{i'\in\{1, \dots, k'\}\backslash\{i\}, j = 1, \dots, k'-1}\right).
		\end{align*}
		In particular, since $a_{k'} \neq 0$ by \eqref{eq:detNonzero}, we infer by $(k,l)$-root-regularity of $\Psi$ that $D(\cdot)$ admits at most $k+l-1$ roots. Evidently, it holds $D(t_j') = 0$ for each $j \in\{1, \dots,k'-1\}$, implying that $D$ admits at most $k+l-1 - (k'-1)$ remaining zeros. Hence, there must exist some $t\in \{t_1, \dots, t_{k+l}\}\backslash\{t_1', \dots, t_{k-1}'\}$ such that $D(t)\neq 0$ and the assertion follows. \qed

		\subsection{Proof of Theorem \ref{thm:identifiabilityFromFunctionals}}\label{app:pf:thm:identifiabilityFromFunctionals}
			Denote by $S = \supp(\mu)\cup\supp(\nu)$ the union of support points of $\mu$ and $\nu$ and label $S = \{x_1, \dots, x_{p}\}$ for $p\in \{1, \dots, 2k\}$. For fixed $t\in \RR$ we will interpolate the indicator $\mathds{1}(\cdot \leq t)$ on $S$ in terms of $\Psi(t_1-\cdot), \dots, \Psi(t_{2k+l}-\cdot)$. Since $\Psi$ is $(2k,l)$-root-regular it follows by Lemma \ref{lem:fullRankRRFunction} that  
			there exists a selection of $p$ distinct indices $\{j_1, \dots, j_{p}\}\subseteq \{1, \dots, 2k+l\}$ such that 
			\begin{align*}
				\det\begin{pmatrix}
					\Psi(t_{j_1}-x_1) & \cdots & \Psi(t_{j_{p}}-x_1)\\
					\vdots & \ddots & \vdots \\
					\Psi(t_{j_1}-x_{p}) & \cdots & \Psi(t_{j_{p}}-x_{p})
				\end{pmatrix}\neq 0.
			\end{align*}
			Hence, there exist coefficients $a_1, \dots, a_{p}\in \RR$, dependent on $t$, such that  
				$$\left(\mathds{1}(x_1\leq t),
					\dots, 
					\mathds{1}(x_{s}\leq t) \right)^\top =  \sum_{s = 1}^{p} a_s\left(\Psi(t_{j_s}-x_1), \dots, \Psi(t_{j_s}-x_{i})\right)^\top.$$ 
			From Assumption \eqref{eq:idenfiableCondition} we therefore infer that 
			\begin{align*}
				F_\mu(t) = \int \mathds{1}(x\leq t)\dif\mu(x) %
				&=\sum_{s = 1}^{p} a_s \int \Psi( t_{j_s}-x)\dif\mu(x) \\
				&= \sum_{s = 1}^{p} a_s \int \Psi(t_{j_s}-y)\dif\nu(y) =  \int \mathds{1}(y\leq t)\dif\nu(y) = F_\nu(t).
			\end{align*}
			Equality of $\mu$ and $\nu$ now follows since $t\in \RR$ was arbitrary. \qed
	
	\subsection{Proof of Lemma \ref{lem:rootregViaDerivative}}
		To show the first assertion, assume there existed distinct points $x_1< \dots < x_{k}$ and $(a_1, \dots, a_{k})\in \RR^k\backslash\{0\}$ such that the function $\Lambda \colon \RR\to \RR, t \mapsto \sum_{i =1}^{k} a_i \Psi(t-x_i)$ admitted $k+l+2$ distinct roots $z_1 < \dots, z_{k+l+2}\in \RR$. Then, by Rolle's theorem there would exist distinct points $\xi_1< \dots <\xi_{k+l+1}\in \RR$ where for each $j \in \{1, \dots, k+l+1\}$ we have $\xi_{j} \in (z_j, z_{j+1})$ and $0=\Lambda'(\xi_j)  = \sum_{i = 1}^{k}a_i\Psi'(\xi_j-x_i)$. However, this would contradict the $(k,l)$-root-regularity of $\Psi'$, and we thus conclude that $\Psi$ is $(k,l+1)$-root regular. \qed %

	\subsection{Proof of Proposition \ref{prop:rootRegularityGaussian}}
		First note by the fundamental theorem of calculus that the function 
		$$\Psi(x) \coloneqq \int_{x}^{x+\delta}K(t)\dif t = \int_{x}^{x+\delta}(2\pi \sigma^2)^{-1/2}\exp(-t^2/(2\sigma^2))\dif t$$ is differentiable with  derivative
		\begin{align*}
			 \Psi'(x) =  (2\pi \sigma^2)^{-1/2}\left[\exp(-(x+\delta)^2/(2\sigma^2)) - \exp(-x^2/(2\sigma^2))\right].
		\end{align*}
		Once we show that $\Psi'$ is $(k,k)$-root-regular the assertion follows from Lemma \ref{lem:rootregViaDerivative}. To this end, consider distinct points $x_1< \dots < x_{k} \in \RR$ and coefficients $a_1, \dots, a_{k}\in \RR$ with $\sum_{i =1}^{k} a_i\neq 0$ and observe by a direct expansion of the terms that 
		\begin{align*}
			&\quad (2\pi \sigma^2)^{1/2}\sum_{i = 1}^{k} a_i \Psi'(t-x_i) \\
			&= \sum_{i = 1}^{k} a_i\left[\exp\left(-\frac{(t-x_i-\delta)^2}{2\sigma^2}\right) - \exp\left(-\frac{(t-x_i)^2}{2\sigma^2}\right)\right]\\
			&=  \exp\left(-\frac{t^2}{2\sigma^2}\right)\sum_{i = 1}^{k} a_i \left[\exp\left(-\frac{(x_i+\delta)^2}{2\sigma^2}\right)\exp\left(\frac{x_i +\delta}{\sigma^2}t\right) - \exp\left(-\frac{x_i^2}{2\sigma^2}\right)\exp\left(\frac{x_i +\delta}{\sigma^2}t\right)\right].
		\end{align*}  
		The factor $\exp(-t^2/(2\sigma^2))$ is strictly positive on $\RR$ whereas the sum is equal to an exponential polynomial of degree at most $2k$ in $t$, i.e., a function $t\mapsto \sum_{i = 1}^{2k} b_i \exp(c_i t)$ for appropriate $b_i, c_i \in \RR$, which are known to admit at most $2k-1$ roots (e.g., \citealt[p.\ 10]{karlin1966tchebycheff} or \citealt[p.\ 48, Aufgabe 75]{polya1925aufgaben}), and $(k,k)$-root-regularity of $\Psi'$ follows.  \qed

	\subsection{Proof of Lemma \ref{prop:rationalFunctionRootRegular}}
			Let $x_{1}< \dots< x_{k}\in \RR$ for $k\in \NN$ be  distinct points and consider $(a_1, \dots,a_k)\in \RR^k\backslash\{0\}$ with $\sum_{i = 1}^{k} a_i^2 \neq 0$. Further, assume without loss of generality that $P$ and $Q$ do not share any common roots in $\CC$, for otherwise they could be canceled out. Hence, since $\Psi$ is real-valued it follows that $Q$ does not have any roots in $\RR$. We thus infer for $t\in \RR$ that $\sum_{i = 1}^{k} a_i P(t-x_i)/Q(t-x_i)=0$ if and only if 
		\begin{align*}
			0 &= \prod_{j = 1}^{k}Q(t-x_j)\sum_{i = 1}^{k} a_i P(t-x_i)/Q(t-x_i)\\
			  &= \sum_{i = 1}^{k} a_i P(t-x_i) \prod_{\substack{j =1\\ j \neq i}}^{k} Q(t-x_j) \eqqcolon f(t),
		\end{align*}
		which is a polynomial of degree at most $p+ (k-1)q$ with respect to $t$. Hence, it follows that $f(t)$ is either equal to zero on $\RR$ or admits at most $k+p + (k-1)(q-1)-1$ roots. However, $f$ cannot be identically zero, since otherwise $\Psi(t-x_1), \dots, \Psi(t-x_n)$ would be linearly dependent on $\RR$, contradicting Lemma \ref{lem:shiftedAnalyticFunctionsLinearlyIndependent} below due to integrability of $\Psi$. \qed
	
		\begin{lemma}\label{lem:shiftedAnalyticFunctionsLinearlyIndependent}
			Let $\phi\colon \RR^d\to \RR$ be an integrable, non-constant function and define $\phi_t(\cdot)\coloneqq \phi(\cdot-t)$. Then, for any finite set $T\subseteq \RR^d$ then functions $\{\phi_t  \,\colon\, t\in T\}$ are linearly independent.  
		\end{lemma}

	\subsection{Proof of Lemma \ref{lem:shiftedAnalyticFunctionsLinearlyIndependent}}
	Let $\varphi$ be integrable such that there are coefficients $\{c_t\}_{t\in T}\subseteq \RR$ with at least one non-zero with $\sum_{t\in T} c_t \varphi(\cdot - t) = 0$ on $\RR^d$. Taking the Fourier transform yields for all $ \omega \in \RR^d$,
	\begin{align}\label{eq:identityFourierTransform}
		\sum_{t\in T} c_t \exp(-i\langle t, \omega\rangle)\hat \varphi(\omega) =0
	\end{align} 
	Note that $s(\omega) \coloneqq \sum_{t\in T} c_t \exp(-i\langle t, \omega\rangle)$ is an analytic function in $\omega\in \RR^d$ that is not constant, hence by \citet[Section 3.1.24, p.\ 240]{federer2014geometric} its set of zeros $s^{-1}(\{0\})$ must be a Lebesgue null-set. Since $\varphi$ is integrable, its Fourier transform $\hat \varphi$ is continuous and by \eqref{eq:identityFourierTransform} we conclude that $\hat \varphi = 0$ on $\RR^d$, which in turn implies $\varphi = 0$ on $\RR^d$. \qed

\section{Proofs for Section~\ref{sec:discussion}}
\subsection{Proof of Corollary~\ref{cor:matchings}}
\label{app:pf_cor_matchings}
To prove the minimax lower bound, notice that for any measure $\mu\in \calU_k(\Theta)$, 
it holds that $(\mu,\mu)\in \calT_\lambda(\Theta)$. Thus, letting $\pi_0$ denote the identity coupling, we have
\begin{align*}
\inf_{\hat\pi_t} \sup_{(\mu,\nu)\in \calT_\lambda(\Omega)}
\bbE_{\mu,\nu} W_2(\hat\pi_t,\pi_0) 
 &\geq  \inf_{\hat\pi_t} \sup_{\mu\in \calU_k(\Theta)} \bbE_{\mu} W_2(\hat\pi_{t},\pi_0)
 =  \inf_{\hat\pi_t} \sup_{\mu\in \calU_k(\Theta)} \bbE_{\mu} W_2(\hat\pi_{t1},\mu),
\end{align*}
where $\hat \pi_{t1}$ denotes the first marginal of the joint distribution $\hat \pi_t$. It follows that 
\begin{align*}
\inf_{\hat\pi_t} \sup_{(\mu,\nu)\in \calT_\lambda(\Omega)}
\bbE_{\mu,\nu} W_2(\hat\pi_t,\pi_0) 
 &\geq   
  \inf_{\hat\mu_t} \sup_{\mu\in \calU_k(\Theta)} \bbE_{\mu} W_2(\hat\mu_{t},\mu) \asymp t^{-1/2k},
\end{align*}
by Theorem~\ref{thm:global_minimax_risk}. 
Conversely, let $\hat\mu_t,\hat\nu_t$ be minimax-optimal estimators for $\mu$ and $\nu$ respectively,
and let $\hat\pi_t$ be any quadratic optimal transport coupling between $\hat\mu_t$ and $\hat\nu_t$. 
Then, by Lemma~\ref{lem:stab_w2}, there exists a positive constant $C = C(\lambda)>0$ such that
\begin{align*}
\sup_{(\mu,\nu)\in \calT_\lambda(\Theta)}\bbE_{\mu,\nu} W_2(\hat\pi_t,\pi_0)
 &\leq C \sup_{(\mu,\nu)\in \calT_\lambda(\Theta)}\bbE_{\mu,\nu} \Big[W_2(\hat\mu_t,\mu) + W_2(\hat\nu_t,\nu)\Big] \\ 
 &\lesssim \sup_{ \mu\in \calU_k(\Theta)}\bbE_{\mu,\nu} \Big[W_2(\hat\mu_t,\mu) + \sup_{\nu\in\calU_k(\Theta)}W_2(\hat\nu_t,\nu)\Big] 
 \lesssim t^{-1/2k},
\end{align*}
by Theorem~\ref{thm:global_minimax_risk}. The claim follows.\qed

	\hyphenation{Hausdorff}

	\section{Comparing the (local) Wasserstein and Hausdorff distance}\label{app:auxiliary}

	In this appendix, we provide some quantitative bounds between different Wasserstein distances, the local Wasserstein divergence as well as the Hausdorff distance for uniform distributions on general Euclidean spaces $\RR^d$. All proofs are deferred to individual subsections at the end of this appendix. Notably, all assertions carry over \emph{mutatis mutandis} to uniform distributions on general metric spaces. 

	We start with our equivalence statement of $p$-Wasserstein for varying $p$ for $k$-atomic uniform measures. 
	
	\begin{lemma}[Equivalence of Wasserstein Distances]\label{lem:inequalityWasserstein1Infty}
		For arbitrary $\mu, \nu \in \calU_k(\RR^d)$ and $p,q\in [1,\infty]$ it holds 
		\begin{align*}
			W_p(\mu, \nu) \leq k W_q(\mu, \nu). 
		\end{align*}
		\end{lemma}

	\begin{remark}[On the necessity of uniformity]
	The above inequality crucially relies on the uniformity assumption on the weights of the measures. Indeed, it is easy to see for non-uniform measures $\tilde \mu_{\epsilon} = \frac{1 + \epsilon}{2}\delta_{x} + \frac{1 -\epsilon}{2}\delta_{y}$ with $\epsilon \in (0,1/2)$ and $x\neq y$ that $W_{p}(\mu_0, \mu_{\epsilon}) = \epsilon^{1/p}\|x-y\|$ for $p\in [1,\infty]$ and where $1/\infty \coloneqq 0$. Hence, whenever $p>q$ it follows that $\frac{W_p(\mu_0, \mu_{\epsilon})}{W_q(\mu_0, \mu_{\epsilon})} = \epsilon^{(1/p - 1/q)}\to \infty$ for $\epsilon\to 0$.
	\end{remark}
	
We now state a bound which formalizes the (local) discriminative power of the local Wasserstein divergence compared to vanilla Wasserstein distance. 

\begin{lemma}[Relation between Wasserstein distance and local Wasserstein divergence]\label{lem:inequalityWassersteinLocal}
			Let $1\leq k_0 \leq k$, $r \in \NN^{k_0}$ with $|r|= k$, and $\delta\in (0,1)$. Further, consider $\mu_0 = \frac{1}{k}\sum_{i = 1}^{k_0} r_j\delta_{\theta_{0j}}\in \calU_{k,k_0}(\RR^d;r, \delta)$ and let $\mu, \nu \in \calU_k(\RR^d;\mu_0, \delta/4)$ be two $k$-uniform measures. Then, upon denoting by $\{V_j\}_{j = 1, \dots, k_0}$ a Voronoi partition generated on $\RR^d$ from $\supp(\mu_0)$ it follows for $r^* \coloneqq \max_{j = 1, \dots, k_0}r_j$, $r_*\coloneqq \min_{j = 1, \dots, k_0}r_j$,   %
			and $\Delta(\mu_0)\coloneqq \diam(\supp(\mu_0))\vee 1$, %
			\begin{align}\label{eq:inequalityWassersteinLocal1}
				W_1^{r^*}(\mu, \nu) &\lesssim_{k} \sum_{i =1}^{k_0} W_1^{r_j}(\mu_{V_j}, \nu_{V_j})\lesssim_{k} W_1(\mu, \nu), \quad \text{ and }\\
				\Delta(\mu_0)^{-k+r_*}\calD_{\mu_0}(\mu, \nu) &\lesssim_{k}\sum_{i =1}^{k_0} W_1^{r_j}(\mu_{V_j}, \nu_{V_j})\lesssim_{k} \delta^{-k+r_*} \calD_{\mu_0}(\mu, \nu)\label{eq:inequalityWassersteinLocal2}
			\end{align}
			where $\calD_{\mu_0}$ denotes the local Wasserstein distance from \eqref{eq:local_Wasserstein}. 
			In particular, it holds $W_1^{r^*}(\mu, \nu) \lesssim_k \delta^{-k + r_*}\calD_{\mu_0}(\mu, \nu)$. Furthermore, $\calD_{\mu_0}^{1/r^*}(\mu, \nu)$ defines a metric on $\calU_k(\RR^d; \mu_0, \delta/4)$.
		\end{lemma}

Lemma \ref{lem:inequalityWassersteinLocal} provides various insights into the relationship between the Wasserstein distance and the local Wasserstein divergence. First, for fixed $\delta$ and bounded $\Delta(\mu_0)$ we see that the local Wasserstein divergence $\calD_{\mu_0}(\mu, \nu)$ is for $\mu, \nu$ nearby $\mu_0$ equivalent to a sum of Wasserstein distances $W_1^{r_j}(\mu_{V_j}, \nu_{V_j})$. In particular, for $k_0 = 1$ we note that the local Wasserstein divergence $\calD_{\mu_0}(\mu, \nu)$ is equivalent to $W_1^k(\mu, \nu)$, whereas for $k_0 = k$, since $r=(1, \dots, 1)\in \NN^{k_0}$ is the only feasible choice which fulfills $|r| = k$, it is equivalent to $W_1(\mu, \nu)$. %

		Finally, we show that Hausdorff distance between the supports of $k$-uniform measures is dominated by the Wasserstein distance, see also Figure~\ref{fig:differenceWassersteinHausdorff}.

		\begin{lemma}[Relation between Hausdorff and Wasserstein Distance]\label{lem:RelationHausdorffAndWasserstein}
			Let $\mu, \nu \in \calU_k(\RR^d)$. Then, it follows for the Hausdorff distance between $\mu$ and $\nu$, defined by $d_H(\mu, \nu)\coloneqq d_H(\supp(\mu), \supp(\nu))$, 
			\begin{align*}
				d_H(\mu, \nu) \leq W_\infty(\mu, \nu)\leq k \cdot W_1(\mu, \nu). 
			\end{align*}
			Further, for $1\leq k_0\leq k$ and a measure $\mu_0= \frac{1}{k}\sum_{j = 1}^{k_0}r_j \delta_{\theta_j}\in \calU_{k,k_0}(\RR^d;r,  \delta)$ with $\delta>0$ and upon denoting $r \coloneqq \max_{j =1, \dots, k_0}r_j$ it follows for all $\mu, \nu \in \calU_k(\RR^d; \mu_0, \delta/4)$ that 
			\begin{align*}
				d_{\mu_0}(\mu, \nu) \leq r^r \cdot \calD_{\mu_0}(\mu, \nu),
			\end{align*}
			where $d_{\mu_0}$ is the local Hausdorff-type distance defined in \eqref{eq:local_hausdorff}  and $\calD_{\mu_0}$ denotes the local Wasserstein divergence from \eqref{eq:local_Wasserstein}
		\end{lemma}

		\begin{remark}
			The previous Lemma asserts that the Hausdorff distance between $k$-uniform measures is weaker than the Wasserstein distance. In fact, for $k\geq 3$ it turns out that the Hausdorff distance is strictly weaker than the Wasserstein distance. To see this, consider two $k$-atomic uniform measures $\mu = \frac{r}{k}\delta_{\theta_1} + \frac{k-r}{k}\delta_{\theta_2}$ and $\nu = \frac{s}{k}\delta_{\theta_1} + \frac{k-s}{k}\delta_{\theta_2}$ where $r\neq s \in \{1, \dots, k-1\}$ and $\theta_1 \neq \theta_2$. Then, it holds 
			\begin{align*}
				0=d_H(\mu, \nu)< W_\infty(\mu, \nu), 
			\end{align*}
			and the same inequality also holds for local versions of the Hausdorff and the Wasserstein distance. 
		\end{remark}
		\begin{figure}[t!]
			\centering
			\includegraphics[width = 0.5\textwidth]{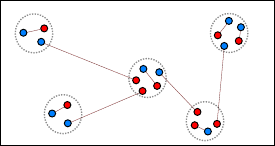}
			\caption{Difference between Hausdorff and Wasserstein distance between $k$-atomic uniforms $\mu$ (blue dots) and $\nu$ (red dots). The Hausdorff distance $d_{H}(\mu, \nu)$ is on the scale of the radius of the dotted circles. The  Wasserstein distance involves a matching between the atoms of $\mu$ and $\nu$ and is on the scale of the longest dashed line, thus larger than $d_H(\mu, \nu)$. }
			\label{fig:differenceWassersteinHausdorff}
		  \end{figure}

		\subsection{Proof of Lemma \ref{lem:inequalityWasserstein1Infty}}
According to \citet[Remark 6.6]{villani2009} it holds by H\"older's inequality for general $s,t \in [1, \infty]$ with $s \leq t$ that $W_s(\mu, \nu) \leq W_t(\mu, \nu).$ Hence, once we show that 
			\begin{align}\label{eq:WassersteinInftyAndOne}
				W_\infty(\mu, \nu) \leq k W_1(\mu, \nu),
			\end{align}
			the assertion follows from the previous two displays by 
			\begin{align*}
				W_p(\mu, \nu) \leq W_\infty(\mu, \nu) \leq k W_1(\mu, \nu) \leq k W_q(\mu, \nu). 
			\end{align*}
			Hence, to confirm \eqref{eq:WassersteinInftyAndOne} we recall that $W_p(\mu, \nu)= \inf_{\pi\in\Pi(\mu, \nu)}\| d\|_{L^p(\pi)}$ for $p \in [1,\infty]$ where $d$ denotes the Euclidean distance. In particular, upon denoting $\mu = \frac{1}{k} \sum_{i = 1}^{k} \delta_{\theta_i}$ and $\nu = \frac{1}{k} \sum_{i = 1}^{k} \delta_{\eta_i}$, it follows by Birkhoff's theorem that there exists an optimal transport plan $\pi_1$ which is a matching, i.e., $\pi_1 =\frac{1}{k} \sum_{i = 1}^{k}\delta_{(\theta_i, \eta_{\sigma i})}$ for some permutation $\sigma \in \calS(k)$. It thus follows that 
		\begin{align*}
			W_1(\mu, \nu)= \sum_{(x,y)\in \textup{supp}(\pi)}\|x- y\| \pi_1(\{(x,y)\})
			&\geq \frac{1}{k} \max_{(x,y) \in\textup{supp}(\pi_1)}\|x- y\| \geq \frac{1}{k} W_\infty(\mu, \nu).\qedhere
		\end{align*}

		\subsection{Proof of Lemma \ref{lem:inequalityWassersteinLocal}}
			
		To show the first inequality in \eqref{eq:inequalityWassersteinLocal1}, first note from $\mu, \nu \in \calU_k(\RR^d; \mu_0, 1\wedge(\delta/4))$ that $W_\infty(\mu, \mu_0)\vee W_{\infty}(\nu, \mu_0)\leq \delta$ and thus $\mu(V_j) = \nu(V_j) = r_j/k$ for each Voronoi cell $V_j$ of $\mu_0$. 
			Next, consider a $W_1$-coupling $\pi_j$ between $\mu|_{V_j}$ and $\nu|_{V_j}$ and note that $\pi = \sum_{j = 1}^{k_0} \pi_j$ is a coupling between $\mu$ and $\nu$. Hence, it follows that 
			\begin{align*}
				W_1(\mu, \nu) &\leq \int_{\RR^{2d}} \|x - y\|\dif\pi(x,y)= \sum_{j = 1}^{k_0} \int_{\RR^{2d}} \|x - y\|\dif\pi_j(x,y)= \sum_{j = 1}^{k_0} W_1(\mu|_{V_j}, \nu|_{V_j})
			\end{align*}
			Since $W_1(\mu|_{V_j}, \nu|_{V_j}) = \frac{r_j}{k}W_1(\mu_{V_j}, \nu_{V_j})$ for each $j = 1, \dots, k_0$, we obtain by H\"older's inequality, 
			\begin{align*}
				W_1(\mu, \nu) %
				&\leq \left(\sum_{j = 1}^{k_0} \left(\frac{r_j}{k}\right)^{\frac{r^*}{r^*-1}}\right)^{(r^*-1)/r^*}\left(\sum_{j = 1}^{k_0}  W_1^{r^*}(\mu_{V_j}, \nu_{V_j})\right)^{1/{r^*}}\\
				&\leq k \left(\sum_{j = 1}^{k_0}  W_1^{r^*}(\mu_{V_j}, \nu_{V_j})\right)^{1/r^*}.
			\end{align*}
			By triangle inequality and since $\delta< 1$ we know that 
			\begin{align}\label{eq:wasserteinBound}
			\begin{aligned}
				W_1(\mu_{V_j}, \nu_{V_j}) &\leq W_1(\mu_{V_j},(\mu_0)_{V_j})+ W_1((\mu_0)_{V_j},\nu_{V_j}) \\ &\leq W_\infty(\mu_{V_j},(\mu_0)_{V_j})+ W_\infty((\mu_0)_{V_j},\nu_{V_j})\leq 2\wedge (\delta/2)\leq 1,
			\end{aligned}
			\end{align}
			which in combination with the penultimate display yields 
			\begin{align*}
					W_1^{r^*}(\mu, \nu) \leq k^{r^*}\sum_{j = 1}^{k_0}  W_1^{r^*}(\mu_{V_j}, \nu_{V_j})\leq k^{r^*} \sum_{j = 1}^{k_0}  W_1^{r_j}(\mu_{V_j}, \nu_{V_j})\leq k^{k} \sum_{j = 1}^{k_0}  W_1^{r_j}(\mu_{V_j}, \nu_{V_j}).%
			\end{align*}
			
			For the second inequality in \eqref{eq:inequalityWassersteinLocal1} first note by \eqref{eq:wasserteinBound} that 
			\begin{align*}
				\sum_{j = 1}^{k_0}  W_1^{r_j}(\mu_{V_j}, \nu_{V_j})\leq \sum_{j = 1}^{k_0} W_1(\mu_{V_j}, \nu_{V_j}). 
			\end{align*}
			Recall that $\mu(V_j) = \nu(V_j) = r_j/k\geq 1/k$ for each Voronoi cell $V_j$. Further, it holds for any pair $\theta \in \supp(\mu|_{V_j})$,  $\eta \in \supp(\nu|_{V_i})$ with $i\neq j$ that 
			\begin{align*}
				\|\theta - \eta\|\geq \|\theta_{0j} - \theta_{0i}\| - \|\theta - \theta_{0j}\| - \|\eta - \theta_{0i}\| > \delta/2,
			\end{align*}
			whereas for $\eta \in \supp(\nu|V_{j})$ it holds 
			\begin{align*}
				\|\theta - \eta\|\leq \|\theta - \theta_{0j}\| + \|\eta - \theta_{0j}\| \leq \delta/2. 
			\end{align*}
			Hence, any optimal coupling $\pi$ between $\mu$ and $\nu$ with respect to $W_1$ must satisfy $\pi(V_i \times V_j) = 0$ for all $i \neq j$, and $\pi$ admits the representation $\pi = \sum_{j = 1}^{k_0} \pi_j$ for $\pi_j\in \Pi(\mu|_{V_j}, \nu|_{V_j})$. We thus infer,
			\begin{align*}
				\sum_{j = 1}^{k_0}  W_1(\mu_{V_j}, \nu_{V_j}) &\leq  k\sum_{j = 1}^{k_0}  W_1(\mu|_{V_j}, \nu|_{V_j})\\
				&\leq k\sum_{j = 1}^{k_0}  \int_{\RR^{2d}} \|x-y\|\dif\pi_j(x,y)\\
				&= k \int_{\RR^{2d}} \|x-y\|\dif\pi(x,y) = k W_1(\mu, \nu),
			\end{align*}
			which confirms the second inequality in \eqref{eq:inequalityWassersteinLocal1}.

			To show \eqref{eq:inequalityWassersteinLocal2} recall that $\delta_j = \prod_{i = 1, \dots, k, i \neq j} \|\theta_{i0} - \theta_{j0}\|^{r_i}$, and observe that $\delta{k-r_*}\leq \min_{j = 1, \dots, k_0}\delta_j$ and $\max_{j = 1, \dots, k_0} \delta_j \leq \Delta(\mu_0)^{k-r_*}$. 
			The first inequality in \eqref{eq:inequalityWassersteinLocal2} then follows from 
			\begin{align*}
				\calD_{\mu_0}(\mu, \nu) &\leq  \sum_{j = 1}^{k_0} \delta_j W_1^{r_j}(\mu_{V_j}, \nu_{V_j})\leq \Delta(\mu_0)^{k-r_*} \sum_{j = 1}^{k_0} W_1^{r_j}(\mu_{V_j}, \nu_{V_j}).
			\end{align*}
			The second inequality in \eqref{eq:inequalityWassersteinLocal2} follows by combining \eqref{eq:wasserteinBound} with our lower bound on $\delta_j$,
			\begin{align*}
				\sum_{j = 1}^{k_0}  W_1^{r_j}(\mu_{V_j}, \nu_{V_j}) &\leq k \left(1 \wedge \sum_{j = 1}^{k_0}  W_1^{r_j}(\mu_{V_j}, \nu_{V_j})\right)\\
				&\leq k\left(1 \wedge \left(\delta^{-k+r_*}  \sum_{j = 1}^{k_0} \delta_j W_1^{r_j}(\mu_{V_j}, \nu_{V_j})\right)\right) \leq \delta^{-k+r_*} k  D_{\mu_0}(\mu, \nu).
			\end{align*}
			Combining the previous two displays yields the desired bound.

				To confirm that $\calD_{\mu_0}^{1/r^*}(\mu, \nu)$ is a metric on $\calU_k(\RR^d; \mu_0, \delta/4)$, we first note that $\calD_{\mu_0}$ is always non-negative. 
				Further, if $\mu = \nu$, then $\mu_{V_j} = \nu_{V_j}$ for all $j = 1, \dots, k_0$ and hence $\calD_{\mu_0}^{1/r^*}(\mu, \nu) = 0$. Conversely, if $\calD_{\mu_0}^{1/r^*}(\mu, \nu) = 0$, then $W_1(\mu_{V_j}, \nu_{V_j}) = 0$ for all $j = 1, \dots, k_0$, which implies $\mu_{V_j} = \nu_{V_j}$ for all $j = 1, \dots, k_0$. Since $\mu(V_j)= \nu(V_j) = \mu_0(V_j)= r_j/k$, it follows that $\mu = \nu$. Symmetry of $\calD_{\mu_0}^{1/r^*}(\mu, \nu)$ follows from symmetry of the Wasserstein distance. 
				Finally, for triangle inequality note for $\mu, \nu, \eta \in \calP(\RR^d)$ since $r_j/r^*\leq 1$ that 
			\begin{align*}
				\delta_j^{1/r^*}W_1^{r_j/r^*}(\mu_{V_j}, \nu_{V_j}) \leq \delta_j^{1/r^*}W_1^{r_j/r^*}(\mu_{V_j}, \xi_{V_j}) + \delta_j^{1/r^*}W_1^{r_j/r^*}(\xi_{V_j}, \nu_{V_j}).
			\end{align*}
			Moreover, by Minkowski's inequality $\|x+y\|_{r^*}\leq \|x\|_{r^*} + \|y\|_{r^*}$ it follows that 
			\begin{align*}
				\left(\sum_{j = 1}^{k_0}\delta_j W_1^{r_j}(\mu_{V_j}, \nu_{V_j})\right)^{1/r^*}
				&\leq \left(\sum_{j = 1}^{k_0} \delta_j W_1^{r_j}(\mu_{V_j}, \xi_{V_j})\right)^{1/r^*} + \left(\sum_{j = 1}^{k_0} \delta_jW_1^{r_j}(\xi_{V_j}, \nu_{V_j})\right)^{1/r^*}.			
			\end{align*}
			Concavity of the function $t\mapsto t\wedge 1$ implies that $D_{\mu_0}^{1/r_*}(\mu, \nu)$ fulfills the triangle inequality. \qed
	
		\subsection{Proof of Lemma \ref{lem:hausdorff_to_wasserstein}}
			Let $\pi\in \Pi(\mu, \nu)$ be an optimal coupling such that $\|d\|_{L^\infty(\pi)} = W_\infty(\mu, \nu)$. Since $\mu = (1/k)\sum_{i = 1}^{k}\delta_{\theta_i}$ and $\nu = (1/k)\sum_{i = 1}^{k}\delta_{\eta_i}$ are uniformly distributed on their atoms, we can assume w.l.o.g. that $\pi = \sum_{i = 1}^{n} \delta_{(\theta_i, \eta_i)}$. 
			Further, note that
		\begin{align*}
			\max_{\theta \in \supp(\mu)} \min_{\eta\in \supp(\nu)} \|\theta - \eta\| &=\min_{\sigma\colon \supp(\mu) \to \supp(\nu)} \max_{\theta\in \supp(\mu)}\|\theta - \sigma(\theta)\|\\
			&\leq \max_{i  = 1, \dots, k} \|\theta_i - \eta_i\| = W_\infty(\mu, \nu).
		\end{align*} Exchanging $\mu$ and $\nu$ yields the first inequality, while the second directly inequality follows from Lemma \ref{lem:inequalityWasserstein1Infty}. 
		This proves the first Assertion.

		For the second Assertion note by assumption on $\mu_0, \mu, \nu$ that $\mu$ and $\nu$ assign mass $r_j/k$ with  $r_j = \mu_0(V_j)$ to each Voronoi cell $V_j$ of $\mu_0$. Uniformity of $\mu$ and $\nu$ therefore implies that $\mu_{V_j} = \mu(\cdot \cap V_j)/\mu(V_j)$ and $\nu_{V_j} = \nu(\cdot \cap V_j)/\nu(V_j)$ are contained in $\calU_{r_j}(V_j)$. We thus conclude from the first Assertion that 
		\begin{align*}
			d_{\mu_0}(\mu,\nu) &=  1\wedge\max_{1 \leq j \leq k_0}
		 \delta_j\cdot d_{H}^{r_j}(\mu_{V_j}, \nu_{V_j})\\
		  &\leq 1\wedge \sum_{j= 1}^{k_0} r_j^{r_j} \delta_j\cdot W_1^{r_j}(\mu_{V_j}, \nu_{V_j})
		  \leq r^r \calD_{\mu_0}(\mu_{V_j}, \nu_{V_j}),
		\end{align*}
		which yields the claim after taking the $r$-th root. \qed
	
		\section{Inequalities between Moment Differences}\label{app:momentDifferenceInequalities}

		In this appendix we derive some useful inequalities between moment differences of random variables. Proofs are provided in the following subsections. 
		
		We begin with a bound that relates the complex moment difference of measures on $\CC$ to the moment difference on $\RR^2$.

		\begin{lemma}\label{lem:complex_real_moments}
			Let $\mu, \nu \in \calP(\RR^2)$ be probability measures such that all moments up to order $\ell\in \NN$ exist. Denote by $\tilde \mu = T_{\#}\mu$ and $\tilde \nu =T_{\#}\nu\in \calP(\CC)$ the corresponding complex measures, where $T\colon \RR^2 \to \CC, (x,y) \to x + iy$ the represents the isomorphism from $\RR^2$ to $\CC$. Then, it follows that 
			\begin{align*}
				|m_\ell(\tilde \mu) - m_\ell(\tilde \nu)|\leq \sum_{\substack{\alpha\in \NN_0^2\backslash\{0\}\\ \|\alpha\|_1 = \ell}}\binom{\ell}{\alpha} |m_\alpha(\mu) - m_\alpha(\nu)|.
			\end{align*}
		\end{lemma}
		
		We additionally provide a bound that relates the moment difference for measures projected onto a one-dimensional subspace with the moment difference of the original measures. 
		
		\begin{lemma}\label{lem:sliced_multivariate_moment_bound}
			Let $\mu, \nu\in \calP(\RR^d)$ be two probability measures such that all moments up to order $\ell\in \NN$ exist. Then, it follows that 
			\begin{align*}
				\sup_{\eta \in \SS^d} |m_{\ell}(\mu^\eta) - m_\ell(\nu^\eta)|\leq d^{\ell/2} \sum_{\substack{\alpha \in \NN_0^d\\ \|\alpha\|_1 =\ell}}\left|m_\alpha(\mu) - m_\alpha(\nu)\right|. 
			\end{align*}
		\end{lemma}
		
		The above inequality only depends on the underlying dimension. For the setting where the probability measures are finitely support with at most $k$ support points, we can also establish a  converse inequality. 
		
		\begin{lemma}\label{lem:moment_sliced_moment}
			Let $k\in \NN$. Then,  for $\mu, \nu \in \calP_k(\RR^d)$ it follows for $\ell \in \NN$ that 
		\begin{align*}
			\sum_{\substack{\alpha \in \NN_0^d\\ \|\alpha\|_1 =\ell}}\left|m_\alpha(\mu) - m_\alpha(\nu)\right| \leq k^{(\ell-1)/2} \sup_{\eta \in \SS^d} |m_{\ell}(\mu^\eta) - m_\ell(\nu^\eta)|.%
		\end{align*}
		\end{lemma}
		
		Finally, the following Lemma shows that the metrics $M_k$ and $M_{2k-1}$ are
		equivalent over the space of uniform measures.
		\begin{lemma}
		\label{lem:Mk_M2k-1_equiv}
		Let $\domain \subseteq \RR^d$ be compact set and let $k\in \NN$. Then, there exists a positive constant $C = C(\Theta, d, k) > 0$ such that for all $\mu,\nu \in \calU_k(\domain)$ it holds 
		$$M_k(\mu,\nu)\leq M_{2k-1}(\mu,\nu) \leq C M_k(\mu,\nu).$$
		\end{lemma}

		\subsection{Proof of Lemma \ref{lem:complex_real_moments}}
		We first express the complex polynomial $z\mapsto z^\ell$ in terms of its real and imaginary part, where we abbreviate $x \coloneqq \Re(z)$ and $y\coloneqq  \Im(z)$, using a binomial expansion 
		\begin{align*}
			z^\ell = & \sum_{h = 0}^{\ell} \binom{\ell}{h} x^{\ell-h} (iy)^{h} = \sum_{\substack{h = 0\\h \text{ even}}}\binom{\ell}{h} (-1)^{h/2} x^{\ell-h} y^{h} + i \sum_{\substack{h = 0\\h \text{ odd}}}\binom{\ell}{h} (-1)^{(h-1)/2} x^{\ell-h} y^{h}.
		\end{align*}
		This immediately yields by triangle inequality and integrability that 
		\begin{align*}
			|m_\ell(\tilde \mu) - m_\ell(\tilde \nu)| &= \left| \int_{\CC} z^\ell \dif(\tilde \mu - (\tilde \nu)(z)\right|\\
			&\leq \sum_{\substack{h = 0}}^{\ell} \binom{\ell}{h} \left| \int_{\RR^d} x^{\ell-h}y^h \dif( \mu - \nu)(x,y) \right|\\
			&= \sum_{\substack{\alpha \in\NN_0^2, \|\alpha\|_1= \ell}}\binom{\ell}{\alpha}\left|m_\alpha(\mu) - m_\alpha(\nu)\right|,
		\end{align*}
		which proves the assertion. \qed

		\subsection{Proof of Lemma \ref{lem:sliced_multivariate_moment_bound}}
			For $X\sim \mu$, $Y\sim \nu$ and $\eta \in \SS^{d-1}$ it holds $\langle \eta, X\rangle \sim \mu^\eta$ and $\langle \eta, Y \rangle \sim \nu^\eta$. 
			A straight-forward computation now yields 
		\begin{align*}
			\left|m_\ell(\mu^\eta) - m_\ell(\nu^\eta)\right|
			 &= \left| \EE\langle \eta, X\rangle^\ell - \EE\langle \eta, Y\rangle^\ell \right|\\
			&\leq   \sum_{\substack{\alpha \in \NN_0^d, \|\alpha\|_1=\ell}}\binom{\ell}{\alpha}|\eta^\alpha|\left| \EE[X^\alpha] - \EE[Y^\alpha]\right|\\
			&\leq \max_{\substack{\alpha \in \NN_0^d, \|\alpha\|_1= \ell}}\left| \EE[X^\alpha] - \EE[Y^\alpha]\right|\sum_{\substack{\beta\in \NN_0^d, \|\beta\|_1 = \ell}}\binom{\ell}{\beta}|\eta^\beta| \\
			&= \max_{\substack{\alpha \in \NN_0^d, \|\alpha\|_1=\ell}}\left| \EE[X^\alpha] - \EE[Y^\alpha]\right| \|\eta\|_1^\ell \\
			&\leq d^{\ell/2}\max_{\substack{\alpha \in \NN_0^d,\|\alpha\|_1= \ell}}\left| \EE[X^\alpha] - \EE[Y^\alpha]\right| \\
			&\leq d^{\ell/2} \sum_{\substack{\alpha\in \NN_0^d, \|\alpha\|_1 = \ell}}\left|m_\alpha(\mu) - m_\alpha(\nu)\right|,
		\end{align*}
		where the last inequality follows from $\|\eta\|_1 \leq \sqrt{d}\|\eta\|_2= \sqrt{d}$.\qed

		\subsection{Proof of Lemma \ref{lem:moment_sliced_moment}}
			The assertion utilizes the moment tensor construction by \citet[Section~4.1]{doss2023}. 
			For the sake of completeness we provide  in the following the full argument, which utilizes some basic theory on tensors, see  \cite{kolda2009tensor}. 
			
			Specifically, for an order-$\ell$ tensor $T\in (\RR^d)^{\otimes \ell}$ we denote its Frobenius norm as $\|T\|_F \coloneqq \sqrt{\langle T, T\rangle}$, where the inner product is defined as $\langle S,T\rangle = \sum_{j \in \{1, \dots, d\}^\ell} S_{j}T_j$. Further, $T$ is symmetric if $T_{j_1, \dots, j_\ell} = T_{j_{\sigma1}, \dots, j_{\sigma \ell}}$ for every permutation $\sigma\in \calS(\ell)$ and in this case the operator norm of $T$ if defined as $\|T\|_{op} = \sup_{u\in \SS^{d-1}} \langle T, u^{\otimes \ell}\rangle$. Further, if $T$ admits the representation $T = \sum_{i = 1}^{r} \alpha_i \theta_i^{\otimes \ell}$ for $(\alpha_i)_{i = 1}^{r} \in \RR^r$ and distinct $\theta_1, \dots, \theta_r \in \RR^d$ the tensor $T$ is called to be of rank $r$, and it holds by \cite{qi2011best}
			\begin{align}\label{eq:operator_frobenius_norm_equivalence}
				r^{-(\ell-1)/2} \|T\|_F \leq \|T\|_{op}.%
			\end{align}
			
			Now, define for a $d$-dimensional random variable $X\sim \xi$ for $\xi \in \{\mu, \nu\}$ the order-$\ell$ moment tensor as 
			\begin{align}
				\bM_\ell(\xi)\coloneqq \EE_{X\sim \xi}[\,\underbrace{X\otimes \dots \otimes X}_{\ell \text{ times}}\,]\in (\RR^d)^{\otimes \ell},\label{eq:MomentTensor1}
			\end{align}
			provided all moments up to order $\ell$ exist. 
			In particular, for $v= (v_1, \dots, v_\ell)\in \{1, \dots, d\}^\ell$ it holds $(\bM_\ell(\xi))_v = \EE[\prod_{i = 1}^{\ell} X_{v_i}]$, which shows that the tensor $\bM_\ell(\xi)$ contains every moment of order $\ell$ of $X$, and that $\bM_\ell(\xi)$ is symmetric. Further, since $\xi\in \calP_k(\RR^d)$ by assumption it follows that $\bM_\ell(\xi)$ is of rank $k$ \citep[Equation 4.10]{doss2023}. Hence, it follows by \eqref{eq:operator_frobenius_norm_equivalence} that 
			\begin{align}\label{eq:MomentTensor2}
				k^{-(\ell-1)/2} \|\bM_\ell(\mu) - \bM_\ell(\nu)\|_{F}\leq \|\bM_\ell(\mu) - \bM_\ell(\nu)\|_{op}. %
			\end{align}
			Based on \citet[Equations (4.5) and (4.9)]{doss2023} it holds 
			\begin{align}\label{eq:MomentTensor3}
				\|\bM_\ell(\mu) - \bM_\ell(\nu)\|_{op}  = \sup_{\eta \in \SS^{d-1}} \left|\EE_{X\sim \mu}[ \langle X, \eta \rangle^\ell] - \EE_{Y\sim \mu}[ \langle Y, \eta \rangle^\ell]\right|= \sup_{\eta \in \SS^{d-1}} \left|m_\ell(\mu^\eta) - m_\ell(\nu^\eta)\right|. 
			\end{align}
			Further, it holds 
			\begin{align*}
				\|\bM_\ell(\mu) - \bM_\ell(\nu)\|_{F}^2  = \sum_{j \in \{1, \dots, d\}^\ell} \left(\bM_\ell(\mu)_j - \bM_\ell(\nu)_j\right)^2\geq \sum_{\substack{\alpha \in \NN_0^d\\ \|\alpha\|_1 = \ell}}\left|m_\alpha(\mu) - m_\alpha(\nu)\right|. %
			\end{align*}
			Combining the previous three displays yields the claim. \qed
		
		\subsection{Proof of Lemma~\ref{lem:Mk_M2k-1_equiv}}
		Let us begin by proving the claim when $\domain\subseteq \bbR$. 
		In this case, it suffices to prove that there exists a constant $C > 0$ such that
		\begin{equation}
		\label{eq:ind_Mk}\big|m_\ell(\mu) - m_\ell(\nu)\big| \leq C M_k(\mu,\nu),
		\end{equation}
		 for all $
		\ell=k,\dots,2k-1,$ and 
		for any measures $\mu=(1/k)\sum_{i=1}^k \delta_{\theta_i},
		\nu=(1/k)\sum_{i=1}^k \delta_{\eta_i} \in \calU_k(\domain)$.
		We prove the claim inductively. The case $\ell=k$ is trivial. 
		Suppose now that there exists a constant $C > 0$ such that
		equation~\eqref{eq:ind_Mk} holds for all $\ell =k,\dots, \ell_0-1$.
		We will prove there is a possibly different constant $C' > 0$ such that
		the claim holds for $\ell=\ell_0.$ 
		Using the notation of Section~\ref{subsec:momentCompBound}, notice that 
		Newton's identities imply for $\ell_0>k$ that  
		\begin{align*}
		0 &= \sum_{i=\ell_0-k}^{\ell_0} (-1)^{i-1} \Big[e_{\ell_0-i}(\theta_1,\dots,\theta_k)m_i(\mu)
		- e_{\ell_0-i}(\eta_1,\dots,\eta_k) m_i(\nu)\Big],
		\end{align*}
		or equivalently, 
		\begin{align*} 
		\sum_{i=\ell_0-k}^{\ell_0} (-1)^{i-1} \Big[e_{\ell_0-i}(\theta_1,\dots,\theta_k)
		- e_{\ell_0-i}(\eta_1,\dots,\eta_k)  \Big] m_i(\mu) \\ 
		  = \sum_{i=\ell_0-k}^{\ell_0} (-1)^{i-1} \big( m_i(\mu)
		- m_i(\nu)\big) e_{\ell_0-i}(\eta_1,\dots,\eta_k) .
		\end{align*}
		By reasoning as in the proof of Lemma~\ref{lem:coeff_lip}, 
		the left-hand side of the above display is bounded above by a multiple of $M_k(\mu,\nu)$, where the underlying constant only depends on $\Theta$ and $k$. 
		We thus obtain, since $e_0 \equiv 1$, 
		\begin{align*}
		|m_{\ell_0}(\mu) - m_{\ell_0}(\nu)|&\leq \sum_{i=\ell_0-k}^{\ell_0} (-1)^{i-1} \Big[e_{\ell_0-i}(\theta_1,\dots,\theta_k)
		- e_{\ell_0-i}(\eta_1,\dots,\eta_k)  \Big] m_i(\mu) \\
		&+ \left| \sum_{i=\ell_0-k}^{\ell_0-1} (-1)^{i-1} e_{\ell_0-i}(\eta_1,\dots,\eta_k) \big(m_i(\mu)-m_i(\nu)\big)\right|
		\\
		&\lesssim M_k(\mu,\nu)
		 + \left| \sum_{i=\ell_0-k}^{\ell_0-1} (-1)^{i-1} e_{\ell_0-i}(\eta_1,\dots,\eta_k) \big(m_i(\mu)-m_i(\nu)\big)\right| \\
		&\lesssim M_k(\mu,\nu) +
		\sum_{i=\ell_0-k}^{\ell_0-1} |m_i(\mu) - m_i(\nu)|.
		\end{align*}
		By the induction hypothesis, the final term is bounded from above by a multiple of $M_k(\mu,\nu)$. 
		The claim readily follows from here in the one-dimensional case.
		
		For the multivariate case, we employ reduce the problem via some slicing argument to the univariate setting, see \Cref{app:momentDifferenceInequalities}. Specifically, 
		for a $k$-atomic measure $\xi \in \calP(\Theta)$ we write $\xi^\eta$ to denote the push-forward of $\xi$ under the projection map $\langle \cdot, \eta\rangle$, by virtue of Lemma \ref{lem:moment_sliced_moment}, then utilizing the first part of this proof, and finally utilizing Lemma \ref{lem:sliced_multivariate_moment_bound}, it follows 
		\begin{align*}
			M_{2k-1}(\mu, \nu)\lesssim \sup_{\eta \in \SS^d} M_{2k-1}(\mu^\eta, \nu^\eta)\lesssim \sup_{\eta \in \SS^d} M_{k}(\mu^\eta, \nu^\eta) \lesssim M_{k}(\mu, \nu), 
		\end{align*}
		where the suppressed constant in the first inequality only depends on $k$, in the second it depends on $\bigcup_{\eta\in \SS^d}\langle \Theta, \eta\rangle$ and $k$, and the third inequality depends on $d$ and $k$. 		\qed 

\section{Orthogonal Polynomials}\label{app:orthoPolynomials}

In this appendix, we display various useful insights about orthogonal systems of polynomials. We first state the key results for Legendre polynomials, which form basis of polynomials for the $L^2([-R,R]^d)$ for $R>0$. 

\begin{lemma}[Legendre polynomials]\label{lem:LegendrePolynomialOrthonormal}
	For $j \in \NN_0$ denote the normalized $j$-th order Legendre polynomial by
	\begin{align}\label{eq:RodriguesFormula}
		\phi_j\colon \RR \to \RR,\quad  x\mapsto \sqrt{j+\frac{1}{2}} \frac{1}{2^j j!}\frac{\dif^j}{\dif x^j}\left[(x^2-1)^j\right].
	\end{align}
	Then, the following assertions hold. 
	\begin{enumerate}
		\item[(i)] For $j\in \NN_0$ it holds 
		\begin{align*}
			\phi_j(x) = 2^{-j}\sqrt{j+\frac{1}{2}}\cdot \sum_{\substack{l = 0\\ j-l \text{ even}}}^{j}(-1)^{(j-l)/2}\binom{j}{(j-l)/2}\binom{j+l}{j} x^l
		\end{align*}
		\item[(ii)] For $R>0$ the set of $R$-rescaled Legendre polynomials $\{\phi_j^R\}_{j\in\NN_0}$, with $\phi_j^R(x) \coloneqq R^{-1/2} \phi_j(x/R)$ for $x\in \RR$, forms an orthonormal basis for $L^2([-R,R])$.
		\item[(iii)] For  $R>0$ and $d\geq 2$ the set of multivariate $R$-rescaled Legendre polynomials $\{\phi_j^R\}_{j\in \NN_0^d}$, with $\phi_j^R(x) \coloneqq \prod_{l = 1}^{d} \phi_{j_l}^{R}(x_l)$ for $x\in\RR^d$, forms an orthonormal basis for $L^2([-R,R]^d)$.
	\end{enumerate}
	\end{lemma}
	
	\begin{proof}%
		For Assertion $(i)$, we note that the representation \eqref{eq:RodriguesFormula} yields (see, e.g., \citealt{lima2022lecture})
		\begin{align*}
			\phi_j(x) &=2^{-j}\sqrt{j+\frac{1}{2}}  \cdot  \sum_{\substack{k = 0}}^{\lfloor j/2\rfloor} (-1)^k \binom{j}{k}\binom{2j-2k}{j} x^{j-2k} \\
			 &= 2^{-j}\sqrt{j+\frac{1}{2}}  \cdot \sum_{\substack{i = 0\\i \text{ even}}}^{2\lfloor j/2 \rfloor}(-1)^{i/2}\binom{j}{i/2}\binom{2j-i}{j} %
			x^{j-i}\\
			&= 2^{-j}\sqrt{j+\frac{1}{2}}\cdot \sum_{\substack{l = 0\\ j-l \text{ even}}}^{j}(-1)^{(j-l)/2}\binom{j}{(j-l)/2}\binom{j+l}{j} x^l
		\end{align*}
		where we consider for the second and third equality the substitutions $i = 2k$ and $l = j-i$. 
	
		For the remaining assertions, we note that for $d = 1$ and  $R =1$ it is well-known that the normalized Legendre polynomials are orthonormal (see, e.g., \cite[Theorem 1]{lima2022lecture}), i.e., for $i,j\in \NN_0$ it  holds 
		\begin{align*}
			\langle \phi_j, \phi_i \rangle_{L^2([-1,1])}= \int_{-1}^{1} \phi_j(x)\phi_i(x)\dif x = \mathds{1}(i = j).
		\end{align*}
		Hence, using integration by substitution it follows for $d =1$ that 
		\begin{align*}
			\langle \phi_j^R, \phi_i^R \rangle_{L^2([-R,R])}= \int_{-R}^R R^{-1} \phi_j(x/R) \phi_i(x/R) \dif x = \int_{-1}^{1} \phi_j(x')\phi_i(x') \dif x' =\mathds{1}(j=i),%
		\end{align*}
		which asserts that $\{\phi_j^R\}_{j \in \NN_0}$ forms an orthonormal system in $L^2([-R,R])$. Similarly, for the multivariate setting with $d\geq 2$ it follows for $i,j\in \NN_0^d$ that 
		\begin{align*}
			\langle \phi_j^R, \phi_i^R \rangle_{L^2([-R,R]^d)} %
			&= \int_{[-R,R]^d}\prod_{l = 1}^{d}\phi^R_{j_l}(x_l)\phi^R_{i_l}(x_l)\dif x \\
			&= \prod_{l = 1}^{d} \int_{[-R,R]}\phi^R_{j_l}(x_l)\phi^R_{i_l}(x_l)\dif x_l\\
			&= \prod_{l = 1}^{d} \mathds{1}(j_l= i_l) = \mathds{1}(j=i).
		\end{align*}
		Hence, $\{\phi_j^R\}_{j\in \NN_0^d}$ forms an orthonormal system in $L^2([-R,R]^d)$. 
		To conclude Assertions (ii) and (iii) it suffices to show for $d\geq 1$ that $\smash{\{\phi_j^R\}_{j\in \NN_0^d}}$ is a basis for $L^2([-R,R]^d)$. This is a consequence of the fact that continuous functions are dense in $L^2([-R,R]^d)$ (e.g., \citealt[Theorem 11.38]{rudin1964principles}) in conjunction with the Stone-Weierstrass theorem which asserts that every continuous function on $[-R,R]^d$ can be approximated in uniform norm using a (multivariate) polynomial.
	\end{proof}

	For the complex plane, the collection of orthogonal polynomials in balls around the origin attain a particularly simple form. The following result formalizes this aspect. 

\begin{lemma}[Complex monomials]\label{lem:ComplexMonomial}
	For $j\in \NN_0$ denote by the normalized complex monomial 
	\begin{align*}
		\phi_j\colon \CC \to \CC, \quad z \mapsto \sqrt{\frac{2j+1}{\pi}}z^j.
	\end{align*}
	Then, the following assertions hold.  
	\begin{enumerate}
		\item[(i)] For $R>0$ the set of $R$-rescaled complex monomials polynomials $\{\phi_j^R\}_{j\in\NN_0}$, with $\phi_j^R(x) \coloneqq R^{-1} \phi_j(x/R)$ for $x\in \RR$, forms an orthonormal system in $L^2(B(0,R))$.
		\item[(ii)] For a complex-valued function $f\in L^2_{\CC}(B(0,R))$ with $R>0$ it holds for every $k\in \NN$, 
		\begin{align*}
			\|f\|_{L^2_\CC(B(0,R))}^2 \geq \sum_{j = 0}^{k} \frac{j+1}{\pi}R^{-2j-2} \left|\int_{B(0,R)} \overline f(z) z^j \dif z\right|^2.
		\end{align*}
	\end{enumerate}
\end{lemma}

\begin{proof}
	To show Assertion (i) we directly compute for $R>0$ and $j,l\in \NN_0$ that 
	\begin{align*}
		\langle \phi_j^R, \phi_l^R\rangle_{L^2_\CC(B(0,R))} &=  \frac{j+1}{\pi} R^{-2-j-l} \int_{B(0,R)}\overline z^j z^l \dif z \\
		&=  \frac{j+1}{\pi} R^{-2-j-l}\int_{0}^{R}  \int_{0}^{2\pi} r^{j+l+1} \exp([-j+l]\theta) \dif \theta \dif r\\ 
		&= (2j+2)R^{-2-j-l} \begin{cases}
		\int_0^R r^{2l+1} \dif r & \text{ if } j = l,\\
		0 & \text{ else},
		\end{cases} \\
		&= \mathds{1}(j = l).
	\end{align*}
For Assertion (ii) denote by $f_k$ the orthogonal projection of $f$ onto the span of $\{\phi_j\}_{j=0, \dots, k}$, which is a finite dimensional vector space of $\CC$ and thus closed. Hence, we obtain 
\begin{align*}
	\|f\|_{L^2_\CC(B(0,R))}^2&\geq \|f_k\|_{L^2_\CC(B(0,R))}^2\\
	 &= \sum_{j = 0}^{j} |\langle f, \phi_j^R\rangle_{L^2_\CC(B(0,R))}|^2\\
	&= \sum_{j = 0}^{j} \frac{j+1}{\pi} R^{-2j-2} \left|\int_{B(0,R)}\overline f(z) z^j \dif z \right|^2.\qedhere
\end{align*} 
\end{proof}

Finally, we state a simple reverse analogue of 
Lemma~\ref{lem:ostrowski}.

\begin{lemma}
\label{lem:forward_bound}
Let $\Theta\subseteq \bbC$ be a bounded set. 
Then, there exists a constant $C_1 = C_1(\Theta,d) > 0$ such that
for all $\mu,\nu\in\calU_k(\Theta)$, 
$$\|f_\mu-f_\nu\|_* \leq C_1 W_1(\mu,\nu).$$
\end{lemma}
\begin{proof}
Since $\Theta$ is bounded, we have for all $\mu = (1/k)\sum_i \delta_{\theta_i}, \nu = (1/k)\sum_i \delta_{\eta_i} \in \calU_k(\Theta)$, 
$$\|f_\mu-f_\nu\|_* \lesssim \|f_\mu-f_\nu\|_{L^\infty(\Theta)} 
= \sup_{z \in \Theta} \left| \prod_i(z-\theta_i) - \prod_i (z-\eta_i)\right| 
\lesssim \sum_i|\theta_{\sigma(i)} - \eta_i|,$$
for any permutation $\sigma \in \calS(k)$. The claim follows.
\end{proof}

\subsection{Implementation to compute polynomials in Theorem \ref{thm:moment_polynomials}}\label{subsec:implementationPolynomials}

In this subsection we detail \Cref{alg:cap} to compute the polynomials $\psi_1, \dots, \psi_\ell$, $\ell \in \NN$  based on the first $\ell$ moments of the kernel $K$. The implementation follows from the proof of \Cref{thm:moment_polynomials} and is applicable for kernels on $\RR$ and $\CC$. 

\begin{algorithm}[h!]
	\caption{Algorithm to compute the polynomials $\psi_1, \dots,\psi_\ell$ from moments of $K$}\label{alg:cap}
	\begin{algorithmic}
	\Require Moments $m_j^K = \int y^j K(y) \dif y$ for $j \in \{1, \dots, \ell\}$	
	\State Define lower triangle matrix $M= (M_{ij})_{i,j = 0}^{\ell}$ with entries:
	\State \quad $M_{ii} \gets 1$ \quad\quad \quad \quad  for $0\leq i \leq \ell,$ 
	\State \quad $M_{ij} \gets \binom{i}{i-j} m_{j}^K$ \!\!\!\! \quad for $0\leq j<i\leq \ell,$
	\State \quad $M_{ji} \gets 0$ \quad \quad \quad \quad\!\! for $0\leq i< j \leq \ell$
	\State Compute $M^{-1}$ and define $A= (a_{ij})_{i,j=0}^{\ell}\gets M^{-1}$ \Comment{Requires $\mathcal{O}(\ell^2)$ operations}
	\For{$i = 1, \dots, \ell$}
		\State $\psi_i(z) \coloneqq \sum_{j = 0}^{i} a_{ij} z^j$ \Comment{Definition of polynomials}
	\EndFor
	\State \Return Polynomials $\psi_1, \dots, \psi_\ell$
	\end{algorithmic}
	\end{algorithm}

\section{Additional Technical Results} 
The following is a standard concentration inequality for Poisson random variables
(see, for instance,~\cite[Theorem A.8]{canonne2022}).
\begin{lemma}
\label{lem:poisson_concentration}
Let $X \sim \mathrm{Poi}(\lambda)$ for some $\lambda > 0$. Then, for all $u > 0$, we have
$$\bbP(|X - \lambda| > t) \leq 2 e^{-\frac{u^2}{2(\lambda+u)}}.$$
\end{lemma}

The following is trivial.
\begin{lemma}
\label{lem:bracketing_number_equiv}
Let $\calF$ be a class of real-valued functions on some domain and let $d,d'$ be two equivalent metrics on $\calF$  such that for some $c,c' > 0$, 
$$c  d(x,y) \leq d'(x,y) \leq c'  d(x,y),\quad \text{ for all } x,y \in \calF.$$
Then, for the bracketing numbers it holds 
$$N_{[]}(\epsilon/c, \calF, d) \leq N_{[]}(\epsilon,\calF,d')
\leq N_{[]}(\epsilon/c',\calF,d),\quad \text{for all } \epsilon > 0.$$
\end{lemma}

The following is a simple bound on the bias of the histogram operator $\Pi_m$
defined in Section~\ref{sec:mle_bracketing}.
\begin{lemma}
\label{lem:Pi_error}
Let $\Omega\subseteq \bbR^d$ be a compact set satisfying Assumption \ref{ass:bins}, and let $L > 0$.
Then, 
there exists a constant $C = C(\Omega,d, L) > 0$ such that for any 
$L$-Lipschitz function $h:\Omega\to\bbR$, 
$$\|\Pi_m h - h\|_{L^2(\Omega)} \leq C m^{-1/d}.$$
\end{lemma}
\begin{proof}
Recalling that the Lebesgue measure of each bin $B_i$ is positive, $\lambda(B_i)>0$, we have 
\begin{align*}
\|h -\Pi_m h\|_{L^2(\Omega)}^2
 &=\left\|h -  \sum_{i=1}^m h(B_i) I_{B_i}/\lambda(B_i) \right\|_{L^2(\Omega)}^2 \\
 &=  \left\| \sum_{i=1}^m (h - h(B_i)/\lambda(B_i)) I_{B_i}\right\|_{L^2(\Omega)}^2 \\
 &= \sum_{i=1}^m  \left\| h - h(B_i)/\lambda(B_i) \right\|_{L^2(B_i)}^2 \\
 &= \sum_{i=1}^m \int_{B_i} \left(  \frac{1}{\lambda(B_i)}\int_{B_i} (h(x)-h(y))dy\right)^2dx.
\end{align*}Thus, 
\begin{align}
\label{eq:discr_err}
\|h -\Pi_m h\|_{L^2(\Omega)}^2
 &\leq L^2 \sum_{i=1}^m \int_{B_i} \left(  \frac{1}{\lambda(B_i)} \int_{B_i} \|x-y\|dy\right)^2dx
 \leq (C m^{-1/d})^2,
\end{align}
for a large enough constant $C$ depending only on $L,\Omega$.
\end{proof}

The following two results formalize some properties on multivariate products of orthonormal systems and convolutions. 

\begin{lemma}[Product of orthonormal basis]\label{lem:productONB}
	For $i = \{1, \dots, d\}$ let $\{f_{n,i}\}_{n \in \NN_0}$ be an orthonormal basis for $L^2(I_i)$ where $I_i\subseteq \RR$ is some (possibly unbounded) interval with non-empty interior. Then,  $\{f_{n_i}\}_{n\in \NN_0^d}$ with $f_{n}(x) = \prod_{i =1}^{d} f_{n_i,i}(x_i)$ forms a basis for $L^2(\bigtimes_{i = 1}^{d} I_i).$
	\end{lemma}
	
	\begin{proof}
		First, we notice for $n, n'\in \NN_0^d$ by Fubini's theorem that 
		\begin{align*}
			\langle f_n,f_{n'}\rangle_{L^2(\bigtimes_{i = 1}^{d} I_i)} &= \int_{\bigtimes_{i = 1}^{d}I_i} \prod_{i = 1}^{d} f_{n_i,i}(x_i) f_{n'_i,i}(x_i)\dif x \\
			&= \prod_{i = 1}^{d} \int_{I_i}  f_{n_i,i}(x_i) f_{n'_i,i}(x_i)\dif x \\
			&= \prod_{i = 1}^{d} \mathds{1}(n_i = n'_i) = \mathds{1}(n = n'),
		\end{align*}
		which yields that $\{f_n\}_{n\in \NN_0^d}$ is an orthonormal system. It remains to show that $\{f_n\}$ forms a basis, for which we show that the orthogonal complement equals zero.	Let $g \in L^2(\bigtimes_{i = 1}^{d} I_i)$ be orthogonal to all $f_n$. Then, for each $n \in \NN_0^d$, again invoking Fubini's theorem it holds, 
		\begin{align*}
			0 &= \langle g, f_n \rangle_{L^2(\bigtimes_{i = 1}^{d} I_i)} \\
			&= \int_{\bigtimes_{i = 1}^{d} I_i} g(x) \prod_{i = 1}^{d} f_{n_i,i}(x_i) \dif x \\
			&= \int_{I_d} \cdots \int_{I_1} g(x) f_{n_1,1}(x_1) \dif x_1 \cdots  f_{n_d,d}(x_d)\dif x_d.
		\end{align*}
		Since $\{f_{n_d,d}\}_{n_d\in \NN_0}$ is an orthonormal basis for $L^2(I_d)$, it follows that $$x_1 \mapsto \int_{I_{d_1}}\dots \int_{I_1} g(x_1, x_2, \dots, x_d) f_{n_1,1}(x_1)\dif x_1 \dots f_{n_{d-1},d-1}\dif x_{d-1}$$ is equal to zero for Lebesgue for almost every $x_d$ on $I_d$. Repeating this argument for $x_1, \dots, x_{d-1}$ yields that $g =0$ almost everywhere on $\bigtimes_{i = 1}^{d} I_i$, and we conclude that $\{f_n\}_{n\in \NN_0^d}$ forms a basis for $L^2(\bigtimes_{i = 1}^{d}I_i)$.
	\end{proof}

\begin{lemma}[Convolution of product]\label{lem:convolutionProduct}
	For $d\in \NN$ consider $f\colon \RR^d\to \RR, x\mapsto \prod_{i = 1}^{d} f_i(x_i)$, $g\colon \RR^d\to \RR, x\mapsto \prod_{i = 1}^{d} g_i(x_i)$ for Lebesgue integrable functions $f_1, \dots, f_d , g_1, \dots, g_d\colon \RR\to \RR$. Then it holds 
	\begin{align*}
		f\ast g(x) = \prod_{i = 1}^{d} (f_i\ast g_i)(x_i) \quad \text{ for all } x\in \RR^d.
	\end{align*}
\end{lemma}

\begin{proof}
	We prove the Assertion for $d = 2$, the assertion then follows by an induction argument over $d\in \NN$. To show the assertion first assume that $f_1, f_2, g_1, g_2$ are non-negative. Then it follows by Tonelli's theorem and linearity of the integral that 
	\begin{align*}
		f\ast g(x) &= \iint_{\RR^2} f_1(x_1-y_1)f(x_2-y_2) g_1(y_1)g_2(y_2) \dif y_1 \dif y_2\\
		&= \int f_1(x_1 -y_1)g_1(y_1) \left(\int f_2(x_2-y_2)g(y_2)\dif y_2 \right) \dif y_1  \\
		&= \int f_1(x_1 -y_1)g_1(y_1) \dif y_1 \int f_2(x_2-y_2)g(y_2)\dif y_2 = f_1\ast g_1(x_1) \cdot f_2\ast g_2(x_2).
	\end{align*}
	Moreover, if $f_1, f_2, g_1, g_2$ are not non-negative, but Lebesgue integrable, Fubini's theorem ensures the validity of the second equation and the assertion follows.
\end{proof}
	
The following result can be deduced from Proposition 2.6 and Theorem 2.7 of~\cite{ding2022}.

\begin{lemma}
\label{lem:ding}
Suppose that $\mu,\nu\in \calU_k(\Theta)$ are two measures with support size exactly equal to $k$. 
Let $\tau_0$ be an optimal matching from $\mu$ to $\nu$.
 Then, 
the following assertions are equivalent.
\begin{enumerate}
\item There is a unique optimal transport coupling from $\mu$ to $\nu$, which is induced by  the matching
$\tau_0$. Furthermore, there is a  unique optimal transport coupling from $\nu$ to $\mu$, which is induced by the matching $\tau_0^{-1}$. 
\item There exists a constant $\lambda > 0$ and a twice continuously differentiable convex function $\varphi:\Omega \to \bbR$ 
such that
$$\lambda^{-1} I_d \preceq \nabla^2 \varphi_0(x) \preceq \lambda I_d,\quad \text{over } \Omega,$$
and such that 
$$\nabla\varphi_0(\theta_{\tau_0(i)}) = \eta_i,\quad \text{and} \quad \nabla\varphi_0^*(\eta_i) = \theta_{\tau_0(i)},
\quad i=1,\dots,k,$$
where $\varphi_0^*$ is the Legendre-Fenchel conjugate of $\varphi_0$.
\end{enumerate}
\end{lemma}

The following result is due to~\citet[Corollary 4]{balakrishnan2025}. 
\begin{lemma}
\label{lem:stab_w2}
Let $\lambda > 0$, and let $\Theta\subseteq \bbR^d$ be a convex set. 
Let $(\mu,\nu)\in \calT_\lambda(\Theta)$, and let $\pi_0$ denote the quadratic optimal
transport coupling between $\mu$ and $\nu$. Then, 
there exists a constant $C = C(\lambda) > 0$ such that for all 
couplings $\hat\pi \in \calP(\Theta\times\Theta)$ with marginal distributions $\hat\mu$ and $\hat\nu$, it holds that
$$W_2(\hat\pi,\pi_0) \leq C\Big(W_2(\hat\mu,\mu) + W_2(\hat\nu,\nu)\Big).$$
\end{lemma}

\section{EM Algorithm for Computation of Maximum Likelihood Estimator}\label{app:EM}

Computation of MLE is a challenging issue in the context mixture learning.
Ever since its introduction by  \cite{dempster1977} the EM algorithm is perhaps the most widely-used method for finding the MLE in mixture models. 
In this appendix we describe our implementation of the EM algorithm for our binned Poisson convolution model \eqref{eq:model}.
To this end, we first recall the EM algorithm for mixture models \eqref{eq:mixture_model}.

Let the observations be i.i.d. random variables from the $k$-atomic mixture model
\begin{equation}
    \label{eq:em:mixture_model}
    Y_1,\dots,Y_n \sim K\star \mu,~~\text{for some } \mu\in \calP_k(\Theta),
\end{equation}
where $\calP_k(\Theta)$ is the set of $k$-atomic probability measures
on $\Theta$, and $K$ is a probability kernel.
Let $Z_1,\dots,Z_n$ i.i.d.\ be the latent variables, such that $\PP(Z_i = j) = w_j$ and $Y_i|Z_i = j \sim K\star \delta_{\theta_j}$ for $j=1,\dots,k$, and $\theta_1,\dots,\theta_k\in\Theta$ are the atoms of $\mu$ and $w_1,\dots,w_k\geq 0$ are the weights of $\mu$ with $\sum_{j=1}^k w_j = 1$.
Therefore, the likelihood is given by $f_{Y,Z}(y,z) = \prod_{i=1}^n w_{z_i} K(y_i-\theta_{z_i})$.
Given $\tilde{\mu} \coloneq \sum_{j=1}^{k} \tilde{w}_j \delta_{\tilde{\theta}_j} \in \calP_k(\Theta)$, we compute the conditional expectation of the log-likelihood
$$
    \EE_{\tilde{\mu}}[\ln f_{Y, Z}(Y,Z) | Y=y]
    = \sum_{i=1}^{n} \EE_{\tilde{\mu}}[\ln(w_{Z_i} K(y_i-\theta_{Z_i})) | Y_i=y_i]
    = \sum_{i=1}^{n} \sum_{j=1}^{k} \tilde{p}_{i,j} \ln(w_{j} K(y_i-\theta_{j}))
    ,
$$
with $\tilde{p}_{i,j} =  \tilde{w}_j K(y_i-\tilde{\theta}_j) / \sum_{h=1}^{k} \tilde{w}_h K(y_i-\tilde{\theta}_h)$. The EM algorithm is described in \Cref{alg:mixture:EM-Algorithm}.

\begin{algorithm}[h!]
    \caption{EM Algorithm for Mixture model \eqref{eq:em:mixture_model} with general kernel $K$}
    \label{alg:mixture:EM-Algorithm}
    \begin{algorithmic}
        \Require{Data $Y_1, \dots, Y_n$, initial atoms $\theta_1^{(0)}, \dots, \theta_k^{(0)}$ and weights $w_1^{(0)}, \dots, w_k^{(0)}\geq 0$ with $\sum_{j =1}^{k} w_j^{(0)} = 1$, number of iterates~$L$, probability kernel $K$,  object domain $\Theta$.}
        \Ensure{Mixing measure $\mu^{(L)} = \sum_{j = 1}^{k}w_j^{(L)}\delta_{\theta_j^{(L)}}$ and its atoms $\theta_1^{(L)}, \dots, \theta_k^{(L)}$.}
        \State Initialize $\mu^{(0)} \coloneq \sum_{j = 1}^{k}w_j^{(0)}\delta_{\theta_j^{(0)}}$.
        \For{$l = 1, \ldots, L$}
        \State Define $\tilde{\mu} \coloneqq \mu^{(l-1)}$, take atoms $\tilde \theta_1, \dots, \tilde \theta_k$ and weights $\tilde w_1, \dots, \tilde w_k$ of $\tilde \mu$,  
        \State Define  $\tilde \lambda_{i,j} \coloneqq \tilde{w}_j K(Y_i- \tilde \theta_j)$ and $\tilde p_{i,j} \coloneqq \tilde \lambda_{i,j}/ \sum_{h = 1}^{k}\tilde \lambda_{i,h}$
        \State Define function
        $Q(\cdot, \tilde{\mu})\colon \calP_k(\Theta)\to \RR$,\hfill \text{(Expectation step)}
        \State \quad \quad  $\sum_{j=1}^{k} w_j \delta_{\theta_j}
            \mapsto \sum_{i=1}^{n} \sum_{j=1}^{k} \tilde{p}_{i,j} \ln(w_{j} K(Y_i-\theta_{j}))$
        \State Compute $\mu^{(l)} \coloneqq \argmax_{\mu \in \calP_k(\Theta)} Q(\mu, \tilde{\mu})$ \hfill \text{(Maximization step)}
        \EndFor
    \end{algorithmic}
\end{algorithm}

To derive the EM algorithm for the binned Poisson convolution model \eqref{eq:model} let us first recall the model. For $i=1, \dots, m$ we have independent observations
\begin{align}
    \label{eq:em:model}
    X_i \sim \textup{Poi}(t\cdot \lambda_i)
    \quad \text{with } \quad
    \lambda_i = (K\ast \mu)(B_i) = \int_{B_i} \left[ \int_{\Theta} K(x-y) \mu(dy)\right] dx
    ,
\end{align}
where $\mu = \frac{1}{k}\sum_{j = 1}^{k} \delta_{\theta_j}$ for some $k\in\NN$ and $\theta_1,\dots\theta_k \in \Theta$, while $(B_i)_{i=1}^{m}$ are disjoint bins.
The main difference to the mixture model \eqref{eq:em:mixture_model} is that instead of explicit locations $Y_1,\dots,Y_n$, we observe the number of their occurrences $X_1,\dots,X_m$ in each bin $(B_i)_{i=1}^{m}$. In addition, we assume uniform weights $w_j = 1/k$ for $j=1,\dots,k$, which can also be generalized to arbitrary weights.

Following along the notation for the classical EM algorithm, we note that the parameters $\lambda_i$ can be expressed as $$\lambda_i = \frac{1}{k} \sum_{j=1}^k \int_{B_i} K(x-\theta_j) dx = \sum_{j=1}^k \lambda_{i,j}, \quad \text{ where } \quad
    \lambda_{i,j}
    = \frac{1}{k} \int_{B_i} K(x-\theta_j) dx
    .
$$
Further, we introduce latent variables 
 $Z_{i,j} \sim \textup{Poi}(t\lambda_{i,j})$ for $i=1,\ldots,m$, $j=1,\ldots,k$, such that
$$ X_i = \sum_{j=1}^k Z_{i,j}, $$
with corresponding densities
$$f_{X_i}(x) = \frac{(t\lambda_i)^{x} e^{-t\lambda_i}}{x!},
    \quad \quad
    f_{Z_{i,j}}(z) = \frac{(t\lambda_{i,j})^{z} e^{-t\lambda_{i,j}}}{z!},
$$
$$
    f_{X_i, Z_{i}}(x,z) = \mathds{1}\left(x=\textstyle\sum_{j=1}^k z_j\right) \prod_{j=1}^k \frac{(t\lambda_{i,j})^{z_j} e^{-t\lambda_{i,j}}}{z_j!}.
$$
To simplify the notation, we assume that $x = \sum_{j=1}^k z_j$. Then we derive that $Z_{i}|X_i=x \sim \textup{Mult}(x,p_{i,1},\dots p_{i,k})$ with
$ p_{i,j} = \frac{\lambda_{i,j}}{\lambda_i}$, since the density of $Z_{i}|X_i=x$ is given by
\begin{align*}
    f_{Z_{i}|X_i=x}(z) & = \frac{f_{X_i, Z_{i}}(x,z)}{f_{X_i}(x)}
    = \frac{x!}{(t\lambda_i)^{x} e^{-t\lambda_i}} \prod_{j=1}^k \frac{(t\lambda_{i,j})^{z_j} e^{-t\lambda_{i,j}}}{z_j!}
    = \frac{x!}{(t\lambda_i)^{\sum_{j=1}^k z_j}} \prod_{j=1}^k \frac{(t\lambda_{i,j})^{z_j}}{z_j!} \\
                       & = x! \prod_{j=1}^k \frac{(\lambda_{i,j}/\lambda_i)^{z_j}}{z_j!}
    .
\end{align*}

Given $\tilde{\mu} = \frac{1}{k}\sum_{j = 1}^{k} \delta_{\tilde{\theta}_j}$, we compute the conditional expectation of the log-likelihood

\begin{align*}
    \EE_{\tilde{\mu}}[\ln f_{X_i, Z_{i}}(X_i,Z_i) | X_i=x]
    &=-t\lambda_i + \sum_{j=1}^k  \EE_{\tilde{\mu}}[Z_{i,j}\ln(t\lambda_{i,j}) - \ln(Z_{i,j}!) | X_i=x]\\
    &=-t\lambda_i + \sum_{j=1}^k  \EE_{\tilde{\mu}}[Z_{i,j}| X_i=x]\ln(t\lambda_{i,j})  + c(\tilde{\mu})\\
    & =-t\lambda_i + \sum_{j=1}^k  x \tilde{p}_{i,j}\ln(t\lambda_{i,j})  + c(\tilde{\mu})\\
    &= \sum_{j=1}^k  (x \tilde{p}_{i,j}\ln(t\lambda_{i,j}) -t\lambda_{i,j})  + c(\tilde{\mu}).
\end{align*}
with $\tilde{p}_{i,j}$ defined as $p_{i,j}$ (using $\tilde{\mu}$ instead of $\mu$) and
$c(\tilde{\mu}) = -\sum\limits_{j=1}^k  \EE_{\tilde{\mu}}[\ln(Z_{i,j}!) | X_i=x]$, which does not depend on $\mu$, so $c(\tilde{\mu})$ can be omitted during maximization step.
For the maximization step we need to maximize
$$Q(\mu, \tilde{\mu}) = \sum_{i=1}^m \sum_{j=1}^k  (X_i \tilde{p}_{i,j}\ln(t\lambda_{i,j}) -t\lambda_{i,j}).$$
Now, given data $X$ and initial measure $\mu^{(0)}$, the EM algorithm for \eqref{eq:em:model} is  described in \Cref{alg:EM-Algorithm}.

\begin{algorithm}[h!]
    \caption{EM Algorithm for binned Poisson convolution model \eqref{eq:em:model} with kernel~$K$}\label{alg:EM-Algorithm}
    \begin{algorithmic}
        \Require{Data $X_1, \dots, X_m$, bins $B_1, \dots, B_m$, initial atoms $\theta_1^{(0)}, \dots, \theta_k^{(0)}$, number of iterates~$L$, probability kernel $K$,  object domain $\Theta$.}
        \Ensure{Mixing measure $\mu^{(L)} = \frac{1}{k}\sum_{j = 1}^{k}\delta_{\theta_j^{(L)}}$ and its atoms $\theta_1^{(L)}, \dots, \theta_k^{(L)}$.}
        \For{$l = 1, \ldots, L$}
        \State Define $\tilde{\mu} \coloneqq \mu^{(l-1)}$, take atoms $\tilde \theta_1, \dots, \tilde \theta_k$ of $\tilde \mu$,          \State Define  $\tilde \lambda_{i,j} \coloneqq K(B_i- \tilde \theta_j)$ and $\tilde p_{i,j} \coloneqq \tilde \lambda_{i,j}/ \sum_{h = 1}^{k}\tilde \lambda_{i,h}$        \State Define function
        $Q(\cdot, \tilde{\mu})\colon \calU_k(\Theta)\to \RR$,\hfill \text{(Expectation step)}
        \State \quad \quad  $\frac{1}{k}\sum_{j = 1}^{k} \delta_{\theta_j}\mapsto \sum_{i=1}^m \sum_{j=1}^k  (X_i \tilde{p}_{i,j}\ln(t K(B_i-\theta_j)) -tK(B_i-\theta_j))$
        \State Compute $\mu^{(l)} \coloneqq \argmax_{\mu \in \calU_k(\domain)} Q(\mu, \tilde{\mu})$ \hfill \text{(Maximization step)}
        \EndFor
    \end{algorithmic}
\end{algorithm}
Notably, to perform the maximization step in the EM algorithm we use a public implementation of L-BFGS-B algorithm.\footnote{\href{https://docs.scipy.org/doc/scipy/reference/optimize.minimize-lbfgsb.html}{https://docs.scipy.org/doc/scipy/reference/optimize.minimize-lbfgsb.html}} As our initializer we employ the moment estimator. 
As for the number of iterations, we use $L = 50$ and we also apply early stopping when $W_1(\mu^{(l)},\mu^{(l-1)}) < 10^{-9}$.

\section{Simulations}\label{sec:simulations}
In this appendix, we explore the performance of our method of moments estimator and the maximum likelihood estimator (computed using the EM algorithm detailed in \Cref{app:EM} and initialized with our method of moments estimator) on synthetic data.
First, we consider atom locations $\theta_1,\dots,\theta_k \in\RR^2$ according to the three distinct configurations: grid, corners, and u-shape, illustrated in \Cref{fig:simulations:points}.
Then we generate data $X_1,\dots,X_m$ according to our model \eqref{eq:model} for $d=2$
and where the kernel $K$ is given by an isotropic Gaussian density with a standard deviation $\sigma>0$, i.e., $K(x) = (2\pi \sigma)^{-1}\exp(-\frac{1}{2 \sigma^2}x^\top x)$, that is
\begin{equation}
    \label{eq:simulations:data}
    X_i \sim \textup{Poi}\left(
    t  K\ast \mu(B_i)\right),
    \quad i = 1, \dots, m.
\end{equation}
Additionally, we consider the case where we have access to the true bin densities,
\begin{align*}
    K\ast \mu(B_i) \quad i = 1, \dots, m,
\end{align*}
which corresponds to the noiseless regime and is labelled as $t = \infty$.
We apply the estimators to the data and compare the resulting estimator $\hat \mu$ with the underlying true measure $\mu$. The error between the estimator and truth is quantified using the $1$-Wasserstein distance from \eqref{eq:def_Wasserstein}, the corresponding risk is computed by averaging over $ N = 100$ simulation runs.

\begin{figure}[h!]
    \centering
    \includegraphics[width=0.95\linewidth]{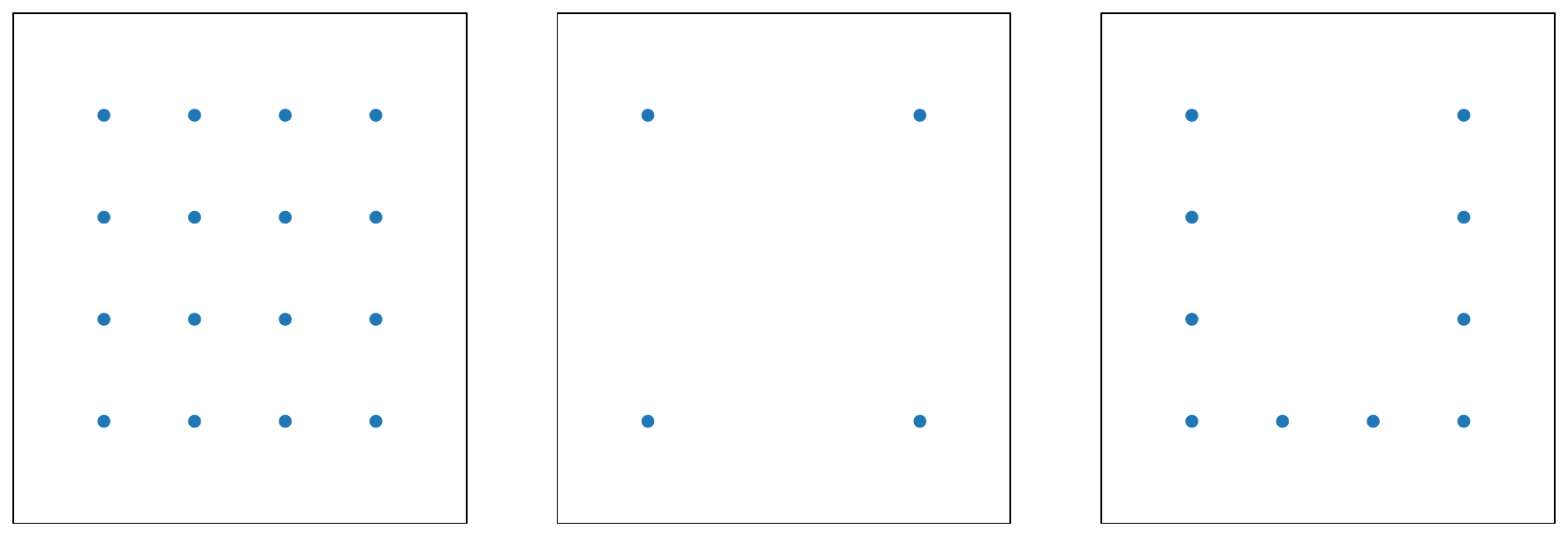}
    \caption{Three exemplary configurations of atoms in a unit square which we denote as "grid", "corners", and "u-shape" (from left to right).}
    \label{fig:simulations:points}
\end{figure}

\begin{figure}[h!]
    \centering
    \includegraphics[width=0.95\linewidth]{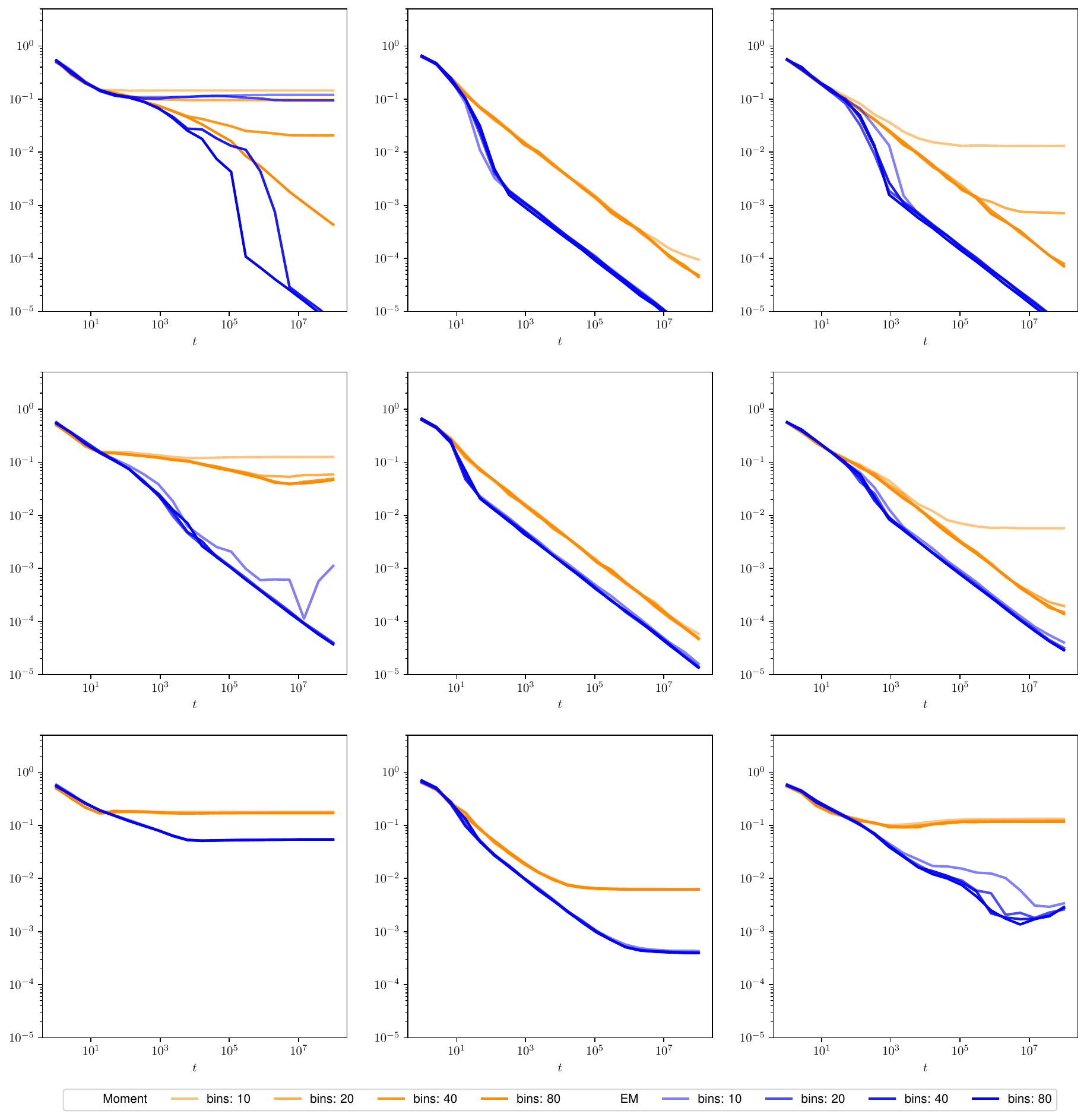}
    \caption{Average $1$-Wasserstein errors along $t$ for the method of moments estimator (orange) and the MLE using the EM algorithm (blue) (from right to left: grid, corners, u-shape setting; from top to bottom standard deviation: $0.01$, $0.05$, $0.1$).}
    \label{fig:simulations:errors:moment-em}
\end{figure}

In \Cref{fig:simulations:individual:corners,fig:simulations:individual:grid} the $1$-Wasserstein risk of the method of moments estimator and the EM algorithm for the grid and corners settings are depicted, respectively, with $\sigma=0.05$ and different values of $m$ and $t$. We observe that the method of moments estimator performs well for $k = 4$ well-separated atoms and poorly for $k = 16$ atoms, while the EM algorithm performs well in both cases.
Further, we average the $1$-Wasserstein loss over several simulations, conducting experiments for different values of $\sigma$, $m$, and $t$ in grid, corners and u-shape settings, illustrated in \Cref{fig:simulations:errors:moment-em}.
We observe that increasing the number of bins $m$ beyond a certain value does not improve the quality of the estimations, while increasing $t$ improves the quality of the estimations at a rate proportional to $t^{-1/2}$.
Additionally, for large standard deviations $\sigma$, the quality of the estimations decreases.

\begin{figure}[h!]
    \centering
    \includegraphics[width=0.95\linewidth]{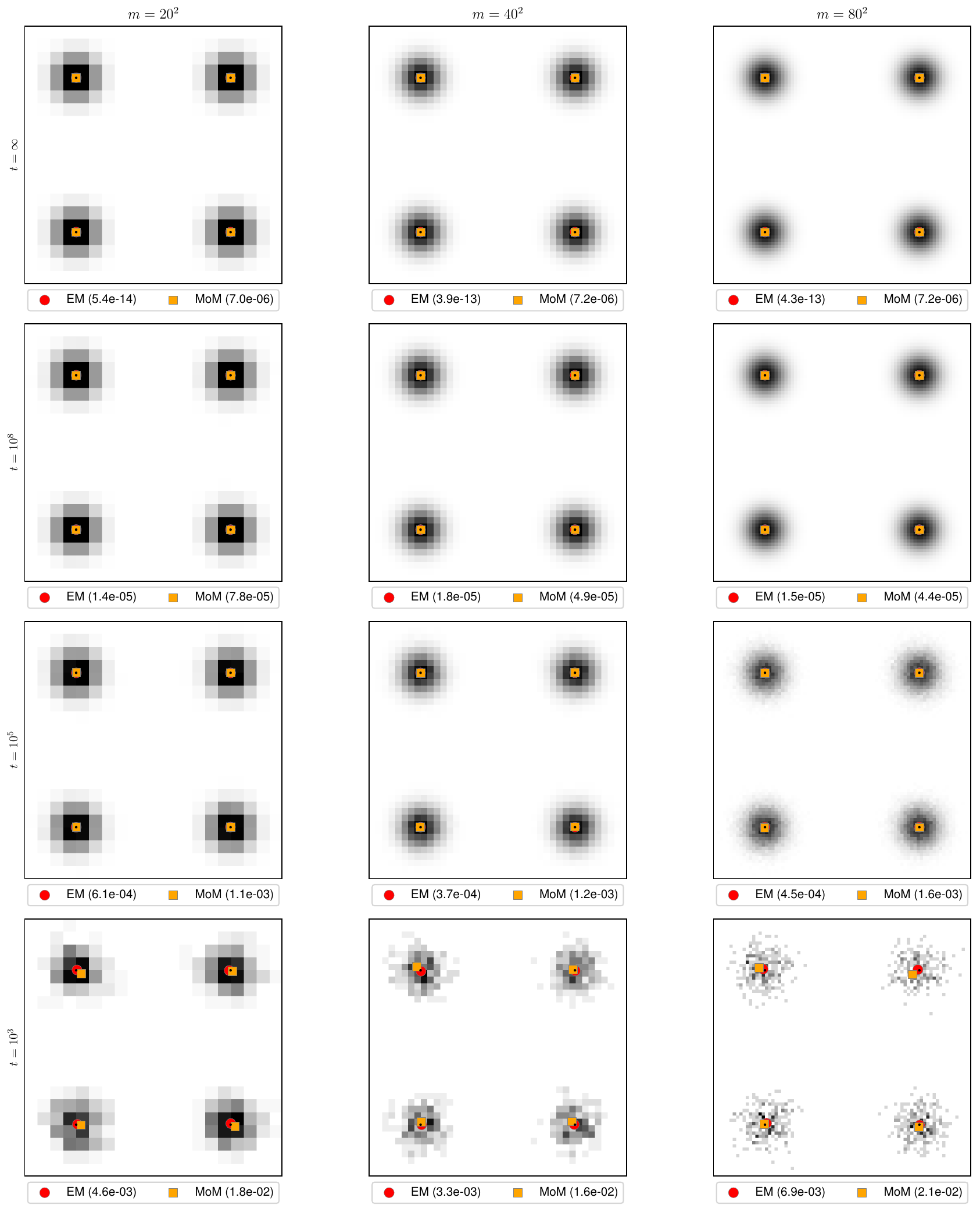}
    \caption{Simulated data for the corners setting together with the atom estimations and $1$-Wasserstein errors for the method of moments estimator (orange squares) and the MLE using the EM algorithm (red circles) for $\sigma = 0.05$, $m=20^2,40^2,80^2$ (from left to right), $t=10^3,10^5,10^8,\infty$ (from bottom to top).}
    \label{fig:simulations:individual:corners}
\end{figure}
\begin{figure}[h!]
    \centering
    \includegraphics[width=0.95\linewidth]{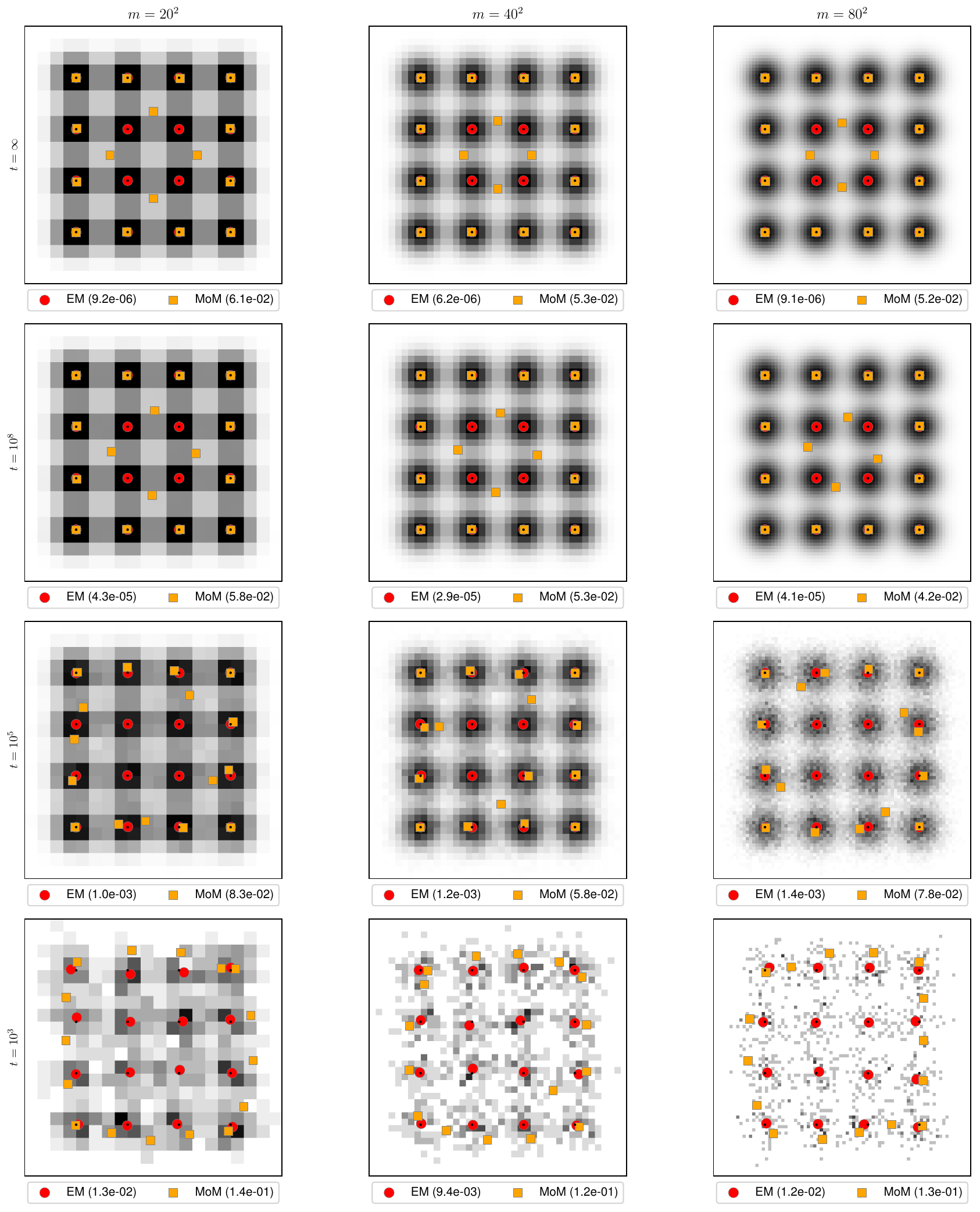}
    \caption{Simulated data for the grid setting together with the atom estimations and $1$-Wasserstein errors for the method of moments estimator (orange squares) and the MLE using the EM algorithm  (red circles) for $\sigma = 0.05$, $m=20^2,40^2,80^2$ (from left to right), $t=10^3,10^5,10^8,\infty$ (from bottom to top).}
    \label{fig:simulations:individual:grid}
\end{figure}

Finally, we compare in \Cref{fig:simulations:time-complexity} the runtime of the method of moments estimator and MLE using the EM algorithm. The method of moments estimator is computed extremely fast. Meanwhile, the MLE exhibits a significantly smaller risk while taking a much longer computational runtime.

\begin{figure}[h!]
    \centering
    \includegraphics[width=0.95\linewidth]{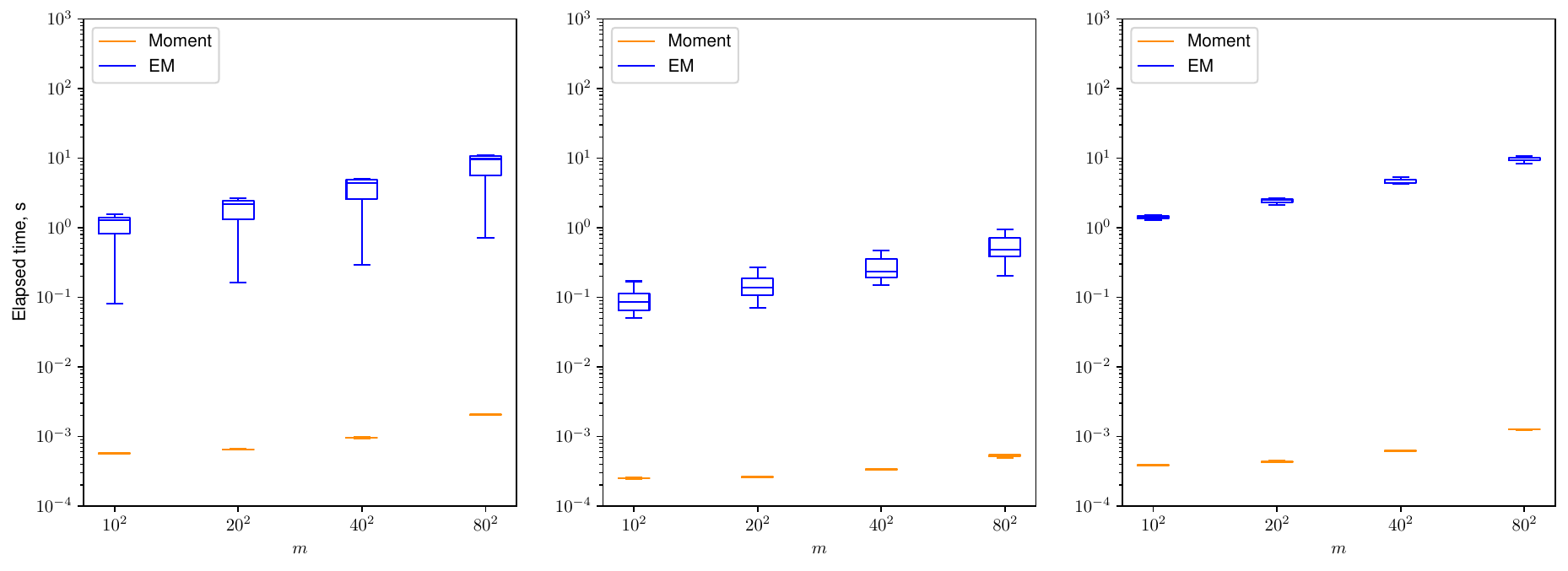}\;\;\;
    \caption{Computational runtime for the method of moments estimator and the MLE using the EM algorithm (from right to left: grid, corners, u-shape setting).}
    \label{fig:simulations:time-complexity}
\end{figure}

\newpage

\section{Application: Methodology and Additional Analysis}\label{app:application}

In this appendix, we discuss addition aspect to our application related to the STED image data provided by \cite{proksch2018multiscale} on the DNA origami samples. The following subsections are organized as follows. In \Cref{app:apl:partition} we describe the procedure with which we partition the image data into several smaller components. In \Cref{app:apl:estimation} we explain the denoising step and state how to apply the EM algorithm to compute the MLE. \Cref{app:apl:robustness} is concerned with robustness properties of the EM-based MLE and compares the performance of the method of moment estimator with the EM-based maximum likelihood estimator. The full code to the different parts of the data analysis is made available on our github repository.

\subsection{Partition Procedure} \label{app:apl:partition}

As a preliminary step in the data analysis, we partition the image data into smaller components. As explained in \Cref{sec:applications}, partitioning the image data into several is beneficial for computational purposes, as the number of location parameters is reduced from $k$ to the number of atoms for each respective cluster.

Our specific partition method is described as follows, though we stress that alternative image segmentation methods would likely also be suitable. We first compute a preliminary Voronoi partition consisting  $\tilde k$ Voronoi partitions where the centers are determined via some mode selection procedure (see Algorithm~\ref{alg:mode-like-selection} and \Cref{fig:mode-like-selection}) with the uniform kernel on $[-180\textup{nm},180\textup{nm}]^2$. 
 Next, we construct a graph whose vertices are the selected modes, and two modes are connected if their distance is less than a threshold $\delta>0$. By identifying the connected components of this graph, we obtain our partition $P_1, \dots, P_s$ as the collection of  unions of Voronoi cells associated with the modes in each connected component.  
The number of modes $\tilde k$ and the threshold $\delta$ are chosen to ensure that the residuals are sufficiently small and the resulting samples are well separated.

In \Cref{fig:mode-like-selection} we illustrate the number of determined modes for different choices for $\tilde k$ and the resulting residuals. In particular, for $\tilde k = 35$ we see that the residuals still contain certain high-intensity spots, whereas for $\tilde k = 39$ and $\tilde k = 45$ the residuals are mostly homogenous with at most $10$ to $15$ photon counts per pixel. However, for $\tilde k = 45$ more Voronoi cells are grouped together, which is not desirable from estimation point of view. Therefore, for the computation in \Cref{sec:applications} we considered $\tilde k = 39$ and $\delta = 270\textup{nm}$, which assigns distinct, non-overlapping modes to different partition cells and overall results in $s=37$ partition elements (see \Cref{fig:origami:delta-em}, top middle). As a general rule of thumb we recommend choosing the partition as fine as possible as long as it does not irregularly cut one kernel component into multiple parts.

\begin{algorithm}[b!]
	\caption{Mode Selection}\label{alg:mode-like-selection}
	\begin{algorithmic}
		\Require Data $X$, kernel $K$, number of modes $\tilde k$, bins $B_1, \dots, B_m$ with centers $\gamma_1, \dots, \gamma_m$
		\Ensure Estimated mode locations $\tilde{\theta}_1,\dots,\tilde{\theta}_k$ and residuals $X^{(\tilde k)}$
		\State $X^{(0)} \gets X$
		\For{$i = 1, \dots, \tilde k$}
			\State Identify bin $j$ with maximum count: $j \gets \argmax_{\overline{j}=1,\dots,m} X_{\overline{j}}^{(i-1)}$
			\State Set mode: $\theta_i \gets \gamma_j$
			\State Compute blurred mode: $X^* \gets \big(K \star \delta_{\theta_i}(B_1),\, \dots,\, K \star \delta_{\theta_i}(B_m)\big)$
			\State Update residuals: $X^{(i)} \gets \max\!\big(X^{(i-1)} - \frac{X_j}{\max_{\tilde{j}=1,\dots,m} X^*_{\tilde{j}}}\, X^*,\,0\big)$
		\EndFor
	\end{algorithmic}
\end{algorithm}

\begin{figure}[h!]
    \centering
    \includegraphics[width=\textwidth, trim={0 0.7cm 0 0}, clip]{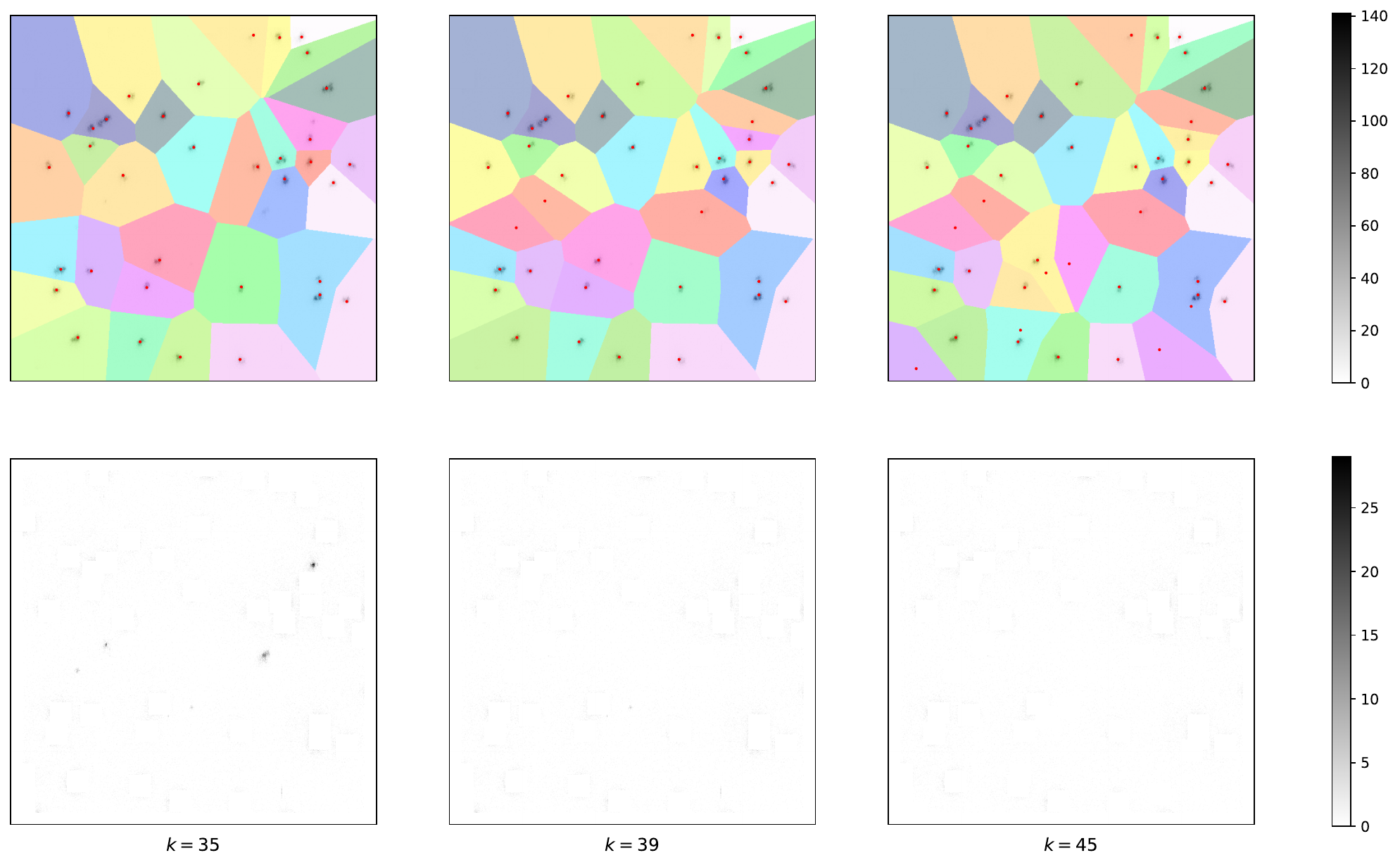}
    \caption{
        Top: Partition of STED image data of the DNA origami sample by \cite{proksch2018multiscale} according to method in \Cref{app:apl:partition} with the uniform kernel on $[-180\textup{nm},180\textup{nm}]^2$, $\delta = 270\textup{nm}$         and $\tilde k = 35, 39, 45$ (from left to right) of modes computed using \Cref{alg:mode-like-selection} and depicted as red points. Bottom: Residuals output from mode selection algorithm for different choices of $\tilde k$.}
    \label{fig:mode-like-selection}
\end{figure}

\subsection{Denoising and Estimation procedure}\label{app:apl:estimation}
For a prescribed total number of components $k$, 
we assign to each partition cell $P\in \{P_1, \dots, P_s\}$ a number of components given by $k_{P} = \text{evenRound}(r\cdot k),$
where $r$ is the ratio of the mass in the cell $P$ to the total mass.  Here we utilize our knowledge about the DNA origami structure,  as it involves fluorophores along two parallel lines.  This also has the regularizing effect that partition cells which were created due to the existence of a mode but which admit too little mass may not be assigned a positive number of atoms. At the same time, this approach prohibits the use of too fine partitions. Recall that the total number of components assigned to all cells may differ from $k$.

To estimate the $k_P$ atoms assigned to partition cell $P$, we first subtract the residuals obtained from the image data, effectively denoising the data. Additionally, we remove all outer zero-valued columns and rows, so that we reduce the image dimensions, which leads due to the smaller domain to a reduction of the computation time and (empirically) improves our estimations. For this reduced partition, we compute the complex method of moments estimator. We then compute for each partition cell the MLE by applying the EM algorithm  (details in \Cref{app:EM}), initialized with the complexmethod of moments estimator, and the resulting estimates are finally combined to yield the overall reconstruction of atoms.

In \Cref{fig:origami_data_2}$b)-c)$ in \Cref{sec:applications} we consider $k = 70$ atoms. In \Cref{fig:origami:em-estimations} we also showcase the output of the EM algorithm for $k = 40, 50, 60, 70$ atoms and note that for every $k$ each local cluster admitting roughly $2$ points, also for $k= 70$ the clustered samples are better discerned. This models the situation where each DNA origami sample (recall \Cref{fig:origami_data_2}$a)$) exhibits a fluorophore group at the opposing ends. While this situation is conceptually appealing and justified by the underlying structure of the data set, we do find that even for different choices of $k$ the recovered point configurations appear fairly similar. This highlights some remarkable robustness properties of the EM-based MLE.

\subsection{Performance and Comparison of Related Estimation Procedures}\label{app:apl:robustness}

In the following we discuss the relevance of the choice of the partition and the denoising step in applying our method. The top row of \Cref{fig:origami:delta-em} shows the STED image data of \cite{proksch2018multiscale} segmented according to the procedure in \Cref{app:apl:partition} into increasingly coarse partitions. The second (resp.\ fourth) row shows (resp.\ a zoomed-in region of) the output of the EM algorithm for $k = 70$ without employing the denoising step from \Cref{app:apl:estimation}, whereas the third (resp.\ fifth) row shows the output of output with the denoising step. The zoomed-in regions (fourth and fifth row) reveal that coarser partitions lead to worse estimates. Moreover, we observe that the difference between denoised and original data is substantial for fine partitions and even more pronounced for coarser partitions. Also, the estimations on the denoised data are rather stable across different choices of $\delta$.

In \Cref{fig:origami:delta-estimations}, we see  the same zoomed-in region as in \Cref{fig:origami:delta-em}. The figure displays the output of the method of moment estimator and EM algorithm for denoised and original data, for different choices of partitions and different numbers of $k$. The analysis showcases that the method of moments becomes increasingly unstable as the partition becomes coarser. In all settings, the method of moment estimator is outperformed by the MLE computed with the EM algorithm.  We also observe that the denoising step leads to method of moment estimates which are closer to the actual signal points. Since the method of moments estimator serves as the initializer for the EM algorithm, this potentially also explains the better performance of the MLE for denoised data over original data. 

\newpage
\begin{figure}[t!]
    \centering
    \includegraphics[width=\textwidth, trim={0 0 0 0}, clip]{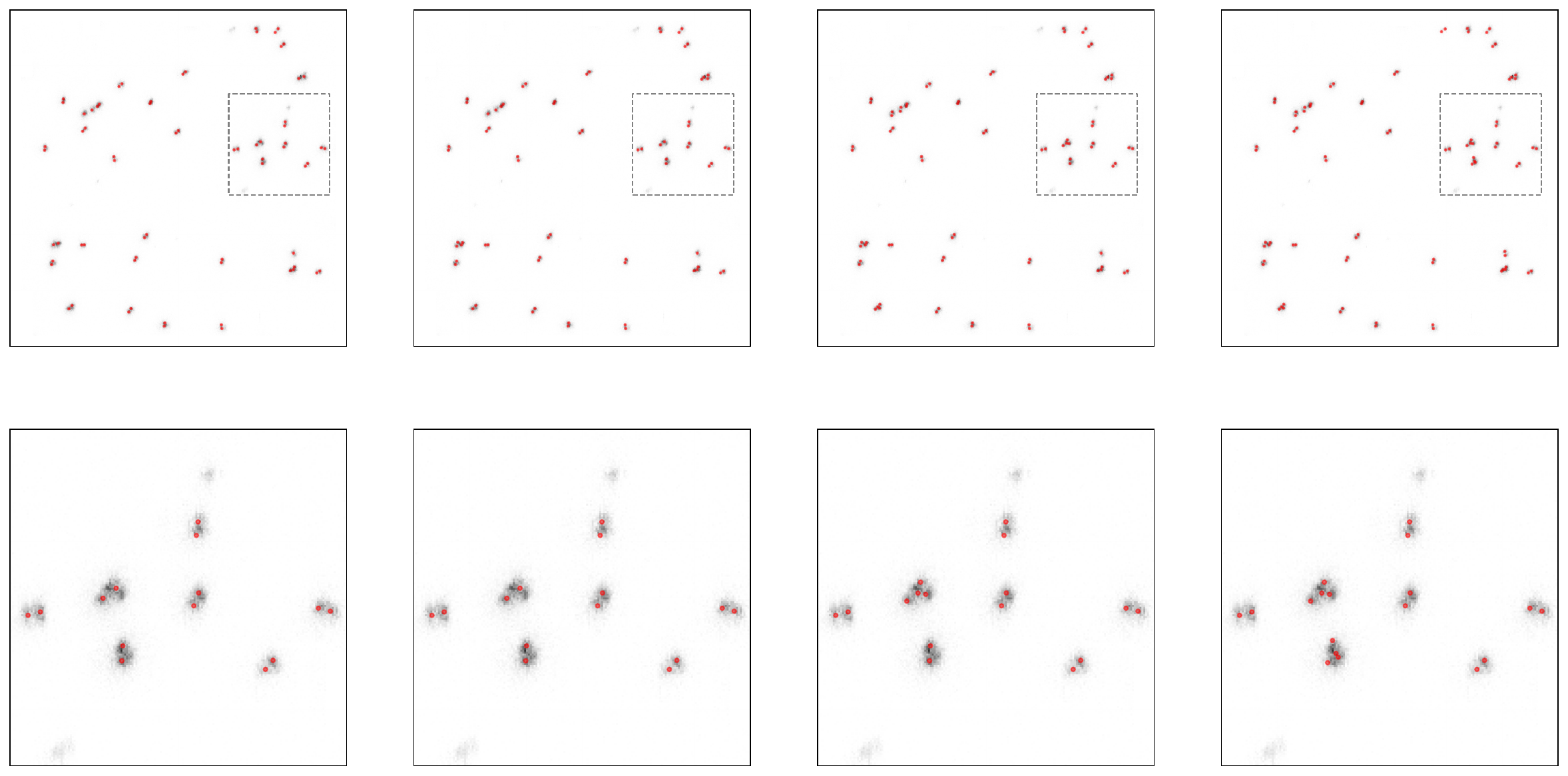}
    \caption{MLE computed using EM algorithm (red points) for STED image data by \cite{proksch2018multiscale} using partition procedures from \Cref{app:apl:partition} for $\tilde k=39$ and $\delta = 270\textup{nm}$  with different $k= 40, 50, 60, 70$ atoms (from left to right). Top: Full image; Bottom: Zoomed-in region.}
    \label{fig:origami:em-estimations}
\end{figure}

\begin{figure}[t!]
    \centering
    \includegraphics[width=0.9\textwidth]{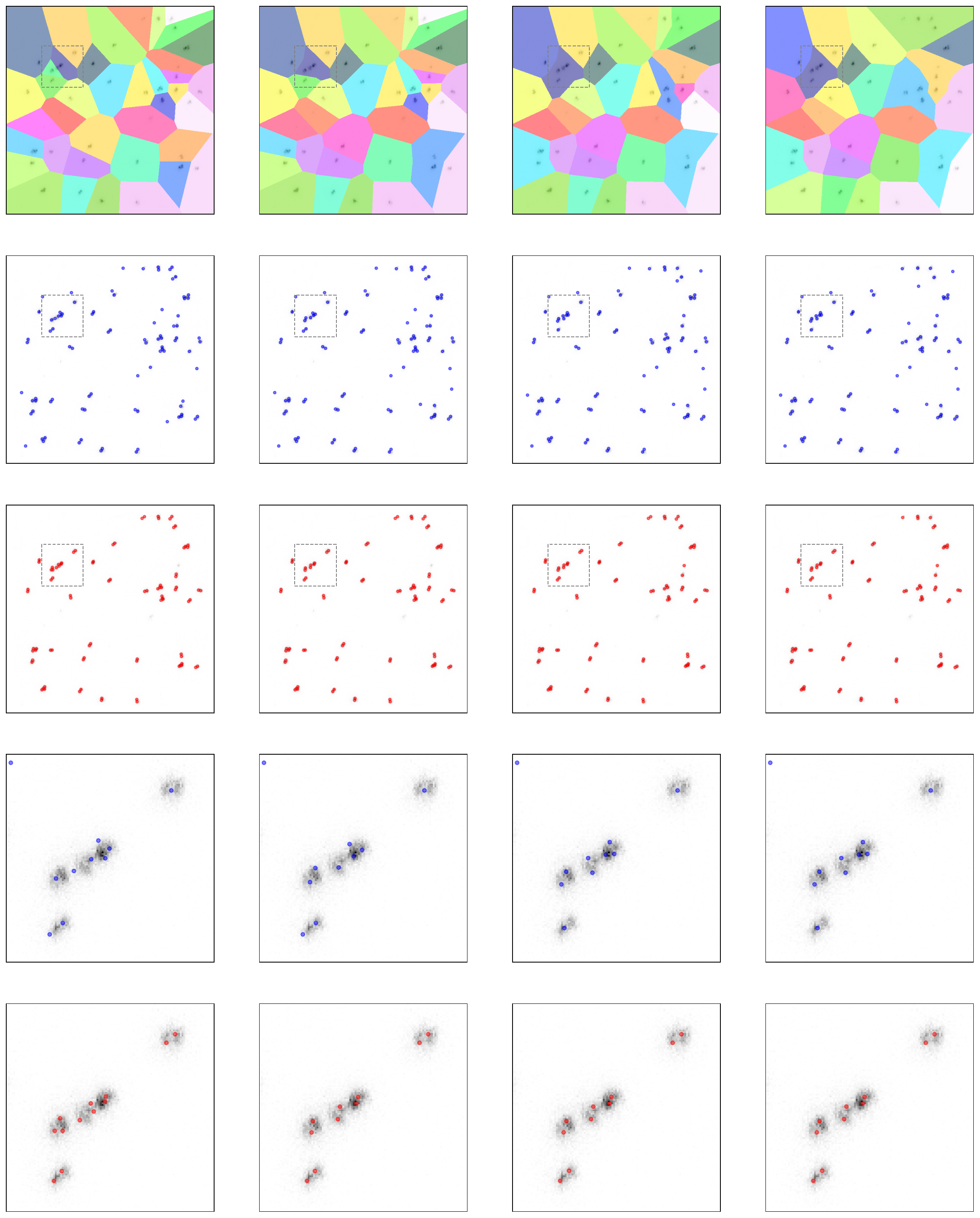}
    \caption{Top row: Partition of STED image data by \cite{proksch2018multiscale} according to method in \Cref{app:apl:partition} for $\tilde k = 120$ and $\delta=180\textup{nm}$, $270\textup{nm}$, $360\textup{nm}$, $450\textup{nm}$ (from left to right).
    Second and third row: Output of EM algorithm procedure from \Cref{app:apl:estimation} for $k = 70$ atoms, initialized for each partition cell with the method of moment estimator without denoising (blue points, second row) and with denoising (red points, third row). Fourth and fifth row: Zoomed-in region of second and third row respectively.}
    \label{fig:origami:delta-em}
\end{figure}

\begin{figure}[t!]
    \centering 
    \includegraphics[width=0.9\textwidth]{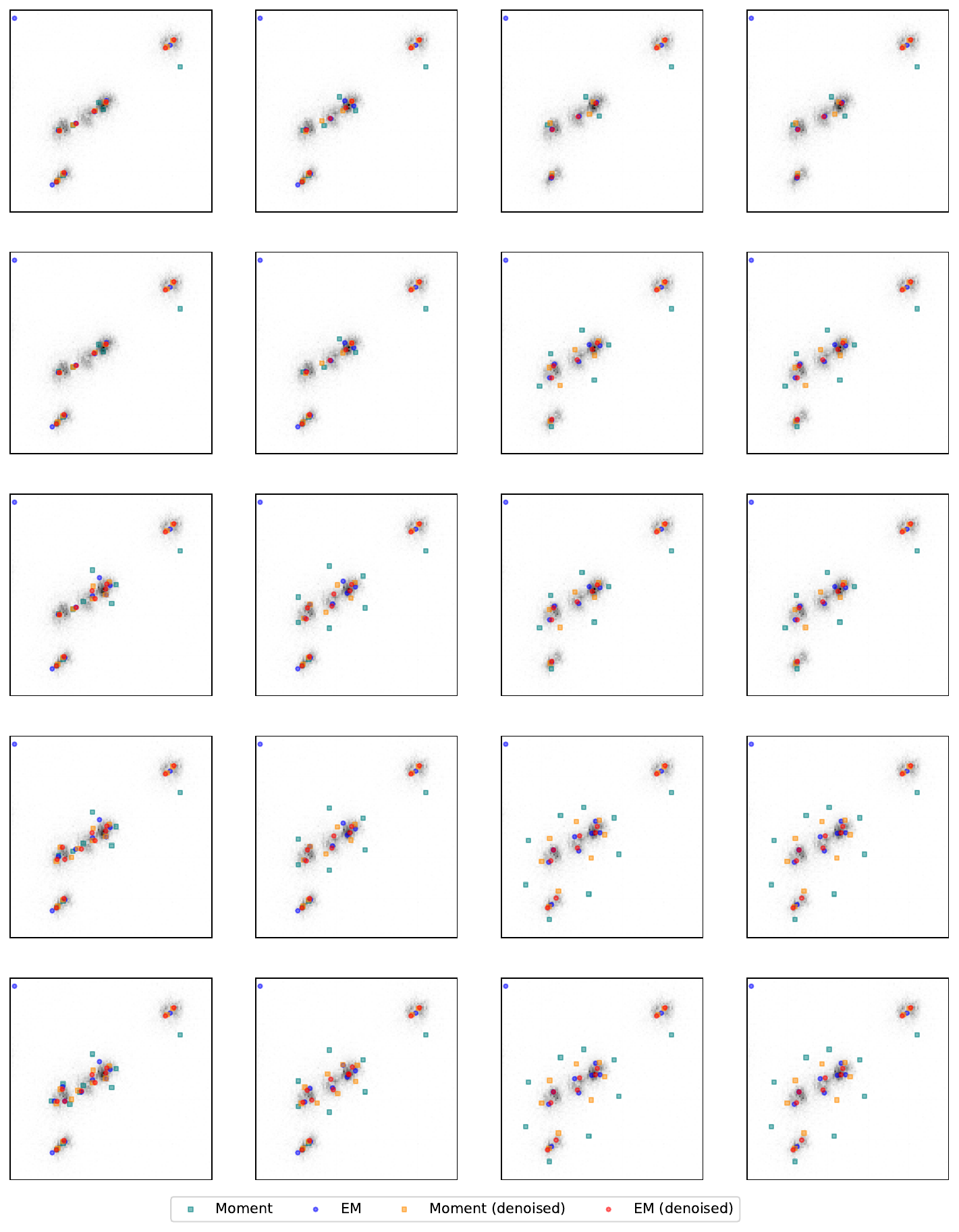}
    \caption{Comparison of method of moment estimator and EM algorithm estimator computed with and without denoising step for zoomed-in STED image region (same as in \Cref{fig:origami:delta-em}), 
     computed for different partitions according to  \Cref{app:apl:partition} for $\tilde k = 39$ and $\delta=180\textup{nm}$, $270\textup{nm}$, $360\textup{nm}$, $450\textup{nm}$ (from left to right) and for $k = 40, 50, 60, 70, 80$ (from top to bottom) atoms.}
    \label{fig:origami:delta-estimations}
\end{figure}

\newpage

\addtocontents{toc}{\protect\setcounter{tocdepth}{2}}

\clearpage
\addcontentsline{toc}{section}{References}

\end{document}